\documentclass[a4paper]{amsart}

\usepackage[T1]{fontenc}
\usepackage[margin=3cm]{geometry}
\usepackage{amsmath, amsfonts, amssymb, mathtools}
\usepackage{wasysym}  % better bowtie
\usepackage[only,Mapsto]{stmaryrd} % \Mapsto

\usepackage{graphicx}
\usepackage{enumitem}
\usepackage{young}
\usepackage{multirow}
\usepackage[svgnames]{xcolor}
\usepackage{pdflscape}  % Allow switch to landscape
\usepackage{multido, xargs, colortbl, multicol, multirow}
\usepackage{calc, ifthen, xstring, blkarray, stackrel}
\usepackage{hvfloat}
\usepackage[normalem]{ulem}
\usepackage{dutchcal}
\usepackage{nicematrix}

\usepackage[bb=stix]{mathalpha}

\usepackage{caption}
\usepackage[backend=biber, style=alphabetic, sorting=nty]{biblatex}
\usepackage{amsthm,thmtools}
\usepackage{url}
\usepackage[hypertexnames=false]{hyperref}
\usepackage{hypcap}
\hypersetup{colorlinks=true, citecolor=blue, linkcolor=blue}
\usepackage[noabbrev,capitalise]{cleveref}
\usepackage{pgfkeys}
\usepackage{tikz}
\usetikzlibrary{trees, shapes, arrows, arrows.meta, matrix, calc, graphs,
  plotmarks, backgrounds}
\usetikzlibrary{decorations,
  decorations.markings,decorations.pathmorphing,decorations.pathreplacing}
\pgfdeclaredecoration{cappedcurveto}{initial}{%
  \state{initial}[width=\pgfdecoratedinputsegmentlength/100]
  {\pgfpathlineto{\pgfpointorigin} }%
  \state{final}{}%
}%
\tikzset{
  draw left/.style={
    decorate, decoration={cappedcurveto, raise={.5*\pgflinewidth}}
  },
  draw right/.style={
    decorate, decoration={cappedcurveto, raise={-.5*\pgflinewidth}}
  }
}

\usetikzlibrary{external}
\usepackage{pdftexcmds} % To detect shell escape
\makeatletter % if not in a .sty file
\ifnum\pdf@shellescape=1
\fi
\makeatother

\usepackage{todonotes}
\setuptodonotes{inline}
\makeatletter
\renewcommand{\todo}[2][]{\tikzexternaldisable\@todo[#1]{#2}\tikzexternalenable}
\define@key{todonotes}{AA}[]{%
  \setkeys{todonotes}{author=AA,color=orange}}%
\define@key{todonotes}{FH}[]{%
  \setkeys{todonotes}{author=FH,color=green!70}}%
\define@key{todonotes}{JS}[]{%
  \setkeys{todonotes}{author=JS,color=cyan}}%
\makeatother

\makeatletter
\def\l@section{\@tocline{1}{7pt}{0pc}{}{}}
\makeatother
\let\oldtocpart=\tocpart
\renewcommand{\tocpart}[2]{\bf\large\oldtocpart{#1}{#2}}
\let\oldtocsection=\tocsection
\renewcommand{\tocsection}[2]{\bf\oldtocsection{#1}{#2}}
\let\oldtocsubsubsection=\tocsubsubsection
\renewcommand{\tocsubsubsection}[2]{\quad\oldtocsubsubsection{#1}{#2}}

\crefname{enumi}{Condition}{Conditions}

\title[Diagrammatic Okada Algebra]{%
  Diagrammatic Okada monoid \\ and cellularity of the Okada algebra}
\thanks{}

\author[F.~Hivert]{Florent Hivert}
\address[F.~Hivert]{Université Paris-Saclay, CNRS, Laboratoire Interdisciplinaire des Sciences du Numérique, 91405, Orsay, France}
\email{florent.hivert@lisn.fr}
\urladdr{http://www.lisn.fr/~hivert/}

\author[J.~Scott]{Jeanne Scott}
\address[J.~Scott]{Department of Mathematics, University of Minnesota, 127 Vincent Hall, 206 Church Street SE, Minneapolis, MN 55455, United States}
\email{scot1526@umn.edu}

\newtheorem{theorem}{Theorem}[section]
\newtheorem{corollary}[theorem]{Corollary}
\newtheorem{proposition}[theorem]{Proposition}
\newtheorem{lemma}[theorem]{Lemma}
\newtheorem{definition}[theorem]{Definition}
\newtheorem{conjecture}[theorem]{Conjecture}

\theoremstyle{definition}
\newtheorem{example}[theorem]{Example}
\newtheorem{remark}[theorem]{Remark}
\newtheorem{question}[theorem]{Question}

\definecolor{green}{RGB}{57,181,74} % green color
\newcommand{\white}{\color{white}} % white command
\newcommand{\blue}{\color{blue}} % blue command
\newcommand{\green}{\color{green}} % green command
\newcommand{\red}{\color{red}} % red command
\newcommand{\cyan}{\color{cyan}} % cyan command
\newcommand{\N}{\mathbb{N}} % naturals
\newcommand{\K}{\mathbb{K}} % field
\newcommand{\C}{\mathbb{C}} % set of summands

\newcommand{\isom}{\simeq}  % isomorphic
\newcommand{\set}[2]{\left\{ #1 \;\middle|\; #2 \right\}} % set notation
\newcommand{\biggset}[2]{\bigg\{ #1 \;\bigg|\; #2 \bigg\}} % bigg set notation
\newcommand{\eqdef}{\coloneq} % :=
\newcommand{\card}{\operatorname{\#}}            % Cardinal d'un ensemble
\DeclareMathOperator{\End}{End}  % endomorphisms

\newcommand{\tMonoid}{M}
\newcommand{\leJJ}{\leq_\mathcal{J}}
\newcommand{\leLL}{\leq_\mathcal{L}}
\newcommand{\leRR}{\leq_\mathcal{R}}
\newcommand{\leHH}{\leq_\mathcal{H}}
\newcommand{\leKK}{\leq_\mathcal{K}}
\newcommand{\JJ}{\mathrel{\mathcal{J}}}
\newcommand{\LL}{\mathrel{\mathcal{L}}}
\newcommand{\RR}{\mathrel{\mathcal{R}}}
\newcommand{\HH}{\mathrel{\mathcal{H}}}
\newcommand{\KK}{\mathrel{\mathcal{K}}}

\newcommand{\ssp}{\hspace{1pt}}
\newcommand{\ie}{\textit{i.e.}~} % id est
\newcommand{\ordinal}{\textsuperscript{th}} % th for ordinals
\newcommand{\ordinalst}{\textsuperscript{st}} % st for ordinals
\newcommand{\qandq}{\quad\text{and}\quad} % space around and
\newcommand{\defn}[1]{\emph{\textsf{\color{blue} #1}}} % emphasis of a definition

\newcommand{\greed}[1]{{\Delta^\mathrm{Greed}(#1)}}  % greedy filling
\newcommand{\perm}[1]{\sigma(#1)} % permutation associated to a diagram
\newcommand{\rewto}{\stackrel*{\mapsto}}
\newcommand{\weight}{\operatorname{wt}}

\newcommandx{\setN}[1][1=N]{[#1]}
\newcommandx{\setNN}[1][1=N]{\overline{[#1]}}

\newcommand{\unpair}{\operatorname{unpair}}
\newcommand{\lpair}{\operatorname{lpair}}
\newcommand{\height}{\operatorname{height}}
\newcommand{\loops}{\operatorname{Loops}}
\newcommand{\walks}{\operatorname{Walks}}

\newcommand{\bl}{|[text=white,fill=black]|}%
\newcommand{\gr}{|[fill=lightgray]|}%
\newcommand{\nn}{|[draw=none]|}%
\tikzset{diamond/node/.style={
    scale=0.55, minimum size=0.9cm, inner sep=0mm,
    draw, regular polygon, regular polygon sides=4,
    yscale=-1, rotate=45, shape border rotate=45
  }}%
\tikzset{diamond/matrix/.style={matrix of nodes,
  transform canvas={yscale=-1, rotate=45},
  nodes={diamond/node},row sep=-\pgflinewidth,column sep=-\pgflinewidth}}
\tikzset{diamond/.style={
    baseline={([yshift=-.5ex]current bounding box.center)},
  }}
\pgfkeys{/tikz/diamond/bound/.code 2 args={
  \useasboundingbox[scale=0.25]
  ($#1*(-1,1) + {1-#2}*(1,0)$) rectangle (${#1+#2-1}*(1,0) + (0,-1)$);
}}
\tikzset{debug/.style={
    show background rectangle,    % to debug margins around diagrams
  }}
\newcommand{\straightLoop}[3]{
  \node[diamond#1]{};
  \begin{scope}
    \clip (-5.0pt,-5.0pt) rectangle (5pt,5pt);
    \draw[ultra thick, #2] (-6pt,1pt) -- (1pt,-6pt);
    \draw[ultra thick, #3] (6pt,-1pt) -- (-1pt,6pt);
  \end{scope}
}
\newcommand{\stC}[2]{\straightLoop{}{#1}{#2}}
\newcommand{\st}{\stC{black}{black}}
\newcommand{\stbC}[1]{\straightLoop{, draw=none}{white}{#1}}
\newcommand{\stb}{\stbC{black}}
\newcommand{\sttC}[1]{\straightLoop{, draw=none}{#1}{white}}
\newcommand{\stt}{\sttC{black}}
\newcommand{\tuC}[2]{
  \node[diamond]{};
  \begin{scope}
    \clip (-5.0pt,-5.0pt) rectangle (5pt,5pt);
    \draw[line width=1.8pt, #1] (-3.6pt,3.6pt) circle (3.85pt);
    \draw[line width=1.8pt, #2] (3.6pt,-3.6pt) circle (3.85pt);
  \end{scope}
}
\newcommand{\tu}{\tuC{black}{black}}
\newcommand{\TLlab}[1]{\nn\raisebox{-8pt}[0pt][0pt]{$#1$}}%

\newcommand{\ov}[1]{\overline{#1}}

\newcommand{\arctip}{\rule[0.1ex]{0.5pt}{1ex}}
\newcommand{\arcdrawing}[1]{\arctip\rule[0.55ex]{#1}{0.5pt}}%

\newlength{\arcwidth}
\newlength{\arcwidthB}
\newsavebox{\arclabel}
\newsavebox{\arcdiag}
\newcommand{\arc}[4][]{%
  \savebox{\arclabel}{$\scriptstyle#3$}%
  \savebox{\arcdiag}{$\scriptstyle#1$}%
  \settowidth{\arcwidth}{\usebox{\arclabel}}%
  \settowidth{\arcwidthB}{\usebox{\arcdiag}}%
  \ifthenelse{\lengthtest{\arcwidthB>\arcwidth}}{%
    \setlength{\arcwidth}{\arcwidthB}}{}%
  \addtolength{\arcwidth}{1em}%
  \ifthenelse{\equal{#4}{}}{\def\endtip{}}{\def\endtip{\arctip}}%
  \ifthenelse{\equal{#1}{}}{%
    #2\mathrel{\stackrel{\usebox{\arclabel}}{\arcdrawing{\arcwidth}\endtip}}#4%
  }{#2\mathrel{\stackrel[{\usebox{\arcdiag}}]{\usebox{\arclabel}}%
      {\arcdrawing{\arcwidth}\endtip}}#4%
  }%
}
\newcommand{\harc}[3][{}]{\arc[#1]{#2}{#3}{}}
\newcommand{\arcp}[4][{}]{(\arc[#1]{#2}{#3}{#4})}
\newcommand{\barc}{\mathrel{\arctip\rule[0.55ex]{2mm}{0.5pt}\!\raisebox{0.1pt}{$\bullet$}\!\rule[0.55ex]{2mm}{0.5pt}\arctip}}

\newcommand{\glue}[2]{%
  #1\,%
  \mathchoice%
    {\raisebox{-.2ex}{$\boldsymbol{\Join}$}} % Display style
    {\raisebox{-.2ex}{$\boldsymbol{\Join}$}} % Text style
    {\raisebox{-.25ex}{$\boldsymbol{\Join}$}} % Script style (indices)
    {\raisebox{-.25ex}{$\boldsymbol{\Join}$}} % Script-script style (indices of indices)
  \,#2}  % or \talloblong

\newcommand{\leftd}[1]{\langle #1|}
\newcommand{\rightd}[1]{\!\left|#1\right\rangle\!}
\newcommand{\linkd}[2]{\rightd{#1, #2}}
\newcommand{\topd}[1]{\linkd{\setN}{#1}}

\newcommand{\freeld}[1]{\linkd{#1}{#1}}
\newcommand{\freerd}[1]{\linkd{#1}{#1}}

\newcommand{\isdomeq}{\preceq}  %{\trianglelefteqslant}
\newcommand{\isdom}{\prec}  %{\vartriangleleft}
\newcommand{\YoungCover}{\subset\!\!\!\!\boldsymbol{\cdot}\ }

\newcommandx{\SG}[1][1=N]{\mathfrak{S}_{#1}} % symmetric group
\newcommandx{\Okada}[1][1=N]{\mathbf{O}_{#1}}
\newcommandx{\OkArc}[1][1=N]{{\mathbf{Arc}}_{#1}}
\newcommandx{\TL}[1][1=N]{\mathbf{TL}_{#1}}
\newcommandx{\Blob}[1][1=N]{\mathbf{Blob}_{#1}}
\newcommandx{\Jones}[1][1=N]{\mathbf{J}_{#1}}
\newcommandx{\FreeMonoid}[1][1=N]{\mathbf{F}_{#1}}
\newcommand{\E}[1]{\mathsf{E}_{#1}}   % Okada algebra
\newcommand{\e}[1]{\mathsf{e}_{#1}}   % Okada monoid
\newcommand{\A}[1]{\mathsf{A}_{#1}}   % Arc monoid (temporarily)
\newcommand{\G}[1]{\mathsf{g}_{#1}}   % generic generators
\newcommand{\f}[1]{\mathsf{f}_{#1}}   % free monoid
\newcommand{\tl}[1]{\mathsf{t}_{#1}}  % TL monoid
\newcommand{\blob}{\mathsf{b}}        % blob monoid

\newcommand{\ind}[1]{\mathbf{#1}}
\newcommand{\id}{\mathsf{id}}      % identity map / functor...
\newcommand{\code}{\mathsf{code}}  % code of a permutation
\newcommandx{\lab}[1][1=]{\operatorname{lab}_{#1}}

\DeclareMathOperator{\Des}{Des}          % descents
\DeclareMathOperator{\LDes}{LDes}        % largest descents
\DeclareMathOperator{\lexmin}{LexMin}
\DeclareMathOperator{\norf}{Can}
\DeclareMathOperator{\length}{Length}

\newcommand{\Sys}{\mathsf{S}}
\newcommand{\Comm}{\mathsf{Comm}}
\newcommand{\OSys}{\mathsf{O}}

\newcommand{\evenD}{\mathbb{EW}}
\newcommand{\oddD}{\mathbb{OW}}

\newcommand{\YFS}{\mathbb{YFS}}
\newcommandx{\YFSn}[1][1=N]{\YFS_{#1}}
\newcommand{\YF}{\mathbb{YF}}
\newcommandx{\YFn}[1][1=N]{\YF_{#1}}
\DeclareMathOperator{\Grow}{Gr}
\DeclareMathOperator{\Shape}{Sh}
\DeclareMathOperator{\PropLab}{PLab}   % Propagating labels
\DeclareMathOperator{\PropInd}{PInd}   % Propagating indexes
\DeclareMathOperator{\Chain}{Chain}
\DeclareMathOperator{\Diag}{Diag}
\DeclareMathOperator{\RS}{RS}

\DeclareMathOperator{\FibSet}{Set}
\DeclareMathOperator{\FibWord}{Fib}

\begin{document}
\begin{abstract}
  It is well known that the Young lattice is the Bratelli diagram of the
  symmetric groups, expressing how irreducible representations restrict from
  $\SG[N]$ to $\SG[N-1]$. In 1975, Stanley discovered a similar lattice called
  the Young-Fibonacci lattice which was identified as the Bratelli diagram of a
  family of algebras $\{\Okada(X,Y)\}_{N \geq 0}$ by Okada in 1994.

  In this paper, we first realize the Okada algebra $\Okada(X,Y)$ and
  the associated monoid $\Okada$ using a labelled version of
  non-crossing arc-diagrams appearing in the description of the
  Temperley-Lieb algebra and Jones monoid.  We establish, for general
  parameters $(X,Y)$, that the dimension of the Okada algebra
  $\Okada(X,Y)$ is $N!$, noting that Okada proved this result only in
  the semisimple case. We interpret a natural bijection between
  permutations and labelled arc-diagrams as an incarnation of Fomin's
  version of the Robinson-Schensted correspondence associated to the
  Young-Fibonacci lattice. The arc-diagram formalism allow us to
  probe the structure of the Okada monoid and algebra. In particular
  we prove that the Okada monoid is a regular, aperiodic $*$-monoid and
  we describe its Green relations and order.

  These results allow us to construct
  a cellular basis of the Okada algebra and to show that
  $\{\Okada(X,Y)\}_{N \geq 0}$ forms a coherent tower of cellular
  algebras in the sense of Goodman and Graber. We present
  some conjectures expressing the Gram determinant of the 
  invariant bilinear form attached to each cell module
  in terms of Okada's clone Schur functions. We conclude the paper by
  presenting two follow-up, ongoing projects along with a series of questions
  pushing further the analogy between the symmetric groups and the
  Okada algebras.
\end{abstract}

\maketitle

\tableofcontents

\listoffigures

\section{Introduction}

\subsection{The Young lattice and the symmetric groups}
The lattice of integer partitions $\mathbb{Y}$, otherwise known as the
\defn{Young lattice}, is an object of central importance, appearing in
algebraic combinatorics, representation theory, statistical mechanics,
and the study of random matrix models. Our interest in the Young
lattice $\mathbb{Y}$ comes from the role it plays in the
representation theory of the symmetric groups.  Up to isomorphism, the
complex, irreducible representations $V^\lambda$ of the symmetric
group $\SG$ are classified by the set $\mathbb{Y}_N$ of partitions
$\lambda$ of rank $N$.  Moreover the Young lattice $\mathbb{Y}$ is the
{\it Bratteli diagram} of the tower of symmetric groups
$\SG[0] \subset \SG[1] \subset \SG[2] \subset \SG[3] \subset \cdots $
and thus the covering relations of $\mathbb{Y}$ express how
irreducible representations of $\SG$ decompose when {\it restricted}
to the the subgroup $\SG[N-1]$ for all $N \geq 1$. Specifically
\begin{equation}
\label{young-lattice-restriction}
\mathrm{\bf Res}_{\ssp \SG[N-1]}^{\ssp \SG} \big( V^\lambda  \big) \ssp \cong \ssp \bigoplus_{\mu \ssp \subset \!\!\! \boldsymbol{\cdot} \, \ssp \lambda } V^\mu
\end{equation}
for $\lambda \in \mathbb{Y}_N$, where $\mu \subset \!\!\!\!\! \boldsymbol{\cdot} \ \lambda$ indicates $\mu \subset \lambda$
is a covering relation.
General results of Okounkov-Vershik~\cite{Okounkov1996}
tell us that for $\lambda \in \mathbb{Y}_N$ the irreducible
module~$V^\lambda$ admits a {\it Gelfand-Tsetlin}
basis which is indexed by saturated chains $\boldsymbol{\underline{\lambda}} = \lambda^{(0)} \subset \lambda^{(1)} \subset \cdots \subset \lambda^{(N)}$
in the $\mathbb{Y}$-lattice starting at $\lambda^{(0)} = \varnothing$ and
ending at $\lambda^{(N)} = \lambda$.
Beyond this the restriction and induction functors obey a {\it categorified} version of the Weyl commutation identity:
\begin{equation}
\label{young-lattice-induction-restriction}
\mathrm{\bf Res}_{\ssp \SG[N]}^{\ssp \SG[N+1]}  \mathrm{\bf Ind}_{\ssp \SG[N]}^{\ssp \SG[N+1]} \big( V^\lambda  \big) \,
\cong \, V^{\lambda} \, \boldsymbol{\oplus} \,
\mathrm{\bf Ind}_{\ssp \SG[N-1]}^{\ssp \SG}  \mathrm{\bf Res}_{\ssp \SG[N-1]}^{\ssp \SG[N]} \big( V^\lambda  \big)\,.
\end{equation}

\subsection{Differential Posets}
\Cref{young-lattice-induction-restriction} is an algebraic incarnation of the {\it differential} property of
the Young lattice $\mathbb{Y}$, where the {\it up} and {\it down} operators
$\mathrm{\bf U}$ and $\mathrm{\bf D}$ respectively implement the action of the induction and restriction
functors $\mathrm{\bf Ind}$ and $\mathrm{\bf Res}$.
Recall that a ranked, locally finite poset $(\mathbb{P}, \unlhd)$ equipped with
a unique, minimal element $\varnothing$ is a said to be
\defn{differential} if the Weyl commutation identity $\mathrm{\bf D} \mathrm{\bf U} - \mathrm{\bf U} \mathrm{\bf D}= \mathrm{\bf Id}_{\ssp \mathbb{P}}$
holds (\cite{Stanley1988}). Here $\mathrm{\bf U}$ and $\mathrm{\bf D}$ are, respectively, the \defn{up} and \defn{down} operators
defined on the vector space $\mathbb{C}[\mathbb{P}]$ of complex-valued functions on $\mathbb{P}$ by
\begin{equation}
\label{def:up+down-operators}
\mathrm{\bf U} \delta_v \ssp \eqdef \ssp \sum_{\stackrel{\scriptstyle w \ssp \rhd \ssp v}{|w| \ssp = \ssp |v| + 1}} \delta_w
\qquad \text{and} \qquad
\mathrm{\bf D} \delta_v \ssp \eqdef \ssp \sum_{\stackrel{\scriptstyle u \ssp \lhd \ssp v}{|u| \ssp = \ssp |v| - 1}} \delta_u
\end{equation}
where $|w|$ and $\delta_w$ denote the rank and indicator function of an element $w \in \mathbb{P}$.
An immediate consequence of \Cref{def:up+down-operators} is that
\begin{equation}
\label{eq:differential-poset-identity}
\mathrm{\bf D}^N \mathrm{\bf U}^N \varnothing = N! \ssp \varnothing.
\end{equation}
The
term {\it differential} winks at the fact that differentiation $\partial_x$ and multiplication by $x$
satisfy the Weyl commutation identity, together with
the observation that \Cref{eq:differential-poset-identity}
can be understood as a poset generalization of the identity $\partial_x \ssp x^N = N!$.
\Cref{eq:differential-poset-identity} is suggestive for another reason, as it indicates a
possible bijection between pairs of saturated chains sharing a common endpoint in $\mathbb{P}_N$
(elements of $w \in \mathbb{P}$ of rank $|w| = N$) and permutations $\sigma \in \SG$.
In the case of the Young lattice $\mathbb{Y}$ this bijection is the celebrated
{\it Robinson-Schensted} correspondence which, from an algebraic perspective,
is just the Wedderburn decomposition of the group algebra $\mathbb{C}[\SG]$.

Indeed, the theory of differential posets was developed by
R. Stanley~\cite{Stanley1988} as a framework for generalizing the
Robinson-Schensted correspondence beyond the combinatorics of the Young
lattice $\mathbb{Y}$ and integer partitions.  A similar undertaking was made by
S. Fomin in his work on dual graded graphs and growth processes, where the later
technique was used to construct an explicit RS-correspondence for
Stanley's \defn{Young-Fibonacci} lattice \defn{$\YF$}~\cite{Fomin1995,Stanley1988}.
Both $\mathbb{Y}$ and $\YF$ are differential and they were
shown to be the only lattices having this property
in the thesis work of P. Byrnes~\cite{byrnes2012structural}.
Fomin's approach involves a Fibonacci version of standard tableaux; a notion
later examined independently by T. Roby, K. Killpatrick, and J. Nzeutchap
(whose formulation by-passes the growth construction altogether),
see~\cite{Nzeutchap2009} and the references therein.

\subsection{The Young-Fibonacci lattice and the Okada algebras}
What is the Young-Fibonacci lattice? Stanley's original
construction~\cite{Stanley1975} involves defining a partial order on
the set of \defn{Fibonacci words}, i.e. binary words in the alphabet
$\{1, 2\}$ which are ranked by the sum of their digits; we postpone a
discussion of this account until \cref{subsection:YF-lattice}. Instead we
work with a new, alternative model using \defn{Fibonacci sets},
which is better adapted to our diagram calculus.
\begin{definition}
\label{def:intro-fibonacci-sets}
A \defn{Fibonacci set of rank $N$} is a subset
$S = \{s_1<s_2<\dots<s_k\}$ of $\{1, \dots, N \}$ whose size $k$ has
the same parity as $N$ and such that $s_\ell$ have the same parity as
$\ell$. We write $\YFS$ for the collection of Fibonacci sets.
\end{definition}
With this definition $\YFS$ becomes a ranked differential lattice
where there is a covering relation $S \lessdot T$ whenever
$T = S \setminus \max S$ or $S = T \setminus \max T$
(see~\cref{def:fibonacci-sets,fig.YF-lattice+introduction,def:YFS-lattice}).
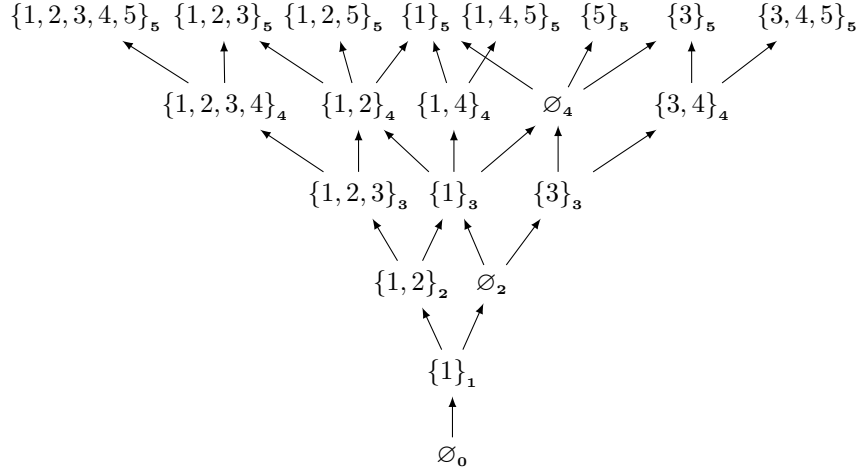
\begin{figure}[ht]
\[    \begin{tikzpicture}[>=latex,line join=bevel, xscale=0.55, yscale=0.68]
\node (node_0) at (303.5bp,3.5bp) [draw,draw=none] {$ \varnothing_{\bf \scriptscriptstyle 0}$};
\node (node_1) at (303.5bp,49.5bp) [draw,draw=none] {$\left\{1 \right\}_{\bf \scriptscriptstyle 1}$};

  \node (node_2) at (275.5bp,98.5bp) [draw,draw=none] {$\left\{1, 2 \right\}_{\bf \scriptscriptstyle 2}$};
  \node (node_7) at (330.5bp,98.5bp) [draw,draw=none] {$\varnothing_{\bf \scriptscriptstyle 2}$};

  \node (node_3) at (239.5bp,147.5bp) [draw,draw=none] {$\left\{1,2,3 \right\}_{\bf \scriptscriptstyle 3}$};
  \node (node_12) at (304.5bp,147.5bp) [draw,draw=none] {$\left\{1\right\}_{\bf \scriptscriptstyle 3}$};
  \node (node_8) at (375.5bp,147.5bp) [draw,draw=none] {$\left\{3\right\}_{\bf \scriptscriptstyle 3}$};

  \node (node_4) at (148.5bp,196.5bp) [draw,draw=none] {$\left\{1,2,3,4\right\}_{\bf \scriptscriptstyle 4}$};
  \node (node_16) at (239.5bp,196.5bp) [draw,draw=none] {$\left\{1,2\right\}_{\bf \scriptscriptstyle 4}$};
  \node (node_9) at (466.5bp,196.5bp) [draw,draw=none] {$\left\{3,4 \right\}_{\bf \scriptscriptstyle 4}$};
  \node (node_22) at (375.5bp,196.5bp) [draw,draw=none] {$\varnothing_{\bf \scriptscriptstyle 4}$};
  \node (node_13) at (304.5bp,196.5bp) [draw,draw=none] {$\left\{1,4\right\}_{\bf \scriptscriptstyle 4}$};

  \node (node_5) at (53.5bp,245.5bp) [draw,draw=none] {$\left\{1,2,3,4,5\right\}_{\bf \scriptscriptstyle 5}$};
  \node (node_19) at (147.5bp,245.5bp) [draw,draw=none] {$\left\{1,2,3\right\}_{\bf \scriptscriptstyle 5}$};
  \node (node_10) at (545.5bp,245.5bp) [draw,draw=none] {$\left\{3,4,5 \right\}_{\bf \scriptscriptstyle 5}$};
  \node (node_25) at (466.5bp,245.5bp) [draw,draw=none] {$\left\{3\right\}_{\bf \scriptscriptstyle 5}$};
  \node (node_14) at (343.5bp,245.5bp) [draw,draw=none] {$\left\{1,4,5\right\}_{\bf \scriptscriptstyle 5}$};
  \node (node_28) at (285.5bp,245.5bp) [draw,draw=none] {$\left\{ 1\right\}_{\bf \scriptscriptstyle 5}$};
  \node (node_17) at (222.5bp,245.5bp) [draw,draw=none] {$\left\{1,2,5\right\}_{\bf \scriptscriptstyle 5}$};
  \node (node_23) at (407.5bp,245.5bp) [draw,draw=none] {$\left\{5\right\}_{\bf \scriptscriptstyle 5}$};

  \draw [black,->] (node_0) -- (node_1);
  \draw [black,->] (node_1) -- (node_2);
  \draw [black,->] (node_1) -- (node_7);
  \draw [black,->] (node_2) -- (node_3);
  \draw [black,->] (node_2) -- (node_12);
  \draw [black,->] (node_3) -- (node_4);
  \draw [black,->] (node_3) -- (node_16);
  \draw [black,->] (node_4) -- (node_5);
  \draw [black,->] (node_4) -- (node_19);
  \draw [black,->] (node_7) -- (node_8);
  \draw [black,->] (node_7) -- (node_12);
  \draw [black,->] (node_8) -- (node_9);
  \draw [black,->] (node_8) -- (node_22);
  \draw [black,->] (node_9) -- (node_10);
  \draw [black,->] (node_9) -- (node_25);
  \draw [black,->] (node_12) -- (node_13);
  \draw [black,->] (node_12) -- (node_16);
  \draw [black,->] (node_12) -- (node_22);
  \draw [black,->] (node_13) -- (node_14);
  \draw [black,->] (node_13) -- (node_28);
  \draw [black,->] (node_16) -- (node_17);
  \draw [black,->] (node_16) -- (node_19);
  \draw [black,->] (node_16) -- (node_28);
  \draw [black,->] (node_22) -- (node_23);
  \draw [black,->] (node_22) -- (node_25);
  \draw [black,->] (node_22) -- (node_28);
\end{tikzpicture}\]
 \caption[The Young-Fibonacci lattice]{Fibonacci sets and the Young-Fibonacci lattice up to rank
    $5$.}
  \label{fig.YF-lattice+introduction}
\end{figure}

\defn{Okada algebas}
$\{ \Okada(X,Y) \}_{\scriptscriptstyle \mathrm{N} \geq 0}$ were
introduced in~\cite{Okada1994} in order to give an algebraic account
of the Young-Fibonacci lattice.  As such, they play a role for the
Young-Fibonacci lattice similar to that played by the symmetric groups
for the Young lattice. These algebras depend on two sequences of
parameters $X = (x_1, x_2, x_3, \dots )$ and
$Y=(y_1, y_2, y_3, \dots)$ selected from a field $\K$.
\begin{definition}[\cite{Okada1994}]
\label{intro-def:okada-algebras}
  The \defn{Okada algebra $\Okada(X, Y)$} is the algebra generated by
  $\E{1},\dots,\E{N-1}$ with the relations
\begin{alignat*}{3}
  \E{i}^2&=x_i\E{i} &\qquad& 1\leq i \leq N-1\,, \\
  \E{i}\E{j}&=\E{j}\E{i} &\qquad&\vert i-j\vert \geq 2\,,  \\
  \E{i+1}\E{i}\E{i+1}&=y_i\E{i+1} &\qquad&  1\leq i \leq N-2\,.
\end{alignat*}
\end{definition}
These algebras are finite dimensional, associative, and generically semi-simple.
An \defn{Okada Monoid $\Okada$} can also be defined (\cref{def:Okada-monoid})
by specializing $x_k = y_k = 1$, however the corresponding
Okada algebra (i.e. the monoid algebra) is not semi-simple. In order
to distinguish the monoid generators, we write $\e{i}$ instead of $\E{i}$.
The obvious inclusion  $\Okada[N-1](X,Y) \subset \Okada[N](X,Y)$
for $N \geq 1$ gives rise to a tower of finite dimensional algebras:
\[ \mathbb{K} = \Okada[0](X,Y) \subset \Okada[1](X,Y) \subset \Okada[2](X,Y) \subset \Okada[3](X,Y) \subset \cdots \]
Okada's chief observation is that the Bratelli diagram of this tower is the Young-Fibonacci lattice, in the semi-simple case.
Therefore, up to isomorphism, irreducible $\Okada(X,Y)$-modules $V^S$ are in bijection with
Fibonacci sets $S \in \YFS_N$ and the decomposition of the restriction of $V^S$ to
$\Okada[N-1](X,Y)$ is captured by the covering relations of the $\YFS$-lattice
\begin{equation}
\label{young-fibonacci-lattice-restriction}
\mathrm{\bf Res}_{\ssp \Okada[N-1](X,Y)}^{\ssp \Okada(X,Y)} \big( V^S  \big) \ssp \cong \ssp \bigoplus_{T \lessdot S} V^{T}\,.
\end{equation}
The
Okounkov-Vershik approach is further evinced: As in the
classical case of the Young lattice, each irreducible
representation $V^S$ aquires a {\it Gelfand-Tsetlin} basis indexed by
saturated chains
$\underline{\mathrm{\bf S}} = S_0 \lessdot S_1 \lessdot \cdots
\lessdot S_N$ in the $\YFS$-lattice beginning at $S_0 = \varnothing_0$
and ending at $S_N = S$, where the action of each generator $\E{k}$ on
$\underline{\mathrm{\bf S}}$ is local and only depends on the covering
relations $S_{k-1} \lessdot S_k \lessdot S_{k+1}$ for
$1 \leq k \leq N-1$~\cite{Okada1994}.

Since the $\YFS$-lattice is differential it follows that the induction
and restriction functors $\mathrm{\bf Ind}$ and $\mathrm{\bf Res}$ for
irreducible representations of the Okada algebras obey the
categorified Weyl commutation identity given in
\Cref{young-lattice-induction-restriction}.  In the semi-simple case,
differentiability also allows one to show that $\dim \Okada(X,Y) = N!$
and that $\Okada(X,Y)$ has a basis $\{ \E{\sigma} \}_{\sigma \in \SG}$
consisting of monomials of the form
$\E{\sigma} \eqdef \E{i_1} \cdots \E{i_k}$ where the tuple
$(i_1, \dots , i_k)$ indexes the unique, lexicographically minimal
factorization $\sigma = s_{i_1} \cdots s_{i_k}$ into simple
transpositions $s_i = (i, i+1)$. We will show that both results remain
valid in the non semi-simple case, however more involved rewriting
techniques are needed to establish this; see~\cref{subsection:rewriting}.

A companion theory of \defn{clone functions}
was also developed in~\cite{Okada1994} which
serves as a counterpart to the theory of symmetric functions for the $\YF$-lattice.
Clone analogues of the classical bases
(e.g. complete homogeneous, Schur, and power-sum symmetric functions)
are constructed as well as variants of the Pieri and Littlewood-Richardson rules
for clone Schur functions. Notably, the covering relations of the
$\YF$-lattice can be recovered from this clone Pieri rule.
Clone Schur functions are important in this story as they simultaneously
model an induction product for irreducible representations of Okada algebras
and also appear as matrix coefficients describing the action of the generators
in these representations. They also appear in Goodman-Kerov's determination of the
{\it Martin boundary} of the
$\YF$-lattice~\cite{Goodman1997} and are related to central Markov
processes which evolve on the $\YF$-lattice
(see~\cite{GNEDIN_KEROV_2000,Petrov_Scott_2026}).

\subsection{Diagram algebras}
For appropriate values of the $X, Y$ parameters, the Okada algebra has
two well known quotients: The Temperley-Lieb algebra~\cite{TemperleyLieb1971}
$\mathrm{TL}_N(\delta)$ and the Martin-Saleur blob
algebra~\cite{Martin1994}
$\mathrm{Blob}_N(\gamma,\delta)$. Indeed, both $\mathrm{TL}_N(\delta)$
and $\mathrm{Blob}_N(\gamma,\delta)$ are defined by relations which
refine Okada's presentation. In addition their Bratelli diagrams
can be embedded into the $\YF$-lattice.

The Temperley-Lieb algebras famously appear in the work of V. Jones on
subfactors and the theory of knots. In Jones' framework any
braid with $N$ strands can be {\it evaluated} in
$\mathrm{TL}_N(\delta)$ using skein relations, and this
evaluation is a topological invariant of the knot obtained by
closing-up the braid~\cite{Jones1983}.  Both families of algebras have
central roles in various statistical models arising in physics, such
as the {\it Potts} and {\it six-vertex}
models~\cite{TemperleyLieb1971}.  Notably the generators of
$\mathrm{TL}_N(\delta)$ in the $\delta = 1$ specialization give rise
to {\it transition probabilities} defining a Markov process on the
space non-crossing link patterns (i.e. {\it half-diagrams}) with $N$
vertices; likewise the $\gamma = \delta =1$ specialization of
$\mathrm{Blob}_N(\gamma,\delta)$ governs a Markov process on
non-crossing link patterns having an open (left) boundary
(see~\cite{DeGier2004}).

Combinatorial aspects of the behaviour of the stationary
probability distribution in each case are described by the
Razumov–Stroganov conjecture and its
blob-variant~\cite{DeGier2004,ZinnJustin2009}. The original
formulation of the Razumov-Stroganov conjecture, involving the
transfer matrix for completely packed loops on a semi-infinite
cylinder and the {\it affine} Temperley-Lieb algebra, was settled by
L. Cantini and A. Sportiello in 2011~\cite{Cantini2011}.

%The Temperley-Lieb algebra $\mathrm{TL}_N(\delta)$
%and the Blob algebra $\mathrm{Blob}_N(\gamma,\delta)$
%can be realized as quotients of the Okada algebra $\Okada(X,Y)$ after specializing
%the $X,Y$ parameters appropriately. In addition their Bratelli diagrams
%can be embedded into the $\YF$-lattice. Both $\mathrm{TL}_N(\delta)$
%and $\mathrm{Blob}_N(\gamma,\delta)$ are {\it diagram algebras}
%and carry a {\it cellular algebra structure}, and so it's natural to ask whether the $\Okada(X,Y)$ has
%such a description.

%In this paper we
%realize the Okada algebra $\Okada(X,Y)$
%as a diagram algebra
%with a multiplicative/monoidal basis
%expressed in terms of certain arc-labelled, non-crossing
%perfect matchings (as appear in both  the Temperley-Lieb
%and Martin-Saleur Blob algebras~\cite{Martin1994a}).
%This basis is cellular and affords us with a
%novel, diagrammatic presentation of the irreducible
%representations of $\Okada(X,Y)$
%(i.e. as \defn{cell modules}).
%We interpret Fomin's RS-correspondence
%diagrammatically. This involves constructing a
%bijection between saturated chains in the $\YF$-lattice
%(presented in terms of sequences of \defn{Fibonacci sets})
%and \defn{Okada half arc-diagrams}.
%In addition we examine the structure theory of the Okada algebra
%and monoid via a \defn{dominance order} on Fibonacci sets.

The Temperley-Lieb algebra $\mathrm{TL}_N(\delta)$
is an example of a {\it diagram algebra} and
possesses a multiplicative basis consisting of
{\it non-crossing arc-diagrams} of rank $N$.
Such a diagram is a planar visualization of a
perfect matching, linking vertices
$ \{1, \dots\, N\}$ and $\{\ov{1}, \dots, \ov{N} \}$ situated
respectively on the left and right boundaries of a rectangle
by {\it non-crossing} arcs (drawn in the interior of the
rectangle). Each pair $\{a, b \}$ in the matching corresponds to
an arc, denoted $\arc{a}{}{b}$, joining vertices $a , b \in \setN \cup \setNN$
where $\setN = \{1, \dots N \}$ and where $\setNN \eqdef \{\ov{1}, \dots, \ov{N} \}$.
The ensemble of arcs is drawn modulo isotopy.
Up to a scalar depending on the parameter $\delta$,
the product $C \boldsymbol{\cdot} D$ of two non-crossing arc-diagrams $C$ and $D$
corresponds to merging, respectively, the right and left boundary vertices of $C$
and $D$ and then removing any loops which may have been created.
Each loop created in the merger contributes a factor of $\delta$
to the overall scalar.
The blob algebra $\mathrm{Blob}_N(\gamma,\delta)$ has a
similar description, involving a diagram basis consisting of
non-crossing arc-diagrams, whose {\it exposed} arcs
are allowed to be marked or {\it blobbed}.
\medskip

This leads us to the main question of this paper: \textbf{Both
$\mathrm{TL}_N(\delta)$ and $\mathrm{Blob}_N(\gamma,\delta)$ are \defn{\it
  diagram algebras} and carry a \defn{\it cellular algebra structure}, and
so it's natural to ask whether $\Okada(X,Y)$ has such a
description.}
\medskip

\subsection{Okada arc-diagrams}

The chief result of this paper is to realize the Okada algebra
$\Okada(X,Y)$ as a diagram algebra with a multiplicative basis
consisting of Temperley-Lieb diagrams whose arcs are assigned positive
integer~\textit{labels} satisfying the following constraints:
\begin{definition}
  \label{def.okada.arc.introduction}
  An \defn{Okada arc-diagram} of rank $N$ is a is a
  non-crossing arc-diagram of rank $N$ whose arcs
  carry integer labels subject to the following conditions:
  \begin{enumerate}
  \item the label  of each arc $\arc{a}{}{b}$ must
    be at least $1$ and at most $\min(|a|, |b|)$
  \item the label of each arc $\arc{a}{}{b}$
    must have the same parity as $\min(|a|, |b|)$
  \item if an arc $\arc{a}{}{b}$ is nested in an arc $\arc{c}{}{d}$ then
    the label of $\arc{a}{}{b}$ is \newline strictly larger than the label of $\arc{c}{}{d}$.
  \end{enumerate}
\end{definition}
Two Okada arc-diagrams $C$ and $D$ of rank $N$ can be multiplied
\defn{$C \boldsymbol{\cdot} D$} by merging their respective underlying
Temperley-Lieb diagrams, removing loops, and labelling each newly created
arc by the {\it minimum} of the integer labels of the individual arcs from
$C$ and $D$ which form it. It is a non-trivial result that labelling
the arcs in this manner produces a new Okada arc-diagram of rank $N$.
In this way we obtain a monoid structure on the set \defn{$\OkArc$} of
all Okada arc-diagram of rank $N$ (see
\cref{fig:intro-okada-arc-diagrams}). $\OkArc$ is generated by
\defn{elementary diagrams $\A{1}, \dots, \A{N-1}$} (see~\cref{def:gen-arc-okada}),
which obey the defining relations of the Okada monoid and the map
$\A{i}\mapsto\e{i}$ extends to a monoid isomorphism $\Theta: \OkArc
\rightarrow  \Okada$~\cref{cor:isomEG}.

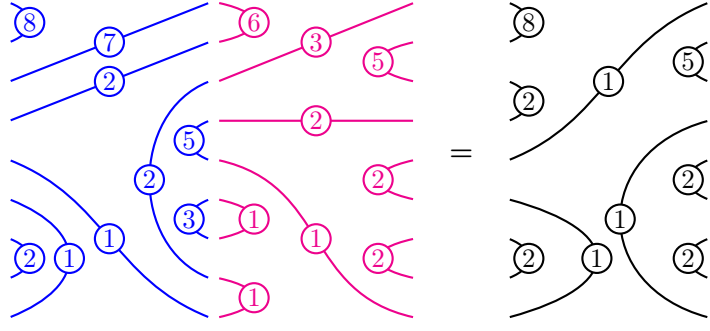
\begin{figure}[ht]
  \centering
  \begin{tikzpicture}[xscale=0.5,yscale=0.4, thick, baseline={(current bounding box.east)}]
\tikzstyle{vertex} = [shape=rectangle, minimum height=15pt, minimum width=0pt, inner sep=0pt]
\tikzstyle{mid} = [draw, fill=white, shape=circle, minimum size=2pt, inner sep=1pt]
\node[vertex] (G--9) at (2.64, 10.40) {};
\node[vertex] (G--8) at (2.64, 9.10) {};
\node[vertex] (G--7) at (2.64, 7.80) {};
\node[vertex] (G--6) at (2.64, 6.50) {};
\node[vertex] (G--5) at (2.64, 5.20) {};
\node[vertex] (G--4) at (2.64, 3.90) {};
\node[vertex] (G--3) at (2.64, 2.60) {};
\node[vertex] (G--2) at (2.64, 1.30) {};
\node[vertex] (G--1) at (2.64, 0.00) {};
\node[vertex] (G-1) at (-2.64, 0.00) {};
\node[vertex] (G-2) at (-2.64, 1.30) {};
\node[vertex] (G-3) at (-2.64, 2.60) {};
\node[vertex] (G-4) at (-2.64, 3.90) {};
\node[vertex] (G-5) at (-2.64, 5.20) {};
\node[vertex] (G-6) at (-2.64, 6.50) {};
\node[vertex] (G-7) at (-2.64, 7.80) {};
\node[vertex] (G-8) at (-2.64, 9.10) {};
\node[vertex] (G-9) at (-2.64, 10.40) {};
\draw[color=blue] (G--1) .. controls +(-2.59, 1.30) and +(2.59, -1.30) .. (G-5) node[pos=0.5,mid] (M-1-5) { 1 };
\draw[color=blue] (G--8) -- (G-6) node[pos=0.5,mid] (M-8-6) { 2 };
\draw[color=blue] (G-1) .. controls +(2.10, 1.05) and +(2.10, -1.05) .. (G-4) node[pos=0.5,mid] (M1-4) { 1 };
\draw[color=blue] (G--7) .. controls +(-2.10, -1.05) and +(-2.10, 1.05) .. (G--2) node[pos=0.5,mid] (M-7--2) { 2 };
\draw[color=blue] (G-2) .. controls +(0.70, 0.35) and +(0.70, -0.35) .. (G-3) node[pos=0.5,mid] (M2-3) { 2 };
\draw[color=blue] (G--4) .. controls +(-0.70, -0.35) and +(-0.70,0.35) .. (G--3) node[pos=0.5,mid] (M-4--3) { 3 };
\draw[color=blue] (G--9) -- (G-7) node[pos=0.5,mid] (M-9-7) { 7 };
\draw[color=blue] (G-8) .. controls +(0.70, 0.35) and +(0.70, -0.35) .. (G-9) node[pos=0.5,mid] (M8-9) { 8 };
\draw[color=blue] (G--6) .. controls +(-0.70, -0.35) and +(-0.70, 0.35) .. (G--5) node[pos=0.5,mid] (M-6--5) { 5 };
\end{tikzpicture}
%$\cdot$
\begin{tikzpicture}[xscale=0.9,yscale=0.4, thick, baseline={(current bounding box.east)}]
\tikzstyle{vertex} = [shape=rectangle, minimum height=15pt, minimum width=0pt, inner sep=0pt]
\tikzstyle{mid} = [draw, fill=white, shape=circle, minimum size=2pt, inner sep=1pt]
\node[vertex] (G--9) at (1.44, 10.40) {};
\node[vertex] (G--8) at (1.44, 9.10) {};
\node[vertex] (G--7) at (1.44, 7.80) {};
\node[vertex] (G--6) at (1.44, 6.50) {};
\node[vertex] (G--5) at (1.44, 5.20) {};
\node[vertex] (G--4) at (1.44, 3.90) {};
\node[vertex] (G--3) at (1.44, 2.60) {};
\node[vertex] (G--2) at (1.44, 1.30) {};
\node[vertex] (G--1) at (1.44, 0.00) {};
\node[vertex] (G-1) at (-1.44, 0.00) {};
\node[vertex] (G-2) at (-1.44, 1.30) {};
\node[vertex] (G-3) at (-1.44, 2.60) {};
\node[vertex] (G-4) at (-1.44, 3.90) {};
\node[vertex] (G-5) at (-1.44, 5.20) {};
\node[vertex] (G-6) at (-1.44, 6.50) {};
\node[vertex] (G-7) at (-1.44, 7.80) {};
\node[vertex] (G-8) at (-1.44, 9.10) {};
\node[vertex] (G-9) at (-1.44, 10.40) {};
\draw[magenta] (G--1) .. controls +(-1.63, 0.82) and +(1.63, -0.82) .. (G-5) node[pos=0.5,mid] (M-1-5) { 1 };
\draw[magenta] (G--6) -- (G-6) node[pos=0.5,mid] (M-6-6) { 2 };
\draw[magenta] (G--9) -- (G-7) node[pos=0.5,mid] (M-9-7) { 3 };
\draw[magenta] (G-3) .. controls +(0.70, 0.35) and +(0.70, -0.35) .. (G-4) node[pos=0.5,mid] (M3-4) { 1 };
\draw[magenta] (G--5) .. controls +(-0.70, -0.35) and +(-0.70, 0.35) .. (G--4) node[pos=0.5,mid] (M-5--4) { 2 };
\draw[magenta] (G-8) .. controls +(0.70, 0.35) and +(0.70, -0.35) .. (G-9) node[pos=0.5,mid] (M8-9) { 6 };
\draw[magenta] (G--8) .. controls +(-0.70, -0.35) and +(-0.70, 0.35) .. (G--7) node[pos=0.5,mid] (M-8--7) { 5 };
\draw[magenta] (G-1) .. controls +(0.70, 0.35) and +(0.70, -0.35) .. (G-2) node[pos=0.5,mid] (M1-2) { 1 };
\draw[magenta] (G--3) .. controls +(-0.70, -0.35) and +(-0.70, 0.35) .. (G--2) node[pos=0.5,mid] (M-3--2) { 2 };
\end{tikzpicture}
%%%%%%%%%%%%%%%%%%%%%%%%%%%
\quad $\boldsymbol{=}$ \quad
%%%%%%%%%%%%%%%%%%%%%%%%%%%
\begin{tikzpicture}[xscale=0.5,yscale=0.4, thick, baseline={(current bounding box.east)}]
\tikzstyle{vertex} = [shape=rectangle, minimum height=15pt, minimum width=0pt, inner sep=0pt]
\tikzstyle{mid} = [draw, fill=white, shape=circle, minimum size=2pt, inner sep=1pt]
\node[vertex] (G--9) at (2.64, 10.40) {};
\node[vertex] (G--8) at (2.64, 9.10) {};
\node[vertex] (G--7) at (2.64, 7.80) {};
\node[vertex] (G--6) at (2.64, 6.50) {};
\node[vertex] (G--5) at (2.64, 5.20) {};
\node[vertex] (G--4) at (2.64, 3.90) {};
\node[vertex] (G--3) at (2.64, 2.60) {};
\node[vertex] (G--2) at (2.64, 1.30) {};
\node[vertex] (G--1) at (2.64, 0.00) {};
\node[vertex] (G-1) at (-2.64, 0.00) {};
\node[vertex] (G-2) at (-2.64, 1.30) {};
\node[vertex] (G-3) at (-2.64, 2.60) {};
\node[vertex] (G-4) at (-2.64, 3.90) {};
\node[vertex] (G-5) at (-2.64, 5.20) {};
\node[vertex] (G-6) at (-2.64, 6.50) {};
\node[vertex] (G-7) at (-2.64, 7.80) {};
\node[vertex] (G-8) at (-2.64, 9.10) {};
\node[vertex] (G-9) at (-2.64, 10.40) {};
\draw (G--9) .. controls +(-2.59, -1.30) and +(2.59, 1.30) .. (G-5) node[pos=0.5,mid] (M-9-5) { 1 };
\draw (G-6) .. controls +(0.70, 0.35) and +(0.70, -0.35) .. (G-7) node[pos=0.5,mid] (M6-7) { 2 };
\draw (G--6) .. controls +(-3.10, -1.05) and +(-3.10, 1.05) .. (G--1) node[pos=0.5,mid] (M-6--1) { 1 };
\draw (G-1) .. controls +(3.10, 1.05) and +(3.10, -1.05) .. (G-4) node[pos=0.5,mid] (M1-4) { 1 };
\draw (G--5) .. controls +(-0.70, -0.35) and +(-0.70, 0.35) .. (G--4) node[pos=0.5,mid] (M-5--4) { 2 };
\draw (G-2) .. controls +(0.70, 0.35) and +(0.70, -0.35) .. (G-3) node[pos=0.5,mid] (M2-3) { 2 };
\draw (G--8) .. controls +(-0.70, -0.35) and +(-0.70, 0.35) .. (G--7) node[pos=0.5,mid] (M-8--7) { 5 };
\draw (G-8) .. controls +(0.70, 0.35) and +(0.70, -0.35) .. (G-9) node[pos=0.5,mid] (M8-9) { 8 };
\draw (G--3) .. controls +(-0.70, -0.35) and +(-0.70, 0.35) .. (G--2) node[pos=0.5,mid] (M-3--2) { 2 };
\end{tikzpicture}
\caption{Two Okada arc-diagrams (in \textcolor{blue}{blue}
and \textcolor{magenta}{magenta}) and their product (in black) within the monoid
$\OkArc[9]$.}
\label{fig:intro-okada-arc-diagrams}
\end{figure}

This description extends to the algebra: $\Okada(X,Y)$ has a natural
diagram basis $(\E{D})_{D\in\OkArc}$ and the $X$
and $Y$ parameter can be incorporated into the monoid product as
\begin{equation}
\label{eq:multiplication-diagrams-okada-algebra-intro}
 \E{C} \ssp \E{D} \, = \,  \lambda(C,D) \ssp \E{C \boldsymbol{\cdot} D}
 \end{equation}
where $\lambda(C,D) \in \mathbb{K}$ is a monomial
in the $X$ and $Y$-parameters depending on
the labels of both the loops and arcs formed when merging $C$ and $D$,
see~\cref{prop:lambda-coefficient,coro:coeff-mult}.
In the simplest case,
when $y_k =1$ for all $k \geq 1$, each loop in the merger
contributes a factor of $x_\ell$ to the scalar $\lambda(C,D)$
where $\ell$ is the minimum label of the loop. The
formula is simultaneously a multivariate and {\it tropical} extension of the rule for
multiplying diagrams in the Temperley-Lieb algebra $\mathrm{TL}_N(\delta)$.

\subsection{Diagrammatic Robinson-Schensted correspondence}

From a combinatorial point of view, there are two ways to justify the
fact that the cardinality of $\Okada$ and the dimension of
$\Okada(X, Y)$ are both $N!$. 

The first approach uses rewriting techniques to show that the map
$\sigma\mapsto \e{\sigma}$ from $\SG$ to $\Okada$ is a bijection where
$\e{\sigma} \eqdef \e{i_1} \cdots \e{i_k}$ matches the unique,
lexicographically minimal factorization
$\sigma = s_{i_1} \cdots s_{i_k}$ into simple transpositions (see
\cref{theo:bij-Sn_Okada,prop:lexmin-normal-form}).

The second approach makes use of our description of the
$\YFS$-lattice in terms of Fibonacci sets and, remarkably,
recovers Fomin's RS-correspondance.
In particular, this approach is designed to leverage
the following key observation: For any Okada arc-diagram $D$ of rank
$N$ the set \defn{$\PropLab(D)$} consisting of the labels of
propagating arcs of $D$ (i.e. those connecting the left and right-hand
sides of the diagram) is always a Fibonacci set of rank $N$
(\cref{prop:proplab-index-fibo}). On the other hand, an Okada
arc-diagram $D$ can be split into its left and right half diagrams
\defn{$\rightd{D}$} and \defn{$\leftd{D}$} whose sets of propagating
labels agree. Furthermore, any rank $N$ half diagram $H$ can be
restricted in an obvious way to $\{1, \dots\, k \}$ for $k \leq N$ to
obtain a rank $k$ half diagram $H/\setN[k]$.  The sequence
$S_k \eqdef \PropLab(H/\setN[k])$ of propagating labels of $H/\setN[k]$
for $k=0,\dots,N$ is always a saturated chain
\[ \Chain(H) \, \eqdef \, S_0 \lessdot \cdots \lessdot S_N \] ending at
$\PropLab(H)\in\YFS_N$ and the map $H\mapsto\Chain(H)$ is a bijection
(\cref{prop:chain-diag-bij}). Recall that we have a monoid
isomorphism $\Theta: \Okada \to \OkArc$. Then
(\cref{theo:RS-isomorphism}) Fomin's bijection from
permutation to pair of chains is equal to
\[ \sigma \mapsto   \Big( \Chain \rightd{\Theta(\e{\sigma})}, \ssp  \Chain \leftd{\Theta(\e{\sigma})}\Big)\,. \]

\subsection{Cellularity vis-à-vis  Green structure}
The main algebraic outcome of this paper is the fact that the diagram basis
$\{E_D\}_{D \in \OkArc}$ of $\Okada(X,Y)$ is \defn{cellular}.
Typically, a cellular basis for a finite dimensional algebra is meant to
capture features of a generalized or modified RS-correspondence
which, in the semi-simple case, manifests the Wedderburn decomposition of the algebra.
A characteristic of this structure is that multiplication should be
{\it non-increasing} with respect to a canonical {\it filtration} of the
algebra by two-sided \defn{cell ideals} which can be defined from
the basis.

General results have been known for some time which exhibit a cellular
structure for both monoid and twisted-monoid algebras~(see e.g.:
\cite{WILCOX200710, EAST2006505, GUO200971}). Examples of the later include,
the Temperley-Lieb, blob, partition, and Brauer algebras. In this situation,
the cellular structure mirrors the relationship between ideals of
the monoid as expressed by Green's relations and related partial
orders. In the twisted case, this approach most certainly works for
the Okada algebra, except perhaps for its most degenerate
specializations. However, that would require a good, a-priori
understanding of the $\lambda(C, D)$ coefficients. In our case the
evaluation of $\lambda(C, D)$ is non-trivial
(see~\cref{coro:coeff-mult}) and instead we opt for a direct approach.
\smallskip

Indications that the diagram basis of $\Okada(X,Y)$ is cellular can
already be seen at the level of the arc monoid $\OkArc$.  One
important feature of $\OkArc$ is that multiplication is {\it
  non-increasing} with respect to Okada's \defn{dominance order
  $\preceq$} on $\YFS_N$, i.e. for $C, D \in \OkArc$
\begin{equation}
\label{proplabs+dominance+inequality}
\PropLab(C \boldsymbol{\cdot} D) \preceq \PropLab(C) \wedge \PropLab(D)
\end{equation}
where $\wedge$ is the meet operation for the dominance order which
endows $\YFS_N$ with the structure of a distributive
lattice~(\cref{prop.dominance.lattice}). This implies that two
diagrams $C$ and $D$ belong to the same $\JJ$-class if and only if
$\PropLab(C)=\PropLab(D)$ and that the $\JJ$-order is isomorphic to
the dominance order (\cref{theo:j-order-okada}). The linear span of a
two-sided monoid ideal is precisely a two-sided cell ideal, and
therefore the dominance order must be the partial order responsible for the
cellular structure on $\Okada(X,Y)$~(\cref{theo-cellular}).  Going
back to the monoid, an important consequence of
\cref{proplabs+dominance+inequality} is that $\Okada$ is an aperiodic
and regular $*$-monoid~(\cref{theo:aperiodic,cor:regular}). A more
thorough analysis also allows us to describe the one-sided Green order
as a ranked semi-lattice on
half-diagrams~(See~\cref{prop:r-combi-order,theo:r-order}).
\smallskip

From a representation theoretic standpoint, a cellular basis affords a
finite dimensional algebra with a supply of \defn{cell modules}.  In
the semi-simple case, these cell modules will be pairwise distinct,
simple, and will exhaust all simple modules. In the non semi-simple
case, there may be more cell modules than simple modules. Nevertheless
each simple module can be obtained as a quotient of {\it some} cell
module.

For the Okada Algebra $\Okada(X,Y)$
there is a distinct cell module $V^S$ for each
$S \in \YFS_N$ whose basis consists
of left half-diagrams $\rightd{D}$ with
$\PropLab \rightd{D} = S$ (see~\cref{sec:cell-module}). The (left) action
of $\Okada(X,Y)$ on half diagrams
in $V^S$ is given by diagram multiplication
\begin{equation}
  \label{eq.cell-module-action-intro}
  \E{C}  \bullet \rightd{D} \ \eqdef
  \begin{cases}
    \lambda(C,D) \ssp \rightd{C \boldsymbol{\cdot} D}
    &\text{if $\PropLab(C \boldsymbol{\cdot} D) = S$} \\
    0
    &\text{otherwise}
  \end{cases}
\end{equation}
where $\lambda(C,D)$ is the scalar appearing
in \Cref{eq:multiplication-diagrams-okada-algebra-intro}.
The action defined by
 \Cref{eq.cell-module-action-intro}
 might be called {\it quasi}-combinatorial
 insofar as $\E{C}$ maps
 half diagrams to half diagrams, albeit the
 occurrence of scalars. This is a hallmark of
 algebras possessing diagram bases which are cellular,
 unlike, for example, the Murphy basis of the Hecke algebra $\mathbcal{H}_N$.
 We emphasize that the half diagram basis for
$V^S$ differs from its Gelfand-Tsetlin basis
appearing in Okada's description of $V^S$.
While Okada's formulation is valid only in the
semi-simple case, the cell module structure prescribed by
\Cref{eq.cell-module-action-intro} makes sense
for any specialization of the $X$ and $Y$-parameters.

We show that the cellular structures for $\Okada(X,Y)$ and
$\Okada[N+1](X,Y)$ are \defn{coherent} with respect to the standard
inclusion $\Okada(X,Y) \subset \Okada[N+1](X,Y)$ for all
$N \geq 0$~(\cref{theo:coherent-celullar}).  This compatibility,
introduced by Goodman-Graber~\cite{Goodman2011}, is needed to
formalize the notion of a Bratteli diagram which describes the
branching of cell modules, even in the non semi-simple case.  In
general the covering relations of a Bratteli diagram are determined by
{\it composition factors} and not direct summands which arise when
restricted and/or inducing a cell module. We prove that for any choice
of parameters, the Goodman-Graber version of the Bratteli diagram for
the tower of Okada algebras $\Okada(X,Y)$ is always
$\YFS$~(\cref{prop:restr-okada}).  This relies on the combinatorial
fact that successors and predecessors of any Fibonacci set in $\YFS$
are totally ordered with respect to the dominance
order~(\cref{lem:YF-pred_totally-ordered}).

\medskip

A important feature of cell modules, is that each comes
equipped with a canonical, \defn{invariant bilinear form} $\varphi$.
The radical of this bilinear form is non-zero precisely when the cell
module is simple; the former condition being characterized by the
non-vanishing of the associated \defn{Gram determinant}.  The bilinear
form $\varphi_S$ attached to the cell module $V^S$ is addressed
in~\cref{sec:clone} and we conjecture an explicit factorization
of the associated Gram determinant $\det G^S_N $ in terms of biserial
clone Schur functions in \ref{conj::SGram-Determinant}.  This is
consistent with Okada's result~\cite[Theorem 2.8]{Okada1994}
where the semi-simplicity of
$\Okada(X,Y)$ is characterized by the non-vanishing of certain clone
Schur functions, which are present in our conjectured factorization.
\Cref{conj::SGram-Determinant}
models a result of \cite{DeGier2009} for the {\it two-boundary
  Tempereley-Lieb algebra} $\boldsymbol{\mathrm{2BTL}}_N$ and also
agrees with well-known results concerning the factorization of Gram
determinants for both $\mathrm{TL}_N(\delta)$ and
$\mathrm{Blob}_N(\gamma,\delta)$, see~\cite{Di_Francesco1998-ot,Jacobsen2008}.

\subsection*{Acknowledgments}
F. Hivert would like to thank James
Mitchell and Nicolas Thiéry for fruitful discussions. Likewise, J. Scott
thanks Amritanshu Prasad, Darij Grinberg, Leonid Petrov, and Tom Roby 
for discussions, and well as
Philippe Biane and
Anatoly Vershik (posthumously) for their input and guidance.
Significant progress on this project was made 
during a joint visit to the Mathematics department at the
University of Minnesota in January 2025. We would 
both like to thank Vic Reiner for insightful discussions 
and for his hospitality during our stay in Minneapolis.

F. Hivert was partially supported by the European projet FRESCO
(European Union’s Horizon 2020 research and innovation program grant
No.101001995) and by the French--Austrian project PAGCAP
(ANR-21-CE48-0020 \& FWF I 5788). Computations were
done using the SageMath and Mathematica computer algebra systems.

\newpage

\section{Background}

Throughout this paper $N$ designates a non-negative integer.  
We often write negative numbers as overlined numbers;
for example, writing $\ov{4}$ instead of $-4$.  The cardinality of a set $S$ is denoted $\card S$.
The set $\{1,\dots,N\}$ is denoted  $\setN$ while $\{ \ov{1} , \dots, \ov{N} \}$ is
denoted $\ov{[N]}$.
For a
non-negative integer $N$ we endow $[N] \cup \ov{[N]}$ with the total order
\begin{equation}
  \label{eq:orderN-NN}
  \{1 < 2< \dots < N < \ov{N} < \dots < \ov{2} < \ov{1}\}\,.
\end{equation}
This non-standard ordering is chosen because we want to ensure that the Okada arc-monoid $\OkArc$
(defined in \cref{subsection:the-labelled-arc-diagram-monoid})
naturally embeds into $\OkArc[N+1]$ while being consistent with 
~\cref{relSxy}. Indicating negative numbers by overlines
should help remind the reader of this unusual ordering.

\subsection{Poset, covers, chains and lattices}
At least seven different posets are defined in this paper! We feel
that we need to recall standard terminology and fix some notation 
so that our exposition can be as self-contained as possible.

A \defn{pre-order} $\leq$ on a set $\mathbb{P}$ is a binary relation which is
reflexive ($x \leq x$ for all $x\in \mathbb{P}$) and transitive (if $x \leq y$
and $y \leq z$, then $x \leq z$ for all all $x, y, z\in \mathbb{P}$). A pre-order
which is anti-symmetric (if $x \leq y$ and $y \leq x$, then
$x = y$) is called a \emph{partial order}. We also says that $(\mathbb{P}, \leq)$ is
a partially ordered set (or \defn{poset} for short). Any pre-order
$\leq$ gives rise to an equivalence relation defined by
$x \equiv y$ whenever $x\leq y$ and $y\leq x$; this equivalence makes the quotient
$\mathbb{P}/\!\equiv$ into a poset. We will use different notation such as
$\leq ,\subseteq, \preceq, \trianglelefteq$ to distinguish between 
various (pre)-orders in the text.  When we exclude the possibility that $x=y$, we
write $x < y$, $x\subset y$, $x\prec y$ or $x\triangleleft y$.

A pair $(x,y)$ such that $x < y$ and for which there is no $z\in \mathbb{P}$ satisfying
$x < z < y$ is called a \defn{covering}. We also says that $y$ covers
$x$ or $x$ is covered by $y$. We usually denote coverings by adding a
dot in the inequality symbol, as in $x\lessdot y$ or $x\YoungCover y$.  The
\defn{Hasse diagram} of $(\mathbb{P},\leq)$ is the oriented graph whose vertices
are the elements $x\in \mathbb{P}$ and where there is an edge directed 
from $x$ to $y$ whenever $x\lessdot y$.

A poset $(\mathbb{P},\leq)$ is \defn{graded} if there is a map $|.| : \mathbb{P} \to \N$
such that $|x| + 1 = |y|$ whenever $x\lessdot y$. Note that the
grading map is only defined up to a constant. In this paper, all 
graded posets will have a unique, minimal element $\varnothing$
(i.e. $\varnothing \leq x$ for all $x$), so requiring that $|\varnothing|=0$
fixes the constant. A \defn{chain} $\mathrm{\bf C}$ in a poset is a totally ordered
subset. Therefore a \emph{finite chain} can be written as a sequence
$(x_1<x_2<\dots<x_n)$. A chain $\mathrm{\bf C}$ is \defn{saturated} if the
only way to extend $\mathrm{\bf C}$ to a larger chain $\{x\} \cup \mathrm{\bf C}$ is 
when either $x \leq y$ for all $y\in \mathrm{\bf C}$ or $x \geq y$ for all $y\in \mathrm{\bf C}$. In
the case of a \emph{finite saturated chain} consecutive pairs in the sequence
form coverings, i.e.  $\mathrm{\bf C} = (x_1\lessdot x_2\lessdot\dots\lessdot
x_n)$. Such a chain can be visualized as a path in the Hasse diagram.

Finally, recall that the \defn{join} (also called the
least upper bound or supremum and denoted $x \vee y$ ) 
of two elements $x, y \in \mathbb{P}$ is an element $z$ such that for
all $t$, the relation $z\leq t$ is equivalent to $x\leq t$ and
$y\leq t$. Clearly the join is unique if it exists.
If all pairs $x ,y \in \mathbb{P}$ admit a join, we says that $\mathbb{P}$ is a
\defn{join semilattice}. Reversing the direction of comparison
defines the \defn{meet} (or greatest lower bound or infimum and denoted
$x \wedge y$) together with the notion of a \defn{meet semilattice}. If $\mathbb{P}$ is
simultaneously both a meet and join semilattice, then $\mathbb{P}$ is said to be a \defn{lattice}.

\subsection{Words}
A finite word $\ind{i} = i_1 \cdots  i_k$
composed of integers in $\setN[N-1]$ will be called an \defn{index word}
of rank $N$. Let $\mathcal{l}(\ind{i})=k$ be the \defn{length} of the
word, $|\ind{i}| = i_1 + \cdots  + i_k$ its sum, and
$\ind{i}^\star = i_k \cdots  i_1$ its reversal. The empty word is denoted
$\boldsymbol{\epsilon}$ and $\ind{i}\cdot \ind{j}$ denotes the concatenation of two index words $\ind{i}$ and $\ind{j}$.
Obviously $\mathcal{l}(\ind{i}\cdot\ind{j})=\mathcal{l}(\ind{i})+\mathcal{l}(\ind{j})$ and
$| \ind{i}\cdot\ind{j}| = |\ind{i}| + |\ind{j}|$ and $(\ind{i} \cdot \ind{j})^\star = \ind{j}^\star \cdot \ind{i}^\star$.
An index word $\ind{i}= i_1 \cdots i_k$ is said to be \defn{decreasing} if
$i_1 > \cdots > i_k$.
Given a monoid or a $\K$-algebra $G$ generated by a family $(\G{i})_{i}$ then
$\G{\ind{i}}$ will denote the product $\G{i_1} \! \cdots \G{i_k} \in G$. Such a
product is called a \defn{word} when $G$ is a monoid and a
\defn{monomial} when $G$ is an algebra. Clearly $\G{\epsilon}=1$ is
the neutral element of $G$ and $\G{\ind{i}\cdot\ind{j}} = \G{\ind{i}} \, \G{\ind{j}}$.
Let $\FreeMonoid$ denote the
\defn{free monoid} consisting
of all rank $N$ index words
under the concatenation product.

\subsection{Monoids}

We present here basic facts about monoids.  We refer to \cite{Pin.2010} or
\cite{Steinberg.2016} for more details. Throughout this paper, all monoids are
assumed to be \emph{finite}.
\medskip

Recall that the left (resp.\ right, resp.\ two-sided) ideal of a
monoid $\tMonoid$ generated by an element $x\in\tMonoid$ is the set
$Mx\eqdef\{m x\mid m\in\tMonoid\}$ (resp.\
$xM\eqdef\{x m\mid m\in\tMonoid\}$, resp.\
$MxM\eqdef\{m x n\mid m, n\in\tMonoid\})$.  In 1951,
Green~\cite{Green.1951} introduced several pre-orders on monoids
related to inclusion of ideals.  The standard terminology is to write
$\RR$ for right ideals, $\LL$ for left, $\JJ$ for two-sided. Let
$\KK\in\lbrace\RR,\LL,\JJ\rbrace$ and let $\tMonoid$ be a monoid. For
$x,y\in \tMonoid$, we write $x \leKK y$ when the $\KK$-ideal generated
by $x$ is contained in the $\KK$-ideal generated by $y$.  For example,
when $\KK=\LL$, this means that $x \leLL y$ if
$\tMonoid x\subseteq \tMonoid y$ or equivalently if $x=uy$ for some
$u\in \tMonoid$. There is a fourth relation writen $\HH$ which is
defined by $x \leHH y$ if and only if $x \leRR y$ and $x \leLL
y$. Each of these relations is clearly a pre-order (reflexive and
anti-symmetric) and therefore each naturally gives rise to an equivalence
relation denoted simply by $\KK$. The corresponding equivalence
classes for the relation are called $\KK$-classes. For example, $x \, \LL \, y$ if
and only if $\tMonoid x = \tMonoid y$.

\begin{definition}
  A monoid $\tMonoid$ is called $\KK$-trivial if all $\KK$-classes have 
  cardinality one, that is if the $\KK$ pre-order is anti-symmetric and
  therefore an actual partial order.
  In particular, $\tMonoid$ is $\HH$-trivial if $\tMonoid x = \tMonoid y$ and
  $x \tMonoid = y \tMonoid$ imply $x = y$.
\end{definition}
For the reader who is more familiar with Cayley graphs, this means that the
$\KK$-sided Cayley graph has only trivial (i.e. singletons) strongly connected
components.
\begin{definition}
  A monoid $\tMonoid$ is \emph{aperiodic} if for each $x\in\tMonoid$, there is
  an integer $n > 0$ such that $x^n = x^{n+1}$.
\end{definition}
As a consequence, $x^m = x^n$ for all $m\geq n$. Since we assume finiteness,
the quantifiers in the definition can be interchanged: $\tMonoid$ is aperiodic if
there is a global integer $N > 0$ such that, for all $x\in\tMonoid$,
$x^N = x^{N+1}$. In contrast to groups, a monoid can contain many
idempotents (i.e. an element $e$ such that $e^2=e$). Each 
idempotent is the unit of some groups, which are a sub-semigroups
of the monoid, the smallest of which is the trivial group $\{e\}$. 
Here is a well known characterization of aperiodic Monoids:
\begin{proposition}[\cite{Pin.2010} Proposition 4.22]
  Let $\tMonoid$ be a finite monoid. The following conditions are
  equivalent:
  \vspace{-0.4cm}
  \begin{multicols}{2}
    \begin{enumerate}
    \item $\tMonoid$ is aperiodic,
    \item $\tMonoid$ is $\HH$-trivial,
    \item all the groups in $\tMonoid$ are trivial.
    \end{enumerate}
  \end{multicols}
\end{proposition}

\section{The Okada monoid as a labelled arc-diagrams monoid}

The main goal of this section is to realize the Okada monoid as a
diagram monoid. After presenting basic definitions we develop a
rewriting technique to construct a bijection between the symmetric
group and the Okada monoid. A combinatorial description of normal
forms of this rewriting system naturally leads us to a new
concept of arc-diagrams which will be the main focus of the paper.

\subsection{Basic definition}

We start by defining the Okada algebras and monoids:
\begin{definition}
  Fix a positive integer $N$. Given a field $\K$, let $X=(x_1,\dots,x_{N-1})$
  and $Y=(y_1,\dots, y_{N-2})$ be two sequences of element of $\K$.
  The \defn{Okada algebra $\Okada(X, Y)$} is the algebra generated by
  $\set{\E{i}}{i=1\dots{N-1}}$ and subject to the relations
\begin{alignat}{3}
  \E{i}^2&=x_i\E{i} &\qquad& 1\leq i \leq N-1, \tag{$I(X,Y)$}\label{relIxy}\\
  \E{i}\E{j}&=\E{j}\E{i} &\qquad&\vert i-j\vert \geq 2,   \tag{C(X,Y)}\label{relCxy}\\
  \E{i+1}\E{i}\E{i+1}&=y_i\E{i+1} &\qquad&  1\leq i \leq N-2, \tag{S(X,Y)}\label{relSxy}
\end{alignat}
\end{definition}
If all the $X$'s and the $Y$'s are equal to $1$, then the Okada algebra
$\Okada(1,\dots,1; 1,\dots,1)$ is actually
the algebra of a monoid:
\begin{definition}\label{def:Okada-monoid}
  The \defn{Okada Monoid $\Okada$} is the monoid generated by
  $\set{\e{i}}{i=1\dots{N-1}}$ with relations
\begin{alignat}{3}
  \e{i}^2&=\e{i} &\qquad& 1\leq i \leq N-1, \tag{I}\label{relI}\\
  \e{i}\e{j}&=\e{j}\e{i} &\qquad&\vert i-j\vert \geq 2,   \tag{C}\label{relC}\\
  \e{i+1}\e{i}\e{i+1}&=\e{i+1} &\qquad&  1\leq i \leq N-2,      \tag{S}\label{relS}
\end{alignat}
\end{definition}

%\begin{definition}
%  The \defn{Jones Monoid $\Jones[N]$} is the quotient of the Okada
%  monoid $\Okada[N]$ by the extra relation
%\begin{alignat}{3}
%  \e{i}\e{i+1}\e{i}&=\e{i} &\qquad&  1\leq i \leq N-2,      \tag{T}\label{relJ}
%\end{alignat}
%\end{definition}

We denote the symmetric group, consisting of all permutations of $\setN$, by
$\SG[N]$. Permutations will be denoted by Greek letters such as $\sigma$ and
$\tau$. The symmetric group is generated by the elementary transpositions $s_i$
exchanging $i$ and $i+1$ and leaving all the other elements fixed.
\pagebreak[3]

We now recall Okada's construction of a basis of the Okada algebra:
\begin{definition}
  The \defn{code} of a permutation $\sigma\in\SG[N]$ is the $N$-tuple
  $\code(\sigma)=(c_1,\dots c_N)$ where
  \begin{equation}\label{eq.def.code}
    c_i \eqdef \card\set{k<i}{\sigma^{-1}(k)>\sigma^{-1}(i)}
    \stackrel{k = \sigma(j)}{=} \card\set{j>\sigma^{-1}(i)}{\sigma(j)< i }
    \,.
  \end{equation}
  Said otherwise, $c_i$ is the number of indices smaller than $i$ that are on the
  right of $i$ in the one line notation of $\sigma$.
  \end{definition}
It is well known that the product
\begin{equation}\label{eq.min.redword}
  \prod_{i=1}^n s_{i-1}s_{i-2}  \cdots s_{i - c_i}
\end{equation}
is the lexicographically minimal reduced factorization of the permutations $\sigma$.
We denote the associated index word by $\lexmin(\sigma)$.
\begin{definition}
  Fix a family of generators $(\G{i})$ of a monoid or an algebra. To a
  permutation $\sigma$ we associate the element $\G{\sigma}$ defined by
  \begin{equation}
    \G{\sigma} \eqdef \G{\lexmin(\sigma)}
    = \prod_{i=1}^n \G{i-1}\G{i-2}  \cdots \G{i - c_i} \,.
  \end{equation}
  where $(c_1,\dots c_N)=\code(\sigma)$, and the product is increasing
  with $i$.
\end{definition}
Example: If $\sigma = (4, 1, 5, 3, 2) \in \SG[5]$ then $\code(4, 1, 5, 3, 2) = (0, 0, 1, 3, 2)$ and
\[
  \G{\sigma}=
  1\cdot1\cdot\G{2}\cdot\G{3}\G{2}\G{1}\cdot\G{4}\G{3}
\]
Okada showed~\cite{Okada1994} that for generic values of $X$ and $Y$, the family
$\{\E{\sigma}  \mid  \sigma\in\SG[n]\}$ is a basis of the Okada
algebra. However, the proof relies on a non-trivial argument using representation
theory which doesn't apply to the degenerate case of the monoid. The next
section is devoted to showing that Okada's basis result is always true using elementary
rewriting techniques.

\subsection{Diamond diagram, heaps of dimers and rewriting}
\label{subsection:rewriting}
The goal of this section is to build a bijection between the symmetric group
and the the Okada monoid and to find a basis for the Okada algebra. A well-known,
standard technique for dealing with generators and relations is
rewriting. However we feel, in our case, that \emph{string} rewriting leads to a
somewhat cumbersome rewriting system and, most of all, to a poor combinatorial
understanding. Inspired by Viennot's theory of heaps of
dimers~\cite{ViennotHeap,ViennotWeb} and Knutson-Miller's pipe-dreams, 
we use instead planar tilings which we call \emph{diamond diagrams}. We start by
defining these tilings along with a proper notion of rewriting.

\begin{definition}
  A \defn{diamond diagram} is an arbitrary wide diagram of trapezoidal shape
  of height $N-1$ starting with a north-east diagonal and ending with a
  south-east diagonal, where each box can be either black or white. The rows
  are indexed starting with $1$ from bottom to top. We identify diagrams differing
  by empty south-east diagonals on the right.

  The \defn{reading} of a diamond diagram $\Delta$ is the index word $\ind{r}(\Delta)$
  obtained by recording the indices of the rows of the black boxes encountered
  when traversing each south-east diagonal from top to bottom, starting from
  the left-most diagonal.
  Its \defn{$\G{}$-evaluation} $\G{\Delta} \in G$ is the product $\G{\ind{r}(\Delta)}$
  associated to the reading $\ind{r}(\Delta)$ when $G$ is a monoid or an algebra.
\end{definition}
Note that the identification of diagrams differing only by
empty diagonals on the right is compatible with the reading. Here is an
example of two equivalent diagrams and their row numbering, together with the
reading order and their reading:
\[\vspace{5mm}
  \begin{tikzpicture}[diamond, diamond/bound={4}{5}]\matrix[diamond/matrix] {
      1 \\
      2 & 1 \\
      \bl3 & 2 & \bl1 \\
      \bl4 & 3 & 2 & 1 \\
      & 4 & 3 & \bl2 & 1 \\
      && \bl4 & \bl3 & 2 &\bl1\\
      &&& 4 & 3 & 2 & 1 \\
      &&&& 4 & 3 & 2 & 1 \\};
  \end{tikzpicture}=
  \begin{tikzpicture}[diamond, diamond/bound={4}{3}]\matrix[diamond/matrix] {
      1 \\
      2 & 1 \\
      \bl3 & 2 & \bl1 \\
      \bl4 & 3 & 2 & 1 \\
      & 4 & 3 & \bl2 & 1 \\
      && \bl4 & \bl3 & 2 & \bl1 \\};
  \end{tikzpicture}\quad
  \begin{tikzpicture}[diamond, diamond/bound={4}{3}]\matrix[diamond/matrix] {
      1 \\
      2 & 3 \\
      4 & 5 & 6 \\
      7 & 8 & 9 & 10 \\
      & 11 & 12 & 13 & 14 \\
      && 15 & 16 & 17 & 18 \\};
  \end{tikzpicture}\quad
  3 1 \cdot 4 \cdot 2 \cdot 431\vspace{5mm}
\]
\begin{definition}
%  Given an index sequence $\ind{i}$, we define its associated \defn{greedy
%    factorization} as the sequence obtained by iteratively factoring out, from left
%  to right, the largest prefix which is a subsequence of
%  $(n-1, n-2, \dots, 1)$.
    Given an index word $\ind{i}$, we define its associated \defn{greedy
    factorization} as $\ind{i} = \ind{i}^{(1)} \cdots \ind{i}^{(m)}$
    where $\ind{i}^{(\ell)}$ is a decreasing index word for each
    $\ell=1, \dots, m$ and $m \geq 1$ is minimal.

  From a greedy factorization we construct a \defn{greedy diamond diagram}
  $\greed{\ind{i}}$ as follows: For each greedy factor $\ind{i}^{(\ell)} = i_1^{(\ell)} \cdots i_{k_\ell}^{(\ell)}$
  we color the cells of a south-east diagonal, putting a black
  box in the row indexed $i^{(\ell)}_d$ for each $d$ from $1$ to $k_\ell$. Then
  $\greed{\ind{i}}$ is obtained by stacking, from left to right, the diagonals
  associated to its greedy factors such that each diagonal appears as far west
  as possible and so that their black boxes fit together in a diamond diagram.
%  From a greedy factorization we obtain a \defn{greedy diamond diagram}
%  $\greed{\ind{i}}$ as follows: for each subsequence $(j_1,\dots, j_k)$ of
%  $(n-1, n-2, \dots, 1)$, we build a south-east diagonal by putting a black
%  box in the row indexed $j_\ell$ for $\ell$ from $1$ to $k$. Then
%  $\greed{\ind{i}}$ is obtained by stacking, from left to right, the diagonals
%  associated to its greedy factors such that each diagonal appears as far west
%  as possible and so that their black boxes fit together in a diamond diagram.
\end{definition}
The last sentence of the previous definition means that the first occupied
diagonal begins at row $i^{(1)}_1$ while the second occupied
diagonal starts at row $i^{(2)}_1$ and so on.
%$\max(i_1+1,i_\ell)$ where $i_\ell$ is the first letter
%of the second factor, and so on.
Here are two greedy factorized
sequences together with their associated greedy diamond diagrams:
\[
  \greed{31\cdot42\cdot431} = \hskip-5mm
  \begin{tikzpicture}[diamond, diamond/bound={4}{3}]\matrix[diamond/matrix] {
      \ \\
      \  &\  \\
      \bl&\  &\bl \\
      \bl &\  &\bl&\  \\
          &\bl&\bl&\  &\bl \\};
  \end{tikzpicture}
  \qquad
  \greed{2\cdot42\cdot431\cdot32\cdot41} =  \hskip-5mm
  \begin{tikzpicture}[diamond, diamond/bound={4}{5}]\matrix[diamond/matrix] {
      \ \\
      \bl&\  \\
      \  &\  &\  \\
      \bl&\  &\bl&\  \\
         &\bl&\bl&\  &\bl \\
         &   &\  &\bl&\bl&\  \\
         &   &   &\bl&\  &\  &\bl\\};
  \end{tikzpicture}
\]
By construction, we have the following lemma
\begin{lemma}
  For any index word $\ind{i}$ the reading of the greedy diamond diagram
  $\greed{\ind{i}}$ is $\ind{i}$.
\end{lemma}

To handle Okada's presentation, we will use a notion of rewriting for diamond
diagrams: A \defn{pattern} is any $2\times2$ square diagram with black and white
boxes. When we say that a pattern appear in a diagram, we allow the top-most or
bottom-most box of the pattern to lie outside the diagram by extending the
diagram by two empty rows at the top and bottom.
A \defn{rewriting} entails replacing one pattern by another. When a
pattern extends outside the diagram, the extra box must be white in both
patterns. Here are two examples of rewritings of the same diagram:
\[
  \begin{tikzpicture}[diamond, diamond/bound={4}{3}]\matrix[diamond/matrix] {
      \ \\
      \ & \  \\
      \bl & \ & \bl\\
      \bl & \ & \ & \ \\
      & \ & |[fill=green!20]| & |[fill=green]| & \ \\
      && |[fill=green]| & |[fill=green]| & \ &\bl\\
    };
  \end{tikzpicture}
  \qquad
  \begin{tikzpicture}[diamond, diamond/bound={3}{0}]\matrix[diamond/matrix] {
      \nn \\
      & |[fill=green!20]| & |[fill=green]| \\
      & |[fill=green]| &|[fill=green]|  &  & \nn & \nn \\
    };
  \end{tikzpicture}
  \longmapsto
  \begin{tikzpicture}[diamond, diamond/bound={3}{0}]\matrix[diamond/matrix] {
      \nn \\
      & \bl & \ \\
      & \bl & \  & \nn & \nn \\
    };
  \end{tikzpicture}
  \qquad
  \begin{tikzpicture}[diamond, diamond/bound={4}{3}]\matrix[diamond/matrix] {
      \ \\
      \ & \  \\
      \bl & \ & \bl\\
      \bl & \ & \ & \ \\
      & \ & \bl & \ & \ \\
      && \bl & \ & \ &\bl\\
    };
  \end{tikzpicture}
\]
\[
  \begin{tikzpicture}[diamond, diamond/bound={5}{2}]\matrix[diamond/matrix] {
      \nn \\
      \ \\
      \ & |[fill=green!20]| & |[fill=green!20]| \\
      \bl & |[fill=green!20]| &|[fill=green]|\\
      \bl & \ & \ & \ \\
      & \ & \ & \bl & \ \\
      && \bl & \bl & \ &\bl &\nn \\
    };
  \end{tikzpicture}
  \qquad
  \begin{tikzpicture}[diamond, diamond/bound={3}{0}]\matrix[diamond/matrix] {
      \nn \\
      & |[fill=green!20]| & |[fill=green!20]| \\
      & |[fill=green!20]| & |[fill=green]| &  \nn & \nn \\
\\
    };
  \end{tikzpicture}
  \longmapsto
  \begin{tikzpicture}[diamond, diamond/bound={3}{0}]\matrix[diamond/matrix] {
      \nn \\
      & \bl & \ \\
      & \ & \ &  \nn & \nn  \\
    };
  \end{tikzpicture}
  \qquad
  \begin{tikzpicture}[diamond, diamond/bound={5}{2}]\matrix[diamond/matrix] {
      \nn \\
      \ \\
      \ & \bl  \\
      \bl & \ & \ \\
      \bl & \ & \ & \ \\
      & \ & \bl & \ & \ \\
      && \bl & \ & \ &\bl & \nn\\
    };
  \end{tikzpicture}
\]

A \defn{diagram rewriting system} is a set of pairs $(p_1, p_2)$ of
$2\times2$ patterns which allows one to replace an occurrence of the pattern
$p_1$ with $p_2$. Given a rewriting system $\Sys$, and two diagrams
$\Delta_1$ and $\Delta_2$ we write $\Delta_1 \mapsto_{\Sys} \Delta_2$
if there is a rewriting from $\Delta_1$ to $\Delta_2$ where the pair
of patterns used for the rewriting belongs to $\Sys$. We may omit the
system $\Sys$ in the subscript if it is clear from the context. Similarly, 
$\Delta_1\rewto_\Sys \Delta_2$ means there is a sequence (possibly empty)
of rewritings from $\Delta_1$ to $\Delta_2$.  Otherwise said,
$\rewto_\Sys$ is the reflexive and transitive closure of the relation
$\mapsto_{\Sys}$. \medskip

Let's recall some general, well-know terminology and facts about
rewriting. We refer the reader to the standard reference~\cite{RewAllThat} for
more details. Given a rewriting system $\Sys$, a \defn{normal form} for $\Sys$
is a diagram which cannot be rewritten using the rules of $\Sys$. A rewriting
system $\Sys$ is \defn{locally confluent}, if for all
$\Delta, \Delta_1, \Delta_2$ such that $\Delta\mapsto_\Sys \Delta_1$ and
$\Delta\mapsto_\Sys \Delta_2$ there exists a $\Delta_0$ such that
$\Delta_1\rewto_\Sys \Delta_0$ and $\Delta_2\rewto_\Sys \Delta_0$. In this
case we say that the pair $(\Delta_1, \Delta_2)$ is \defn{joinable}. 
In order to show local confluence it is sufficient to check the joinability of
all \defn{critical pairs}, i.e. pairs $(\Delta_1, \Delta_2)$ where the
rewritten pattern in $\Delta\mapsto_\Sys \Delta_1$ and
$\Delta\mapsto_\Sys \Delta_2$ overlap:
\begin{lemma}
  A rewriting system is locally confluent if and only if all critical pairs
  are joinable.
\end{lemma}
If a rewriting system \defn{terminates} (there is no infinite sequence of rewriting)
and is locally confluent then one can associate to any $\Delta$ a normal form
$\norf(\Delta)$ obtained by rewriting $\Delta$ until no possible
rewriting can be implemented. Confluence ensures that the result does not depend on the choice
of rewriting steps. The main result is the following fundamental
theorem (see \textit{e.g.} \cite[Theorem 2.1.9]{RewAllThat}):
\begin{proposition}
  Let $S$ be a terminating and locally confluent rewriting system. Then two
  objects are equivalent under the reflexive, symmetric and transitive closure
  of $\Sys$ if and only if their associated normal forms are equal.
\end{proposition}
\bigskip

We now describe the two rewriting systems which will allow us to deal with
Okada's relations:
\begin{definition}
  We say that two words in the free monoid $\FreeMonoid$
  are in the same
  \defn{commutation class} if they are related by a sequence
  of consecutive, pairwise exchanges of the form:
  \begin{equation*}
    i j \Mapsto j i\quad\text{for $|i-j|\geq2$}
  \end{equation*}
We refer to this elementary transformation as Relation~(\ref{relC}).
\end{definition}

Consider the rewriting system $\Comm$ containing the unique rule
$C\eqdef
\bigg(\begin{tikzpicture}[diamond, diamond/bound={3}{0}]\matrix[diamond/matrix] {
    \nn \\
    & \ & \ \\
    & \ & \bl & \nn& \nn\\
  };
\end{tikzpicture},
\begin{tikzpicture}[diamond, diamond/bound={3}{0}]\matrix[diamond/matrix] {
    \nn \\
    & \bl & \ \\
    & \ & \ & \nn& \nn\\
  };
\end{tikzpicture}\bigg)$. Clearly, for any rewriting $\Delta_1\to \Delta_2$
using this rule, the $\f{}$-readings of $\Delta_1$ and $\Delta_2$ are in the
same commutation class. It is also obvious that any diagram $\Delta$ whose
reading is $\ind{v}$ can be rewritten to $\greed{\ind{v}}$ by $\Comm$.
Finally, it is easy to see (this is basically Viennot's theory of heaps of
dimers) that the rewriting system $\Comm$ is locally confluent and
terminating, and so one has the following lemma:
\begin{lemma}
  Two words $\ind{v}$ and $\ind{w}$ are in the same commutation class if and
  only if any two diagrams $\Delta(\ind{v})$ and $\Delta(\ind{w})$ with
  respective readings $\ind{v}$ and $\ind{w}$ rewrite to the same
  normal form under the rewriting system $\Comm$.
\end{lemma}
We consider now the following two new rules
\[
I \eqdef \bigg(
\begin{tikzpicture}[diamond, diamond/bound={3}{0}]\matrix[diamond/matrix] {
    \nn \\
    & \bl & \ \\
    & \ & \bl & \nn& \nn\\
  };
\end{tikzpicture},
\begin{tikzpicture}[diamond, diamond/bound={3}{0}]\matrix[diamond/matrix] {
    \nn \\
    & \bl & \ \\
    & \ & \ & \nn& \nn\\
  };
\end{tikzpicture}
\bigg),\quad
S \eqdef \bigg(
\begin{tikzpicture}[diamond, diamond/bound={3}{0}]\matrix[diamond/matrix] {
    \nn \\
    & \bl & \bl \\
    & \ & \bl & \nn& \nn\\
  };
\end{tikzpicture},
\begin{tikzpicture}[diamond, diamond/bound={3}{0}]\matrix[diamond/matrix] {
    \nn \\
    & \bl & \ \\
    & \ & \ & \nn& \nn\\
  };
\end{tikzpicture}
\bigg)
% \qandq     %% Unused TL rewriting.
% T \eqdef \bigg(
% \begin{tikzpicture}[diamond, diamond/bound={3}{0}]\matrix[diamond/matrix] {
%     \nn \\
%     & \bl & \ \\
%     & \bl & \bl & \nn& \nn\\
%   };
% \end{tikzpicture},
% \begin{tikzpicture}[diamond, diamond/bound={3}{0}]\matrix[diamond/matrix] {
%     \nn \\
%     & \bl & \ \\
%     & \ & \ & \nn& \nn\\
%   };
% \end{tikzpicture}
% \bigg)
\,,
\]
and define the \defn{Okada rewriting system} as $\OSys\eqdef\{C, I, S\}$.
% \begin{equation}\label{eq.rewrite_rule}
% \OSys\eqdef\{C, I, S\}
% \qandq
% \TLSys\eqdef\{C, I, S, T\}\,.
% \end{equation}
\medskip

We shall explain how the $\OSys$ rewriting system for diamond diagrams is related with
Okada's monoid and algebra relations.
We define first the \defn{$\Okada$-evaluation} of a diamond diagram $\Delta$
with reading $\ind{v}$
to be the product $\e{\Delta} = \e{\ind{v}}$ understood as an
element of the monoid $\Okada$.

\begin{lemma}\label{lemma:diamond_rewriting}
  Any rewriting of a diamond diagram doesn't change its $\Okada$-evaluation.
  Moreover if two words $\ind{v}$ and $\ind{w}$ in $\FreeMonoid$
  differ by
  an Okada monoid relation, then there exists two diagrams $\Delta(\ind{v})$ and
  $\Delta(\ind{w})$ with respective readings $\ind{v}$ and
  $\ind{w}$ together with a rewriting between $\Delta(\ind{v})$ and $\Delta(\ind{w})$.
\end{lemma}
\begin{proof}
  Consider a diamond diagram rewriting
  $\Delta(\ind{v})\mapsto \Delta(\ind{w})$ with respective readings $\ind{v}$
  and $\ind{w}$. Let $i$ be the height of the middle row
  of the $2\times2$ pattern involved in the rewriting.
  Furthermore let $a < a+1 < b <b+1$ be the reading orders
  of the respective west, south, north, and east boxes of this
   $2\times2$ square. Accordingly, the readings of the diamond diagrams
   $\Delta(\ind{v})$ and $\Delta(\ind{w})$ share three common subwords:

  \begin{itemize}
  \item the word $ \ind{l}$ formed by reading the boxes with reading order positions $p<a$
  \item the word $\ind{r}$ formed by reading the boxes with reading order positions $p > b+1$
  \item the word $\ind{m}$ formed by reading the boxes with reading order positions $a+1 < p <b$
  \end{itemize}
  Note that $\ind{m}$ is necessarily a subword of $(i-2) \! \cdot  \! (i-3)  \cdots 1 \cdot  (N-1) \cdots (i+2)$.
  Depending of the rewriting, the word $\ind{v}$ has one of the following three forms
  \begin{equation*}
    \ind{v} = \ind{l}  \! \cdot \!  i  \!  \cdot  \! \ind{m} \! \cdot \!  i \! \cdot \! \ind{r}, \quad
    \ind{l}  \!  \cdot \! \ind{m} \! \cdot \! i  \! \cdot \!  \ind{r},
    \quad  \ind{l} \! \cdot \! i \! \cdot \! (i-1) \! \cdot \! \ind{m} \! \cdot \!  i \! \cdot \!  \ind{r}
    \end{equation*}
  while it must be the case that $\ind{w} =  \ind{l} \! \cdot \! i \! \cdot \! \ind{m} \! \cdot \! \ind{r}$.
  All indices occurring in $\ind{m}$ differ from $i$ by at least $2$ so
  we may interchange $i$ and the word $\ind{m}$ by repeated
  application of Relation ~\ref{relC} and consequently $\ind{v}$ is
  equivalent to one of the following three forms
  \begin{equation*}
    \ind{l}  \! \cdot \!  i  \!  \cdot  \! i \! \cdot \!  \ind{m} \! \cdot \! \ind{r}, \quad
    \ind{l}  \!  \cdot \! i \! \cdot \! \ind{m}  \! \cdot \!  \ind{r},
    \quad  \ind{l} \! \cdot \! i \! \cdot \! (i-1) \! \cdot \! i \! \cdot \!  \ind{m} \! \cdot \!  \ind{r}
    \end{equation*}
%  \begin{equation*}
%    \f{\ind{l}} \f{i}\f{i} \f{\ind{m}} \f{\ind{r}}
%    \text{ or } \f{\ind{l}} \f{i} \f{\ind{m}} \f{\ind{r}}
%    \text{ or } \f{\ind{l}} \f{i} \f{i-1} \f{i} \f{\ind{m}} \f{\ind{r}}
%  \end{equation*}
  The first is equivalent to $\ind{w}$ using Relation~\ref{relI}, the
  second is already $\ind{w}$ and the third is equivalent to $\ind{w}$
  using Relation~\ref{relS}. In all cases $\ind{v}\equiv \ind{w}$ modulo
  Okada's relations.

  Finally if two words $\ind{v}$ and $\ind{w}$ differ by an Okada
  relation one can build two associated diamond diagrams by writing the two patterns of the
  relation and completing the rest of the diagram by coloring
  a single box in each diagonal for each the remaining indices (i.e. those not participating
   in the $2\times2$ pattern).
\end{proof}
\pagebreak[2]

We now explain how to handle the $X$ and $Y$-parameters
for the Okada algebra when rewriting:
\begin{definition}
  The \defn{weight} of a rewriting $\Delta_1\mapsto \Delta_2$ is the element of $\K$
  defined by
  \begin{equation}
    \weight(\Delta_1\mapsto \Delta_2)\eqdef
    \begin{cases}
      1 & \text{if $\Delta_2$ is obtained by rule $C$ applied to $\Delta_1$} \\
      x_i & \text{if $\Delta_2$ is obtained by rule $I$ applied to $\Delta_1$ in row $i$} \\
      y_{i-1} & \text{if $\Delta_2$ is obtained by rule $S$ applied to $\Delta_1$ in row $i$}
    \end{cases}
  \end{equation}
  The weight of a rewriting sequence is the product of the weights of its
  rewriting steps.
\end{definition}
Like the monoid case, weights of a rewriting sequences relate to the Okada algebra as
follows:
\begin{lemma}
  For any rewriting sequence $\Delta_1\rewto \Delta_2$
  of diamond diagrams
  the equality
  \begin{equation}
    \E{\Delta_1} = \weight(\Delta_1\rewto \Delta_2) \E{\Delta_2}
  \end{equation}
  holds in the Okada algebra
  $\Okada(X, Y)$.
\end{lemma}
\begin{proof}
  One just needs to check that the equality holds for individual rewriting steps. 
  Indeed this is the case because each diamond rewriting step corresponds to a word rewriting up to
  commutation as in \cref{lemma:diamond_rewriting}.
\end{proof}
\medskip

The crucial observation is that rewriting system $\OSys$ has all
properties needed to faithfully compute within the Okada monoid and
algebra:
\begin{theorem}
  The rewriting system $\OSys$ terminates and is locally confluent.
  Moreover for any diamond diagram $\Delta$ and any rewriting sequence
  $\Delta\rewto\norf(\Delta)$, the weight $\weight(\Delta\rewto\norf(\Delta))$ depends only on $\Delta$
  and not on the rewriting sequence.
\end{theorem}
\begin{proof}
  Black boxes either disappear or move to the west. Therefore, the sum
  of the column positions of the black boxes strictly decrease with each
  rewriting step. Consequently it is not possible to have an infinite rewriting sequence.

  Regarding local confluence: Notice first that
  $2 \times 2$ patterns are either disjoint or
  else overlap on a single  box (shared either vertically or horizontally). \medskip

%  First, one can easily check that
%  there is no overlapping rewriting with the two rewriten patterns sharing two
%  entries. Therefore we have to consider two cases: overlapping patterns
%  sharing their top and bottom boxes and overlapping patterns sharing their
%  right and left boxes.  \medskip

  \cref{fig:column_Spairs} shows all overlapping patterns sharing 
  top and bottom boxes:
  \begin{figure}[ht]
    \[\begin{array}{c@{\qquad}c@{\qquad}c@{\qquad}c@{\qquad}c@{\qquad}c}
  \begin{tikzpicture}[diamond, diamond/bound={5}{-2}]\matrix[diamond/matrix] {
    &&&&\nn \\ &&&\nn\\
       &\  &\  \\
    \  &\  &\bl\\
    \  &\bl\\
  };\end{tikzpicture}&
  \begin{tikzpicture}[diamond, diamond/bound={5}{-2}]\matrix[diamond/matrix] {
    &&&&\nn \\ &&&\nn\\
       &\  &\  \\
    \bl&\  &\bl\\
    \  &\bl\\
  };\end{tikzpicture}&
  \begin{tikzpicture}[diamond, diamond/bound={5}{-2}]\matrix[diamond/matrix] {
    &&&&\nn \\ &&&\nn\\
       &\bl&\  \\
    \  &\  &\bl\\
    \  &\bl\\
  };\end{tikzpicture}&
  \begin{tikzpicture}[diamond, diamond/bound={5}{-2}]\matrix[diamond/matrix] {
    &&&&\nn \\ &&&\nn\\
       &\bl&\  \\
    \bl&\  &\bl\\
    \  &\bl\\
  };\end{tikzpicture}&
  \begin{tikzpicture}[diamond, diamond/bound={5}{-2}]\matrix[diamond/matrix] {
    &&&&\nn \\ &&&\nn\\
       &\bl&\bl\\
    \  &\  &\bl\\
    \  &\bl\\
  };\end{tikzpicture}&
  \begin{tikzpicture}[diamond, diamond/bound={5}{-2}]\matrix[diamond/matrix] {
    &&&&\nn \\ &&&\nn\\
       &\bl&\bl\\
    \bl&\  &\bl\\
    \  &\bl\\
  };\end{tikzpicture}\\
      1&x_{i+1}&x_{i-1}&x_{i-1}x_{i+1}&y_{i-2}&y_{i-2}x_{i+1}
    \end{array}
\]
\caption[Overlapping patterns sharing their top and bottom boxes]{%
  Overlapping patterns sharing their top and bottom boxes together
  with the common weight of the two rewriting steps (taken in either order). 
  The middle box is situated in the $i\ordinal$ row.}
    \label{fig:column_Spairs}.
  \end{figure}
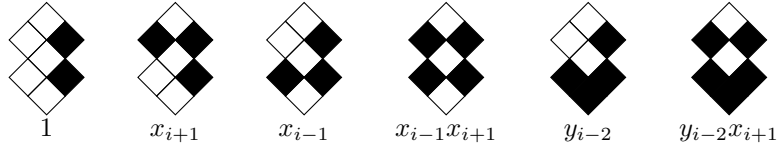
These cases are easy to handle since the middle box always remains white
regardless of whether we rewrite the top or the bottom pattern. Consequently
both rewritings can be executed, one after the other, and the order does not matter.
In each case we obtain
\begin{tikzpicture}[diamond, diamond/bound={5}{-2}]\matrix[diamond/matrix] {
    &&&&\nn \\ &&&\nn\\
       &\bl &\  \\
    \bl&\  &\  \\
    \  &\  \\
};\end{tikzpicture}. The weight is independent of the order in which the two possible
rewriting steps are taken, as shown in
\cref{fig:column_Spairs}.
\medskip

We now turn to the situation where the $2 \times 2$ patterns overlap horizontally,
sharing a single box as shown in \cref{fig:row_Spairs}.
\begin{figure}[ht]
  \[
    \begin{array}{c@{\qquad}c@{\qquad}c@{\qquad}c@{\qquad}c@{\qquad}c}
      \begin{tikzpicture}[diamond, diamond/bound={3}{1}]\matrix[diamond/matrix] {
    \nn \\
    &\  &\  \\
    &\  &\bl&\ \\
    &   &\  &\bl&\nn&\nn\\
  };\end{tikzpicture}&
\begin{tikzpicture}[diamond, diamond/bound={3}{1}]\matrix[diamond/matrix] {
    \nn \\
    &\bl&\  \\
    &\  &\bl&\ \\
    &   &\  &\bl&\nn&\nn\\
  };\end{tikzpicture}&
\begin{tikzpicture}[diamond, diamond/bound={3}{1}]\matrix[diamond/matrix] {
    \nn \\
    &\bl&\bl\\
    &\  &\bl&\ \\
    &   &\  &\bl&\nn&\nn\\
  };\end{tikzpicture}&
\begin{tikzpicture}[diamond, diamond/bound={3}{1}]\matrix[diamond/matrix] {
    \nn \\
    &\  &\  \\
    &\  &\bl&\bl\\
    &   &\  &\bl&\nn&\nn\\
  };\end{tikzpicture}&
\begin{tikzpicture}[diamond, diamond/bound={3}{1}]\matrix[diamond/matrix] {
    \nn \\
    &\bl&\  \\
    &\  &\bl&\bl\\
    &   &\  &\bl&\nn&\nn\\
  };\end{tikzpicture}&
\begin{tikzpicture}[diamond, diamond/bound={3}{1}]\matrix[diamond/matrix] {
    \nn \\
    &\bl&\bl\\
    &\  &\bl&\bl\\
    &   &\  &\bl&\nn&\nn\\
  };\end{tikzpicture}\\
        x_i & x_i^2 & x_iy_{i-1}&y_{i-1}&x_{i}y_{i-1}&y_{i-1}^2\\
      \end{array}
\]
\caption[Overlapping patterns sharing their left and right boxes]{%
  Overlapping patterns sharing their left and right boxes together with the
  weight of the rewriting to the normal form.}
    \label{fig:row_Spairs}.
  \end{figure}
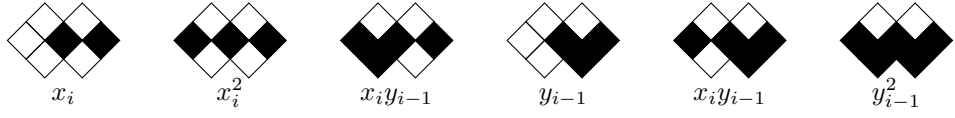
The first three pairs on the left are easy: If we rewrite the righthand pattern first
and then subsequently rewrite the lefthand pattern we obtain
\begin{tikzpicture}[diamond, diamond/bound={3}{1}]\matrix[diamond/matrix] {
    \nn \\
    &\bl&\ \\
    &\  &\  &\  \\
    &   &\  &\  &\nn&\nn\\
  };\end{tikzpicture}.
Alternatively, rewriting the lefthand pattern first produces
\begin{tikzpicture}[diamond, diamond/bound={3}{1}]\matrix[diamond/matrix] {
    \nn \\
    &\bl&\ \\
    &\  &\  &\  \\
    &   &\  &\bl&\nn&\nn\\
  };\end{tikzpicture}
which then rewrites to
\begin{tikzpicture}[diamond, diamond/bound={3}{1}]\matrix[diamond/matrix] {
    \nn \\
    &\bl&\ \\
    &\  &\bl&\  \\
    &   &\  &\  &\nn&\nn\\
  };\end{tikzpicture}
from which we finally obtain
\begin{tikzpicture}[diamond, diamond/bound={3}{1}]\matrix[diamond/matrix] {
    \nn \\
    &\bl&\ \\
    &\  &\  &\  \\
    &   &\  &\  &\nn&\nn\\
  };\end{tikzpicture}\,.
  In all three cases, the two possible rewriting sequences produce the same weights
  (displayed in \cref{fig:row_Spairs}).
  \bigskip
In each of the last three cases, we may rewrite the lefthand side first and obtain
\begin{tikzpicture}[diamond, diamond/bound={3}{1}]\matrix[diamond/matrix] {
    \nn \\
    &\bl&\ \\
    &\  &\  &\bl\\
    &   &\  &\bl&\nn&\nn\\
  };\end{tikzpicture} 
whose righthand side cannot be rewriten. 
Now move south-west from the west-most bottom box until we find a
rewritable pattern. This must exist since we will eventually encounter a $2 \times 2$ pattern containing either a
southern white box (possbily outside the diagram) or a western black box.

We obtain a picture, such as the first diagram on the left of the following sequence of
rewriting steps; the gray boxes can have any color provided they form a
rewritable pattern (the bottom-most gray box may by outside the shape, but in
this case it is white). We then may proceed with the rewriting sequence:
\[
  \begin{tikzpicture}[diamond, diamond/bound={7}{-2}]\matrix[diamond/matrix] {
    &&&\nn\\&&&\nn \\&&&\nn \\
    &   &\gr&\gr\\
    &   &\  &\bl\\
    &   &\  &\bl\\
    &\bl&\  &\bl\\
    &\  &\  &\bl\\
    &   &\  &\bl&\nn&\nn&\nn&\nn\\
  };\end{tikzpicture}\longmapsto
  \begin{tikzpicture}[diamond, diamond/bound={7}{-2}]\matrix[diamond/matrix] {
    &&&\nn\\&&&\nn \\&&&\nn \\
    &   &\bl&\  \\
    &   &\  &\  \\
    &   &\  &\bl\\
    &\bl&\  &\bl\\
    &\  &\  &\bl\\
    &   &\  &\bl&\nn&\nn&\nn&\nn\\
  };\end{tikzpicture}\longmapsto
  \begin{tikzpicture}[diamond, diamond/bound={7}{-2}]\matrix[diamond/matrix] {
    &&&\nn\\&&&\nn \\&&&\nn \\
    &   &\bl&\  \\
    &   &\bl&\  \\
    &   &\  &\  \\
    &\bl&\  &\bl\\
    &\  &\  &\bl\\
    &   &\  &\bl&\nn&\nn&\nn&\nn\\
  };\end{tikzpicture}\longmapsto\dots\longmapsto
%   \begin{tikzpicture}[diamond, diamond/bound={7}{-2}]\matrix[diamond/matrix] {
%     &&&\nn\\&&&\nn \\&&&\nn \\
%     &   &\bl&\  \\
%     &   &\bl&\  \\
%     &   &\bl&\  \\
%     &\bl&\  &\  \\
%     &\  &\  &\bl\\
%     &   &\  &\bl&\nn&\nn&\nn&\nn\\
%   };\end{tikzpicture}\]
% \[
% \longmapsto
%   \begin{tikzpicture}[diamond, diamond/bound={7}{-2}]\matrix[diamond/matrix] {
%     &&&\nn\\&&&\nn \\&&&\nn \\
%     &   &\bl&\  \\
%     &   &\bl&\  \\
%     &   &\bl&\  \\
%     &\bl&\bl&\  \\
%     &\  &\  &\  \\
%     &   &\  &\bl&\nn&\nn&\nn&\nn\\
%   };\end{tikzpicture}\longmapsto
  \begin{tikzpicture}[diamond, diamond/bound={7}{-2}]\matrix[diamond/matrix] {
    &&&\nn\\&&&\nn \\&&&\nn \\
    &   &\bl&\  \\
    &   &\bl&\  \\
    &   &\bl&\  \\
    &\bl&\bl&\  \\
    &\  &\bl&\  \\
    &   &\  &\  &\nn&\nn&\nn&\nn\\
  };\end{tikzpicture}\longmapsto
  \begin{tikzpicture}[diamond, diamond/bound={7}{-2}]\matrix[diamond/matrix] {
    &&&\nn\\&&&\nn \\&&&\nn \\
    &   &\bl&\  \\
    &   &\bl&\  \\
    &   &\bl&\  \\
    &\bl&\  &\  \\
    &\  &\  &\  \\
    &   &\  &\  &\nn&\nn&\nn&\nn\\
  };\end{tikzpicture}\,.
\]

Had we instead started by rewriting the righthand side of each of the last three overlapping patterns 
in \cref{fig:row_Spairs}, we would have obtained  
the following three kinds of subdiagrams:
\[
  \begin{tikzpicture}[diamond, diamond/bound={7}{-2}]\matrix[diamond/matrix] {
    &&&\nn\\&&&\nn \\&&&\nn \\
    &   &\gr&\gr\\
    &   &\  &\bl\\
    &   &\  &\bl\\
    &\  &\  &\bl\\
    &\  &\bl&\  \\
    &   &\  &\  &\nn&\nn&\nn&\nn\\
  };\end{tikzpicture}\,,\quad
  \begin{tikzpicture}[diamond, diamond/bound={7}{-2}]\matrix[diamond/matrix] {
    &&&\nn\\&&&\nn \\&&&\nn \\
    &   &\gr&\gr\\
    &   &\  &\bl\\
    &   &\  &\bl\\
    &\bl&\  &\bl\\
    &\  &\bl&\  \\
    &   &\  &\  &\nn&\nn&\nn&\nn\\
  };\end{tikzpicture}\,,\quad
  \begin{tikzpicture}[diamond, diamond/bound={7}{-2}]\matrix[diamond/matrix] {
    &&&\nn\\&&&\nn \\&&&\nn \\
    &   &\gr&\gr\\
    &   &\  &\bl\\
    &   &\  &\bl\\
    &\bl&\bl&\bl\\
    &\  &\bl&\  \\
    &   &\  &\  &\nn&\nn&\nn&\nn\\
  };\end{tikzpicture} 
\]
all of which admit the rewriting sequence 
\[
    \begin{tikzpicture}[diamond, diamond/bound={7}{-2}]\matrix[diamond/matrix] {
    &&&\nn\\&&&\nn \\&&&\nn \\
    &   &\gr  &\gr  \\
    &   &\  &\bl\\
    &   &\  &\bl\\
    &\bl&\  &\bl\\
    &\  &\ &\  \\
    &   &\  &\  &\nn&\nn&\nn&\nn\\
  };\end{tikzpicture}\longmapsto
    \begin{tikzpicture}[diamond, diamond/bound={7}{-2}]\matrix[diamond/matrix] {
    &&&\nn\\&&&\nn \\&&&\nn \\
    &   &\bl&\  \\
    &   &\  &\  \\
    &   &\  &\bl\\
    &\bl&\  &\bl\\
    &\  &\ &\  \\
    &   &\  &\  &\nn&\nn&\nn&\nn\\
  };\end{tikzpicture}\longmapsto\dots\longmapsto
  \begin{tikzpicture}[diamond, diamond/bound={7}{-2}]\matrix[diamond/matrix] {
    &&&\nn\\&&&\nn \\&&&\nn \\
    &   &\bl&\  \\
    &   &\bl&\  \\
    &   &\bl&\  \\
    &\bl&\  &\  \\
    &\  &\  &\  \\
    &   &\  &\  &\nn&\nn&\nn&\nn\\
  };\end{tikzpicture}\,.
\]
We have shown that all critical pairs are convergent and so the system is
locally confluent. Moreover we can check, in all cases, that the weights of both
rewritings agree. It follows, by a standard local to global rewriting argument, 
that the weight does not depend on the rewriting sequence.
\end{proof}

\begin{corollary}\label{cor:bij-mon-norf}
  The map $\Delta\mapsto\e{\Delta}$ sending a diagram to its evaluation in
  $\Okada$ defines a bijection between the set of diagram normal forms and
  $\Okada$.
\end{corollary}
The inverse bijection is denoted \defn{$\e{} \mapsto \Delta_{\e{}}$}, and we call
$\Delta_{\e{}}$ the \defn{canonical diamond diagram}
associated to $\e{}$ in $\Okada$.

\begin{corollary}\label{coro-weight-basis}
  For any diamond diagram $\Delta$, the equality
  \begin{equation}
    \E{\Delta} = \weight(\Delta\rewto \norf(\Delta))\,\E{\norf(\Delta)}
  \end{equation}
  holds in the Okada algebra $\Okada(X, Y)$. Moreover the family $(\E{\Delta})_{\Delta}$
  where $\Delta$ varies over the set of diagram normal forms is a basis of the Okada algebra.
 \end{corollary}
\begin{proof}
  Consider the quotient map $\phi$ from the free algebra
  $\K\FreeMonoid$ to $\Okada(X,Y)$
  mapping $\ind{v} \mapsto \E{\ind{v}}$ for each index word
  $\ind{v} \in \FreeMonoid$.
  We partition the set of index words into equivalence
  classes based on their evaluations within the Okada monoid, i.e.
  $\ind{v}$ and $\ind{w}$ are equivalent if and only
  if $\e{\ind{v}}= \e{\ind{w}}$.

 For any $\mathsf{e} \in\Okada$, let
  $W_\mathsf{e}$ denote the linear span of $\ind{w} \in \FreeMonoid$ such that $\e{\ind{w}} =
  \mathsf{e}$. Obviously
  $\K\FreeMonoid =\bigoplus_{\mathsf{e} \in\Okada} W_\mathsf{e}$.

  The kernel of $\phi$ is the ideal generated by the Okada relations. It has a
  basis composed of either individual index words or differences $\ind{v} - \lambda \ind{w}$ where
  $\lambda \in\K$ and $\ind{v}, \ind{w}$ are equivalent index words. As a
  consequence the kernel of $\phi$ can be written as a direct sum
  \[
    \ker(\phi) = \bigoplus_{\mathsf{e} \in\Okada} \ker(\phi) \cap W_\mathsf{e} \,.
  \]
  Therefore, the Okada algebra is isomorphic as a vector space to
  \[
    \Okada(X,Y)
    \isom \left.\left(\bigoplus_{\mathsf{e} \in\Okada} W_\mathsf{e}  \right)
          \middle/ \left(\bigoplus_{ \mathsf{e} \in\Okada} \ker(\phi) \cap W_\mathsf{e} \right)\right.
    \isom \bigoplus_{\mathsf{e} \in\Okada} (W_\mathsf{e}  / \ker(\phi) \cap W_\mathsf{e})
    \isom \bigoplus_{\mathsf{e} \in\Okada} \K \E{\Delta(\mathsf{e})}\,.
  \]
  where $\Delta(\mathsf{e})$ is the normal form of the diamond diagram associated to $\mathsf{e}$.
\end{proof}
We now describe the normal forms:
\begin{proposition}
  A diamond diagram $\Delta$ is a normal form if and only if it has a triangular
  shape (possibly after removing irrelevant empty diagonals from the east),
  and if each black box has either a black box to its northwest or else lies on
  the western border of the diamond diagram.
\end{proposition}
\begin{proof}
  First, all rewritable patterns in $\OSys$ have a black eastern box and a
  white northern box. So they all have a black box along with a white box to its
  northwest. If there are no such boxes, the diagram is in normal form.

  Conversely, suppose that $\Delta$ has a black box which is not on the western border
  and has a white box to its north-west. Call $b$ one of the west-most 
  boxes of this kind and consider the $2\times2$ pattern whose eastern box is $b$. The pattern
  \begin{tikzpicture}[diamond, diamond/bound={3}{0}]\matrix[diamond/matrix] {
      \nn \\
      & \ & \bl \\
      & \ & \bl & \nn& \nn\\};
  \end{tikzpicture}
  is not possible because there is a white box north-west
  of the southern box of the pattern (which is black), contradicting the assumption that $b$ is as west as
  possible. The three remaining possible colors of the western and southern
  boxes make rewritable patterns.
\end{proof}
\begin{proposition}\label{prop:lexmin-normal-form}
  Let $\Delta$ be a normal form diamond diagram. Define $c(\Delta)=(c_0, \dots c_{N-1})$ where
  $c_0=0$ and $c_i$ is the number of black boxes in the $i\ordinal$
  southeast diagonal. Then
  $c(\Delta)$ is a code (i.e.  $0\leq c_i \leq i$ for all $i$)
  and $\Delta\mapsto c(\Delta)$ is a bijection between normal forms of
  diamond diagrams and codes such that the reading $\ind{r}(\Delta)$
  is the $\lexmin$ word associated to $c(\Delta)$.
\end{proposition}
\begin{proof}
  The length of the diagonals ensure that $c(\Delta)$ is a code and since the black
  boxes must lie north-west within each diagonal, one can recover $\Delta$ from
  its code. The equality of the reading and LexMin words follows from the
  definitions.
\end{proof}
\cref{fig.six_normal_forms} shows the $6$ normal forms for $N=3$ and their
readings. \cref{fig.example_normal_forms} shows a few larger examples.
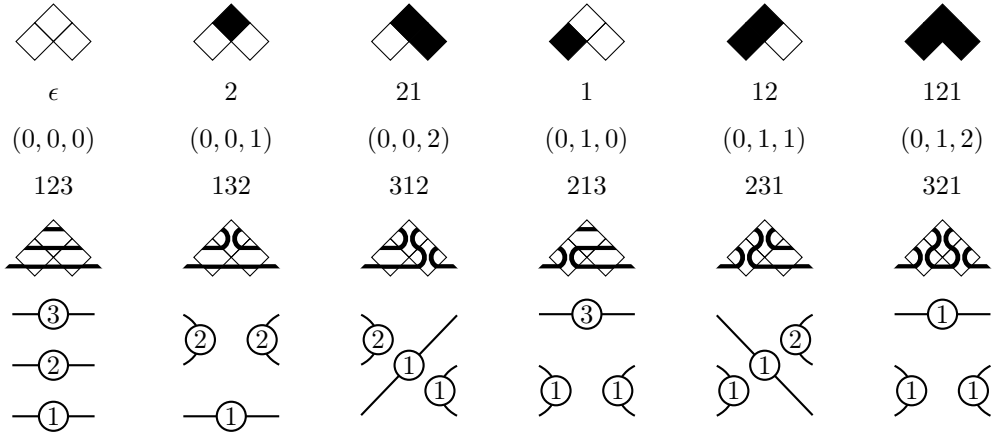
\begin{figure}[ht]
  \[\begin{array}{cccccc}
  \begin{tikzpicture}[diamond, diamond/bound={2}{1}]\matrix[diamond/matrix] {
      \  \\
      \  &\  \\
    };\end{tikzpicture}&
  \begin{tikzpicture}[diamond, diamond/bound={2}{1}]\matrix[diamond/matrix] {
      \  \\
      \bl&\  \\
    };\end{tikzpicture}&
  \begin{tikzpicture}[diamond, diamond/bound={2}{1}]\matrix[diamond/matrix] {
      \  \\
      \bl&\bl\\
    };\end{tikzpicture}&
  \begin{tikzpicture}[diamond, diamond/bound={2}{1}]\matrix[diamond/matrix] {
      \bl\\
      \  &\  \\
    };\end{tikzpicture}&
  \begin{tikzpicture}[diamond, diamond/bound={2}{1}]\matrix[diamond/matrix] {
      \bl\\
      \bl&\  \\
    };\end{tikzpicture}&
  \begin{tikzpicture}[diamond, diamond/bound={2}{1}]\matrix[diamond/matrix] {
      \bl\\
      \bl&\bl\\
    };\end{tikzpicture} \\[4mm]
      %%%%%%%%%%%%%%%%
      \epsilon&2&21&1&12&121 \\[2mm]
      (0,0,0)&(0,0,1)&(0,0,2)&(0,1,0)&(0,1,1)&(0,1,2) \\[2mm]
      123&132&312&213&231&321 \\[2mm]
      %%%%%%%%%%%%%%%%
  \begin{tikzpicture}[diamond, diamond/bound={3}{2}]\matrix[diamond/matrix] {
      \stt\\
      \st&\stt\\
      \st&\st&\stt\\
    };\end{tikzpicture}&
  \begin{tikzpicture}[diamond, diamond/bound={3}{2}]\matrix[diamond/matrix] {
      \stt\\
      \st&\stt\\
      \tu&\st&\stt\\
    };\end{tikzpicture}&
  \begin{tikzpicture}[diamond, diamond/bound={3}{2}]\matrix[diamond/matrix] {
      \stt\\
      \st&\stt\\
      \tu&\tu&\stt\\
    };\end{tikzpicture}&
  \begin{tikzpicture}[diamond, diamond/bound={3}{2}]\matrix[diamond/matrix] {
      \stt\\
      \tu&\stt\\
      \st&\st&\stt\\
    };\end{tikzpicture}&
  \begin{tikzpicture}[diamond, diamond/bound={3}{2}]\matrix[diamond/matrix] {
      \stt\\
      \tu&\stt\\
      \tu&\st&\stt\\
    };\end{tikzpicture}&
  \begin{tikzpicture}[diamond, diamond/bound={3}{2}]\matrix[diamond/matrix] {
      \stt\\
      \tu&\stt\\
      \tu&\tu&\stt\\
    };\end{tikzpicture}
      \\
  %%%%%%%%%%%%%%%%%%%%%%%%%%%%%%%%%%%%%%%%%%%%%%%%%
\begin{tikzpicture}[scale=0.45, thick, baseline={(current bounding box.east)}]
\tikzstyle{vertex} = [shape=rectangle, minimum height=15pt, minimum width=0pt, inner sep=0pt]
\tikzstyle{mid} = [draw, fill=white, shape=circle, minimum size=2pt, inner sep=1pt]
\node[vertex] (G--1) at (1.24, 0.0) {};
\node[vertex] (G-1) at (-1.24, 0.0) {};
\node[vertex] (G--2) at (1.24, 1.5) {};
\node[vertex] (G-2) at (-1.24, 1.5) {};
\node[vertex] (G--3) at (1.24, 3.0) {};
\node[vertex] (G-3) at (-1.24, 3.0) {};
\draw (G--1) -- (G-1) node[pos=0.5,mid] (M-1-1) { 1 };
\draw (G--2) -- (G-2) node[pos=0.5,mid] (M-2-2) { 2 };
\draw (G--3) -- (G-3) node[pos=0.5,mid] (M-3-3) { 3 };
\end{tikzpicture}
&\begin{tikzpicture}[scale=0.45, thick, baseline={(current bounding box.east)}]
\tikzstyle{vertex} = [shape=rectangle, minimum height=15pt, minimum width=0pt, inner sep=0pt]
\tikzstyle{mid} = [draw, fill=white, shape=circle, minimum size=2pt, inner sep=1pt]
\node[vertex] (G--1) at (1.44, 0.0) {};
\node[vertex] (G-1) at (-1.44, 0.0) {};
\node[vertex] (G-2) at (-1.44, 1.5) {};
\node[vertex] (G-3) at (-1.44, 3.0) {};
\node[vertex] (G--3) at (1.44, 3.0) {};
\node[vertex] (G--2) at (1.44, 1.5) {};
\draw (G--1) -- (G-1) node[pos=0.5,mid] (M-1-1) { 1 };
\draw (G-2) .. controls +(0.7, 0.35) and +(0.7, -0.35) .. (G-3) node[pos=0.5,mid] (M2-3) { 2 };
\draw (G--3) .. controls +(-0.7, -0.35) and +(-0.7, 0.35) .. (G--2) node[pos=0.5,mid] (M-3--2) { 2 };
\end{tikzpicture}&
\begin{tikzpicture}[scale=0.45, thick, baseline={(current bounding box.east)}]
\tikzstyle{vertex} = [shape=rectangle, minimum height=15pt, minimum width=0pt, inner sep=0pt]
\tikzstyle{mid} = [draw, fill=white, shape=circle, minimum size=2pt, inner sep=1pt]
\node[vertex] (G--3) at (1.44, 3.0) {};
\node[vertex] (G-1) at (-1.44, 0.0) {};
\node[vertex] (G-2) at (-1.44, 1.5) {};
\node[vertex] (G-3) at (-1.44, 3.0) {};
\node[vertex] (G--2) at (1.44, 1.5) {};
\node[vertex] (G--1) at (1.44, 0.0) {};
\draw (G--3) -- (G-1) node[pos=0.5,mid] (M-3-1) { 1 };
\draw (G-2) .. controls +(0.7, 0.35) and +(0.7, -0.35) .. (G-3) node[pos=0.5,mid] (M2-3) { 2 };
\draw (G--2) .. controls +(-0.7, -0.35) and +(-0.7, 0.35) .. (G--1) node[pos=0.5,mid] (M-2--1) { 1 };
\end{tikzpicture}&
\begin{tikzpicture}[scale=0.45, thick, baseline={(current bounding box.east)}]
\tikzstyle{vertex} = [shape=rectangle, minimum height=15pt, minimum width=0pt, inner sep=0pt]
\tikzstyle{mid} = [draw, fill=white, shape=circle, minimum size=2pt, inner sep=1pt]
\node[vertex] (G-1) at (-1.44, 0.0) {};
\node[vertex] (G-2) at (-1.44, 1.5) {};
\node[vertex] (G--2) at (1.44, 1.5) {};
\node[vertex] (G--1) at (1.44, 0.0) {};
\node[vertex] (G--3) at (1.44, 3.0) {};
\node[vertex] (G-3) at (-1.44, 3.0) {};
\draw (G-1) .. controls +(0.7, 0.35) and +(0.7, -0.35) .. (G-2) node[pos=0.5,mid] (M1-2) { 1 };
\draw (G--2) .. controls +(-0.7, -0.35) and +(-0.7, 0.35) .. (G--1) node[pos=0.5,mid] (M-2--1) { 1 };
\draw (G--3) -- (G-3) node[pos=0.5,mid] (M-3-3) { 3 };
\end{tikzpicture}&
\begin{tikzpicture}[scale=0.45, thick, baseline={(current bounding box.east)}]
\tikzstyle{vertex} = [shape=rectangle, minimum height=15pt, minimum width=0pt, inner sep=0pt]
\tikzstyle{mid} = [draw, fill=white, shape=circle, minimum size=2pt, inner sep=1pt]
\node[vertex] (G--1) at (1.44, 0.0) {};
\node[vertex] (G-3) at (-1.44, 3.0) {};
\node[vertex] (G-1) at (-1.44, 0.0) {};
\node[vertex] (G-2) at (-1.44, 1.5) {};
\node[vertex] (G--3) at (1.44, 3.0) {};
\node[vertex] (G--2) at (1.44, 1.5) {};
\draw (G--1) -- (G-3) node[pos=0.5,mid] (M-1-3) { 1 };
\draw (G-1) .. controls +(0.7, 0.35) and +(0.7, -0.35) .. (G-2) node[pos=0.5,mid] (M1-2) { 1 };
\draw (G--3) .. controls +(-0.7, -0.35) and +(-0.7, 0.35) .. (G--2) node[pos=0.5,mid] (M-3--2) { 2 };
\end{tikzpicture}&
\begin{tikzpicture}[scale=0.45, thick, baseline={(current bounding box.east)}]
\tikzstyle{vertex} = [shape=rectangle, minimum height=15pt, minimum width=0pt, inner sep=0pt]
\tikzstyle{mid} = [draw, fill=white, shape=circle, minimum size=2pt, inner sep=1pt]
\node[vertex] (G--3) at (1.44, 3.0) {};
\node[vertex] (G-3) at (-1.44, 3.0) {};
\node[vertex] (G-1) at (-1.44, 0.0) {};
\node[vertex] (G-2) at (-1.44, 1.5) {};
\node[vertex] (G--2) at (1.44, 1.5) {};
\node[vertex] (G--1) at (1.44, 0.0) {};
\draw (G--3) -- (G-3) node[pos=0.5,mid] (M-3-3) { 1 };
\draw (G-1) .. controls +(0.7, 0.35) and +(0.7, -0.35) .. (G-2) node[pos=0.5,mid] (M1-2) { 1 };
\draw (G--2) .. controls +(-0.7, -0.35) and +(-0.7, 0.35) .. (G--1) node[pos=0.5,mid] (M-2--1) { 1 };
\end{tikzpicture}
  \end{array}\]
\caption[The 6 normal forms in rank $3$]{%
  The 6 normal forms of the $\Okada[3]$-rewriting system, their readings,
  associated codes, permutations, fully-packed loop configurations, and labelled arc-diagrams.
}
  \label{fig.six_normal_forms}
\end{figure}
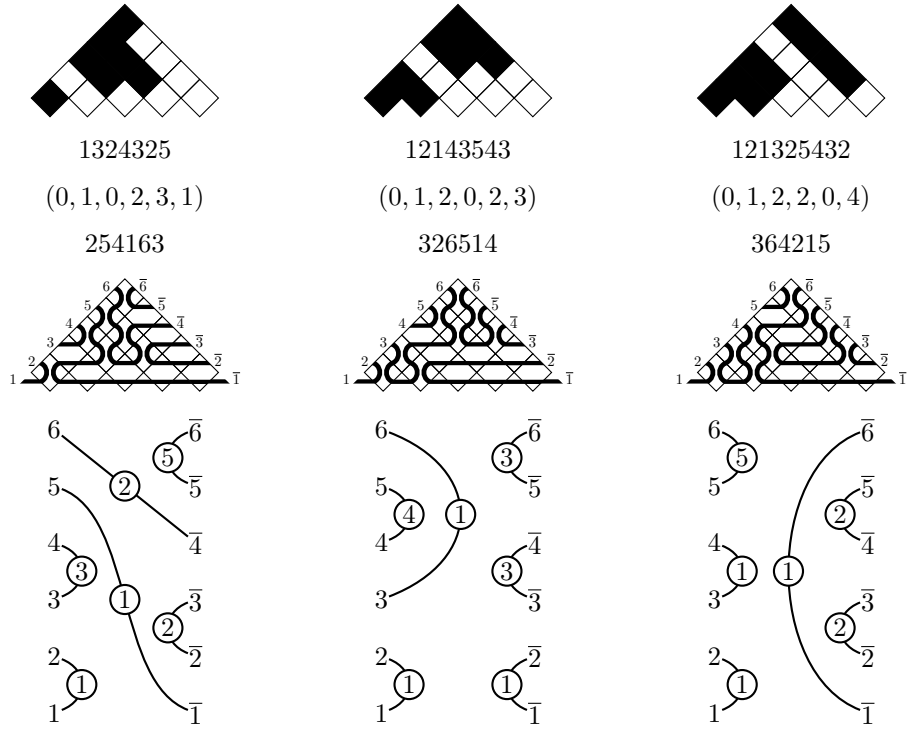
\begin{figure}
\[
  \begin{array}{c@{\qquad\qquad}c@{\qquad\qquad}c}
\begin{tikzpicture}[diamond,diamond/bound={5}{1}]\matrix[diamond/matrix] {
    \bl\\
    \  &\  \\
    \bl&\bl&\  \\
    \bl&\bl&\bl&\  \\
    \bl&\  &\  &\  &\  \\
  };\end{tikzpicture}
&
\begin{tikzpicture}[diamond, diamond/bound={5}{1}]\matrix[diamond/matrix] {
    \bl\\
    \bl&\bl\\
    \  &\  &\  \\
    \bl&\bl&\  &\  \\
    \bl&\bl&\bl&\  &\  \\
  };\end{tikzpicture}
&
\begin{tikzpicture}[diamond, diamond/bound={5}{1}]\matrix[diamond/matrix] {
    \bl\\
    \bl&\bl\\
    \bl&\bl&\  \\
    \  &\  &\  &\  \\
    \bl&\bl&\bl&\bl&\  \\
  };\end{tikzpicture}
    \\[8mm]
    %%%%%%%%%%%%%%%%%%%%%%%%%%%%%
    1324325 & 12143543 & 121325432 \\[2mm]
    (0,1,0,2,3,1) & (0,1,2,0,2,3) & (0,1,2,2,0,4) \\[2mm]
    254163 & 326514 & 364215 \\[2mm]
%%%%%%%%%%%%%%%%%%%%%%%%%%%%
\begin{tikzpicture}[diamond, diamond/bound={6}{1}]\matrix[diamond/matrix] {
    \nn\\
    \TLlab1 &\stt\\
    \TLlab2 &\tu&\stt\\
    \TLlab3 &\st&\st&\stt\\
    \TLlab4 &\tu&\tu&\st&\stt\\
    \TLlab5 &\tu&\tu&\tu&\st&\stt\\
    \TLlab6 &\tu&\st&\st&\st&\st&\stt\\
    &\TLlab{\ov6}&\TLlab{\ov5}&\TLlab{\ov4}&\TLlab{\ov3}&\TLlab{\ov2}&\TLlab{\ov1}
    &\nn\\
  };\end{tikzpicture}
&
\begin{tikzpicture}[diamond, diamond/bound={6}{1}]\matrix[diamond/matrix] {
    \nn\\
    \TLlab1 &\stt\\
    \TLlab2 &\tu&\stt\\
    \TLlab3 &\tu&\tu&\stt\\
    \TLlab4 &\st&\st&\st&\stt\\
    \TLlab5 &\tu&\tu&\st&\st&\stt\\
    \TLlab6 &\tu&\tu&\tu&\st&\st&\stt\\
    &\TLlab{\ov6}&\TLlab{\ov5}&\TLlab{\ov4}&\TLlab{\ov3}&\TLlab{\ov2}&\TLlab{\ov1}
    &\nn\\
  };\end{tikzpicture}
&
\begin{tikzpicture}[diamond, diamond/bound={6}{1}]\matrix[diamond/matrix] {
    \nn\\
    \TLlab1 &\stt\\
    \TLlab2 &\tu&\stt\\
    \TLlab3 &\tu&\tu&\stt\\
    \TLlab4 &\tu&\tu&\st&\stt\\
    \TLlab5 &\st&\st&\st&\st&\stt\\
    \TLlab6 &\tu&\tu&\tu&\tu&\st&\stt\\
    &\TLlab{\ov6}&\TLlab{\ov5}&\TLlab{\ov4}&\TLlab{\ov3}&\TLlab{\ov2}&\TLlab{\ov1}
    &\nn\\
  };\end{tikzpicture}
  \\
  \begin{tikzpicture}[xscale=0.65,yscale=0.5, thick, baseline={(current bounding box.east)}]
\tikzstyle{vertex} = [shape=rectangle, minimum height=15pt, minimum width=0pt, inner sep=0pt]
\tikzstyle{mid} = [draw, fill=white, shape=circle, minimum size=2pt, inner sep=1pt]
\node[vertex] (G--1) at (1.44, 0.0) {$\ov{1}$};
\node[vertex] (G-5) at (-1.44, 6.0) {$5$};
\node[vertex] (G--4) at (1.44, 4.5) {$\ov{4}$};
\node[vertex] (G-6) at (-1.44, 7.5) {$6$};
\node[vertex] (G-1) at (-1.44, 0.0) {$1$};
\node[vertex] (G-2) at (-1.44, 1.5) {$2$};
\node[vertex] (G--3) at (1.44, 3.0) {$\ov{3}$};
\node[vertex] (G--2) at (1.44, 1.5) {$\ov{2}$};
\node[vertex] (G-3) at (-1.44, 3.0) {$3$};
\node[vertex] (G-4) at (-1.44, 4.5) {$4$};
\node[vertex] (G--6) at (1.44, 7.5) {$\ov{6}$};
\node[vertex] (G--5) at (1.44, 6.0) {$\ov{5}$};
\draw (G--1) .. controls +(-1.632, 0.816) and +(1.632, -0.816) .. (G-5) node[pos=0.5,mid] (M-1-5) { 1 };
\draw (G--4) -- (G-6) node[pos=0.5,mid] (M-4-6) { 2 };
\draw (G-1) .. controls +(0.7, 0.35) and +(0.7, -0.35) .. (G-2) node[pos=0.5,mid] (M1-2) { 1 };
\draw (G--3) .. controls +(-0.7, -0.35) and +(-0.7, 0.35) .. (G--2) node[pos=0.5,mid] (M-3--2) { 2 };
\draw (G-3) .. controls +(0.7, 0.35) and +(0.7, -0.35) .. (G-4) node[pos=0.5,mid] (M3-4) { 3 };
\draw (G--6) .. controls +(-0.7, -0.35) and +(-0.7, 0.35) .. (G--5) node[pos=0.5,mid] (M-6--5) { 5 };
\end{tikzpicture}&
\begin{tikzpicture}[xscale=0.65,yscale=0.5, thick, baseline={(current bounding box.east)}]
\tikzstyle{vertex} = [shape=rectangle, minimum height=15pt, minimum width=0pt, inner sep=0pt]
\tikzstyle{mid} = [draw, fill=white, shape=circle, minimum size=2pt, inner sep=1pt]
\node[vertex] (G-3) at (-2.16, 3.0) {$3$};
\node[vertex] (G-6) at (-2.16, 7.5) {$6$};
\node[vertex] (G--2) at (0.96, 1.5) {$\ov{2}$};
\node[vertex] (G--1) at (0.96, 0.0) {$\ov{1}$};
\node[vertex] (G-1) at (-2.16, 0.0) {$1$};
\node[vertex] (G-2) at (-2.16, 1.5) {$2$};
\node[vertex] (G--6) at (0.96, 7.5) {$\ov{6}$};
\node[vertex] (G--5) at (0.96, 6.0) {$\ov{5}$};
\node[vertex] (G-4) at (-2.16, 4.5) {$4$};
\node[vertex] (G-5) at (-2.16, 6.0) {$5$};
\node[vertex] (G--4) at (0.96, 4.5) {$\ov{4}$};
\node[vertex] (G--3) at (0.96, 3.0) {$\ov{3}$};
\draw (G-3) .. controls +(2.10, 1.05) and +(2.10, -1.05) .. (G-6) node[pos=0.5,mid] (M3-6) { 1 };
\draw (G--2) .. controls +(-0.7, -0.35) and +(-0.7, 0.35) .. (G--1) node[pos=0.5,mid] (M-2--1) { 1 };
\draw (G-1) .. controls +(0.7, 0.35) and +(0.7, -0.35) .. (G-2) node[pos=0.5,mid] (M1-2) { 1 };
\draw (G--6) .. controls +(-0.7, -0.35) and +(-0.7, 0.35) .. (G--5) node[pos=0.5,mid] (M-6--5) { 3 };
\draw (G-4) .. controls +(0.7, 0.35) and +(0.7, -0.35) .. (G-5) node[pos=0.5,mid] (M4-5) { 4 };
\draw (G--4) .. controls +(-0.7, -0.35) and +(-0.7, 0.35) .. (G--3) node[pos=0.5,mid] (M-4--3) { 3 };
\end{tikzpicture}&
\begin{tikzpicture}[xscale=0.65,yscale=0.5, thick, baseline={(current bounding box.east)}]
\tikzstyle{vertex} = [shape=rectangle, minimum height=15pt, minimum width=0pt, inner sep=0pt]
\tikzstyle{mid} = [draw, fill=white, shape=circle, minimum size=2pt, inner sep=1pt]
\node[vertex] (G-3) at (-0.96, 3.0) {$3$};
\node[vertex] (G-4) at (-0.96, 4.5) {$4$};
\node[vertex] (G--6) at (2.16, 7.5) {$\ov{6}$};
\node[vertex] (G--1) at (2.16, 0.0) {$\ov{1}$};
\node[vertex] (G-1) at (-0.96, 0.0) {$1$};
\node[vertex] (G-2) at (-0.96, 1.5) {$2$};
\node[vertex] (G--5) at (2.16, 6.0) {$\ov{5}$};
\node[vertex] (G--4) at (2.16, 4.5) {$\ov{4}$};
\node[vertex] (G-5) at (-0.96, 6.0) {$5$};
\node[vertex] (G-6) at (-0.96, 7.5) {$6$};
\node[vertex] (G--3) at (2.16, 3.0) {$\ov{3}$};
\node[vertex] (G--2) at (2.16, 1.5) {$\ov{2}$};
\draw (G-3) .. controls +(0.7, 0.35) and +(0.7, -0.35) .. (G-4) node[pos=0.5,mid] (M3-4) { 1 };
\draw (G--6) .. controls +(-2.10, -1.05) and +(-2.10, 1.05) .. (G--1) node[pos=0.5,mid] (M-6--1) { 1 };
\draw (G-1) .. controls +(0.7, 0.35) and +(0.7, -0.35) .. (G-2) node[pos=0.5,mid] (M1-2) { 1 };
\draw (G--5) .. controls +(-0.7, -0.35) and +(-0.7, 0.35) .. (G--4) node[pos=0.5,mid] (M-5--4) { 2 };
\draw (G-5) .. controls +(0.7, 0.35) and +(0.7, -0.35) .. (G-6) node[pos=0.5,mid] (M5-6) { 5 };
\draw (G--3) .. controls +(-0.7, -0.35) and +(-0.7, 0.35) .. (G--2) node[pos=0.5,mid] (M-3--2) { 2 };
\end{tikzpicture}
  \end{array}\]
    \caption[Examples of normal forms]{Some examples of normal forms, their readings,
  associated codes, permutations, fully-packed loop configurations, and labelled arc-diagrams.}
\label{fig.example_normal_forms}
  \end{figure}

We've established now the first main result of this section:
\begin{theorem}\label{theo:bij-Sn_Okada}
  For any non-negative integer $N$, the monoid $\Okada$ has cardinality
  $N!$. Moreover the map $\sigma\mapsto \e{\sigma}$ is a bijection from
  $\SG[N]$ to $\Okada$.

  For any non-negative integer $N$ and any $X$, $Y$, the algebra $\Okada(X,Y)$
  has dimension $N!$ and the family $(\E{\sigma})_{\sigma\in\SG}$ is a basis.
\end{theorem}
\begin{remark}
  The previous description implies that
  \begin{align}
    \label{eq:left-coset-decomposition}
    \Okada[N+1] &= \Okada[N] \ssp \sqcup \ssp \Okada[N]\E{N}
                  \ssp \sqcup \ssp \Okada[N]\E{N}\E{N-1} \ssp \sqcup \dots
                  \sqcup \ssp  \Okada[N]\E{N}\E{N-1} \cdots \E{2}\E{1} \\
    \label{eq:right-coset-decomposition}
                &= \Okada[N] \ssp \sqcup \ssp \E{N}\Okada[N]
                  \ssp \sqcup \ssp \E{N-1}\E{N}\Okada[N] \ssp \sqcup \dots
                  \sqcup \ssp \E{1}\E{2}\dots\E{N-1}\E{N}\Okada[N]\,.
    \end{align}
  Similar direct sum decompositions hold for $\Okada[N+1](X,Y)$.
\end{remark}
Recall that $\ind{w} \in \FreeMonoid$ is a \defn{reduced} word for
$\mathsf{e} \in \Okada$ if $\e{\ind{w}} = \mathsf{e}$ and $\ind{w}$ has
minimal length among the words $\ind{v}$ with $\e{\ind{v}} = \mathsf{e}$. We
denote $\length(\mathsf{e})$ the \defn{length} of any reduced word for
$\mathsf{e}$. The
rewriting system shows the following important fact about reduced words
in $\Okada$:
\begin{proposition}\label{prop:commutation-class}
  Let $\ind{w} , \ind{w'} \in \FreeMonoid$ be two reduced words for respective elements
  $\mathsf{e}, \mathsf{e'} \in \Okada$. Then $\mathsf{e}= \mathsf{e'}$ if and only if $\ind{w}$ and $\ind{w'}$ are in the same
  commutation class. In particular, reduced words are exactly the elements of the
  commutations classes of canonical words.
\end{proposition}
\begin{proof}
  The reverse implication is obvious since the rewriting system $\OSys$
  contains $\Comm =\{C\}$.

  Consider a diagram $\Delta$ whose reading is a reduced word $\ind{v}$. During the
  rewriting of $\Delta$ to its normal form $\norf(\Delta)$, one can only use the rule
  $C$, since $\Delta$ and $\norf(\Delta)$ have the same number of black boxes and all
  the other rewriting rules strictly decrease the number of black boxes.
\end{proof}
As a consequence, the set of reduced words for an Okada element is a
subset of the set of reduced words of the associated permutation.
This implies the following corollaries:
\begin{corollary}\label{lemma:redword-Okada-Sn}
  If $\ind{w} \in \FreeMonoid$ is a reduced word for $\e{\ind{w}} \in \Okada$, 
  then $\ind{w}$ is also a reduced word for $\sigma_{\ind{w}} \in \SG$. In particular
  $\length(\e{\ind{w}}) = \operatorname{inv}(\sigma_{\ind{w}})$, where
  $\operatorname{inv}$ denotes the Coxeter length (which is also the number
  of inversions).
\end{corollary}
\begin{corollary}\label{lemma:redword-eq-Okada-Sn}
  Suppose that $\ind{w}_1, \ind{w}_2 \in \FreeMonoid$ are two reduced
  words respectively for $\e{\ind{w}_1}, \e{\ind{w}_2} \in \Okada$. 
  Then $\e{\ind{w}_1}=\e{\ind{w}_2}$ if and only if
  $\sigma_{\ind{w}_1}=\sigma_{\ind{w}_2}$.
\end{corollary}

\begin{corollary}[Cancellation lemma]\label{lemma:cancel}
  Let $\e{}, \f1, \f2 \in \Okada$. Suppose that
  $\e{}\f1=\e{}\f2$ and that these product are reduced, that is
  $\length(\e{}\f1)=\length(\e{})+\length(\f1) = \length(\e{})+\length(\f2)$.
  Then $\f1=\f2$. The left-handed version of the statement also holds.
\end{corollary}
\begin{proof}
  By \cref{theo:bij-Sn_Okada}, there exist three reduced words
  $\ind{w},\ind{w}_1,\ind{w}_2$, such that $\e{}=\e{\ind{w}}$,
  $\f1=\e{\ind{w}_1}$ and $\f2=\e{\ind{w}_2}$. Equality of lengths shows
  that $\ind{w}\ind{w}_1$ and $\ind{w}\ind{w}_2$ are both reduced words for
  $\e{}\f1=\e{}\f2$. Then \cref{lemma:redword-Okada-Sn} shows that
  $\sigma_{\ind{w}}\sigma_{\ind{w}_1}=\sigma_{\ind{w}}\sigma_{\ind{w}_2}$ and
  thus $\sigma_{\ind{w}_1}=\sigma_{\ind{w}_2}$. The conclusion follows
  by applying \cref{lemma:redword-eq-Okada-Sn}.
\end{proof}

\bigskip

Notice that the relations defining the Okada monoid and Okada algebra are
preserved when reversed. This means that there's an involution of the
Okada monoid/algebra which preserves the set of reduced words. The following proposition
elucidates the effect of this involution on $\e{\sigma}$.
\begin{proposition}\label{prop:revword}
  For any permutation $\sigma$, a word $\ind{w}$ is a reduced word for
  $\e{\sigma}$ if and only if $\ind{w}^\star$ is a reduced word for
  $\e{\sigma^{-1}}$.
\end{proposition}
\begin{proof}
By symmetry, we only need to prove that
$\ind{w}^\star$ is a reduced word for $\sigma^{-1}$
whenever $\ind{w}$ is a reduced word for
$\e{\sigma}$.
Start with the normal form diagram associated to $\sigma$. After widening the
diagram if necessary, we migrate each black box eastward
using the reverse rewriting rule
$\begin{tikzpicture}[diamond, diamond/bound={3}{0}]\matrix[diamond/matrix] {
  \nn \\
  & \bl & \ \\
  & \ & \ & \nn& \nn\\
};
\end{tikzpicture}\longmapsto
\begin{tikzpicture}[diamond, diamond/bound={3}{0}]\matrix[diamond/matrix] {
  \nn \\
  & \ & \ \\
  & \ & \bl & \nn& \nn\\
};
\end{tikzpicture}$. We do this 
until the quarter plane between the southwest and southeast
diagonals emanating from each black box
contains no other black box but itself. Let 
$\Delta$ denote the diamond diagram obtained.
The quarter plane condition ensures
that the reading of $\Delta$ is the mirror image of the reading of $\Delta^\star$ (i.e the mirror
image diamond diagram). 
Consequently the reading of $\Delta^\star$ is a reduced word
for $\sigma^{-1}$ and therefore the reading of its associated normal form $\norf(\Delta^\star)$ is
$\lexmin(\sigma^{-1})$. This proves
that $\ind{w}^\star$ and $\lexmin(\sigma^{-1})$ are in the same commutation class
and are both reduced words for $\e{\sigma^{-1}}$.
\end{proof}

As a consequence $\Okada$ is a $\star$-monoid:
\begin{corollary}
  The map $\star \, :\ \e{\sigma} \mapsto \e{\sigma}^\star\eqdef\e{\sigma^{-1}}$
  is an anti-automorphism of $\Okada$.
\end{corollary}
\begin{proof}
  Since the relations defining $\Okada$ are symmetric, the map sending any
  word to its mirror image is compatible with the relations and therefore
  defines an anti-automorphism of $\Okada$. Thanks to
  Proposition~\ref{prop:revword}, this is nothing but the map
  $\e{\sigma} \mapsto \e{\sigma^{-1}}$.
\end{proof}

Each permutation has another important reduced word, namely the reverse word
of the inverse permutation, namely $\lexmin(\sigma^{-1})^\star$. 
The two preceding
lemmas imply the following corollary:
\begin{corollary}
  For any permutation $\sigma$, the elements $ \lexmin(\sigma)$ and
  $\lexmin(\sigma^{-1})^\star$ belong to the same commutation class.
\end{corollary}
This is yet another word in the same commutation class which is
the reading of the normal form for the reverse rewriting rule. In general,
$ \lexmin(\sigma)$ is not the same word as $\lexmin(\sigma^{-1})^\star$.
\medskip

The rewriting system also allows us to understand how to factor elements in the
Okada Monoid. A (right) \defn{$\Okada$-descent} of an element $\mathsf{e} \in \Okada$ is the last
index $w_k$ of a reduced word $\ind{w} = w_1 \cdots w_k$
for $\mathsf{e}$. This definition emulates the well-known characterization
of descents of a of permutation as the terminal indices in reduced words
for the permutation. Since the set of reduced words for $\e{\sigma} \in \Okada$
is a subset of the set of reduced words for the corresponding
permutation $\sigma \in \SG$, it follows that the
set of $\Okada$-descents of $\e{\sigma} \in \Okada$ is a subset of
the set of descents of  $\sigma \in \SG$. We only deal
with right descents since, thanks to \cref{prop:revword}, left $\Okada$-descents of $\mathsf{e}$
are the right $\Okada$-descents of $\mathsf{e}^\star$.
\begin{lemma}\label{lemma:descents-loop}
  Let $ \mathsf{e} \in\Okada$ and let $\Delta(\mathsf{e})$ be its associated canonical
  diamond diagram. Then $i$ is an $\Okada$-descent of $\mathsf{e}$ if and only
\begin{itemize}
  \item the $i\ordinal$ row of $\Delta(\mathsf{e})$ is occupied by a black box;
  \item there are no black boxes in either the $(i-1)\ordinalst$ or the
    $(i+1)\ordinalst$~rows of $\Delta(\mathsf{e})$ which are situated to the right
    of the right-most black box in the $i\ordinal$~row.
  \end{itemize}
\end{lemma}
\begin{proof}
  Recall that reduced words for $\mathsf{e} \in \Okada$ are the row readings of diamond diagrams
  within the commutation class of $\Delta(\mathsf{e})$. Now $i$ is an $\Okada$-decent of $\mathsf{e}$ 
  if and only if one can move, using only commutation, a black box to the right-most position in the
  $i\ordinal$~row of $\Delta(\mathsf{e})$.
\end{proof}
\begin{corollary}
  $\Okada$-descent sets are exactly the subsets of
  $\setN[N-1]$ which don't contain consecutive members. Consequently, the total number
  of $\Okada$-descent sets is the $N\ordinal$ Fibonacci number.
\end{corollary}
\begin{proof}
  The only remaining thing to prove is that for any subset $I$ without
  consecutive members, there is an element of $\Okada$ which has $I$ as an $\Okada$-descent
  set. Clearly $\prod_{i\in I} \e{i}$ is such an element.
\end{proof}
%On can see from \cref{fig.six_normal_forms} that $\e{123}$ has no descents,
%$\e{312}, \e{213}$ and $\e{321}$ each descent set $\{1\}$, and both $\e{132}$ and
%$\e{231}$ have descent set $\{2\}$.
One can see from \cref{fig.six_normal_forms} that the monoid unit $\e{\boldsymbol{\epsilon}}$ has no $\Okada$-descents, that
$\e{\ind{1}}$, $\e{\ind{2}\ind{1}}$ and $\e{\ind{1}\ind{2}\ind{1}}$ all have $\{1\}$ as an $\Okada$-descent set, and that both $\e{\ind{2}}$ and $\e{\ind{1}\ind{2}}$
have $\{2\}$ as an $\Okada$-descent set.

In general, terminal indices of any word are $\Okada$-descents:
\begin{proposition}
  If $\ind{w}=w_1 \cdots w_k$ is \underline{any} index word then
  $w_k$ is an $\Okada$-descent of $\e{\ind{w}} \in \Okada$.
\end{proposition}
Of course if $i$ is an $\Okada$-descent of $\e{\sigma}$ then $i$ is a descent of
the permutation $\sigma$. The converse is false. For example, both $i = 1,2$ are
descents of the permutation $\sigma = (3,2,1)$ while
$i = 1$ is
the only $\Okada$-descent of $\e{\sigma}$.
\begin{proof}
  Let $w_k = i$ and
  pick a diamond diagram $\Delta$ whose reading is $\ind{w}$ and
  consider a rewriting sequence
  $\Delta=\Delta_1\mapsto \Delta_2\mapsto \dots \mapsto
  \Delta_\ell=\norf(\Delta)$. We argue by induction that for all $j\leq \ell$
  \begin{itemize}
  \item the $i\ordinal$ row of $\Delta_j$ is occupied by a black box;
  \item there is no black box in the $(i-1)\ordinalst$~row of $\Delta_j$ which is situated to the right of the right-most
    black box in the $i\ordinal$~row.
  \end{itemize}
%  \begin{itemize}
%  \item the $i\ordinal$ row of $\Delta_j$ is not empty;
%  \item there is no black box on the right of the rightmost black box in the
%    $i$-th column in both the $i-1\ordinalst$~row and the $i+1\ordinalst$~row.
%  \end{itemize}
  The base case of the induction is trivial since the reading of $\Delta_1=\Delta$ ends
  with $i$. We now turn to the inductive step. Call $b$ the right-most black box in
  the $i\ordinal$ row of $\Delta_j$. Again, we emphasize that rewriting only move black boxes to the
  left. So nothing changes to the right of $b$ provided $b$ does not belong to the
  pattern of the rewriting step $\Delta_j \mapsto \Delta_{j+1}$. Otherwise,
  the assumptions above force $b$ to be the eastern box of the pattern. Since the pattern can only be rewritten to
  \begin{tikzpicture}[diamond, diamond/bound={3}{0}]\matrix[diamond/matrix] {
    \nn \\
    & \bl & \ \\
    & \ & \ & \nn& \nn\\
  }; \end{tikzpicture}
  the induction hypothesis also holds for $\Delta_{j+1}$.
\end{proof}

\subsection{Loop configurations}
\label{subsection:loop-configurations}

To better understand elements of the Okada monoid, we'll modify diamond diagrams by
replacing black and white squares respectively by \defn{double U-turn} and
\defn{double straight squares} as:
\mbox{$\begin{tikzpicture}[diamond,
    diamond/bound={1}{1}]\matrix[diamond/matrix] {
      \bl\\
    };
\end{tikzpicture}\to
\begin{tikzpicture}[diamond, diamond/bound={1}{1}]\matrix[diamond/matrix] {
    \tu\\
};
\end{tikzpicture}$}
and
\mbox{$\begin{tikzpicture}[diamond, diamond/bound={1}{1}]\matrix[diamond/matrix] {
    \ \\
};
\end{tikzpicture}\to
\begin{tikzpicture}[diamond, diamond/bound={1}{1}]\matrix[diamond/matrix] {
    \st\\
};
\end{tikzpicture}$}. If necessary, we complete the paths by adding straight lines
at the top and bottom of the diagram, as if there where half straight squares. The left and right end-points of the paths are numbered from bottom to top,
with positive numbers on the left and negative numbers (overlined) on
the right. The picture obtained is called the \defn{fully-packed loop configuration (FPLC)}
associated to the diamond diagram. See \cref{fig.path-def} for an example.
\begin{figure}[ht]
  % \makebox[\textwidth]{$
  \[
\begin{tikzpicture}[diamond, diamond/bound={6}{3}]\matrix[diamond/matrix] {
  \  \\
  \bl&\bl\\
  \bl&\  &\  \\
  \bl&\  &\bl&\  \\
  \bl&\  &\  &\bl&\bl\\
  &\  &\bl&\  &\bl&\  \\
  &   &\  &\  &\bl&\bl&\bl\\
  &   &   &\  &\  &\  &\  &\bl\\
};
\end{tikzpicture}
\hspace{1cm}
\begin{tikzpicture}[diamond, diamond/bound={6}{4}]\matrix[diamond/matrix] {
  \TLlab1 & \stt\\
  \TLlab2 &\st& \stt\\
  \TLlab3 &\tu&\tuC{black}{green}& \sttC{green}\\
  \TLlab4 &\tuC{black}{green}&\stC{green}{green}&\stC{green}{green}& \sttC{green}\\
  \TLlab5 &\tuC{green}{black}&\stC{black}{green}&\tuC{green}{black}&\stC{black}{green}& \sttC{green}\\
  \TLlab6 &\tu&\st&\st&\tuC{black}{green}&\tuC{green}{green}& \sttC{green}\\
          &\stb&\st&\tuC{black}{blue}&\st\stC{blue}{green}&\tuC{green}{blue}&\stC{blue}{green}& \sttC{green}\\
          &   &\stb&\stC{black}{blue}&\stC{blue}{blue}&\tuC{blue}{blue}&\tuC{blue}{green}&\tuC{green}{black}& \stt\\
          &   &   &\stb&\stC{black}{blue}&\stC{blue}{blue}&\stC{blue}{green}&\stC{green}{black}&\tu& \stt\\
  & & & & \TLlab{\ov6} & \TLlab{\ov5} &\TLlab{\ov4} &\TLlab{\ov3} &\TLlab{\ov2} &\TLlab{\ov1} \\
};
\end{tikzpicture}
\hspace{1cm}
\begin{tikzpicture}[yscale=0.35,xscale=0.6, thick, baseline={(current bounding box.east)}]
\tikzstyle{vertex} = [shape=rectangle, minimum height=15pt, minimum width=0pt, inner sep=0pt]
\tikzstyle{mid} = [draw, fill=white, shape=circle, minimum size=2pt, inner sep=1pt]
\foreach \i in {1,...,6} {
  \node[vertex] (G-\i) at ($(-3, \i*1.5-1.5)$) {$\i$};
  \node[vertex] (G--\i) at ($(1.44, \i*1.5-1.5)$) {$\ov{\i}$};
}
\draw[green] (G--3) -- (G-5) node[pos=0.35294117647058826,mid] (M-3-5) { 1 };
\draw (G-1) .. controls +(2.3, 1.05) and +(2.3, -1.05) .. (G-4) node[pos=0.5,mid] (M1-4) { 1 };
\draw[blue] (G--5) .. controls +(-0.7, -0.35) and +(-0.7, 0.35) .. (G--4) node[pos=0.5,mid] (M-5--4) { 2 };
\draw (G--6) -- (G-6) node[pos=0.35294117647058826,mid] (M-6-6) { 4 };
\draw (G-2) .. controls +(0.7, 0.35) and +(0.7, -0.35) .. (G-3) node[pos=0.5,mid] (M2-3) { 2 };
\draw (G--2) .. controls +(-0.7, -0.35) and +(-0.7, 0.35) .. (G--1) node[pos=0.5,mid] (M-2--1) { 1 };
\draw[color=red, <-] (M2-3) -- (M1-4);
\draw[color=red, <-] (M-5--4) -- (M-3-5);
\draw[color=red, <-] (M-6-6) -- (M-3-5);
\end{tikzpicture}
\]
  \caption[Associated diamond, path and arc-diagrams]{\label{fig.path-def}%
    A diamond diagram with its associated fully-packed loop
    configuration and labelled arc-diagram. The red arrows show the Hasse
    diagram of the nesting order on the arcs. Two paths and their corresponding
    edges are highlighted in green and blue.}
\end{figure}
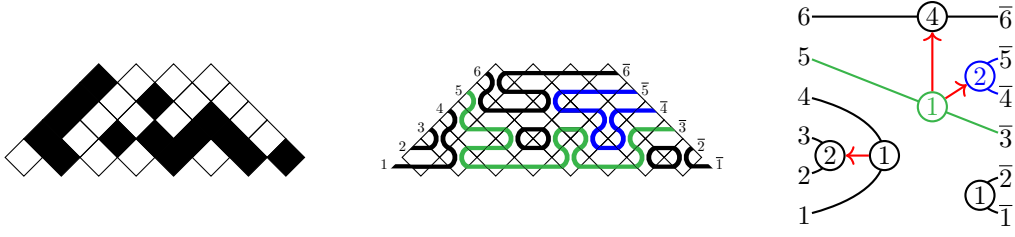
A path in an FPLC which doesn't terminate at the border of the diagram is
called a \defn{loop}.

The rewriting system $\OSys$ can be reinterpreted
using straight and U-turn squares:
\[
  \OSys\eqdef
  \left\{
\begin{tikzpicture}[diamond, diamond/bound={3}{0}]\matrix[diamond/matrix] {
    \nn \\
    & \st & \st \\
    & \st & \tu & \nn& \nn\\
  };
\end{tikzpicture}
\longmapsto
\begin{tikzpicture}[diamond, diamond/bound={3}{0}]\matrix[diamond/matrix] {
    \nn \\
    & \tu & \st \\
    & \st & \st & \nn& \nn\\
  };
\end{tikzpicture}
\ ,\quad
\begin{tikzpicture}[diamond, diamond/bound={3}{0}]\matrix[diamond/matrix] {
    \nn \\
    & \tu & \st \\
    & \st & \tu & \nn& \nn\\
  };
\end{tikzpicture}
\longmapsto
\begin{tikzpicture}[diamond, diamond/bound={3}{0}]\matrix[diamond/matrix] {
    \nn \\
    & \tu & \st \\
    & \st & \st & \nn& \nn\\
  };
\end{tikzpicture}
\ ,\quad
\begin{tikzpicture}[diamond, diamond/bound={3}{0}]\matrix[diamond/matrix] {
    \nn \\
    & \tu & \tu \\
    & \st & \tu & \nn& \nn\\
  };
\end{tikzpicture}
\longmapsto
\begin{tikzpicture}[diamond, diamond/bound={3}{0}]\matrix[diamond/matrix] {
    \nn \\
    & \tu & \st \\
    & \st & \st & \nn& \nn\\
  };
\end{tikzpicture}\right\}\,.
\]

We export the notation used for diamond diagrams to FPLCs,
replacing freely each diamond diagram $\Delta$ by its
associated FPLC $\mathcal{D}$. For example:
\begin{multicols}{2}
\begin{itemize}
\item $\e{\mathcal{D}} \eqdef \e{\Delta}$
\item $\mathcal{D}_{\e{}}$ denotes the FPLC associated to $\Delta_{\e{}}$ for any $\e{} \in \Okada$
\item $\norf(\mathcal{D})$ replaces $\norf(\Delta)$
\item $\mathcal{D}_1\mapsto \mathcal{D}_2$
replaces  $\Delta_1\mapsto \Delta_2$
\item $\mathcal{D}_1\rewto \mathcal{D}_2$
replaces $\Delta_1\rewto \Delta_2$
\end{itemize}
\end{multicols}

Notice that a path must pass through an intermediate straight-square whenever
the path goes up and later goes down. In particular, loops must always pass
through some straight squares. This simple observation has a important
consequence:
\begin{proposition}\label{prop:red-expr-loop}
  The reading $\ind{r}(\Delta)$ of a diamond diagram $\Delta$ is a
  reduced word for an element in the monoid $\Okada$ if and only if its
  associated loop configuration $\mathcal{D}$ contains no paths which go up
  and later down. In particular, it contains no loops.
\end{proposition}
Since the reading of normal forms are particular reduced words, they contain no
paths going up and later down nor loops.

\begin{remark}
  The reading of a diamond diagram is reduced if and only if the number
  of its black squares is minimal, which means that the number of
  U-turn squares in its associated loop configuration is also minimal.
\end{remark}
\begin{proof}[Proof of \cref{prop:red-expr-loop}]
  Rule (C), which implements the
  commutation relations, only changes the horizontal steps of paths.
  If a FPLC $\mathcal{D}$ can be reduced, it must contain either
  \begin{tikzpicture}[diamond, diamond/bound={3}{0}]\matrix[diamond/matrix] {
      \nn \\
      & \tu & \st \\
      & \st & \tu & \nn& \nn\\
    };
  \end{tikzpicture}
  or
  \begin{tikzpicture}[diamond, diamond/bound={3}{0}]\matrix[diamond/matrix] {
      \nn \\
      & \tu & \tu \\
      & \st & \tu & \nn& \nn\\
    };
  \end{tikzpicture}
  in its commutation class. Both contain a path going up and down so $\mathcal{D}$ must
  contain such a path as well.
  \medskip

  Conversely, we need to show that if a FPLC has a path which goes up and
  down, it can be reduced. For any path going up and then down, there
  must be an intermediate sequence of horizontal steps following the up
  step and preceding the down step.

  Consider one of the lowest possible sequences of horizontal steps confined by an up
  and down step. It follows a certain number of straight square. Under
  those square there must be at most one U-turn square. Indeed, if there were
  more than one, the straight line between the two leftmost U-turn squares
  would build a path going up and down, contradicting the assumption that the
  path inhabits the lowest possible level of the FPLC. This is illustrated below:
  \[
    \begin{tikzpicture}[diamond, diamond/bound={2}{5}]
      \matrix[diamond/matrix] {
        \nn&\nn&\nn\\
        \tu&\stt\\
        \stb&\st&\tu\\
        &\stb&\st&\tu\\
        &&\stb&\st&\stt\\
        &&&\stb&\st&\stt\\
        &&&&\stb&\tu&\nn\\
      };
    \end{tikzpicture}
  \]
  Thus, there are only two cases: Either there is no U-turn square below or
  there is exactly one.

  If there is no U-turn square below, we are left with a loop containing two
  straight lines and two U-turns. We can apply the following rewriting sequence:
  \[
    \begin{tikzpicture}[diamond, diamond/bound={1}{5}]
      \matrix[diamond/matrix] {
        \tu&\stt\\
        \stb&\st&\stt\\
        &\stb&\nn\dots&\stt\\
        &&\stb&\st&\stt\\
        &&&\stb&\tu\\
      };
    \end{tikzpicture}\longmapsto
    \begin{tikzpicture}[diamond, diamond/bound={1}{5}]
      \matrix[diamond/matrix] {
        \tu&\stt\\
        \stb&\st&\stt\\
        &\stb&\nn\dots&\stt\\
        &&\stb&\tu&\stt\\
        &&&\stb&\st\\
      };
    \end{tikzpicture}\longmapsto\dots\longmapsto
    \begin{tikzpicture}[diamond, diamond/bound={1}{5}]
      \matrix[diamond/matrix] {
        \tu&\stt\\
        \stb&\tu&\stt\\
        &\stb&\nn\dots&\stt\\
        &&\stb&\st&\stt\\
        &&&\stb&\st\\
      };
    \end{tikzpicture}\longmapsto
    \begin{tikzpicture}[diamond, diamond/bound={1}{5}]
      \matrix[diamond/matrix] {
        \tu&\stt\\
        \stb&\st&\stt\\
        &\stb&\nn\dots&\stt\\
        &&\stb&\st&\stt\\
        &&&\stb&\st\\
      };
    \end{tikzpicture}
  \]
  The last step removes a U-turn square showing that the reading of $\mathcal{D}$
  is not a reduced index word.

  What remains is the case where there is a single U-turn square
  below. We proceed similarly:
  \[
    \begin{tikzpicture}[diamond, diamond/bound={2}{5}]
      \matrix[diamond/matrix] {
        \nn&\nn&\nn\\
        \tu&\stt\\
        \stb&\nn\dots&\stt\\
        &\stb&\st&\tu\\
        &&\stb&\st&\stt\\
        &&&\stb&\nn\dots&\stt\\
        &&&&\stb&\tu&\nn\\
      };
    \end{tikzpicture}\longmapsto
    \begin{tikzpicture}[diamond, diamond/bound={2}{5}]
      \matrix[diamond/matrix] {
        \nn&\nn&\nn\\
        \tu&\stt\\
        \stb&\nn\dots&\stt\\
        &\stb&\st&\tu\\
        &&\stb&\tu&\stt\\
        &&&\stb&\nn\dots&\stt\\
        &&&&\stb&\st&\nn\\
      };
    \end{tikzpicture}\longmapsto
    \begin{tikzpicture}[diamond, diamond/bound={2}{5}]
      \matrix[diamond/matrix] {
        \nn&\nn&\nn\\
        \st&\stt\\
        \stb&\nn\dots&\stt\\
        &\stb&\tu&\tu\\
        &&\stb&\tu&\stt\\
        &&&\stb&\nn\dots&\stt\\
        &&&&\stb&\st&\nn\\
      };
    \end{tikzpicture}\longmapsto
    \begin{tikzpicture}[diamond, diamond/bound={2}{5}]
      \matrix[diamond/matrix] {
        \nn&\nn&\nn\\
        \st&\stt\\
        \stb&\nn\dots&\stt\\
        &\stb&\tu&\st\\
        &&\stb&\st&\stt\\
        &&&\stb&\nn\dots&\stt\\
        &&&&\stb&\st&\nn\\
      };
    \end{tikzpicture}
  \]
  In both cases, we manage to remove at least one U-turn square.
\end{proof}
\cref{fig.example_normal_forms} shows more examples of normal form FPLCs.

\subsection{The labelled arc-diagram monoid}
\label{subsection:the-labelled-arc-diagram-monoid}

At this stage, we know that Okada monoid elements are in bijection with
commutation classes of normal forms of FPLCs. Within commutation
classes horizontal steps of FPLCs can move freely. This
suggest that we can ignore them and record only how much each path goes down
and then up. This is reinforced by the following key observation:
\begin{lemma}
\label{lemma-FPLC-height-invariance}
  Consider a FPLC $\mathcal{D}$ of rank $N$ and a path in $\mathcal{D}$ joining endpoints
  $i, j \in \setN \cup \setNN$. Then the lowest height $h\in\setN$ of
  a horizontal step attained by the path is invariant under any rewriting of $\mathcal{D}$. The
  normal form of $\mathcal{D}$ must therefore contain a path joining $i, j \in \setN \cup \setNN$
  whose minimal height is also $h \in \setN$.
\end{lemma}
\begin{proof}
  This it clear from the three possible rewriting rules.
\end{proof}
Accordingly, Lemma \ref{lemma-FPLC-height-invariance}
allows us to abbreviate the structure of a FPLC $\mathcal{D}$ by removing its loops
and recording the connectivity of its paths along with their heights. Formally
we remove all loop of $\mathcal{D}$, label each
path by its respective height and taking the isotopy class of what remains. The
result is called the \defn{height-labelled isotopy class} and  is denoted $[\mathcal{D}]$; an example is depicted in the third image
of \cref{fig.path-def}.
In view of the previous remarks
$[\mathcal{D}_1] = [\mathcal{D}_2]$ whenever $\mathcal{D}_1$ and $\mathcal{D}_2$ are
two FPLCs of rank $N$ which are related by a sequences of rewriting moves. It turns out
that this is actually an equivalence (which we will later prove
in \cref{prop:loop-to-diag}), providing us with a diagram model for the
Okada algebra. Let us elaborate.

%The idea of the present section is to find a data-structure to record
%precisely this information. This will be provided by labelled arc-diagram: If
%there is a path from $i\in\setNN$ going down to height $k$ and then up to
%$j\in\setNN$, we will draw an arc between $i$ and $j$ labelled by $k$. This is
%illustrated on \cref{fig.path-def}. We now give the formal definitions.

Recall that a \defn{perfect matching} of $\setN \cup \setNN$ is a partition of
$\setN \cup \setNN$ into pairs. A perfect matching is \defn{crossing} if it
contains two pairs $\{a,b\}$ and $\{c, d\}$ of elements in $\setN \cup \setNN$,
where $a < c < b < d$ with respect to the order of~\cref{eq:orderN-NN}.  A
rank $N$ \defn{non-crossing arc-diagram} is a visualization of a perfect
matching linking vertices $\{1, \dots , N\}$ and $\{\ov{1}, \dots, \ov{N} \}$,
situated on the left and right boundaries of a rectangle, by non-crossing arcs 
(drawn in the interior of the rectangle). A pair $\{a,b\}$ in the matching is depicted
by an arc, in the rectangle, joining vertices $a,b \in [N] \cup \ov{[N]}$ and
is denoted by $\arc{a}{}{b}$. Either $a,b$ are both positive, both negative,
or else have different signs; in the later case we say the arc $\arc{a}{}{b}$
is a \defn{propagating}.  Arcs linking positive (resp. negative) numbers are
called left (resp. right) arcs.  Only the incidence relations of the
arc-diagram are relevant, and so isotopic diagrams are considered
equivalent. It is well known that the number of perfect matching is the
$N\ordinal$ Catalan number $\mathrm{Cat}(N)$.

An arc $\arc{a}{}{b}$ is said to be \defn{nested} in another arc
$\arc{c}{}{d}$ if $c<a<b<d$. Nesting defines a partial order on the arcs of a
non-crossing arc-diagram. The reader should be aware that any arc situated
above a propagating arc is nested in the later. In particular, given
any pair of propagating arcs, one arc must be nested in the other;
consequently the nesting order is total when restricted to propagating
arcs. \cref{fig.Okada_arc_diagrams} shows some arc-diagrams where the Hasse diagram
of the nesting order is depicted by red arrows.

\begin{figure}
  \makebox[\textwidth]{
    \begin{tikzpicture}[xscale=0.6,yscale=0.45, thick, baseline={(current bounding box.east)}]
\tikzstyle{vertex} = [shape=circle, minimum size=7pt, inner sep=1pt]
\tikzstyle{mid} = [draw, fill=white, shape=circle, minimum size=2pt, inner sep=1pt]
\foreach \i in {1,...,8} {
  \node[vertex] (G-\i) at ($(-1.20, \i*1.5-1.5)$) {$\i$};
  \node[vertex] (G--\i) at ($(1.20, \i*1.5-1.5)$) {$\overline{\i}$};
}
\draw (G--1) -- (G-1) node[pos=0.5,mid] (M-1-1) { 1 };
\draw (G--2) -- (G-2) node[pos=0.5,mid] (M-2-2) { 2 };
\draw (G--3) -- (G-3) node[pos=0.5,mid] (M-3-3) { 3 };
\draw (G--4) -- (G-4) node[pos=0.5,mid] (M-4-4) { 4 };
\draw (G--5) -- (G-5) node[pos=0.5,mid] (M-5-5) { 5 };
\draw (G--6) -- (G-6) node[pos=0.5,mid] (M-6-6) { 6 };
\draw (G--7) -- (G-7) node[pos=0.5,mid] (M-7-7) { 7 };
\draw (G--8) -- (G-8) node[pos=0.5,mid] (M-8-8) { 8 };
\draw[color=red, <-] (M-8-8) -- (M-7-7);
\draw[color=red, <-] (M-7-7) -- (M-6-6);
\draw[color=red, <-] (M-6-6) -- (M-5-5);
\draw[color=red, <-] (M-5-5) -- (M-4-4);
\draw[color=red, <-] (M-4-4) -- (M-3-3);
\draw[color=red, <-] (M-3-3) -- (M-2-2);
\draw[color=red, <-] (M-2-2) -- (M-1-1);
\end{tikzpicture}
\hfill
\begin{tikzpicture}[xscale=0.6,yscale=0.45, thick, baseline={(current bounding box.east)}]
\tikzstyle{vertex} = [shape=circle, minimum size=7pt, inner sep=1pt]
\tikzstyle{mid} = [draw, fill=white, shape=circle, minimum size=2pt, inner sep=1pt]
\foreach \i in {1,...,8} {
  \node[vertex] (G-\i) at ($(-2.64, \i*1.5-1.5)$) {$\i$};
  \node[vertex] (G--\i) at ($(2.64, \i*1.5-1.5)$) {$\overline{\i}$};
}
\draw (G--3) -- (G-1) node[pos=0.5,mid] (M-3-1) { 1 };
\draw (G--8) .. controls +(-2.592, -1.296) and +(2.592, 1.296) .. (G-4) node[pos=0.5,mid] (M-8-4) { 2 };
\draw (G-5) .. controls +(2.1, 1.05) and +(2.1, -1.05) .. (G-8) node[pos=0.5,mid] (M5-8) { 3 };
\draw (G--2) .. controls +(-0.7, -0.35) and +(-0.7, 0.35) .. (G--1) node[pos=0.5,mid] (M-2--1) { 1 };
\draw (G-2) .. controls +(0.7, 0.35) and +(0.7, -0.35) .. (G-3) node[pos=0.5,mid] (M2-3) { 2 };
\draw (G--7) .. controls +(-2.1, -1.05) and +(-2.1, 1.05) .. (G--4) node[pos=0.5,mid] (M-7--4) { 4 };
\draw (G-6) .. controls +(0.7, 0.35) and +(0.7, -0.35) .. (G-7) node[pos=0.5,mid] (M6-7) { 6 };
\draw (G--6) .. controls +(-0.7, -0.35) and +(-0.7, 0.35) .. (G--5) node[pos=0.5,mid] (M-6--5) { 5 };
\draw[color=red, <-] (M6-7) -- (M5-8);
\draw[color=red, <-] (M5-8) -- (M-8-4);
\draw[color=red, <-] (M2-3) -- (M-3-1);
\draw[color=red, <-] (M-6--5) -- (M-7--4);
\draw[color=red, <-] (M-7--4) -- (M-3-1);
\draw[color=red, <-] (M-8-4) -- (M-3-1);
\end{tikzpicture}
\hfill
\begin{tikzpicture}[xscale=0.6,yscale=0.45, thick, baseline={(current bounding box.east)}]
\tikzstyle{vertex} = [shape=circle, minimum size=7pt, inner sep=1pt]
\tikzstyle{mid} = [draw, fill=white, shape=circle, minimum size=2pt, inner sep=1pt]
\foreach \i in {1,...,8} {
  \node[vertex] (G-\i) at ($(-1.96, \i*1.5-1.5)$) {$\i$};
  \node[vertex] (G--\i) at ($(1.96, \i*1.5-1.5)$) {$\overline{\i}$};
}
\draw (G-1) .. controls +(2.1, 1.05) and +(2.1, -1.05) .. (G-4) node[pos=0.5,mid] (M1-4) { 1 };
\draw (G--8) .. controls +(-2.1, -1.05) and +(-2.1, 1.05) .. (G--3) node[pos=0.5,mid] (M-8--3) { 1 };
\draw (G-5) .. controls +(0.7, 0.35) and +(0.7, -0.35) .. (G-6) node[pos=0.5,mid] (M5-6) { 3 };
\draw (G--7) .. controls +(-0.7, -0.35) and +(-0.7, 0.35) .. (G--6) node[pos=0.5,mid] (M-7--6) { 2 };
\draw (G-2) .. controls +(0.7, 0.35) and +(0.7, -0.35) .. (G-3) node[pos=0.5,mid] (M2-3) { 2 };
\draw (G--2) .. controls +(-0.7, -0.35) and +(-0.7, 0.35) .. (G--1) node[pos=0.5,mid] (M-2--1) { 1 };
\draw (G-7) .. controls +(0.7, 0.35) and +(0.7, -0.35) .. (G-8) node[pos=0.5,mid] (M7-8) { 7 };
\draw (G--5) .. controls +(-0.7, -0.35) and +(-0.7, 0.35) .. (G--4) node[pos=0.5,mid] (M-5--4) { 4 };
\draw[color=red, <-] (M2-3) -- (M1-4);
\draw[color=red, <-] (M-5--4) -- (M-8--3);
\draw[color=red, <-] (M-7--6) -- (M-8--3);
\end{tikzpicture}
\hfill
  \begin{tikzpicture}[xscale=0.8,yscale=0.45, thick, baseline={(current bounding box.east)}]
\tikzstyle{vertex} = [shape=circle, minimum size=7pt, inner sep=1pt]
\tikzstyle{mid} = [draw, fill=white, shape=circle, minimum size=2pt, inner sep=1pt]
\foreach \i in {1,...,8} {
  \node[vertex] (G-\i) at ($(-1.35, \i*1.5-1.5)$) {$\i$};
  \node[vertex] (G--\i) at ($(1.35, \i*1.5-1.5)$) {$\overline{\i}$};
}
\draw (G--3) -- (G-3) node[pos=0.5,mid] (M-3-3) { 1 };
\draw (G--6) -- (G-4) node[pos=0.5,mid] (M-6-4) { 2 };
\draw (G-1) .. controls +(0.70, 0.35) and +(0.70, -0.35) .. (G-2) node[pos=0.5,mid] (M1-2) { 1 };
\draw (G--2) .. controls +(-0.70, -0.35) and +(-0.70, 0.35) .. (G--1) node[pos=0.5,mid] (M-2--1) { 1 };
\draw (G--7) -- (G-5) node[pos=0.5,mid] (M-7-5) { 5 };
\draw (G--8) -- (G-8) node[pos=0.5,mid] (M-8-8) { 6 };
\draw (G-6) .. controls +(0.70, 0.35) and +(0.70, -0.35) .. (G-7) node[pos=0.5,mid] (M6-7) { 6 };
\draw (G--5) .. controls +(-0.70, -0.35) and +(-0.70, 0.35) .. (G--4) node[pos=0.5,mid] (M-5--4) { 4 };
\draw[color=red, <-] (M6-7) -- (M-7-5);
\draw[color=red, <-] (M-5--4) -- (M-3-3);
\draw[color=red, <-] (M-8-8) -- (M-7-5);
\draw[color=red, <-] (M-7-5) -- (M-6-4);
\draw[color=red, <-] (M-6-4) -- (M-3-3);
\end{tikzpicture}}
  \caption[Okada arc-diagram]{Some Okada arc-diagrams. The red arrows shows
    the Hasse diagram of the nesting order on the edges.}
  \label{fig.Okada_arc_diagrams}
\end{figure}
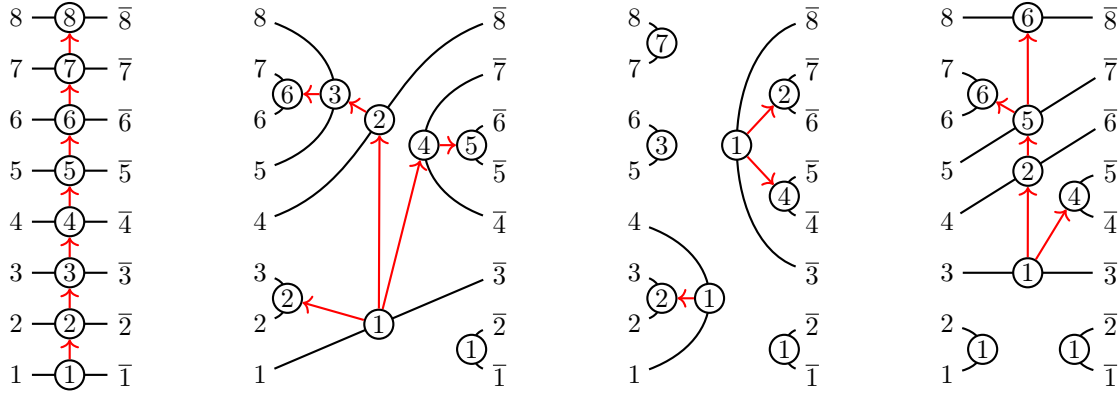

We can merge two non-crossing arc-diagrams $C$ and $D$
of rank $N$
by glueing the right boundary nodes of
$C$ with the left boundary nodes of $D$
and concatenating their respective arcs. The resulting
diagram is call the \defn{composition}
of $C$ and $D$ and is denoted $C \circ D$.
The \defn{product} $C \boldsymbol{\cdot} D$ of $C$ and $D$
is the rank $N$ non-crossing arc-diagram obtained by
removing any loops in $C \circ D$ and
taking the isotopy class of the result.

%Given two arc-diagrams $C, D$ of rank $N$, their composition $C \cdot D$ is
%computed by glueing the two diagrams
%recording in a new arc-diagram how the
%left nodes of $C$ are connected to the right node of $D$ following the paths of
%the diagrams. loops that doesn't touch the left of $C$ or the right of $D$ are
%discarded.

Here are some examples:
\[
\begin{tikzpicture}[scale=0.3, thick, baseline={(current bounding box.east)}]
\tikzstyle{vertex} = [shape=circle, minimum size=7pt, inner sep=1pt]
\node[vertex] (G-1) at (-1.44, -0.1) {1};
\node[vertex] (G-2) at (-1.44, 1.4) {2};
\node[vertex] (G-3) at (-1.44, 2.9) {3};
\node[vertex] (G-4) at (-1.44, 4.4) {4};
\node[vertex] (G-5) at (-1.44, 5.9) {5};
\node[vertex] (G-6) at (-1.44, 7.4) {6};
\node[vertex] (G--1) at (2.64, 0.0) {$\ov1$};
\node[vertex] (G--3) at (2.64, 1.5) {$\ov2$};
\node[vertex] (G--2) at (2.64, 3.0) {$\ov3$};
\node[vertex] (G--4) at (2.64, 4.5) {$\ov4$};
\node[vertex] (G--5) at (2.64, 6.0) {$\ov5$};
\node[vertex] (G--6) at (2.64, 7.5) {$\ov6$};
\draw (G--5) .. controls +(-2.592, -1.296) and +(1.632, 0.816) .. (G-1);
\draw (G--6) -- (G-4);
\draw (G-5) .. controls +(1.05, 0.35) and +(1.05, -0.35) .. (G-6);
\draw (G--4) .. controls +(-2.1, -1.05) and +(-2.1, 1.05) .. (G--1);
\draw (G-2) .. controls +(1.05, 0.35) and +(1.05, -0.35) .. (G-3);
\draw (G--3) .. controls +(-1.05, -0.35) and +(-1.05, 0.35) .. (G--2);
\end{tikzpicture}
\begin{tikzpicture}[scale=0.3, thick, baseline={(current bounding box.east)}]
  \tikzstyle{vertex} = [shape=circle, minimum size=7pt, inner sep=1pt]
\node[vertex] (G-1) at (-1.44, -0.1) {1};
\node[vertex] (G-2) at (-1.44, 1.4) {2};
\node[vertex] (G-3) at (-1.44, 2.9) {3};
\node[vertex] (G-4) at (-1.44, 4.4) {4};
\node[vertex] (G-5) at (-1.44, 5.9) {5};
\node[vertex] (G-6) at (-1.44, 7.4) {6};
\node[vertex] (G--1) at (1.44, 0.0) {$\ov1$};
\node[vertex] (G--2) at (1.44, 1.5) {$\ov2$};
\node[vertex] (G--3) at (1.44, 3.0) {$\ov3$};
\node[vertex] (G--5) at (1.44, 6.0) {$\ov4$};
\node[vertex] (G--4) at (1.44, 4.5) {$\ov5$};
\node[vertex] (G--6) at (1.44, 7.5) {$\ov6$};
\draw (G--3) -- (G-1);
\draw (G-2) .. controls +(1.05, 0.35) and +(1.05, -0.35) .. (G-3);
\draw (G--2) .. controls +(-1.05, -0.35) and +(-1.05, 0.35) .. (G--1);
\draw (G--6) -- (G-6);
\draw (G-4) .. controls +(1.05, 0.35) and +(1.05, -0.35) .. (G-5);
\draw (G--5) .. controls +(-1.05, -0.35) and +(-1.05, 0.35) .. (G--4);
\end{tikzpicture}
\qquad \raisebox{-2.5pt}{=} \qquad
\begin{tikzpicture}[scale=0.3, thick, baseline={(current bounding box.east)}]
  \tikzstyle{vertex} = [shape=circle, minimum size=7pt, inner sep=1pt]
\node[vertex] (G-1) at (-1.44, -0.1) {1};
\node[vertex] (G-2) at (-1.44, 1.4) {2};
\node[vertex] (G-3) at (-1.44, 2.9) {3};
\node[vertex] (G-4) at (-1.44, 4.4) {4};
\node[vertex] (G-5) at (-1.44, 5.9) {5};
\node[vertex] (G-6) at (-1.44, 7.4) {6};
\node[vertex] (G--1) at (1.44, 0.0) {$\ov1$};
\node[vertex] (G--2) at (1.44, 1.5) {$\ov2$};
\node[vertex] (G--3) at (1.44, 3.0) {$\ov3$};
\node[vertex] (G--4) at (1.44, 4.5) {$\ov4$};
\node[vertex] (G--5) at (1.44, 6.0) {$\ov5$};
\node[vertex] (G--6) at (1.44, 7.5) {$\ov6$};
\draw (G--3) -- (G-1);
\draw (G--6) -- (G-4);
\draw (G-5) .. controls +(1.05, 0.35) and +(1.05, -0.35) .. (G-6);
\draw (G--2) .. controls +(-1.05, -0.35) and +(-1.05, 0.35) .. (G--1);
\draw (G-2) .. controls +(1.05, 0.35) and +(1.05, -0.35) .. (G-3);
\draw (G--5) .. controls +(-1.05, -0.35) and +(-1.05, 0.35) .. (G--4);
\end{tikzpicture}
\]

\begin{definition}
  A \defn{labelled non-crossing arc-diagram} is a non-crossing arc-diagram with a positive integer $\ell$
  attached to each arc. We write $\arc[D]{a}{\ell}{b}$ to express that $D$
  contains an arc joining distinct nodes $a, b  \in \setN \cup \setNN$ and labelled by $\ell$. If $D$ is clear from
  the context, we shall omit the subscript, writing only $\arc{a}{\ell}{b}$. In pictures,
  each label is drawn within a circle in the middle of the corresponding arc.

  The \defn{labelled product} of two labelled non-crossing arc-diagrams $C$ and $D$ is
  obtained by labelling each arc $\alpha$ of the underlying product $C \boldsymbol{\cdot} D$
  by the \textbf{minimum} of the labels of the arcs of $C$ and $D$
  which form $\alpha$ when concatenated. Bearing some abuse of notation,
   $C \boldsymbol{\cdot} D$ will simultaneously denote both the product and labelled product of $C$ and $D$.

%  The \defn{(labelled) product} of two labelled arc-diagrams is the labelling
%  of the unlabelled product where the label of an arc $e$ is the
%  \textbf{minimum} of the labels of the arcs in the path associated to $e$.
\end{definition}
See \cref{fig.compo_Okada_arc,fig:right-code-factor} for some examples of
products.
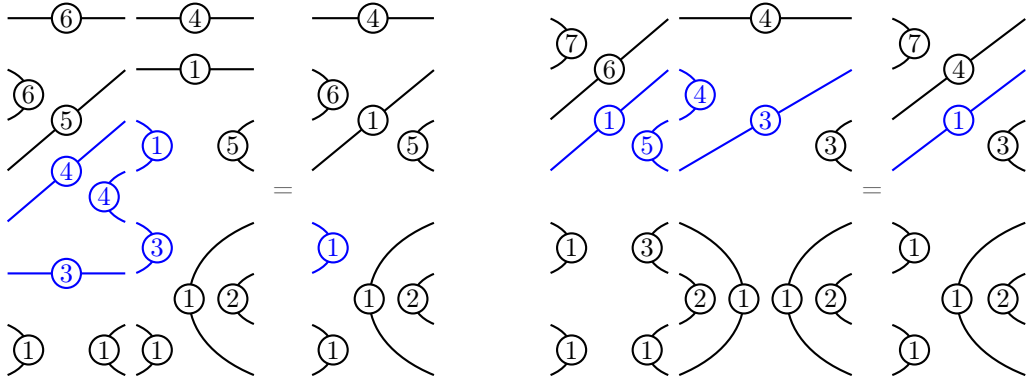
\begin{figure}[ht]
  \makebox[\textwidth]{
  \begin{tikzpicture}[xscale=0.55,yscale=0.45, thick, baseline={(current bounding box.east)}]
\tikzstyle{vertex} = [shape=rectangle, minimum height=15pt, minimum width=0pt, inner sep=0pt]
\tikzstyle{mid} = [draw, fill=white, shape=circle, minimum size=2pt, inner sep=1pt]
\foreach \i in {1,...,8} {
  \node[vertex] (G-\i) at ($(-1.44, \i*1.5-1.5)$) {};
  \node[vertex] (G--\i) at ($(1.44, \i*1.5-1.5)$) {};
}
\draw[color=blue] (G--3) -- (G-3) node[pos=0.5,mid] (M-3-3) { 3 };
\draw[color=blue] (G--6) -- (G-4) node[pos=0.5,mid] (M-6-4) { 4 };
\draw (G-1) .. controls +(0.70, 0.35) and +(0.70, -0.35) .. (G-2) node[pos=0.5,mid] (M1-2) { 1 };
\draw (G--2) .. controls +(-0.70, -0.35) and +(-0.70, 0.35) .. (G--1) node[pos=0.5,mid] (M-2--1) { 1 };
\draw (G--7) -- (G-5) node[pos=0.5,mid] (M-7-5) {$5$};
\draw (G--8) -- (G-8) node[pos=0.5,mid] (M-8-8) { 6 };
\draw (G-6) .. controls +(0.70, 0.35) and +(0.70, -0.35) .. (G-7) node[pos=0.5,mid] (M6-7) { 6 };
\draw[color=blue] (G--5) .. controls +(-0.70, -0.35) and +(-0.70, 0.35) .. (G--4) node[pos=0.5,mid] (M-5--4) { 4 };
% \draw[color=red, <-] (M6-7) -- (M-7-5);
% \draw[color=red, <-] (M-5--4) -- (M-3-3);
% \draw[color=red, <-] (M-8-8) -- (M-7-5);
% \draw[color=red, <-] (M-7-5) -- (M-6-4);
% \draw[color=red, <-] (M-6-4) -- (M-3-3);
\end{tikzpicture}
\begin{tikzpicture}[xscale=0.55,yscale=0.45, thick, baseline={(current bounding box.east)}]
\tikzstyle{vertex} = [shape=rectangle, minimum height=15pt, minimum width=0pt, inner sep=0pt]
\tikzstyle{mid} = [draw, fill=white, shape=circle, minimum size=2pt, inner sep=1pt]
\foreach \i in {1,...,8} {
  \node[vertex] (G-\i) at ($(-1.44, \i*1.5-1.5)$) {};
  \node[vertex] (G--\i) at ($(1.44, \i*1.5-1.5)$) {};
}
\draw (G--7) -- (G-7) node[pos=0.5,mid] (M-7-7) { 1 };
\draw[color=blue] (G-5) .. controls +(0.70, 0.35) and +(0.70, -0.35) .. (G-6) node[pos=0.5,mid] (M5-6) { 1 };
\draw (G--4) .. controls +(-2.10, -1.05) and +(-2.10, 1.05) .. (G--1) node[pos=0.5,mid] (M-4--1) { 1 };
\draw[color=blue] (G-3) .. controls +(0.70, 0.35) and +(0.70, -0.35) .. (G-4) node[pos=0.5,mid] (M3-4) { 3 };
\draw (G--3) .. controls +(-0.70, -0.35) and +(-0.70, 0.35) .. (G--2) node[pos=0.5,mid] (M-3--2) { 2 };
\draw (G-1) .. controls +(0.70, 0.35) and +(0.70, -0.35) .. (G-2) node[pos=0.5,mid] (M1-2) { 1 };
\draw (G--6) .. controls +(-0.70, -0.35) and +(-0.70, 0.35) .. (G--5) node[pos=0.5,mid] (M-6--5) { 5 };
\draw (G--8) -- (G-8) node[pos=0.5,mid] (M-8-8) { 4 };
% \draw[color=red, <-] (M-3--2) -- (M-4--1);
% \draw[color=red, <-] (M-8-8) -- (M-7-7);
\end{tikzpicture}
\ = \
\begin{tikzpicture}[xscale=0.55,yscale=0.45, thick, baseline={(current bounding box.east)}]
\tikzstyle{vertex} = [shape=rectangle, minimum height=15pt, minimum width=0pt, inner sep=0pt]
\tikzstyle{mid} = [draw, fill=white, shape=circle, minimum size=2pt, inner sep=1pt]
\foreach \i in {1,...,8} {
  \node[vertex] (G-\i) at ($(-1.44, \i*1.5-1.5)$) {};
  \node[vertex] (G--\i) at ($(1.54, \i*1.5-1.5)$) {};
}
\draw (G--7) -- (G-5) node[pos=0.5,mid] (M-7-5) { 1 };
\draw[color=blue] (G-3) .. controls +(0.70, 0.35) and +(0.70, -0.35) .. (G-4) node[pos=0.5,mid] (M3-4) { 1 };
\draw (G--4) .. controls +(-2.10, -1.05) and +(-2.10, 1.05) .. (G--1) node[pos=0.5,mid] (M-4--1) { 1 };
\draw (G-1) .. controls +(0.70, 0.35) and +(0.70, -0.35) .. (G-2) node[pos=0.5,mid] (M1-2) { 1 };
\draw (G--3) .. controls +(-0.70, -0.35) and +(-0.70, 0.35) .. (G--2) node[pos=0.5,mid] (M-3--2) { 2 };
\draw (G--8) -- (G-8) node[pos=0.5,mid] (M-8-8) { 4 };
\draw (G-6) .. controls +(0.70, 0.35) and +(0.70, -0.35) .. (G-7) node[pos=0.5,mid] (M6-7) { 6 };
\draw (G--6) .. controls +(-0.70, -0.35) and +(-0.70, 0.35) .. (G--5) node[pos=0.5,mid] (M-6--5) { 5 };
% \draw[color=red, <-] (M6-7) -- (M-7-5);
% \draw[color=red, <-] (M-3--2) -- (M-4--1);
% \draw[color=red, <-] (M-8-8) -- (M-7-5);
\end{tikzpicture}
\hfil\hfil
\begin{tikzpicture}[xscale=0.55,yscale=0.45, thick, baseline={(current bounding box.east)}]
\tikzstyle{vertex} = [shape=rectangle, minimum height=15pt, minimum width=0pt, inner sep=0pt]
\tikzstyle{mid} = [draw, fill=white, shape=circle, minimum size=2pt, inner sep=1pt]
\foreach \i in {1,...,8} {
  \node[vertex] (G-\i) at ($(-1.44, \i*1.5-1.5)$) {};
  \node[vertex] (G--\i) at ($(1.44, \i*1.5-1.5)$) {};
}
\draw[color=blue] (G--7) -- (G-5) node[pos=0.5,mid] (M-7-5) { 1 };
\draw (G-3) .. controls +(0.70, 0.35) and +(0.70, -0.35) .. (G-4) node[pos=0.5,mid] (M3-4) { 1 };
\draw (G--4) .. controls +(-0.70, -0.35) and +(-0.70, 0.35) .. (G--3) node[pos=0.5,mid] (M-4--3) { 3 };
\draw (G-1) .. controls +(0.70, 0.35) and +(0.70, -0.35) .. (G-2) node[pos=0.5,mid] (M1-2) { 1 };
\draw[color=blue] (G--6) .. controls +(-0.70, -0.35) and +(-0.70, 0.35) .. (G--5) node[pos=0.5,mid] (M-6--5) { 5 };
\draw (G--8) -- (G-6) node[pos=0.5,mid] (M-8-6) { 6 };
\draw (G-7) .. controls +(0.70, 0.35) and +(0.70, -0.35) .. (G-8) node[pos=0.5,mid] (M7-8) { 7 };
\draw (G--2) .. controls +(-0.70, -0.35) and +(-0.70, 0.35) .. (G--1) node[pos=0.5,mid] (M-2--1) { 1 };
% \draw[color=red, <-] (M7-8) -- (M-8-6);
% \draw[color=red, <-] (M-8-6) -- (M-7-5);
\end{tikzpicture}
\begin{tikzpicture}[xscale=0.55,yscale=0.45, thick, baseline={(current bounding box.east)}]
\tikzstyle{vertex} = [shape=rectangle, minimum height=15pt, minimum width=0pt, inner sep=0pt]
\tikzstyle{mid} = [draw, fill=white, shape=circle, minimum size=2pt, inner sep=1pt]
\foreach \i in {1,...,8} {
  \node[vertex] (G-\i) at ($(-2.1, \i*1.5-1.5)$) {};
  \node[vertex] (G--\i) at ($(2.1, \i*1.5-1.5)$) {};
}
\draw[color=blue] (G--7) -- (G-5) node[pos=0.5,mid] (M-7-5) { 3 };
\draw[color=blue] (G-6) .. controls +(0.70, 0.35) and +(0.70, -0.35) .. (G-7) node[pos=0.5,mid] (M6-7) { 4 };
\draw (G--4) .. controls +(-2.10, -1.05) and +(-2.10, 1.05) .. (G--1) node[pos=0.5,mid] (M-4--1) { 1 };
\draw (G-1) .. controls +(2.10, 1.05) and +(2.10, -1.05) .. (G-4) node[pos=0.5,mid] (M1-4) { 1 };
\draw (G--6) .. controls +(-0.70, -0.35) and +(-0.70, 0.35) .. (G--5) node[pos=0.5,mid] (M-6--5) { 3 };
\draw (G--8) -- (G-8) node[pos=0.5,mid] (M-8-8) { 4 };
\draw (G-2) .. controls +(0.70, 0.35) and +(0.70, -0.35) .. (G-3) node[pos=0.5,mid] (M2-3) { 2 };
\draw (G--3) .. controls +(-0.70, -0.35) and +(-0.70, 0.35) .. (G--2) node[pos=0.5,mid] (M-3--2) { 2 };
% \draw[color=red, <-] (M6-7) -- (M-7-5);
% \draw[color=red, <-] (M2-3) -- (M1-4);
% \draw[color=red, <-] (M-3--2) -- (M-4--1);
% \draw[color=red, <-] (M-8-8) -- (M-7-5);
\end{tikzpicture}\ =\
\begin{tikzpicture}[xscale=0.58,yscale=0.45, thick, baseline={(current bounding box.east)}]
\tikzstyle{vertex} = [shape=rectangle, minimum height=15pt, minimum width=0pt, inner sep=0pt]
\tikzstyle{mid} = [draw, fill=white, shape=circle, minimum size=2pt, inner sep=1pt]
\foreach \i in {1,...,8} {
  \node[vertex] (G-\i) at ($(-1.64, \i*1.5-1.5)$) {};
  \node[vertex] (G--\i) at ($(1.44, \i*1.5-1.5)$) {};
}
\draw[color=blue] (G--7) -- (G-5) node[pos=0.5,mid] (M-7-5) { 1 };
\draw (G-3) .. controls +(0.70, 0.35) and +(0.70, -0.35) .. (G-4) node[pos=0.5,mid] (M3-4) { 1 };
\draw (G--4) .. controls +(-2.10, -1.05) and +(-2.10, 1.05) .. (G--1) node[pos=0.5,mid] (M-4--1) { 1 };
\draw (G-1) .. controls +(0.70, 0.35) and +(0.70, -0.35) .. (G-2) node[pos=0.5,mid] (M1-2) { 1 };
\draw (G--6) .. controls +(-0.70, -0.35) and +(-0.70, 0.35) .. (G--5) node[pos=0.5,mid] (M-6--5) { 3 };
\draw (G--8) -- (G-6) node[pos=0.5,mid] (M-8-6) { 4 };
\draw (G-7) .. controls +(0.70, 0.35) and +(0.70, -0.35) .. (G-8) node[pos=0.5,mid] (M7-8) { 7 };
\draw (G--3) .. controls +(-0.70, -0.35) and +(-0.70, 0.35) .. (G--2) node[pos=0.5,mid] (M-3--2) { 2 };
% \draw[color=red, <-] (M7-8) -- (M-8-6);
% \draw[color=red, <-] (M-3--2) -- (M-4--1);
% \draw[color=red, <-] (M-8-6) -- (M-7-5);
\end{tikzpicture}}
\caption[{Products of arc-diagrams}]{Two example of products of Okada
  arc-diagrams. Some arcs, along with the associated arc fragments
  which form them, are highlighted in blue.}
  \label{fig.compo_Okada_arc}
\end{figure}

We first gather some simple, but important observations:
\begin{proposition}
  The labelled product of non-crossing arc-diagrams is associative, and therefore the set of labelled
  non-crossing arc-diagrams is a semigroup.
\end{proposition}
\begin{proof}
  Associativity follows from the associativity of unlabelled product and
  the associativity of the minimum.
\end{proof}
\begin{proposition}
  The mirror map defined by
  $D\longmapsto D^\star\eqdef\set{\arcp{-a}{\ell}{-b}}{\arcp{a}{\ell}{b}\in D}$ is an
  anti-automorphism of the semigroup of labelled non-crossing arc-diagrams.
\end{proposition}

Since non-crossing arc-diagrams are in bijection with well-parenthesized
words, it is not unexpected that parity plays an important role. The following
fundamental lemma illustrates this:
\begin{lemma}\label{lemma:propag-parity}
  In a non-crossing arc-diagram, the indices of the endpoints of each arc have the same
  parity if an only if the arc is propagating.
\end{lemma}
\begin{proof}
   The number of nodes nested between
   the two endpoints of any non-crossing arc is always even.
   If the arc is non-propagating (i.e. its endpoints are either
   both positive or both negative) then the difference (of the
   indices) of its endpoints is odd, and hence the parities
   differ. If the arc is propagating then the heights
   of the left and righthand endpoints of the arc must both be
   either odd or even, and hence the parities agree.
%  If an arc is non-crossing, there is exactly twice the number
%  of nested edges node between the two extremities: they parity
%  must differ. Conversely, if the arc is propagating, the parity of its extremity is the
%  opposite of the number of propagating arc it is nested in.
\end{proof}
\bigskip

We are now in a position to define the main object of this paper:
\begin{definition}
  \label{def.okada.arc}
  A labelled, non-crossing arc-diagram is called an \defn{Okada arc-diagram} if
  the following conditions are satisfied:
  \begin{enumerate}
  \item\label{okada-arc-cond-interv} the label  of each arc $\arc{a}{}{b}$ must
    be at least $1$ and at most $\min(|a|, |b|)$,
  \item\label{okada-arc-cond-parity} the label of each arc $\arc{a}{}{b}$
    must have the same parity as $\min(|a|, |b|)$.
  \item\label{okada-arc-cond-nested} if an arc $\arc{a}{}{b}$ is nested in an arc $\arc{c}{}{d}$ then
    the label of $\arc{a}{}{b}$ is \newline strictly larger than the label of $\arc{c}{}{d}$.
  \end{enumerate}
\end{definition}
The reader should examine
\cref{fig.Okada_arc_diagrams,fig.compo_Okada_arc} which show various Okada
arc-diagrams.

\begin{remark}
\label{rem:fusion-FPLCs}
As mentioned earlier, the concept of an Okada arc-diagram
is meant to capture the features of a FPLC which are invariant under the
re-writing system $\OSys$. Indeed, the height-labelled isotopy class
$[\mathcal{D}]$ of any FPLC $\mathcal{D}$ of rank $N$
is an Okada arc-diagram of rank $N$. As we'll see later, each rank $N$ Okada arc-diagram,
is realized as the height-labelled isotopy class of some FPLC of rank $N$.
The labelled product (of Okada arc-diagrams) is intended to condense an obvious product
which can be defined for height-labelled isotopy classes by
\begin{equation}
\label{eqn.product-isotopy-classes}
[\mathcal{D}_1 ] \boldsymbol{\cdot}  [\mathcal{D}_2] \, \eqdef \, [ \mathcal{D}_1 \circ \mathcal{D}_2 ]
\end{equation}
where $\mathcal{D}_1$ and $\mathcal{D}_2$ are representative FPLCs of rank $N$, both in
normal form, and where $\mathcal{D}_1 \circ \mathcal{D}_2$ is the rank $N$ FPLC
obtained by connecting the endpoints of the arcs on the right hand side of $\mathcal{D}_1$ and those on left hand side of $\mathcal{D}_2$
by intermediate horizontal arcs; see~\cref{fig.prod-arc-fplc} for an example.
We'll revisit Formula \ref{eqn.product-isotopy-classes}
later in Proposition \ref{prop:loop-to-diag}. 
\end{remark}

\begin{figure}[ht]
  \centering
  \begin{tikzpicture}[xscale=0.35,yscale=0.4, thick, baseline={(current bounding box.east)}]
\tikzstyle{vertex} = [shape=rectangle, minimum height=15pt, minimum width=0pt, inner sep=0pt]
\tikzstyle{mid} = [draw, fill=white, shape=circle, minimum size=2pt, inner sep=1pt]
\node[vertex] (G--9) at (2.64, 10.40) {};
\node[vertex] (G--8) at (2.64, 9.10) {};
\node[vertex] (G--7) at (2.64, 7.80) {};
\node[vertex] (G--6) at (2.64, 6.50) {};
\node[vertex] (G--5) at (2.64, 5.20) {};
\node[vertex] (G--4) at (2.64, 3.90) {};
\node[vertex] (G--3) at (2.64, 2.60) {};
\node[vertex] (G--2) at (2.64, 1.30) {};
\node[vertex] (G--1) at (2.64, 0.00) {};
\node[vertex] (G-1) at (-2.64, 0.00) {};
\node[vertex] (G-2) at (-2.64, 1.30) {};
\node[vertex] (G-3) at (-2.64, 2.60) {};
\node[vertex] (G-4) at (-2.64, 3.90) {};
\node[vertex] (G-5) at (-2.64, 5.20) {};
\node[vertex] (G-6) at (-2.64, 6.50) {};
\node[vertex] (G-7) at (-2.64, 7.80) {};
\node[vertex] (G-8) at (-2.64, 9.10) {};
\node[vertex] (G-9) at (-2.64, 10.40) {};
\draw[color=blue] (G--1) .. controls +(-2.59, 1.30) and +(2.59, -1.30) .. (G-5) node[pos=0.5,mid] (M-1-5) { 1 };
\draw[color=blue] (G--8) -- (G-6) node[pos=0.5,mid] (M-8-6) { 2 };
\draw[color=blue] (G-1) .. controls +(2.10, 1.05) and +(2.10, -1.05) .. (G-4) node[pos=0.5,mid] (M1-4) { 1 };
\draw[color=blue] (G--7) .. controls +(-2.10, -1.05) and +(-2.10, 1.05) .. (G--2) node[pos=0.5,mid] (M-7--2) { 2 };
\draw[color=blue] (G-2) .. controls +(0.70, 0.35) and +(0.70, -0.35) .. (G-3) node[pos=0.5,mid] (M2-3) { 2 };
\draw[color=blue] (G--4) .. controls +(-0.70, -0.35) and +(-0.70, 0.35) .. (G--3) node[pos=0.5,mid] (M-4--3) { 3 };
\draw[color=blue] (G--9) -- (G-7) node[pos=0.5,mid] (M-9-7) { 7 };
\draw[color=blue] (G-8) .. controls +(0.70, 0.35) and +(0.70, -0.35) .. (G-9) node[pos=0.5,mid] (M8-9) { 8 };
\draw[color=blue] (G--6) .. controls +(-0.70, -0.35) and +(-0.70, 0.35) .. (G--5) node[pos=0.5,mid] (M-6--5) { 5 };
% \draw[color=red, <-] (M8-9) -- (M-9-7);
% \draw[color=red, <-] (M2-3) -- (M1-4);
% \draw[color=red, <-] (M-4--3) -- (M-7--2);
% \draw[color=red, <-] (M-6--5) -- (M-7--2);
% \draw[color=red, <-] (M-7--2) -- (M-1-5);
% \draw[color=red, <-] (M-9-7) -- (M-8-6);
% \draw[color=red, <-] (M-8-6) -- (M-1-5);
\end{tikzpicture}
%$\cdot$
\begin{tikzpicture}[xscale=0.5,yscale=0.4, thick, baseline={(current bounding box.east)}]
\tikzstyle{vertex} = [shape=rectangle, minimum height=15pt, minimum width=0pt, inner sep=0pt]
\tikzstyle{mid} = [draw, fill=white, shape=circle, minimum size=2pt, inner sep=1pt]
\node[vertex] (G--9) at (1.44, 10.40) {};
\node[vertex] (G--8) at (1.44, 9.10) {};
\node[vertex] (G--7) at (1.44, 7.80) {};
\node[vertex] (G--6) at (1.44, 6.50) {};
\node[vertex] (G--5) at (1.44, 5.20) {};
\node[vertex] (G--4) at (1.44, 3.90) {};
\node[vertex] (G--3) at (1.44, 2.60) {};
\node[vertex] (G--2) at (1.44, 1.30) {};
\node[vertex] (G--1) at (1.44, 0.00) {};
\node[vertex] (G-1) at (-1.44, 0.00) {};
\node[vertex] (G-2) at (-1.44, 1.30) {};
\node[vertex] (G-3) at (-1.44, 2.60) {};
\node[vertex] (G-4) at (-1.44, 3.90) {};
\node[vertex] (G-5) at (-1.44, 5.20) {};
\node[vertex] (G-6) at (-1.44, 6.50) {};
\node[vertex] (G-7) at (-1.44, 7.80) {};
\node[vertex] (G-8) at (-1.44, 9.10) {};
\node[vertex] (G-9) at (-1.44, 10.40) {};
\draw[green] (G--1) .. controls +(-1.63, 0.82) and +(1.63, -0.82) .. (G-5) node[pos=0.5,mid] (M-1-5) { 1 };
\draw[green] (G--6) -- (G-6) node[pos=0.5,mid] (M-6-6) { 2 };
\draw[green] (G--9) -- (G-7) node[pos=0.5,mid] (M-9-7) { 3 };
\draw[green] (G-3) .. controls +(0.70, 0.35) and +(0.70, -0.35) .. (G-4) node[pos=0.5,mid] (M3-4) { 1 };
\draw[green] (G--5) .. controls +(-0.70, -0.35) and +(-0.70, 0.35) .. (G--4) node[pos=0.5,mid] (M-5--4) { 2 };
\draw[green] (G-8) .. controls +(0.70, 0.35) and +(0.70, -0.35) .. (G-9) node[pos=0.5,mid] (M8-9) { 6 };
\draw[green] (G--8) .. controls +(-0.70, -0.35) and +(-0.70, 0.35) .. (G--7) node[pos=0.5,mid] (M-8--7) { 5 };
\draw[green] (G-1) .. controls +(0.70, 0.35) and +(0.70, -0.35) .. (G-2) node[pos=0.5,mid] (M1-2) { 1 };
\draw[green] (G--3) .. controls +(-0.70, -0.35) and +(-0.70, 0.35) .. (G--2) node[pos=0.5,mid] (M-3--2) { 2 };
% \draw[color=red, <-] (M8-9) -- (M-9-7);
% \draw[color=red, <-] (M-3--2) -- (M-1-5);
% \draw[color=red, <-] (M-5--4) -- (M-1-5);
% \draw[color=red, <-] (M-8--7) -- (M-6-6);
% \draw[color=red, <-] (M-9-7) -- (M-6-6);
% \draw[color=red, <-] (M-6-6) -- (M-1-5);
\end{tikzpicture}
%%%%%%%%%%%%%%%%%%%%%%%%%%%
\ =\
%%%%%%%%%%%%%%%%%%%%%%%%%%%
\begin{tikzpicture}[xscale=0.35,yscale=0.4, thick, baseline={(current bounding box.east)}]
\tikzstyle{vertex} = [shape=rectangle, minimum height=15pt, minimum width=0pt, inner sep=0pt]
\tikzstyle{mid} = [draw, fill=white, shape=circle, minimum size=2pt, inner sep=1pt]
\node[vertex] (G--9) at (2.64, 10.40) {};
\node[vertex] (G--8) at (2.64, 9.10) {};
\node[vertex] (G--7) at (2.64, 7.80) {};
\node[vertex] (G--6) at (2.64, 6.50) {};
\node[vertex] (G--5) at (2.64, 5.20) {};
\node[vertex] (G--4) at (2.64, 3.90) {};
\node[vertex] (G--3) at (2.64, 2.60) {};
\node[vertex] (G--2) at (2.64, 1.30) {};
\node[vertex] (G--1) at (2.64, 0.00) {};
\node[vertex] (G-1) at (-2.64, 0.00) {};
\node[vertex] (G-2) at (-2.64, 1.30) {};
\node[vertex] (G-3) at (-2.64, 2.60) {};
\node[vertex] (G-4) at (-2.64, 3.90) {};
\node[vertex] (G-5) at (-2.64, 5.20) {};
\node[vertex] (G-6) at (-2.64, 6.50) {};
\node[vertex] (G-7) at (-2.64, 7.80) {};
\node[vertex] (G-8) at (-2.64, 9.10) {};
\node[vertex] (G-9) at (-2.64, 10.40) {};
\draw (G--9) .. controls +(-2.59, -1.30) and +(2.59, 1.30) .. (G-5) node[pos=0.5,mid] (M-9-5) { 1 };
\draw (G-6) .. controls +(0.70, 0.35) and +(0.70, -0.35) .. (G-7) node[pos=0.5,mid] (M6-7) { 2 };
\draw (G--6) .. controls +(-3.10, -1.05) and +(-3.10, 1.05) .. (G--1) node[pos=0.5,mid] (M-6--1) { 1 };
\draw (G-1) .. controls +(3.10, 1.05) and +(3.10, -1.05) .. (G-4) node[pos=0.5,mid] (M1-4) { 1 };
\draw (G--5) .. controls +(-0.70, -0.35) and +(-0.70, 0.35) .. (G--4) node[pos=0.5,mid] (M-5--4) { 2 };
\draw (G-2) .. controls +(0.70, 0.35) and +(0.70, -0.35) .. (G-3) node[pos=0.5,mid] (M2-3) { 2 };
\draw (G--8) .. controls +(-0.70, -0.35) and +(-0.70, 0.35) .. (G--7) node[pos=0.5,mid] (M-8--7) { 5 };
\draw (G-8) .. controls +(0.70, 0.35) and +(0.70, -0.35) .. (G-9) node[pos=0.5,mid] (M8-9) { 8 };
\draw (G--3) .. controls +(-0.70, -0.35) and +(-0.70, 0.35) .. (G--2) node[pos=0.5,mid] (M-3--2) { 2 };
% \draw[color=red, <-] (M8-9) -- (M-9-5);
% \draw[color=red, <-] (M6-7) -- (M-9-5);
% \draw[color=red, <-] (M2-3) -- (M1-4);
% \draw[color=red, <-] (M-3--2) -- (M-6--1);
% \draw[color=red, <-] (M-5--4) -- (M-6--1);
\end{tikzpicture},
    \begin{tikzpicture}[diamond, diamond/bound={ 8 }{ 10 }]
      \matrix[, diamond/matrix] {
    \TLlab{1}&\straightLoop{, draw=none}{blue}{white}\\
\TLlab{2}&\straightLoop{}{blue}{blue}&\straightLoop{, draw=none}{blue}{white}\\
\TLlab{3}&\tuC{blue}{blue}&\tuC{blue}{blue}&\straightLoop{, draw=none}{blue}{white}\\
\TLlab{4}&\tuC{blue}{blue}&\tuC{blue}{blue}&\straightLoop{}{blue}{blue}&\straightLoop{, draw=none}{blue}{white}\\
\TLlab{5}&\tuC{blue}{blue}&\tuC{blue}{blue}&\tuC{blue}{blue}&\straightLoop{}{blue}{blue}&\straightLoop{, draw=none}{blue}{white}\\
\TLlab{6}&\tuC{blue}{blue}&\tuC{blue}{blue}&\tuC{blue}{blue}&\straightLoop{}{blue}{blue}&\straightLoop{}{blue}{blue}&\straightLoop{, draw=none}{blue}{white}\\
\TLlab{7}&\straightLoop{}{blue}{blue}&\straightLoop{}{blue}{blue}&\straightLoop{}{blue}{blue}&\straightLoop{}{blue}{blue}&\straightLoop{}{blue}{blue}&\straightLoop{}{blue}{blue}&\straightLoop{, draw=none}{blue}{white}\\
\TLlab{8}&\straightLoop{}{blue}{blue}&\straightLoop{}{blue}{blue}&\straightLoop{}{blue}{blue}&\straightLoop{}{blue}{blue}&\straightLoop{}{blue}{blue}&\straightLoop{}{blue}{blue}&\straightLoop{}{blue}{blue}&\straightLoop{, draw=none}{blue}{white}\\
\TLlab{9}&\tuC{blue}{blue}&\tuC{blue}{blue}&\tuC{blue}{blue}&\tuC{blue}{blue}&\straightLoop{}{blue}{blue}&\straightLoop{}{blue}{blue}&\straightLoop{}{blue}{blue}&\straightLoop{}{blue}{blue}&\straightLoop{, draw=none}{blue}{white}\\
&\straightLoop{, draw=none}{white}{gray}&\straightLoop{}{gray}{gray}&\straightLoop{}{gray}{gray}&\straightLoop{}{gray}{gray}&\straightLoop{}{gray}{gray}&\straightLoop{}{gray}{gray}&\straightLoop{}{gray}{gray}&\straightLoop{}{gray}{gray}&\tuC{green}{green}&\straightLoop{, draw=none}{green}{white}\\
&&\straightLoop{, draw=none}{white}{gray}&\straightLoop{}{gray}{gray}&\straightLoop{}{gray}{gray}&\straightLoop{}{gray}{gray}&\straightLoop{}{gray}{gray}&\straightLoop{}{gray}{gray}&\straightLoop{}{gray}{gray}&\tuC{green}{green}&\tuC{green}{green}&\straightLoop{, draw=none}{green}{white}\\
&&&\straightLoop{, draw=none}{white}{gray}&\straightLoop{}{gray}{gray}&\straightLoop{}{gray}{gray}&\straightLoop{}{gray}{gray}&\straightLoop{}{gray}{gray}&\straightLoop{}{gray}{gray}&\tuC{green}{green}&\tuC{green}{green}&\straightLoop{}{green}{green}&\straightLoop{, draw=none}{green}{white}\\
&&&&\straightLoop{, draw=none}{white}{gray}&\straightLoop{}{gray}{gray}&\straightLoop{}{gray}{gray}&\straightLoop{}{gray}{gray}&\straightLoop{}{gray}{gray}&\tuC{green}{green}&\tuC{green}{green}&\straightLoop{}{green}{green}&\straightLoop{}{green}{green}&\straightLoop{, draw=none}{green}{white}\\
&&&&&\straightLoop{, draw=none}{white}{gray}&\straightLoop{}{gray}{gray}&\straightLoop{}{gray}{gray}&\straightLoop{}{gray}{gray}&\tuC{green}{green}&\tuC{green}{green}&\tuC{green}{green}&\tuC{green}{green}&\straightLoop{}{green}{green}&\straightLoop{, draw=none}{green}{white}\\
&&&&&&\straightLoop{, draw=none}{white}{gray}&\straightLoop{}{gray}{gray}&\straightLoop{}{gray}{gray}&\tuC{green}{green}&\straightLoop{}{green}{green}&\straightLoop{}{green}{green}&\straightLoop{}{green}{green}&\straightLoop{}{green}{green}&\straightLoop{}{green}{green}&\straightLoop{, draw=none}{green}{white}\\
&&&&&&&\straightLoop{, draw=none}{white}{gray}&\straightLoop{}{gray}{gray}&\tuC{green}{green}&\tuC{green}{green}&\tuC{green}{green}&\straightLoop{}{green}{green}&\straightLoop{}{green}{green}&\straightLoop{}{green}{green}&\straightLoop{}{green}{green}&\straightLoop{, draw=none}{green}{white}\\
&&&&&&&&\straightLoop{, draw=none}{white}{gray}&\tuC{green}{green}&\tuC{green}{green}&\tuC{green}{green}&\tuC{green}{green}&\tuC{green}{green}&\tuC{green}{green}&\tuC{green}{green}&\straightLoop{}{green}{green}&\straightLoop{, draw=none}{green}{white}\\
&&&&&&&&&\TLlab{\ov{9}}&\TLlab{\ov{8}}&\TLlab{\ov{7}}&\TLlab{\ov{6}}&\TLlab{\ov{5}}&\TLlab{\ov{4}}&\TLlab{\ov{3}}&\TLlab{\ov{2}}&\TLlab{\ov{1}}\\
    };
    \end{tikzpicture}
  \caption{Products of two Okada arc-diagrams and their associated FPLCs}
  \label{fig.prod-arc-fplc}
\end{figure}
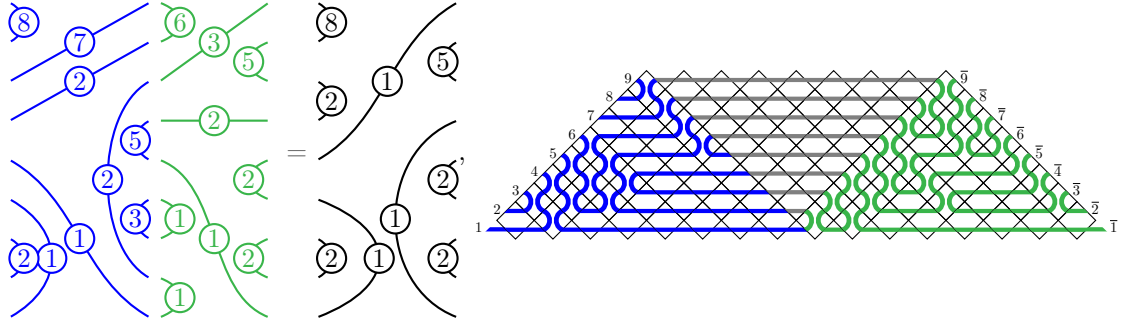

\begin{definition}
  Let \defn{$\id_N$} denote the \defn{identity arc-diagram} of rank $N$.
  This is the unique arc-diagram containing the horizontal arcs $\arc{a}{a}{\ov{a}}$
  for all $1\leq a \leq N$.
\end{definition}
For example, the first arc-diagram of \cref{fig.Okada_arc_diagrams} is $\id_8$.
The identity arc-diagram has the following characterization:
\begin{lemma}\label{lemma:uniq-id_propag}
  $\id_N$ is the unique Okada arc-diagram of rank~$N$ containing only
  propagating arcs.
\end{lemma}
\begin{proof}
  The three conditions in~\cref{def.okada.arc} are clearly satisfied by
  $\id_N$. Now suppose that $D$ is an Okada arc-diagram of rank~$N$ containing only
  propagating edges. Let $\pi_a$ denote the propagating arc with endpoint $a \in \setN$.
  The non-crossing condition together with the pigeonhole principle
  forces the righthand endpoint of $\pi_a$ to be $\ov{a}$.
  \Cref{okada-arc-cond-interv} ensures that $\pi_1$ has label $\ell_1 =1$.
  Clearly $\pi_{a+1}$ is nested in $\pi_a$ for each $1 \leq a \leq N-1$.
  \cref{okada-arc-cond-interv,okada-arc-cond-nested}, along with the fact that $\ell_1 = 1$,
  force $\pi_a$ to have label $\ell_a = a$ for each $1 \leq a \leq N$.
  \end{proof}
%  We denote $\alpha_i$ the arc starting at $i$ and $\beta_i$ the arc
%  starting at $\ov{i}$.
%  We show by induction on $i$ that $\alpha_i$ must be linked to $\ov{i}$ and
%  labelled by $i$. For the base case, we notice that the arc $\alpha_1$ ends
%  at $\ov{1}$ otherwise the arc $\beta_1$ can't be propagating without crossing
%  $\alpha_1$. \Cref{okada-arc-cond-interv} ensures that $\alpha_1$ is labelled $1$.
%
%  We now turn to the inductive step. The arc $\alpha_{i+1}$ has to end at $\ov{j}$
%  for $j>i$, otherwise $j\leq i$ and the end of arc $\alpha_j$ is the
%  node $\ov{j}$ by induction. By the same reasoning used in the base case, the arc
%  $\alpha_{i+1}$ has to ends at $\ov{i+1}$ otherwise the arc $\beta_{i+1}$ can't be
%  propagating without crossing it. Finally, noticing that $\alpha_{i+1}$ is nested
%  in $\alpha_{i}$, we see that its label has to be at least $i+1$ by
%  \cref{okada-arc-cond-nested} and is at most $i+1$ by
%  \cref{okada-arc-cond-interv}.
The set of Okada arc-diagrams is closed under the labelled product:
\begin{theorem}
  \label{prop.Okada_arc_semigroup}
  Rank $N$ Okada arc-diagrams form a sub-semigroup of the semigroup of labelled non-crossing arc-diagrams.
  Furthermore rank $N$ Okada arc-diagrams form a monoid
  whose identity element is $\id_N$.
\end{theorem}
Until we prove that the Okada arc-diagram monoid is isomorphic to the Okada
monoid $\Okada$ we'll denote the former by \defn{$\OkArc$}.
The proof will eventually use the following crucial remark.
\begin{remark}\label{remark:restrict_unnested}
  Let $D_1$ and $D_2$ be two Okada arc-diagrams and $\alpha$ be an arc in the
  product $D =  D_1 \boldsymbol{\cdot} D_2$. Up to isotopy, the arc $\alpha$ can be
  represented as a concatenation of arc fragments from $D_1$ and $D_2$ in the
  composition $D_1\circ D_2$.
  By definition, the label of arc $\alpha$ is the minimum of the labels of these arcs
  fragments. When computing this minimum,
  it is enough, by \cref{okada-arc-cond-nested}, to consider
  only those arc-fragments, among all the arc fragments forming $\alpha$,
  which are minimal with respect to the nesting order; i.e. those arcs
  which are not nested by other arc fragments of $\alpha$.
  \end{remark}
\begin{proof}
  It is clear from \cref{okada-arc-cond-interv} that the arc-diagram
  $\id_N$ is an identity element for the labelled product product. We need to
  show that the product of two Okada arc-diagrams is an Okada arc-diagram. It
  is known that the unlabelled product of
  two non-crossing arc-diagrams is a non-crossing arc-diagram.
  Let $D_1$ and $D_2$ be two Okada arc-diagrams. We need
  to show that the labelled product $D \eqdef D_1 \boldsymbol{\cdot} D_2$ satisfies the three
  conditions of \cref{def.okada.arc}. Any arc in the product $D$ is the
  isotopy class $[\alpha]$ of some path $\alpha$ in the composition
  $D_1\circ D_2$.
  \medskip

  \begin{enumerate}[wide, labelwidth=0pt, labelindent=0pt]
  \item \Cref{okada-arc-cond-interv} is clearly satisfied as the
    minimum of the labels of an arc is smaller than the labels of its two extreme
    arc fragments, which in turn are smaller than the labels of their endpoints.
    \medskip
    
  \item To show \cref{okada-arc-cond-parity}, by symmetry, we only
    have to consider arcs having at least one endpoint on the lefthand side of $D_1\circ D_2$. Let $[\alpha]$ be
    such an arc in $D$. The arc $[\alpha]$ has either one or two endpoints on
    the lefthand side of $D$. Among these choose the endpoint, denoted $a_1$, which is minimal.
    Consider the list $A = (a_1, a_2, \dots, a_\ell)$ of the
    absolute values of node indices encountered while traversing $\alpha$ starting from $a_1$
    and terminating at either $a_\ell$ if $\ell$ is even or $\ov{a_\ell}$ if $\ell$ is odd.
    If $\ell =2$ then $\alpha$ is $\arc{a_1}{}{a_2}$ which can only be an arc in $D_1$.
    Since $D_1$ is an Okada arc-diagram the label $\mathcal{h}$ of $\arc{a_1}{}{a_2}$
    satisfies \cref{okada-arc-cond-parity} and we are done.

    Otherwise $\ell \geq 3$ and the arc $\alpha$ travels first along a propagating arc $\arc{a_1}{}{\ov{a_2}}$ in $D_1$
    then follows an arc $\arc{a_2}{}{a_3}$ on the lefthand side of $D_2$ and subsequently proceeds to
    alternate between
    arcs $\arc{\ov{a_{2k-1}}}{}{\ov{a_{2k}}}$ on the righthand side of $D_1$ and arcs
   $\arc{a_{2k}}{}{a_{2k+1}}$ on lefthand side of $D_2$, finally traversing either a propagating arc
   $\arc{\ov{a_{\ell-1}}}{}{a_{\ell}}$ in $D_1$ or a propagating arc
   $\arc{a_{\ell-1}}{}{\ov{a_{\ell}}}$ in $D_2$ depending of the parity of
   $\ell$.
   See \cref{fig:path-alternate} for an illustration.

\begin{figure}[ht]
\[
\tikzset{mid arrow/.style={
decoration={markings,mark=at position 0.5 with {\arrow{#1}}},
postaction={decorate}}}
\begin{tikzpicture}[scale=0.4, rotate=-90,thick, baseline={(current bounding
    box.east)}, >=Computer Modern Rightarrow] %, >=Straight Barb]
\tikzstyle{vertex} = [shape=rectangle, minimum height=15pt, minimum width=0pt,
inner sep=0pt]
\tikzstyle{minimal} = [ultra thick, red]
\draw[black,thin] (0, 31.5) -- (0,-1.5) -- (10,-1.5) -- (10, 31.5) -- (0, 31.5);
\draw[black,thin] (5, 31.5) -- (5,-1.5);
\node[rotate=-90] at(2,0) {$\mathcal{O}_1$};
\node[rotate=-90] at(8,0) {$\mathcal{O}_2$};
\foreach \i in {1,...,21} {
  \node[vertex] (G-\i) at ($(0.00, \i*1.5-1.5)$) {};
  \node[vertex] (G--\i) at ($(5.00, \i*1.5-1.5)$) {};
  \node[vertex] (G+\i) at ($(10.00, \i*1.5-1.5)$) {};
}
\foreach \i in {1,...,21} {
  \node[left,rotate=-90] at (G-\i) {$\i$};
  \node[right,rotate=-90] at (G+\i) {$\i$};
}
\draw[minimal,mid arrow={<}] (G--15) .. controls +(-4.51, -2.26) and +(1.63, 0.82) .. (G-7);
\draw[mid arrow={<}] (G--21) .. controls +(-3.50, -1.75) and +(-3.50, 1.75) .. (G--16);
\draw[mid arrow={>}] (G--20) .. controls +(-2.10, -1.05) and +(-2.10, 1.05) .. (G--17);
\draw[minimal,mid arrow={>}] (G--12) .. controls +(-0.90, -0.35) and +(-0.90, 0.35) .. (G--11);
\draw[minimal,mid arrow={>}] (G--10) .. controls +(-4.90, -2.45) and +(-4.90, 2.45) .. (G--1);
\draw[mid arrow={<}] (G--9) .. controls +(-3.50, -1.75) and +(-3.50, 1.75) .. (G--2);
\draw[mid arrow={>}] (G--8) .. controls +(-0.70, -0.35) and +(-0.70, 0.35) .. (G--7);
\draw[mid arrow={>}] (G--6) .. controls +(-2.10, -1.05) and +(-2.10, 1.05) .. (G--3);
\draw[mid arrow={<}] (G--5) .. controls +(-0.70, -0.35) and +(-0.70, 0.35) .. (G--4);
\draw[minimal, mid arrow={<}] (G+13) .. controls +(-3.55, -1.78) and +(3.55, 1.78) .. (G--7);
\draw[mid arrow={<}] (G--20) .. controls +(0.70, 0.35) and +(0.70, -0.35) .. (G--21);
\draw[mid arrow={<}] (G--12) .. controls +(2.10, 1.05) and +(2.10, -1.05) .. (G--17);
\draw[mid arrow={>}] (G--15) .. controls +(0.70, 0.35) and +(0.70, -0.35) .. (G--16);
\draw[mid arrow={<}] (G--10) .. controls +(0.70, 0.35) and +(0.70, -0.35) .. (G--11);
\draw[mid arrow={<}] (G--8) .. controls +(0.70, 0.35) and +(0.70, -0.35) .. (G--9);
\draw[minimal, mid arrow={>}] (G--1) .. controls +(3.50, 1.75) and +(3.50, -1.75) .. (G--6);
\draw[mid arrow={<}] (G--2) .. controls +(2.10, 1.05) and +(2.10, -1.05) .. (G--5);
\draw[mid arrow={>}] (G--3) .. controls +(0.70, 0.35) and +(0.70, -0.35) .. (G--4);
\draw[gray!70, dotted] (G--18) .. controls +(-0.90, -0.35) and +(-0.90, 0.35) .. (G--19);
\draw[gray!70, dotted] (G--13) .. controls +(0.90, -0.35) and +(0.90, 0.35) .. (G--14);
\foreach \i in {1,3,...,21} {
  \node[green] at (G-\i) {\pgfuseplotmark{*}};
  \node[green] at (G--\i) {\pgfuseplotmark{*}};
  \node[green] at (G+\i) {\pgfuseplotmark{*}};
}
\foreach \i in {2,4,...,20} {
  \node[color=blue] at (G-\i) {\pgfuseplotmark{square*}};
  \node[color=blue] at (G--\i) {\pgfuseplotmark{square*}};
  \node[color=blue] at (G+\i) {\pgfuseplotmark{square*}};
}
\end{tikzpicture}
\]
\caption[A path in a composition]{%
  Traveling along a path $\alpha$ in the composition $D_1 \circ D_2$. Odd nodes are 
  green circle, while even nodes blue square. Minimal arcs are shown in thick red.}
  \label{fig:path-alternate}
\end{figure}
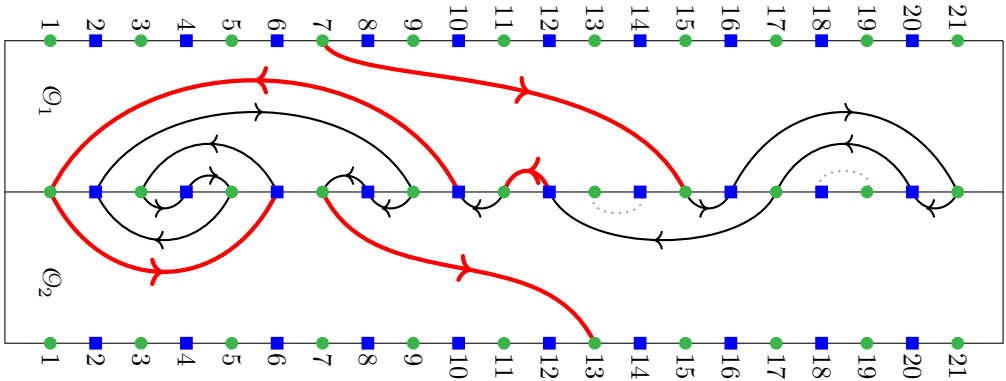

Call $p$ the parity of $a_1$ and let $q\eqdef p+1 \bmod 2$ be the opposite
parity. Thanks to \cref{lemma:propag-parity}, $a_2$ also has parity $p$.
Due to alternation, the arcs in $D_1$ travel in the $q \to p$
direction while the arcs of $D_2$ travel in the $p \to q$
direction. The entire procedure is illustrated below
\[
  \arc[D_1]{\stackrel{\displaystyle p}{\mid}{a_1}\ }{}{\ov{a_2}}
  \stackrel{\displaystyle p}{\mid}
  \arc[D_2]{a_2}{}{a_3}
  \stackrel{\displaystyle q}{\mid}
  \ov{a_3}
  \quad\dots\quad
  {a_{2k-1}}
  \stackrel{\displaystyle q}{\mid}
  \arc[D_1]{\ov{a_{2k-1}}}{}{\ov{a_{2k}}}
  \stackrel{\displaystyle p}{\mid}
  \arc[D_2]{a_{2k}}{}{a_{2k+1}}
  \stackrel{\displaystyle q}{\mid}
  \ov{a_{2k+1}}
  \quad\dots
\]
ending with
\[
  a_{\ell-1}
  \stackrel{\displaystyle q}{\mid}
  \arc[D_1]{\ov{a_{\ell-1}}}{}{a_\ell  \ \stackrel{\displaystyle q}{\mid}}
  \ \text{ if $\ell$ is even, }
  \qquad
  \text{ or }\
%  \ov{a_{\ell-1}}
%  \stackrel{\displaystyle p}{\mid}
%  \arc[D_2]{a_{\ell-1}}{}{\ov{a_{\ell}}  \stackrel{\displaystyle p}{\mid}
  \ov{a_{\ell-1}}
  \stackrel{\displaystyle p}{\mid}
  \arc[D_2]{a_{\ell-1}}{}{\ov{a_\ell} \ \stackrel{\displaystyle p}{\mid}}
  \ \text{ if $\ell$ is odd. }
\]

The statement is a consequence of the following fact:
    \begin{lemma}
      \label{lemma:arc_travel_direction}
      Consider the set $\mathcal{A}_1(\alpha)$ consisting of all arcs of $D_1$ which are fragments of $\alpha$. Order
      the arcs in $\mathcal{A}_1(\alpha)$ with respect to the nesting order in $D_1$. Any minimal right arc in $\mathcal{A}_1(\alpha)$
      is traversed from top to bottom when following $\alpha$ from $a_1$ to $a_\ell$. Likewise,
      let $\mathcal{A}_2(\alpha)$ be the set of all arcs of $D_2$ which are fragments of $\alpha$. Order
      the arcs in $\mathcal{A}_2(\alpha)$ with respect to the nesting order in $D_2$. Any minimal left arc in $\mathcal{A}_2(\alpha)$
      is traversed from bottom to top when following $\alpha$ from $a_1$ to $a_\ell$.
%      Any right arc of $D_1$ which is part of $\alpha$ and which is minimal among such
%      arcs with respect to the nesting order (that is not nested in any other arc fragment of $\alpha$ in $D_1$) is
%      traversed from top to bottom. Any right arc of $D_2$ which is part of $\alpha$ and which is not nested in any
%      other arc fragment of $\alpha$ in $D_2$ is traversed from bottom to top.
    \end{lemma}

    Indeed, \cref{lemma:arc_travel_direction} implies that the index of the bottom endpoint of each
    nesting-minimal arc on the righthand side of $D_1$ which is encountered while
    traversing $\alpha$ must have parity $p$.
    Finally, depending on the parity of its label, the arc $\alpha$ ends
    \begin{itemize}
    \item either with a propagating arc in $D_1$ which, due to the
      hypothesis that $a_1$ is minimal, is above $\arc{a_1}{}{a_2}$ and therefore not
      minimal;
    \item or with a propagating arc in $D_2$ whose endpoints both have parity $p$.
    \end{itemize}
    As a consequence, following \cref{remark:restrict_unnested}, the
    label of $[\alpha]$ can be computed by taking
    the minimum of the labels of its arc fragments which are not nested in any other arcs
    fragments of $\alpha$. Since the labels of those arcs all have parity $p$, the minimum of
    these label also has parity $p$. Thus the label of $[\alpha]$ matches the parity
    of both of its endpoints.

    We need still need to prove the Lemma.
    \begin{proof}[Proof of \cref{lemma:arc_travel_direction}]
This is a consequence of the fact that the diagrams $D_1$ and $D_2$ are
non-crossing.  Embed the composition $D_1 \circ D_2$ into the rectangle
$[0,2] \times [0, n+1]$ so that the $i\ordinal$ left vertex of $D_1$
is situated at position $(0,i)$, the $i\ordinal$ right vertex of $D_1$
and the $i\ordinal$ left vertex of $D_2$ occupy
position $(1,i)$, and the $i\ordinal$ right vertex of $D_2$ is
situated in position $(2,i)$. Orient $\alpha$ from beginning to
end.  As a Jordan curve, $\alpha$
cuts the rectangle into two connected regions: The region to the
right of $\alpha$ containing the bottom segment $[0,2] \times \{ 0 \}$
is called $\mathcal{O}$. Let $\mathcal{O}_1$ and
 $\mathcal{O}_2$ denote the intersections of $\mathcal{O}$
 with the sub-rectangles $[0,1] \times [0,n+1]$ and
 $[1,2] \times [0,n+1]$ respectively.
The region $\mathcal{O}_1$ lies to the right of
each minimal right arc $\arc{\ov{a_k}}{}{\ov{a_{k+1}}}$
in $\mathcal{A}_1(\alpha)$.
Therefore $\arc{\ov{a_k}}{}{\ov{a_{k+1}}}$
must be traversed from top to bottom while following
$\alpha$ from start to finish.
Likewise, the region $\mathcal{O}_2$ lies to the right of
each minimal left arc $\arc{a_k}{}{a_{k+1}}$
in $\mathcal{A}_2(\alpha)$
Consequently $\arc{a_k}{}{a_{k+1}}$
must be traversed from bottom to top while following
$\alpha$ from start to finish.
\end{proof}
  \item Finally, \cref{okada-arc-cond-nested} is an easy consequence of
    the following observation:
    \begin{lemma}
      \label{lemma:path_arc_nested}
      Suppose that $[\alpha]$ and $[\beta]$ are two arcs of the product $D_1 \boldsymbol{\cdot} D_2$
      such that $[\beta]$ is nested in $[\alpha]$. Any arc fragment of $\beta$ in
      $D_1$ or $D_2$ is nested in some arc fragment of $\alpha$ in $D_1$ or $D_2$
      respectively.
    \end{lemma}
    \begin{proof}[Proof of \cref{lemma:path_arc_nested}]
      This is again a consequence of the fact that the diagrams $D_1$ and
      $D_2$ are non-crossing.

      Let's proceed as in  Lemma \ref{lemma:arc_travel_direction} and embed $D = D_1 \circ D_2$ into
      the rectangle $[0,2] \times [0,n+1]$ and let $\mathcal{O}(\alpha)$ be the region inside
      the rectangle to the right $\alpha$. As before $\mathcal{O}_1(\alpha)$ and $\mathcal{O}_2(\alpha)$
      will denote the intersections of $\mathcal{O}(\alpha)$ with the rectangles $[0,1] \times [0,n+1]$
      and $[1,2] \times [0,n+1]$.
      By symmetry we can assume,
      without loss of generality, that the starting vertices of $\alpha$ and $\beta$ are
      situated respectively at $a_1$ and $b_1$ on the left hand side of $D_1$. The nesting
      hypothesis forces $a_1<b_1$.

       If $\arc{b_k}{}{b_{k+1}}$ is an arc-fragment of $\beta$ in $D_1$
       which is not nested in any minimal arc of $\mathcal{A}_1(\alpha)$
       then $\arc{b_k}{}{b_{k+1}}$ must lie entirely in $\mathcal{O}_1(\alpha)$.
       However this situation is in conflict with the non-crossing property
       since that starting vertex $b_1$ of $\beta$ lies outside of
       $\mathcal{O}_1(\alpha)$. Therefore $\arc{b_k}{}{b_{k+1}}$ must be
       nested in some minimal arc of $\mathcal{A}_1(\alpha)$.
       A similar argument handles arc-fragments of $\beta$ in $D_2$.
%      According to \cref{lemma:arc_travel_direction}, nesting-minimal arcs of
%      $D_1$ are traversed from top to bottom, so being on their left
%      implies being nested. Similarly, minimal arcs of $D_2$ are traversed from
%      bottom to top, so being on their left also implies being
%      nested. This reasoning also applies to both the first arc fragment of $\alpha$, and the
%      last arc fragment of $\alpha$ if it is minimal and therefore in $D_2$. We have shown
%      that all arc fragments of $\beta$ are nested in a minimal arc fragments of $\alpha$.
    \end{proof}
  \end{enumerate}

  As a consequence, the label of any arc fragment in $\beta$ is strictly larger than the
  label of $\alpha$, and so the label of $[\beta]$ is also strictly larger than the label of
  $[\alpha]$. This ends the proof of \cref{prop.Okada_arc_semigroup}.
\end{proof}
Before investigating in detail the structure of the Okada arc-diagram monoid,
let's point out a few simple remarks:
\begin{remark}
  $\OkArc$ is stable under the mirror image map $D\mapsto D^\star$. Therefore
  the mirror image map is an anti-isomorphism of $\OkArc$.
\end{remark}
\begin{remark}
  Let $\iota_N$ denote the map from $\OkArc[N]$ to $\OkArc[N+1]$ adding to
  each arc-diagram the labelled arc $\alpha \eqdef\arc{N+1}{N+1}{\ov{N+1}}$. Clearly $\iota_N$
  is an injective monoid morphism whose image is the set of diagrams containing
  $\alpha$. More generally
  \[
    \begin{array}{rcl}
      \iota_{M,M}: \OkArc[M]\times\OkArc[N]&\longrightarrow &\OkArc[M+N]\\
      (D, E)&\longmapsto& D\ \cup\
                          \set{\arcp{i+N}{\ell+N}{j+N}}{\arcp{i}{\ell}{j} \in E}\,
    \end{array}
  \]
  is an injective monoid morphism.
\end{remark}

We now turn to the proof that $\Okada$ and $\OkArc$ are isomorphic monoids. We
first need to describe the generating set for $\OkArc$ matching the $\e{i}$'s.
\begin{definition} \label{def:gen-arc-okada}
  For $1\leq i < N$, let $\A{i}$ denote the Okada arc-diagram containing the arcs
    \quad $\arcp{a}{a}{\ov{a}}$ for $a \neq i, i+1$ as well as 
    $\arcp{i}{i}{i+1}$ and $\arcp{\ov{i}}{i}{\ov{i+1}}$.
\end{definition}
The diagrams $\A{i}$ for $N=5$ and $i = 1, \dots, 4$ are depicted below.
\[
\A{1}=\begin{tikzpicture}[scale=0.4, thick, baseline={(current bounding box.east)}]
\tikzstyle{vertex} = [shape=rectangle, minimum height=15pt, minimum width=0pt, inner sep=0pt]
\tikzstyle{mid} = [draw, fill=white, shape=circle, minimum size=2pt, inner sep=1pt]
\foreach \i in {1,...,5} {
  \node[vertex] (G-\i) at ($(-1.44, \i*1.5-1.5)$) {};
  \node[vertex] (G--\i) at ($(1.44, \i*1.5-1.5)$) {};
}
\draw (G-1) .. controls +(0.7, 0.35) and +(0.7, -0.35) .. (G-2) node[pos=0.5,mid] (M1-2) { 1 };
\draw (G--2) .. controls +(-0.7, -0.35) and +(-0.7, 0.35) .. (G--1) node[pos=0.5,mid] (M-2--1) { 1 };
\draw (G--3) -- (G-3) node[pos=0.5,mid] (M-3-3) { 3 };
\draw (G--4) -- (G-4) node[pos=0.5,mid] (M-4-4) { 4 };
\draw (G--5) -- (G-5) node[pos=0.5,mid] (M-5-5) { 5 };
\end{tikzpicture}\qquad
\A{2}=\begin{tikzpicture}[scale=0.4, thick, baseline={(current bounding box.east)}]
\tikzstyle{vertex} = [shape=rectangle, minimum height=15pt, minimum width=0pt, inner sep=0pt]
\tikzstyle{mid} = [draw, fill=white, shape=circle, minimum size=2pt, inner sep=1pt]
\foreach \i in {1,...,5} {
  \node[vertex] (G-\i) at ($(-1.44, \i*1.5-1.5)$) {};
  \node[vertex] (G--\i) at ($(1.44, \i*1.5-1.5)$) {};
}
\draw (G--1) -- (G-1) node[pos=0.5,mid] (M-1-1) { 1 };
\draw (G-2) .. controls +(0.7, 0.35) and +(0.7, -0.35) .. (G-3) node[pos=0.5,mid] (M2-3) { 2 };
\draw (G--3) .. controls +(-0.7, -0.35) and +(-0.7, 0.35) .. (G--2) node[pos=0.5,mid] (M-3--2) { 2 };
\draw (G--4) -- (G-4) node[pos=0.5,mid] (M-4-4) { 4 };
\draw (G--5) -- (G-5) node[pos=0.5,mid] (M-5-5) { 5 };
\end{tikzpicture}\qquad
\A{3}=\begin{tikzpicture}[scale=0.4, thick, baseline={(current bounding box.east)}]
\tikzstyle{vertex} = [shape=rectangle, minimum height=15pt, minimum width=0pt, inner sep=0pt]
\tikzstyle{mid} = [draw, fill=white, shape=circle, minimum size=2pt, inner sep=1pt]
\foreach \i in {1,...,5} {
  \node[vertex] (G-\i) at ($(-1.44, \i*1.5-1.5)$) {};
  \node[vertex] (G--\i) at ($(1.44, \i*1.5-1.5)$) {};
}
\draw (G--1) -- (G-1) node[pos=0.5,mid] (M-1-1) { 1 };
\draw (G--2) -- (G-2) node[pos=0.5,mid] (M-2-2) { 2 };
\draw (G-3) .. controls +(0.7, 0.35) and +(0.7, -0.35) .. (G-4) node[pos=0.5,mid] (M3-4) { 3 };
\draw (G--4) .. controls +(-0.7, -0.35) and +(-0.7, 0.35) .. (G--3) node[pos=0.5,mid] (M-4--3) { 3 };
\draw (G--5) -- (G-5) node[pos=0.5,mid] (M-5-5) { 5 };
\end{tikzpicture}\qquad
\A{4}=\begin{tikzpicture}[scale=0.4, thick, baseline={(current bounding box.east)}]
\tikzstyle{vertex} = [shape=rectangle, minimum height=15pt, minimum width=0pt, inner sep=0pt]
\tikzstyle{mid} = [draw, fill=white, shape=circle, minimum size=2pt, inner sep=1pt]
\foreach \i in {1,...,5} {
  \node[vertex] (G-\i) at ($(-1.44, \i*1.5-1.5)$) {};
  \node[vertex] (G--\i) at ($(1.44, \i*1.5-1.5)$) {};
}
\draw (G--1) -- (G-1) node[pos=0.5,mid] (M-1-1) { 1 };
\draw (G--2) -- (G-2) node[pos=0.5,mid] (M-2-2) { 2 };
\draw (G--3) -- (G-3) node[pos=0.5,mid] (M-3-3) { 3 };
\draw (G-4) .. controls +(0.7, 0.35) and +(0.7, -0.35) .. (G-5) node[pos=0.5,mid] (M4-5) { 4 };
\draw (G--5) .. controls +(-0.7, -0.35) and +(-0.7, 0.35) .. (G--4) node[pos=0.5,mid] (M-5--4) { 4 };
\end{tikzpicture}
\]
\begin{proposition}\label{prop:relat-G}
  The $(\A{i})_{i\in\setN[N-1]}$ satisfy the Okada relations:
  \begin{alignat}{3}
  \A{i}^2&=\A{i} &\qquad& 1\leq i \leq N-1, \\
  \A{i}\A{j}&=\A{j}\A{i} &\qquad&\vert i-j\vert \geq 2, \\
  \A{i+1}\A{i}\A{i+1}&=\A{i+1} &\qquad&  1\leq i \leq N-2,
\end{alignat}
\end{proposition}
\begin{proof}
  We only show the third relation; the first two are left to the reader. The
  last one is shown in the picture below:
  \[
\begin{tikzpicture}[scale=0.4, thick, baseline={(current bounding box.east)}]
\tikzstyle{vertex} = [shape=rectangle, minimum height=15pt, minimum width=0pt, inner sep=0pt]
\tikzstyle{mid} = [draw, fill=white, shape=circle, minimum size=13pt, inner sep=1pt]
\foreach \i / \a in {1/south,5/center} {
  \node[vertex, anchor=\a] (G-\i) at ($(0, \i*1.5-1.6)$) {$\vdots$};
  \node[vertex, anchor=\a] (G--\i) at ($(6.5, \i*1.5-1.6)$) {$\vdots$};
  \node[vertex, anchor=\a] (G+\i) at ($(11, \i*1.5-1.6)$) {$\vdots$};
  \node[vertex, anchor=\a] (G++\i) at ($(17.5, \i*1.5-1.6)$) {$\vdots$};
}
\foreach \i in {2,...,4} {
  \node[vertex] (G-\i) at ($(0, \i*1.5-1.5)$) {$.$};
  \node[vertex] (G--\i) at ($(6.5, \i*1.5-1.5)$) {$.$};
  \node[vertex] (G+\i) at ($(11, \i*1.5-1.5)$) {$.$};
  \node[vertex] (G++\i) at ($(17.5, \i*1.5-1.5)$) {$.$};
}
\node[anchor=east] at (G-2) {$i$};
\node[anchor=east] at (G-3) {$i+1$};
\node[anchor=east] at (G-4) {$i+2$};
\node[anchor=west] at (G++2) {$\ov{i}$};
\node[anchor=west] at (G++3) {$\ov{i+1}$};
\node[anchor=west] at (G++4) {$\ov{i+2}$};
\draw (G-2) -- (G--2) node[pos=0.5,mid] { $i$ };
\draw (G-3) .. controls +(2.1, 0.35) and +(2.1, -0.35) .. (G-4)
    node[pos=0.5,mid,shape=ellipse] { \scriptsize $i+1$ };
\draw (G--3) .. controls +(-2.1, 0.35) and +(-2.1, -0.35) .. (G--4)
    node[pos=0.5,mid,shape=ellipse] { \scriptsize $i+1$ };
\draw (G--2) .. controls +(1.5, 0.35) and +(1.5, -0.35) .. (G--3)
    node[pos=0.5,mid] { $i$ };
\draw (G+3) .. controls +(-1.5, -0.35) and +(-1.5, 0.35) .. (G+2)
    node[pos=0.5,mid] { $i$ };
\draw (G--4) -- (G+4) node[pos=0.5,mid,shape=ellipse] { \scriptsize $i+2$ };
\draw (G+2) -- (G++2) node[pos=0.5,mid] { $i$ };
\draw (G+3) .. controls +(2.1, 0.35) and +(2.1, -0.35) .. (G+4)
    node[pos=0.5,mid,shape=ellipse] { \scriptsize $i+1$ };
\draw (G++3) .. controls +(-2.1, 0.35) and +(-2.1, -0.35) .. (G++4)
    node[pos=0.5,mid,shape=ellipse] { \scriptsize $i+1$ };
\end{tikzpicture}
\ =\joinrel\!\!=\ 
\begin{tikzpicture}[scale=0.4, thick, baseline={(current bounding box.east)}]
\tikzstyle{vertex} = [shape=rectangle, minimum height=15pt, minimum width=0pt, inner sep=0pt]
\tikzstyle{mid} = [draw, fill=white, shape=circle, minimum size=13pt, inner sep=1pt]
\foreach \i / \a in {1/south,5/center} {
  \node[vertex, anchor=\a] (G-\i) at ($(0, \i*1.5-1.6)$) {$\vdots$};
  \node[vertex, anchor=\a] (G--\i) at ($(6.5, \i*1.5-1.6)$) {$\vdots$};
}
\foreach \i in {2,...,4} {
  \node[vertex] (G-\i) at ($(0, \i*1.5-1.5)$) {$.$};
  \node[vertex] (G--\i) at ($(6.5, \i*1.5-1.5)$) {$.$};
}

\node[anchor=east] at (G-2) {$i$};
\node[anchor=east] at (G-3) {$i+1$};
\node[anchor=east] at (G-4) {$i+2$};
\node[anchor=west] at (G--2) {$\ov{i}$};
\node[anchor=west] at (G--3) {$\ov{i+1}$};
\node[anchor=west] at (G--4) {$\ov{i+2}$};
\draw (G-2) -- (G--2) node[pos=0.5,mid] { $i$ };
\draw (G-3) .. controls +(2.1, 0.35) and +(2.1, -0.35) .. (G-4)
   node[pos=0.5,mid,shape=ellipse] { \scriptsize $i+1$ };
\draw (G--3) .. controls +(-2.1, 0.35) and +(-2.1, -0.35) .. (G--4)
   node[pos=0.5,mid,shape=ellipse] { \scriptsize $i+1$ };
\end{tikzpicture}\qedhere
\]
\end{proof}
\begin{remark}
  We'll write $\A{i}^{(N)}$ to emphasize that $\A{i}$ lives in $\OkArc$;
  then $\iota_N\left(\A{i}^{(N)}\right) = \A{i}^{(N+1)}$.
\end{remark}
\medskip

%To construct the isomorphism from $\Okada$ to $\OkArc$, we need to show that
%$(\A{i})$ generate $\OkArc$. It is clear from the definition of labelled diagram product, that if a
%product ends with $\A{i}$, then it contains an arc
%$\arc{\ov{i}}{i}{\ov{i+1}}$. The converse is also true: If an element
%$\mathsf{e} \in\Okada$ contains an arc linking $\arc{\ov{i}}{i}{\ov{i+1}}$ then it can
%be factored as $\mathsf{e}= \mathsf{f} \ssp \A{i}$. However, constructing $\mathsf{f}$ in general is
%tricky. Therefore, we postpone the proof of the general case until we know
%that $\Okada$ and $\OkArc$ are isomorphic, and focus on the simpler case where
%$i$ is maximal. The goal is to find a such factorization of any element
%$D \in\OkArc$.  \smallskip

To construct the isomorphism from $\Okada$ to $\OkArc$, we need to show that
$(\A{i})_{i=1\dots N-1}$ generate $\OkArc$. It is clear from the definition of the labelled product, that
if an Okada arc-diagram $D$ can be factored as
$D  = D' \boldsymbol{\cdot} \A{i}$ then $D$ contains an arc
$\arc{\ov{i}}{i}{\ov{i+1}}$. The converse is also true: If an element
$D \in \OkArc$ contains an arc linking $\arc{\ov{i}}{i}{\ov{i+1}}$ then it can
be factored as $D= D' \boldsymbol{\cdot} \A{i}$. However, constructing $D'$ is
tricky in general. Therefore, we postpone the proof of the general case until we know
that $\Okada$ and $\OkArc$ are isomorphic, and focus on the simpler case where
$i$ is maximal. The goal is to find a such factorization of any element
$D \in\OkArc$.  \smallskip

The next two lemmas show that any rank $N$ Okada arc-diagram $D \ne \id_N$
must contain a non-propagating arc whose label is maximal on both the left and righthand sides.

%$\arc{\ov{i}}{i}{\ov{i+1}}$.
%\begin{lemma}\label{lemma:id-arc}
%  Let $D\in\OkArc\setminus\{\id_N\}$ be an Okada arc-diagram. Suppose that
%  for some $k \leq N$ the arc ending at $k$ has label $\ell = k$. Then
%  \begin{itemize}
%  \item either $D$ contains all arcs $\arc{i}{i}{\ov{i}}$ for $i\geq k$,
%  \item or $D$ contains at least one arc $\arc{i}{i}{i+1}$ for some $i\geq k$.
%  \end{itemize}
%\end{lemma}
%\begin{proof}
%  Suppose that $D$ doesn't contain any of the arcs $\arc{i}{i}{\ov{i}}$ for
%  $i\geq k$. By \cref{okada-arc-cond-interv,okada-arc-cond-nested}, if the
%  arc ending at $k$ is $\arc{k}{k}{\ov{k}}$, then the arc ending at $k+1$ is
%  nested in the former and must be labelled $k+1$. Therefore, we can assume,
%  without a loss of generality that $k$ is not linked to $\ov{k}$. However
%  since the label of the arc is $k$, its other end must be $i$ or $\ov{i}$
%  for some $i>k$. As a consequence, $D$ must contain a non-propagating arc
%  above $k$ on the left side.
%
%  Let $\alpha$ be such an arc with minimal difference between the indices of its
%  two endpoints. Then $\alpha$ must link some $i$ with $i+1$. Indeed, if this is not the
%  case, then $D$ also contains another arc nested in $\alpha$ contradicting
%  minimality. Let $i_0$ be the minimal index of an arc linking $i$ and
%  $i+1$. Thanks to
%  \cref{okada-arc-cond-interv,okada-arc-cond-nested},
%  a reasoning, similar to the argument in \cref{lemma:uniq-id_propag},
%  can be used to show (by induction on $j$) that all arcs starting at
%  some $j$ where $1 \leq j\leq i_0$ have label $j$.
%\end{proof}

  \begin{lemma}\label{lemma:id-arc}
  Let $D\in\OkArc \setminus \{  \id_N \} $ be an Okada arc-diagram, then
  $D$ contains at least one non-propagating arc of the form $\arcp{a}{a}{a+1}$
  for some $a < N$.
  \end{lemma}
\begin{proof}
  $D \ne \id_N$ so it must contain at least one non-propagating arc of
  the form $\arc{u}{}{v}$ for some pair $u , v \in \setN$ with $u <v$. Among
  these arcs choose one for which $v- u$ is minimal. It must be
  the case that $v = u+1$ otherwise
  the non-crossing property would imply there is an arc $\arc{u'}{}{v'}$
  with $u' < v'$ and nested in $\arc{u}{}{v}$. But then $v'-u' < v - u$ which violates
  the minimality condition. Now select an arc $\pi$ of the form $\arc{a}{\ell}{a+1}$
  whose endpoint $a \in \setN$ is minimal.
  If $a =1$ then we are done because \cref{okada-arc-cond-interv} forces
  $\ell =1$. If $a > 1$ then the arc $\pi_k$ with endpoint $k$
  for $k < a$ must be propagating. Furthermore $\pi_{k+1}$
  is nested in $\pi_k$ for each $1 \leq k \leq a-1$ where
  we set $\pi_a \eqdef \pi$. It follows from
  \cref{okada-arc-cond-interv,okada-arc-cond-nested}
  that $\pi_k$ has label $\ell_k = k$ for each
  $1 \leq k \leq a$. In particular $\pi = \pi_a$ has
  label $\ell = \ell_a =a$.
  \end{proof}

Applying mirror symmetry $D \mapsto D^\star$ we obtain
\begin{corollary}\label{lemma:non_id_desc}
  Let $D\in\OkArc \setminus \{ \id_N \}$ be an Okada arc-diagram
  then $D$ contains at least one arc $\arcp{\ov{a}}{a}{\ov{a+1}}$ for some $a < N$.
\end{corollary}

%Applying the previous lemma with $n=1$, we get:
%\begin{corollary}\label{lemma:non_id_desc}
%  Let $D\in\OkArc\setminus\{\id_N\}$ be an Okada arc-diagram which is not the
%  identity. Then $D$ contains at least one arc $\arc{i}{i}{i+1}$ for some
%  $i$. By mirror symmetry, $D$ contains at least one arc
%  $\arc{\ov{i}}{i}{\ov{i+1}}$ for some $i$.
%\end{corollary}

We will need the following computation:
\begin{remark}\label{remark:descending-product}
  For $d < N$, the product $\A{N-1}\A{N-2} \ssp \boldsymbol{\cdots} \A{d}$ is the Okada arc-diagram with
  the arcs
  $\arcp{a}{a}{\ov{a}}$ for $a<d$, $\arcp{\ov{d}}{d}{\ov{d+1}}$,
  $\arcp{a}{a}{\ov{a+2}}$ for $d\leq a \leq N-2$, and $\arcp{N-1}{N-1}{N}$.
  See \cref{fig:right-code-factor} for some examples.
\end{remark}
We now define a factorization which allows us to prove that $(\A{i})_{i\in\setN[N-1]}$ generate
$\OkArc$.
\begin{definition}\label{def:flat_diagram}
  Suppose that $D$ is an Okada arc-diagram of rank~$N$ which doesn't contain
  the arc $\arc{N}{N}{\ov{N}}$. Define $\LDes(D)$ to be the largest
  index $a  \in \setN$ such that $D$ contains the arc $\arc{\ov{a}}{a}{\ov{a+1}}$ (which
  must exist by \cref{lemma:non_id_desc}). For all
  $a \in \setN[N] \cup \setNN$ such that $a \ne \ov{\LDes(D)}$ and $a \ne \ov{\LDes(D)+1}$ define
  \[
    a^\flat\eqdef
    \begin{cases}
      \ov{N-1} & \text{ if $a=N$ } \\
      a &\text{\ if  $a< N$}
    \end{cases}
    \qandq
    \ov{a}^{\,\flat} \eqdef
    \begin{cases}
      \ov{a - 2} & \text{ if $a \geq\LDes(D)+2$} \\
      \ov{a} &\text{\ if $a < \LDes(D)$}
    \end{cases}
  \]
  We define $D^\flat$ to be the diagram containing
  \begin{itemize}
  \item the arc $\arc{N}{N}{\ov{N}}$
  \item the arc $\arc{i^\flat}{\ell}{j^\flat}$ whenever
    $\arc{i}{\ell}{j}$ is an arc of $D$
    other than $\arc{\ov{d}}{d}{\ov{d+1}}$ where $d = \LDes(D)$
  \end{itemize}
\end{definition}

\noindent
The notation $\LDes$ is an abbreviation for \defn{ largest descent}, a terminology which we'll
 justify later.

 Graphically, we remove the arc
 $\arcp{\ov{\LDes(D)}}{\LDes(D)}{\ov{\LDes(D)+1}}$, shift by two
 places all the nodes in the interval
 $N, \ov{N}, \ov{N-1}, \dots, \ov{\LDes(D)+2}$ to
 $\ov{N-1}, \dots, \ov{\LDes(D)}$ and add
 $\arc{N}{N}{\ov{N}}$. \cref{fig:right-code-factor} shows an example
 of the construction of $D^\flat$.

\begin{proposition} \label{prop:right-code-factor}
  As in \cref{def:flat_diagram}, suppose that $D$ is an Okada arc-diagram of
  rank $N$ which doesn't contains $\arc{N}{N}{\ov{N}}$. Then $D^\flat$
  is an Okada arc-diagram such that $D$ factorizes as
  \[
    D=D^\flat \A{N-1}\A{N-2} \ssp \boldsymbol{\cdots} \A{\LDes(D)}\,.
  \]
\end{proposition}
  % \todo[FH]{Better example : $[3, 4, 9, 5, 7, 8, 2, 6, 1]$, and put back
  %   nesting arcs.}
  % sage: e = M.random_element()
  % sage: e = M.from_code(e.code()[:-1]+(6,))
  % sage: view(L([e, eq, L(e.right_c_factor())]))  # not tested
  % [3, 6, 2, 4, 5, 9, 7, 8, 1] [4, 3, 1, 7, 9, 2, 6, 8, 5]
\begin{figure}[ht]
  \makebox[\textwidth]{
\begin{tikzpicture}[xscale=0.7, yscale=0.4, thick, baseline={(current bounding box.east)}]
\tikzstyle{vertex} = [shape=rectangle, minimum height=15pt, minimum width=0pt, inner sep=0pt]
\tikzstyle{mid} = [draw, fill=white, shape=circle, minimum size=2pt, inner sep=1pt]
\foreach \i in {0,...,9} {
  \node[vertex] (G-\i) at ($(-1.44, \i*1.5-1.5)$) {};
  \node[vertex] (G--\i) at ($(2.64, \i*1.5-1.5)$) {};
}
\draw (G--1) .. controls +(-2.59, 1.30) and +(1.63, -0.82) .. (G-5) node[pos=0.6470588235294118,mid] (M-1-5) { 1 };
\draw (G--2) .. controls +(-2.59, 1.30) and +(1.63, -0.82) .. (G-6) node[pos=0.6470588235294118,mid] (M-2-6) { 2 };
\draw (G--5) .. controls +(-2.59, 1.30) and +(1.63, -0.82) .. (G-9) node[pos=0.6470588235294118,mid] (M-5-9) { 3 };
\draw (G-3) .. controls +(0.70, 0.35) and +(0.70, -0.35) .. (G-4) node[pos=0.5,mid] (M3-4) { 1 };
\draw (G--9) .. controls +(-2.10, -1.05) and +(-2.10, 1.05) .. (G--6) node[pos=0.5,mid] (M-9--6) { 4 };
\draw (G-7) .. controls +(0.70, 0.35) and +(0.70, -0.35) .. (G-8) node[pos=0.5,mid] (M7-8) { 5 };
\draw (G--8) .. controls +(-0.70, -0.35) and +(-0.70, 0.35) .. (G--7) node[pos=0.5,mid] (M-8--7) { 5 };
\draw (G-1) .. controls +(0.70, 0.35) and +(0.70, -0.35) .. (G-2) node[pos=0.5,mid] (M1-2) { 1 };
\draw (G--4) .. controls +(-0.70, -0.35) and +(-0.70, 0.35) .. (G--3) node[pos=0.5,mid] (M-4--3) { 3 };
% \draw[color=red, <-] (M7-8) -- (M-2-6);
% \draw[color=red, <-] (M-4--3) -- (M-2-6);
% \draw[color=red, <-] (M-8--7) -- (M-9--6);
% \draw[color=red, <-] (M-9--6) -- (M-5-9);
% \draw[color=red, <-] (M-5-9) -- (M-2-6);
% \draw[color=red, <-] (M-2-6) -- (M-1-5);
\path (G-0) -- (G--0) node[pos=0.5] {$D$};
\end{tikzpicture}
\quad=\quad
\begin{tikzpicture}[xscale=0.7, yscale=0.4, thick, baseline={(current bounding box.east)}]
\tikzstyle{vertex} = [shape=rectangle, minimum height=15pt, minimum width=0pt, inner sep=0pt]
\tikzstyle{mid} = [draw, fill=white, shape=circle, minimum size=2pt, inner sep=1pt]
\foreach \i in {0,...,9} {
  \node[vertex] (G-\i) at ($(-1.44, \i*1.5-1.5)$) {};
  \node[vertex] (G--\i) at ($(3.84, \i*1.5-1.5)$) {};
}
\draw (G--1) .. controls +(-3.55, 1.78) and +(1.63, -0.82) .. (G-5) node[pos=0.7272727272727272,mid] (M-1-5) { 1 };
\draw (G--2) .. controls +(-3.55, 1.78) and +(1.63, -0.82) .. (G-6) node[pos=0.7272727272727272,mid] (M-2-6) { 2 };
\draw (G-3) .. controls +(0.70, 0.35) and +(0.70, -0.35) .. (G-4) node[pos=0.5,mid] (M3-4) { 1 };
\draw (G--8) .. controls +(-3.50, -1.75) and +(-3.50, 1.75) .. (G--3) node[pos=0.5,mid] (M-8--3) { 3 };
\draw (G-7) .. controls +(0.70, 0.35) and +(0.70, -0.35) .. (G-8) node[pos=0.5,mid] (M7-8) { 5 };
\draw (G--7) .. controls +(-2.10, -1.05) and +(-2.10, 1.05) .. (G--4) node[pos=0.5,mid] (M-7--4) { 4 };
\draw (G-1) .. controls +(0.70, 0.35) and +(0.70, -0.35) .. (G-2) node[pos=0.5,mid] (M1-2) { 1 };
\draw (G--6) .. controls +(-0.70, -0.35) and +(-0.70, 0.35) .. (G--5) node[pos=0.5,mid] (M-6--5) { 5 };
\draw (G--9) -- (G-9) node[pos=0.7272727272727272,mid] (M-9-9) { 9 };
% \draw[color=red, <-] (M7-8) -- (M-2-6);
% \draw[color=red, <-] (M-6--5) -- (M-7--4);
% \draw[color=red, <-] (M-7--4) -- (M-8--3);
% \draw[color=red, <-] (M-8--3) -- (M-2-6);
% \draw[color=red, <-] (M-9-9) -- (M-2-6);
% \draw[color=red, <-] (M-2-6) -- (M-1-5);
\path (G-0) -- (G--0) node[pos=0.5] {$D^\flat$};
\end{tikzpicture}\hskip-0.36cm
%%%
\begin{tikzpicture}[xscale=0.6, yscale=0.4, thick, baseline={(current bounding box.east)}]
\tikzstyle{vertex} = [shape=rectangle, minimum height=15pt, minimum width=0pt, inner sep=0pt]
\tikzstyle{mid} = [draw, fill=white, shape=circle, minimum size=2pt, inner sep=1pt]
\foreach \i in {0,...,9} {
  \node[vertex] (G-\i) at ($(-1.44, \i*1.5-1.5)$) {};
  \node[vertex] (G--\i) at ($(1.44, \i*1.5-1.5)$) {};
}
\draw[color=blue] (G--1) -- (G-1) node[pos=0.5,mid] (M-1-1) { 1 };
\draw[color=blue] (G--2) -- (G-2) node[pos=0.5,mid] (M-2-2) { 2 };
\draw[color=blue] (G--5) -- (G-3) node[pos=0.5,mid] (M-5-3) { 3 };
\draw[color=blue] (G--6) -- (G-4) node[pos=0.5,mid] (M-6-4) { 4 };
\draw[color=blue] (G--7) -- (G-5) node[pos=0.5,mid] (M-7-5) { 5 };
\draw[color=blue] (G--8) -- (G-6) node[pos=0.5,mid] (M-8-6) { 6 };
\draw[color=blue] (G--9) -- (G-7) node[pos=0.5,mid] (M-9-7) { 7 };
\draw[color=blue] (G-8) .. controls +(0.70, 0.35) and +(0.70, -0.35) .. (G-9) node[pos=0.5,mid] (M8-9) { 8 };
\draw[color=blue] (G--4) .. controls +(-0.70, -0.35) and +(-0.70, 0.35) .. (G--3) node[pos=0.5,mid] (M-4--3) { 3 };
% \draw[color=red, <-] (M8-9) -- (M-9-7);
% \draw[color=red, <-] (M-4--3) -- (M-2-2);
% \draw[color=red, <-] (M-9-7) -- (M-8-6);
% \draw[color=red, <-] (M-8-6) -- (M-7-5);
% \draw[color=red, <-] (M-7-5) -- (M-6-4);
% \draw[color=red, <-] (M-6-4) -- (M-5-3);
% \draw[color=red, <-] (M-5-3) -- (M-2-2);
% \draw[color=red, <-] (M-2-2) -- (M-1-1);
\path (G-0) -- (G--0) node[pos=0.5,blue] {$\A{8}\A{7}\A{6}\A{5}\A{4}\A{3}$};
\end{tikzpicture}}
\caption[Examples of right factorization]{%
  An example of right factorization according to
  \cref{prop:right-code-factor}.
  % One can see that the
  % Hasse diagram of the nesting order of $D^\flat$ is the same up to shifting of
  % some node to the right as the one of $D$, plus possibly extra arc linking to
  % $\arc{N}{N}{\ov{N}}$.
  }
  \label{fig:right-code-factor}
\end{figure}
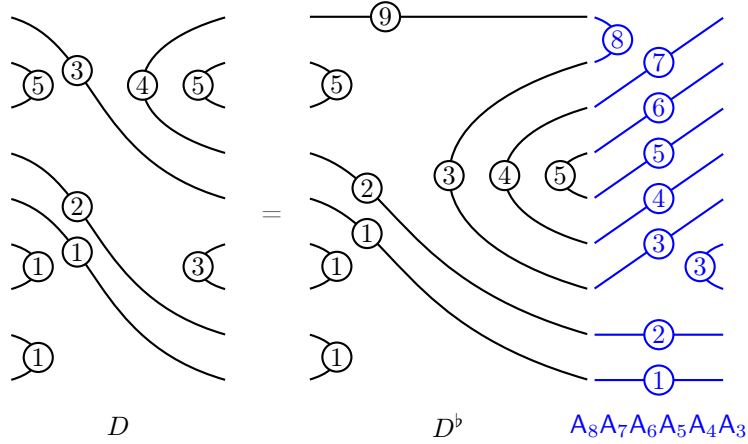
\begin{proof}
 First we show that $D^\flat$ is an Okada arc-diagram. Set $d\eqdef\LDes(D)$.
 We'll show that $D^\flat$ is non-crossing. Going from
  $D$ to $D^\flat$ amounts to removing an arc between two consecutive nodes, then
  shifting a bunch of consecutive nodes, relative to the order on $\setNN[N]$, and
  finally adding an arc between two consecutive nodes. None of these
  operations creates a crossing.

  We now verify the three condition in~\cref{def.okada.arc}:
  \begin{enumerate}[wide, labelwidth=0pt, labelindent=0pt]

  \item It suffices to prove that the arc incident to node $\ov{i}$ for
  each $i=d+2,\dots,N$ has label $\ell \leq i-2$ in $D$. Suppose not.
  Accordingly, there exists an arc incident to $\ov{j}$ for
   $j \geq d+2$ with label $\ell > j-2$.
    \cref{okada-arc-cond-parity,okada-arc-cond-interv,okada-arc-cond-nested}
    imply that for all $i = j, \dots, N$ the
    arc incident to $\ov{i}$ must have label $\ell =i$.
    Furthermore, the arcs incident to $\ov{i}$ for $i= j , \dots, N$ can not
    all be propagating. If they were then, in particular, the arc
    incident to $\ov{N}$ would simultaneously have label $\ell = N$
    and be incident to a node $i' < N$ (since $D$ doesn't contain the
    arc $\arc{N}{N}{\ov{N}}$). But this violates
    \cref{okada-arc-cond-interv}. Adapting the reasoning in the
    proof of \cref{lemma:id-arc}, we can show that among the
    non-propagating arcs of $D$ incident to nodes $\ov{i}$
    for $i \geq j \geq d+2$, there must
    exist an arc of the form $\arc{\ov{k}}{k}{\ov{k+1}}$ with $k \geq d + 2$.
    But this violates the maximality of $d = \LDes(D)$.

  \item The parity condition is trivial for $\arc{N}{N}{\ov{N}}$ and for
    $\arc{\ov{N-1}}{\ell}{j}$ as well, since $|j| \leq N-1$. It is also clear for all
    remaining arcs $\arc{i^\flat}{\ell}{j^\flat}$ since $j$ and
    $j^\flat$ have the same parity for all $j \in \setN$.
  \item The nesting condition also holds:  Two arcs $\alpha^\flat$ and $\beta^\flat$ of
    the form $\arc{i^\flat}{\ell}{j^\flat}$ are nested if and only if $\alpha$ and $\beta$
    are. Moreover the arc $\arc{N}{N}{\ov{N}}$ presents no problem.
  \end{enumerate}
  \smallskip

  Finally the factorization $ D=D^\flat \A{N-1}\A{N-2} \ssp \boldsymbol{\cdots} \A{\LDes(D)}$ follows from
  \cref{remark:descending-product}.
\end{proof}

\begin{proposition}
  Suppose that $N>1$. If we define
  $D^\flat\eqdef D$ and $\LDes(D)\eqdef N$ whenever $D$ contains the arc 
  $\arc{N}{N}{\ov{N}}$, then the maps
  \[
    \begin{array}{cccc}
      \phi: & D & \longmapsto
      & \left(D^\flat\setminus\{\arc{N}{N}{\ov{N}}\},\ \LDes(D)\right)\, \\
      \psi: & (D, d) & \longmapsto
      & \iota_{N-1}(D)\ \A{N-1}\A{N-2} \ssp \boldsymbol{\cdots} \A{d}\,
    \end{array}
  \]
  are mutually inverse bijections between $\OkArc[N]$ and $\OkArc[N-1]\times\setN$.
\end{proposition}
\begin{proof}
  We already proved in \cref{prop:right-code-factor}, that the composition
  $\psi\,\circ\,\phi$ is the identity. Let's prove that
  $\phi(\psi(D,d))=(D,d)$.  If $d=N$, this just adds and removes the arc
  $\arc{N}{N}{\ov{N}}$. If $d\neq N$, one can recover $d$ knowing
  $\psi(D, d)$. Indeed $\psi(D, d)$ contains $\arc{\ov{d}}{d}{\ov{d+1}}$ and
  for all $j\geq d+2$ the label of the arc incident to $\ov{j}$ is at most $j-2$ since
  the label of the arc incident to $\ov{j}$ in the product $\A{N-1}\cdots\A{d}$
  is precisely $j-2$ for all $j\geq d+2$. Therefore $d=\LDes(\psi(D, d))$.

  It remains to check that $(\psi(D, d))^\flat=\iota_{N-1}(D)$.
  This is a consequence of  \cref{remark:descending-product}.
\end{proof}
By induction, we get the size of the Okada arc-diagram monoid:
\begin{corollary}
  The monoid $\OkArc$ has cardinality $N!$.
\end{corollary}
\begin{corollary}
  The elements $\A{1}, \dots, \A{N-1}$ generate the monoid $\OkArc$.
\end{corollary}
\begin{proof}
  We argue by induction on $N$. When $N=0,1$, the result is obviously true for
  $\OkArc[0]$ and $\OkArc[1]$.
  Recall that $\iota(\A{i})=\A{i}$. As a consequence, if
  $D = \psi(E, i)$ then the factorization $E = \A{j_1} \ssp \boldsymbol{\cdots} \A{j_k}$
  obtained by induction gives a factorization
  \[
    D = \A{j_1} \ssp \boldsymbol{\cdots} \A{j_k}\A{N-1}\A{N-2} \ssp \boldsymbol{\cdots} \A{\LDes(D)}\,.\qedhere
  \]
\end{proof}
\begin{definition}
\label{def:A-sigma-diagram}
  For any $\sigma\in\SG$ we define
  $\A{\sigma}\eqdef\A{\lexmin(\sigma)}$. In addition, we define a map
  $\Theta:\Okada\to\OkArc$ by $\Theta(\e{\sigma}) \eqdef \A{\sigma}$.
\end{definition}
\begin{corollary}\label{cor:isomEG}
  The map $\Theta$ is a monoid isomorphism from
  $\Okada$ to $\OkArc$ sending $\e{i}$ to $\A{i}$.
\end{corollary}
\begin{proof}
  We know from \cref{prop:relat-G} that the $\A{i}$ satisfy the defining relations
  of the Okada monoid. Hence, there is a morphism $\Theta$ from $\Okada$
  to $\OkArc$ sending $\e{i}$ to $\A{i}$. Since the $\A{i}$ generate $\OkArc$
  this morphism is surjective, and therefore bijective since the two monoids
  have the same cardinality.
\end{proof}
\cref{fig.six_normal_forms} gives the isomorphism when $N=3$.
\cref{fig.example_normal_forms} gives examples of 
canonical forms of FPLCs together with their
associated Okada diagrams when $N=6$.
\begin{remark}
  The isomorphism $\Theta$ commutes with the star anti-automorphism:
  $\Theta(\e{}^{\star}) = \Theta(\e{})^{\star}$.
\end{remark}
\begin{definition}
\label{def:arcdiagram-permutation-bijection}
  For any Okada arc-diagram $D\in\OkArc$, define $\perm{D}\in\SG$ to be the unique permutation
  such that $D= \A{\perm{D}}$.
\end{definition}
By this definition we have the identity:
\begin{equation}\label{eq.perm-of-diag}
\Theta(\e{\perm{D}}) = D\,.
\end{equation}
\medskip

We conclude this section by making explicit the relation between FPLCs
and arc-diagrams:
\begin{proposition}\label{prop:loop-to-diag}
  Let $\mathcal{D}$ be a FPLC of rank $N$ and let $\e{_\Delta}$ be the
  $\e{}$-evaluation of the diamond diagram $\Delta$ associated ot $\mathcal{D}$. Then
  $\Theta(\e{\Delta}) = [\mathcal{D}]$, that is the rank $N$ Okada arc-diagram
  where $\arc{i}{\ell}{j}$ is an arc whenever there is a path linking $i$
  and $j$ in $\mathcal{D}$ with minimal height $\ell$.
\end{proposition}
See \cref{fig.path-def,fig.example_normal_forms} for some examples.
\begin{proof}
  We know that the minimal height $\ell$ of a path linking $i$ to $j$
  in a $\mathcal{D}$ is preserved by the FPLC rewriting rules. One can therefore suppose
  without loss of generality that $\mathcal{D}$ is a normal form FPLC and
  let $\Delta$ be its associated diamond diagram. 
  Arguing by induction on $N$, let us suppose that $\Delta$ is obtained from a diamond diagram
  $\Delta'$ of rank $N-1$ by adding on it right a south-east diagonal consisting of
  $k$ black boxes followed by $N-k$ white boxes.  The reading $\e{\Delta}$ is then
  \[
    \e{\Delta'} \ssp \e{N-1}\e{N-2}\cdots\e{N-k}\,.
  \]
  This means that $\Theta(\e{\Delta}) = \psi(\Theta( \e{\Delta'}), N-k)$. Assuming that the paths
  of $\mathcal{D}'$ match those of $\Theta(\e{\Delta'})$, its clear that the paths of
  $\mathcal{D}$ match those of $\psi( \Theta(\e{\Delta'}), N-k)$.
\end{proof}

\section{Young-Fibonacci Robinson-Schensted correspondence and arc-diagrams}

The bijection $D \mapsto \sigma(D)$ between $\SG$ and
$\OkArc$ defined in \ref{def:arcdiagram-permutation-bijection}
is not very transparent: Starting from a permutation, first its code is
computed, then the associated FPLC is drawn, from which
an Okada arc-diagram is obtained. It is not obvious, for example, that the
inverse of a permutation corresponds to the mirror of the associated
Okada arc-diagram. The goal of this section is to better explicate
this graphical bijection which turns out to be an
incarnation of Fomin's Robinson-Schensted
correspondence for the Young-Fibonacci lattice.
\medskip

We start by recalling some established facts about the Young-Fibonacci
lattice~\cite{Stanley1975,Fomin1994} and Fomin's approach to the
associated Robinson-Schensted correspondence~\cite{Fomin1995}.

\subsection{The Young-Fibonacci lattice and its Fomin's RS correspondence}
\label{subsection:YF-lattice}
A \defn{Fibonacci word} of rank $N$ is a finite sequence containing only $1$'s
and $2$'s whose sum is $N$. By convention there is only one Fibonacci word of
rank zero, namely the empty word denoted $\varnothing$. We write Fibonacci words
as $212221121$ or in its abbreviated form $212^31^221$ where the exponents indicate repeated digits.
 It is well known and elementary that the number of such sequences
is the $N\ordinal$~Fibonacci number $F_N$. The set of Fibonacci words of rank $N$ is
denoted $\YFn$ and the set of all Fibonacci words, of all ranks, is denoted
$\YF$. The rank of the Fibonacci word $w$
is denoted $|w|$ and it length is denoted $\ell(w)$.

There is a lattice structure on $\YF$, called the \defn{Young Fibonacci
  lattice}, which we define through covering relations: An element $w$ covers
an element $v$ if $w$ is obtained
\begin{itemize}
\item either by replacing the \emph{left-most} occurrence of $1$ in $v$ by a $2$;
\item or by inserting a $1$ in $v$ anywhere \emph{to the left} of all $1$'s of $v$.
\end{itemize}
We write $v \lessdot w$ if $w$ covers $v$ in $\YF$. The definition clearly
implies that $|w|=|v|+1$ and so $\YF$ is a ranked
lattice. \cref{fig.YF-lattice} shows Hasse diagram of the $\YF$-lattice up to rank five.
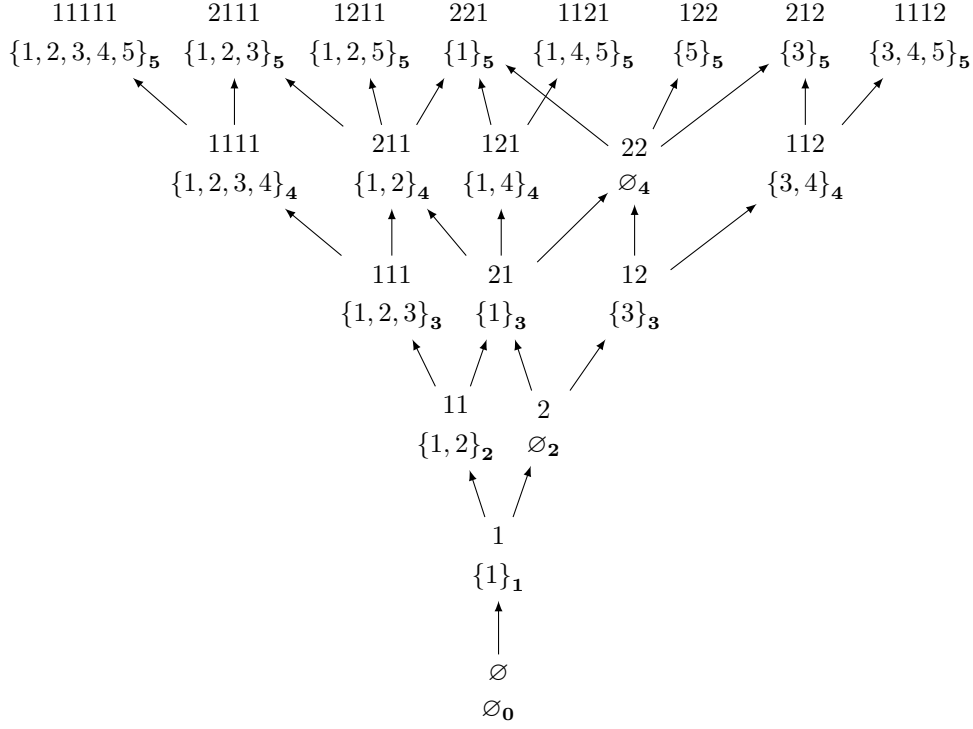
\begin{figure}[ht]
  \[
  \tikzset{every node/.style={rectangle split,rectangle split parts=2}}
  \begin{tikzpicture}[>=latex,yscale=1,line join=bevel]
\node (node_0) at (163.0bp,7.5bp) {$\varnothing$ \nodepart{second} $\varnothing_\mathrm{\bf 0}$};
  \node (node_1) at (163.0bp,57.5bp) {$1$ \nodepart{second} $\left\{1\right\}_\mathrm{\bf 1}$};
  \node (node_2) at (147.0bp,106.5bp) {$11$ \nodepart{second} $\left\{1, 2\right\}_\mathrm{\bf 2}$};
  \node (node_6) at (180.0bp,106.5bp) {$2$ \nodepart{second} $\varnothing_\mathrm{\bf 2}$};
  \node (node_3) at (123.0bp,155.5bp) {$111$ \nodepart{second} $\left\{1, 2, 3\right\}_\mathrm{\bf 3}$};
  \node (node_10) at (164.0bp,155.5bp) {$21$ \nodepart{second} $\left\{1\right\}_\mathrm{\bf 3}$};
  \node (node_4) at (64.0bp,204.5bp) {$1111$ \nodepart{second} $\left\{1, 2, 3, 4\right\}_\mathrm{\bf 4}$};
  \node (node_13) at (123.0bp,204.5bp) {$211$ \nodepart{second} $\left\{1, 2\right\}_\mathrm{\bf 4}$};
  \node (node_5) at (8.0bp,253.5bp) {$11111$ \nodepart{second} $\left\{1, 2, 3, 4, 5\right\}_\mathrm{\bf 5}$};
  \node (node_15) at (64.0bp,253.5bp) {$2111$ \nodepart{second} $\left\{1, 2, 3\right\}_\mathrm{\bf 5}$};
  \node (node_7) at (214.0bp,155.5bp) {$12$ \nodepart{second} $\left\{3\right\}_\mathrm{\bf 3}$};
  \node (node_8) at (278.0bp,204.5bp) {$112$ \nodepart{second} $\left\{3, 4\right\}_\mathrm{\bf 4}$};
  \node (node_16) at (214.0bp,204.5bp) {$22$ \nodepart{second} $\varnothing_\mathrm{\bf 4}$};
  \node (node_9) at (321.0bp,253.5bp) {$1112$ \nodepart{second} $\left\{3, 4, 5\right\}_\mathrm{\bf 5}$};
  \node (node_18) at (278.0bp,253.5bp) {$212$ \nodepart{second} $\left\{3\right\}_\mathrm{\bf 5}$};
  \node (node_11) at (164.0bp,204.5bp) {$121$ \nodepart{second} $\left\{1, 4\right\}_\mathrm{\bf 4}$};
  \node (node_12) at (195.0bp,253.5bp) {$1121$ \nodepart{second} $\left\{1, 4, 5\right\}_\mathrm{\bf 5}$};
  \node (node_19) at (152.0bp,253.5bp) {$221$ \nodepart{second} $\left\{1\right\}_\mathrm{\bf 5}$};
  \node (node_14) at (111.0bp,253.5bp) {$1211$ \nodepart{second} $\left\{1, 2, 5\right\}_\mathrm{\bf 5}$};
  \node (node_17) at (238.0bp,253.5bp) {$122$ \nodepart{second} $\left\{5\right\}_\mathrm{\bf 5}$};
  \draw [black,->] (node_0) -- (node_1);
  \draw [black,->] (node_1) -- (node_2);
  \draw [black,->] (node_1) -- (node_6);
  \draw [black,->] (node_2) -- (node_3);
  \draw [black,->] (node_2) -- (node_10);
  \draw [black,->] (node_3) -- (node_4);
  \draw [black,->] (node_3) -- (node_13);
  \draw [black,->] (node_4) -- (node_5);
  \draw [black,->] (node_4) -- (node_15);
  \draw [black,->] (node_6) -- (node_7);
  \draw [black,->] (node_6) -- (node_10);
  \draw [black,->] (node_7) -- (node_8);
  \draw [black,->] (node_7) -- (node_16);
  \draw [black,->] (node_8) -- (node_9);
  \draw [black,->] (node_8) -- (node_18);
  \draw [black,->] (node_10) -- (node_11);
  \draw [black,->] (node_10) -- (node_13);
  \draw [black,->] (node_10) -- (node_16);
  \draw [black,->] (node_11) -- (node_12);
  \draw [black,->] (node_11) -- (node_19);
  \draw [black,->] (node_13) -- (node_14);
  \draw [black,->] (node_13) -- (node_15);
  \draw [black,->] (node_13) -- (node_19);
  \draw [black,->] (node_16) -- (node_17);
  \draw [black,->] (node_16) -- (node_18);
  \draw [black,->] (node_16) -- (node_19);
\end{tikzpicture}
  \]
  \caption[The Young-Fibonacci lattice]{The Young-Fibonacci lattice up to rank
    $5$. We show the Fibonacci words and their associated Fibonacci sets.}
  \label{fig.YF-lattice}
\end{figure}

Consider the \defn{up and down operators} acting on formal linear combinations of
elements of $\YF$ defined by 
\[
  \mathrm{\bf U}(v)\eqdef\sum_{w : \, w \gtrdot v} w
  \qandq
  \mathrm{\bf D}(v)\eqdef\sum_{u : \, u\lessdot v} u 
\]
and extended by linearity. The Young-Fibonacci lattice is a \defn{differential poset}, meaning that $\mathrm{\bf U}$
and $\mathrm{\bf D}$ satisfy the relation $\mathrm{\bf DU}- \mathrm{\bf UD} =  \mathrm{\bf Id}_{\YF}$. In particular, this implies that
$\mathrm{\bf D}^N\mathrm{\bf U}^N(\varnothing) = N!\, \varnothing$. Otherwise said, there are $N!$ many
pairs of saturated chains in the $\YF$-lattice starting at $\varnothing$ and
sharing a common endpoint at rank $N$. Fomin ~\cite{Fomin1995} devised a way to construct a
bijection between permutations and pairs of such chains using growth diagrams.

A \defn{growth diagram} of rank $N$ is a way to associate an element of $\YF$
to each vertex in the $\{0,1,\dots N\}\times\{0,1,\dots N\}$ grid. We
picture such a grid in cartesian coordinates and identify each square of the
grid with the vertex in its upper-right corner. To compute the growth diagram
$\Grow(\sigma)$ associated to a permutation $\sigma\in\SG$, we begin by putting a box
$\Box$ in each $(i, \sigma(i))$ square cell (latter we'll place a number inside the box). We
assign the empty Fibonacci word $\varnothing$ to each vertex along the left and lower boundary of the grid.
The construction is iterative, and evolves from left to right
and bottom to top. At each step we consider a square cell whose three
left and lower corners $u, v, w$ have already been assigned Fibonacci words.
We then assign a Fibonacci word to the upper
right corner $z$ according to a local rule depending whether there is a box in the
square cell or not.
\[
  \begin{tikzpicture}
    \matrix (M) [matrix of math nodes,
    column sep=0.2cm, row sep=0.2cm,
    nodes={inner sep=3pt, anchor=center}]{
      v && z \\
      & \Box & \\
      u && w\\
    };
    \draw[-latex] (M-1-1) edge node[auto] {$b_2$} (M-1-3);
    \draw[-latex] (M-3-1) edge node[below] {$a_2$} (M-3-3);
    \draw[-latex] (M-3-1) edge node[auto] {$a_1$} (M-1-1);
    \draw[-latex] (M-3-3) edge node[right] {$b_1$} (M-1-3);
  \end{tikzpicture}
  \qquad
 \begin{tikzpicture}
   \matrix (M) [matrix of math nodes,
    column sep=0.2cm, row sep=0.2cm,
    nodes={inner sep=3pt, anchor=center}]{
      v && z \\
      & {\white \Box} & \\
      u && w\\
    };
    \draw[-latex] (M-1-1) edge node[auto] {$b_2$} (M-1-3);
    \draw[-latex] (M-3-1) edge node[below] {$a_2$} (M-3-3);
    \draw[-latex] (M-3-1) edge node[auto] {$a_1$} (M-1-1);
    \draw[-latex] (M-3-3) edge node[right] {$b_1$} (M-1-3);
  \end{tikzpicture}
\]
The rule to assign $z$ is the following:
\begin{itemize}
\item if $v \neq u$ and $w \neq u$, then $z\eqdef 2u$;
\item else if $v=w=u$ and the cell contains a box $\Box$, then $z\eqdef 1u$;
\item otherwise $z$ is assigned is such a way that the two endpoints of the edge $b_i$
  are equal whenever the two endpoints of the edge $a_i$ are as well ($i=1\text{ or }2$).
\end{itemize}
See~\cref{fig.growth-diagram} for an example (the meaning of the colored arcs
will be explained later).
\begin{figure}
  \vspace{0.2cm}
  \includegraphics[height=0.99\textwidth, angle=-90]{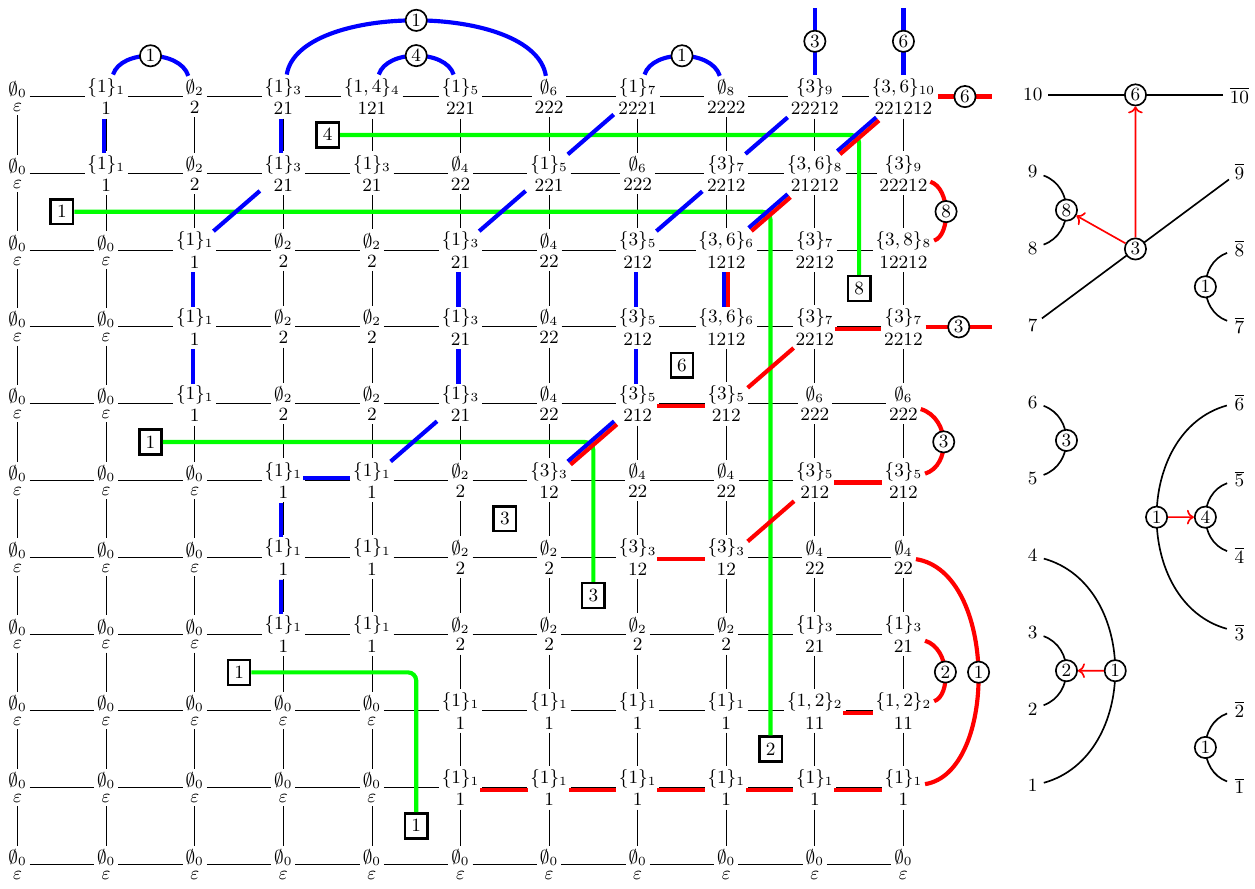}
  \vspace{0.2cm}
  \caption[Example of growth diagram]{%
    The growth diagram of the permutation $(9, 6, 3, 10, 1, 5, 4, 7, 2, 8)$
    and the construction of the associated Okada arc-diagram. Grid positions are
    labelled both by Fibonacci sets and words. Pairs of removed letters in the
    largest last letter procedure are linked in green. \Cref{rem:redblue}
    explains the blue and red lines.}
  \label{fig.growth-diagram}
\end{figure}

We recall a few important facts: The rank of the Fibonacci word 
at grid position $(x,y)$ is equal to the number of square cells with boxes $\Box$
situated to its south-west, that is the number of $i\leq x$ such that $\sigma(i)\leq y$.
Moreover, if $v$ is the Fibonacci word at vertex $(x,y)$ and $w$ is
the Fibonacci word at vertex $(x+1,y)$ or vertex $(x,y+1)$
then either $v=w$ or $v\lessdot w$ depending on whether their ranks are equal or differ
by one. In particular, there are two saturated chains
$P(\sigma) = (u_0 \lessdot u_1 \lessdot u_2 \lessdot \dots \lessdot u_N)$ and
$Q(\sigma) = (v_0 \lessdot v_1 \lessdot v_2 \lessdot \dots \lessdot v_N)$
on the right-most and top borders of the grid, both
beginning at $u_0 = v_0 = \varnothing$.
Since these two chains meet at the upper-right corner of
the grid it follows that $u_N = v_N$ in $\YFn$. The crucial feature is
that the growth process can be reversed starting from the top and rightmost borders and filling
in the intermediate grid positions and permutation boxes along the way. Hence Fomin's
theorem:
\begin{theorem}[\cite{Fomin1995}]
  The map sending a permutation $\sigma \in \SG$ to the pair of saturated chains
  $\left(P(\sigma), Q(\sigma)\right)$
  situated on the right and top borders
  of the associated growth diagram $\Grow(\sigma)$ defines a bijection between $\SG$
  and pairs of saturated chains in $\YF$-lattice
  sharing a common endpoint $\Shape(\sigma) \in \YFn$.
\end{theorem}
\begin{remark}\label{rem:pairing}
  There is an alternative method to compute the Fibonacci word $\Shape(\sigma)$ (and the
  entire growth diagram) due to ~\cite[Definition 5.4.1]{Roby1991}, \cite[Theorem
  1]{Killpatrick2005} and also \cite{Nzeutchap2009}. Express the
  permutation $\sigma$ in one-line notation. If the last entry is also the
  largest, record $1$ and remove the letter. Otherwise, record $2$ and
  remove the last entry together with the largest entry to its left.
  The procedure is repeated until no entries of the permutation remain. The
  resulting Fibonacci word is $\Shape(\sigma)$. In \cref{fig.growth-diagram}, pairs of entries
  removed by the procedure are joined by a bent green arc.
  The Fibonacci word situated at an intermediate grid position $(x, y)$ in $\Grow(\sigma)$ can be
  also obtained by restricting the elimination procedure to the first $x$ letters of the permutation $\sigma$ which are smaller
  than $y$.

\end{remark}
\subsection{Young-Fibonacci chains and half arc-diagrams}
In the classical Robinson-Schensted correspondence,
a standard Young tableaux records
the order in which boxes appear along a saturated chain, while the shape of the tableaux
specifies the endpoint of the saturated chain. This correspondence
is a bijection between permutations and pairs of standard Young tableaux of the same
shape.  In the Young-Fibonacci case,
\cite{Roby1991,Killpatrick2005,Nzeutchap2009} devised several different
ways to encode saturated chains of the $\YF$-lattice
using analogues of standard Young tableaux.
Here we present a graphical version of the Robison-Schensted
correspondence for the $\YF$-lattice
which directly constructs the
Okada arc-diagram associated to a
permutation. In this framework, saturated chains in the
$\YF$-lattice can be bijectively
converted into Okada half arc-diagrams.
The Okada arc-diagram $\Theta(\e{\sigma})$
is obtained by glueing together the two half arc-diagrams
associated to the pair of saturated chains
arising from $\sigma \in \SG$
under the RS-bijection
for the $\YF$-lattice.
We now explain the details of this construction. \medskip

Cutting a labelled arc-diagram $D$ in the middle gives a natural notion of
\defn{labelled half arc-diagram} containing either full arcs $\arc{i}{\ell}{j}$
joining two nodes $i,j\in\setN$ or a labelled half arcs $\harc{i}{\ell}$ with a
free end. Such a free arc is \defn{propagating}. The left half arc-diagram
$\rightd{D}$ of $D$, also called the $\defn{bra}$,  is the restriction of~$D$ to
its positive part. The right
half arc-diagram $\leftd{D}$, also called the \defn{ket}, is the bra $\rightd{D^\star}$
of the mirror $D^\star$.
The notion of
crossing as previously defined naturally restricts to half diagrams.
Likewise, the concept of nesting of half arcs in a half arc-diagram $\rightd{D}$
is simply the restriction of the nesting relation
to the left side of the non-crossing arc-diagram $D$ (clearly this
doesn't depend on the arcs on the right side).
A \defn{half Okada arc-diagram} $H$ is a labelled, non-crossing
half arc-diagram which satisfies 
\cref{okada-arc-cond-interv,okada-arc-cond-parity,okada-arc-cond-nested},
or, equivalently, if gluing $H$ with its mirror produces a
valid Okada arc-diagram.
\begin{definition}
  The \defn{propagating index set} of a half Okada arc-diagram $H$ of rank $N$ is
  the subset $\PropInd(H)$ of $\setN$ containing the indices of the
  nodes of its propagating arcs.

  The \defn{propagating label set} of a half Okada arc-diagram $H$ of rank $N$ is
  the subset $\PropLab(H)$ of $\setN$ containing the labels of its
  propagating arcs.
\end{definition}
The following lemma-definition is of great importance:
\begin{lemma}[Gluing lemma]\label{lemma:gluing}
  For any Okada arc-diagram $D$, the left and right half arc-diagrams $\rightd{D}$
  and $\leftd{D}$ are two Okada half arc-diagrams which have identical
  propagating label sets. Consequently, it makes sense to define
  $\PropLab(D)\eqdef\PropLab\ssp \rightd{D} \ssp=\PropLab\ssp \leftd{D} \ssp$.

  Moreover if $L$ and $R$ are two Okada half arc-diagrams such
  that $\PropLab(L)=\PropLab(R)$, then there is a unique Okada
  arc-diagram~\defn{$\glue{L}{R}$} such that $\rightd{\glue{L}{R}}= L$ and
  $\leftd{\glue{L}{R}}= R$.
\end{lemma}
At this point we need to characterize those sets which can appear as a propagating sets:
\begin{definition}\label{def:fibonacci-sets}
  A \defn{Fibonacci set of rank $N$} is a subset $S = \{s_1<s_2<\dots<s_k\}$ of
  $\setN$ whose size $k$ has the same parity as $N$ and such that $s_i$ has
  the same parity as $i$ for all $i \in \{1, \dots, k \}$. 
  The collection of all Fibonacci set of rank $N$ will be denoted $\YFSn$. The disjoint union of 
  $\YFSn$ for all $N \geq 0$ is denoted $\YFS$.
\end{definition}
The entire set $\setN$ itself is always a rank $N$ Fibonacci set, while $\varnothing$ is
only a rank $N$ Fibonacci set when $N$ is even. \Cref{fig.dominance} shows the full list of Fibonacci
sets of ranks $0$ to $6$.  We emphasize that a Fibonacci set is a set together
with a rank. For example, the set
$\{1,2,5\}$ of rank $5$ is not the same Fibonacci set as $\{1,2,5\}$ of rank
$7$. In order to make this distinction, we sometimes add a subscript $N$ indicating the rank, e.g.
$\{1,2,5\}_\text{\bf 5}$ versus $\{1,2,5\}_\text{\bf 7}$. If $S$ is a
Fibonacci set of rank $N$ then clearly it is also a Fibonacci set of rank $N+2$; the notation \defn{$S^\sharp$}
will be a shorthand to indicate that we are viewing $S \in \YFS_{N+2}$ instead of $S \in \YFS_N$.
\begin{remark}
  The parity conditions for a rank $N$ Fibonacci set $S = \{s_1<s_2<\dots<s_k\}$
  are equivalent to the condition that the consecutive
  differences $s_1-0$, $s_{i+1}-s_{i}$, and $N+1-s_k$ are all odd. Hence Fibonacci
  sets are nothing but descent sets of compositions of $N+1$ into odd parts
  (which are known to be counted by Fibonacci numbers).
\end{remark}
\begin{proposition}\label{prop:proplab-index-fibo}
  For an Okada half arc-diagram $\rightd{D}$ of rank $N$
  both $\PropInd \rightd{D}$ and $\PropLab \rightd{D}$
  are Fibonacci sets of rank $N$.
\end{proposition}
\begin{proof}
  Between two consecutive propagating arcs, the number of nodes is twice
  the number of non-propagating arcs. Therefore the indices of the nodes of
  any pair of consecutive propagating arcs differ by an odd number.
  \Cref{okada-arc-cond-parity} ensures that
  this is also the case for their labels.
\end{proof}
The converse is also true: Any Fibonacci set of rank $N$ is the propagating set of some
rank $N$ Okada half arc-diagram. This is a consequence of \cref{prop:chain-diag-bij}.
\medskip

As one might expect, Fibonacci sets are in bijection with Fibonacci words:
\begin{proposition}\label{prop-bij-fib-set-word}
  Given a Fibonacci word $w=w_1 \cdots w_n$ of rank $N$
  we construct the following subset of $\setN$
  \[
    \FibSet(w)\eqdef\biggset{\sum_{i=j}^{n} w_i}{w_j = 1}
  \]
  Conversely to any Fibonacci set $S=\{s_1<s_2<\dots<s_k\}_{\boldsymbol{N}}$ we associate the word
  \[
    \FibWord(S)\eqdef 2^{(N-s_k)/2} 1 2^{(s_k - s_{k-1} -1)/2} 1 \cdots 2^{(s_2 - s_1 - 1)/2} 1 2^{(s_1 -1)/2 }
%    \qquad 2^{(s_1-1)/2}1 2^{(s_2-s_1-1)/2}1\dots 2^{(N-s_k)/2}
  \]
  These maps $\FibSet$ and $\FibWord$ define mutually inverse
  bijections between Fibonacci words of rank $N$ and Fibonacci sets of rank $N$.
\end{proposition}
For example the Fibonacci set associated to $22212212211$ is
$\{1,2,7,12\}_{\text{\bf 18}}$. Also in \cref{fig.YF-lattice}, the nodes of the
Young-Fibonacci lattice are shown with their associated Fibonacci sets.
\begin{proof}
  Notice that $\FibSet(w)$ is actually a Fibonacci set. Indeed,
  the cardinality $\#S$ is
  the number of $1$'s in $w$
  while $N$ equals the number of $1$'s in $w$
  plus twice the number of $2$'s in $w$. Clearly
  $\#S$ and $N$ have the same parity.
   The first element $s_1$ of $\FibSet(w)$ corresponds to the
  last occurrence of $1$ when $w$ is read from left to right.
  The last $1$ is succeded by $2^m$ for some $m \geq 0$ and therefore the
  $s_1 = 1 + 2m$ whose parity is odd.
  If $s_i < s_{i+1}$ are consecutive elements of $S$ then $s_i$
  and $s_{i+1}$ are the ranks of $1v$ and $12^m 1v$ respectively,
  for some $m \geq 1$ and
  suffix $v$. This forces $s_i$ and $s_{i+1}$ to have opposite parity.
  The reader can check that $\FibWord(S)$ is a Fibonacci word of rank $N$ whenever $S$
  is a rank $N$ Fibonacci set, that
  $\FibSet(\FibWord(S))=S$, and that $\FibWord(\FibSet(w))=w$.
\end{proof}
Before returning to Okada arc-diagrams, let us translate the Young-Fibonacci
lattice and Fomin's Robinson-Schensted correspondence in the language of Fibonacci sets.
The covering relations of the Young-Fibonacci lattice are particularly simple using this
description:
\begin{proposition}\label{def:YFS-lattice}
We transport the lattice structure on the $\YF$-lattice
to $\YFS$ using the bijection $w \mapsto \FibSet(w)$.
Covering relations have the form
$S \lessdot T$ where $S \in \YFS_N$, $T \in \YFS_{N+1}$, and either
\begin{itemize}
\item $T = S \setminus \max S$
\item $S = T \setminus \max T$
\end{itemize}
Given a cell in $\Grow(\sigma)$ (with or without a box $\Box$)

\[
  \begin{tikzpicture}
    \matrix (M) [matrix of math nodes,
    column sep=0.2cm, row sep=0.2cm,
    nodes={inner sep=3pt, anchor=center}]{
      S' && T \\
      & \Box & \\
      S && S'' \\
    };
    \draw[-latex] (M-1-1) edge node[auto] {$b_2$} (M-1-3);
    \draw[-latex] (M-3-1) edge node[below] {$a_2$} (M-3-3);
    \draw[-latex] (M-3-1) edge node[auto] {$a_1$} (M-1-1);
    \draw[-latex] (M-3-3) edge node[right] {$b_1$} (M-1-3);
  \end{tikzpicture}
  \qquad
 \begin{tikzpicture}
   \matrix (M) [matrix of math nodes,
    column sep=0.2cm, row sep=0.2cm,
    nodes={inner sep=3pt, anchor=center}]{
      S' && T \\
      & {\white \Box} & \\
      S && S'' \\
    };
    \draw[-latex] (M-1-1) edge node[auto] {$b_2$} (M-1-3);
    \draw[-latex] (M-3-1) edge node[below] {$a_2$} (M-3-3);
    \draw[-latex] (M-3-1) edge node[auto] {$a_1$} (M-1-1);
    \draw[-latex] (M-3-3) edge node[right] {$b_1$} (M-1-3);
  \end{tikzpicture}
\]
where $S$ is a Fibonacci set of rank $N$
and where $S \lessdot S'$ and $S \lessdot S''$ are
covering relations in $\YFS$, Fomin's local growth rules
become
\begin{itemize}
\item if $S' \neq S$ and $S'' \neq S$, then $T \eqdef S^\sharp$ (the set $S$
  but with rank $N+2$);
\item if $S'=S''=S$ and the cell contains a box $\Box$ then $T \eqdef S \cup\{N+1\}$
  \item otherwise $T$ is assigned so that the two ends of edge $b_i$
  are equal whenever the two ends of edge $a_i$ are as well ($i=1\text{ or }2$).
\end{itemize}
\end{proposition}

\medskip
We now explain how to convert a saturated chain in $\YFS$ consisting of Fibonacci sets
into an Okada half arc-diagram and vice versa. To do this we need a way to {\it restrict}
half arc-diagrams:
\begin{definition}
  For $r\leq N$, the \defn{$r$-restriction} of an Okada half arc-diagram
  $\rightd{D}$ is the half arc-diagram denoted by $\rightd{D}/\setN[r]$ of rank
      $r$ having
  \begin{itemize}
  \item a full arc $\arc{i}{\ell}{j}$ whenever $\rightd{D}$ contains an arc
    $\arc{i}{\ell}{j}$ with $i,j\leq r$
  \item a propagating arc $\harc{i}{\ell}$ whenever $\rightd{D}$ contains either a full arc
    $\arc{i}{\ell}{j}$ with $i\leq r<j$ or a propagating arc $\harc{i}{\ell}$ with
    $i\leq r$
  \end{itemize}
\end{definition}
Clearly $\rightd{D}/\setN[r]$ satisfies \cref{okada-arc-cond-interv,okada-arc-cond-parity,okada-arc-cond-nested}
and is thus a half Okada arc-diagram of rank $r$ in its own right. Moreover if
  $r\leq s$ the $r$-restriction of the $s$-restriction of any $\rightd{D}$ is
    equal to the $r$-restriction of $\rightd{D}$.

\begin{definition}
\label{def:chain-map}
  To any Okada half arc-diagram $\rightd{D}$ of rank $N$ we associate a
  sequence of Fibonacci sets $\defn{$\Chain \ssp \rightd{D} $}\eqdef(C_0,\dots,C_N)$
  defined by $C_i\eqdef\PropLab(\, \rightd{D}/\setN[i])$.
\end{definition}
\begin{definition}
  To each saturated chain ${\bf C} = C_0\lessdot\dots\lessdot C_N$ in $\YFS$ we associate a
  half Okada arc-diagram \defn{$\Diag( {\bf C})$} of rank $N$. Begin with the
  empty half arc-diagram and add an arc incident to node $i$ for each $i \in \setN$
  according to the following rule:
 \begin{itemize}
  \item if $C_i=C_{i-1}\cup\{\ell\}$ create a propagating arc $\harc{i}{\ell}$
  \item if $C_i=C_{i-1}\setminus\{\ell\}$ form an arc $\arc{j}{\ell}{i}$ by
  closing the propagating arc $\harc{j}{\ell}$
  \end{itemize}
\end{definition}
 \Cref{fig.ex-half-arcs} shows examples of saturated chains together with their associated half
Okada diagrams when $N=7$. 
\begin{proposition}\label{prop:chain-diag-bij}
  The $\Chain$ and $\Diag$ maps are mutually inverse bijections between rank $N$ Okada half
  arc-diagrams and saturated chains of $\YFS$ terminating at rank level $N$.
\end{proposition}
\begin{proof}
  We first show by induction on $N$ that if $\rightd{D}$ is a half Okada arc-diagram, then
  $\Chain \ssp \rightd{D}$ is a saturated chain whose associated diagram is
  $\rightd{D}$. Given a half Okada arc-diagram $\rightd{D}$ of rank $N$
  and $i\in\setN$, consider the restricted diagrams
  $\rightd{D}_{i-1}\eqdef \rightd{D}/\setN[i-1]$ and
  $\rightd{D}_{i}\eqdef \rightd{D}/\setN[i]$. By induction we know that
  $\Chain \ssp \rightd{D}_{i-1}$ is a saturated chain whose associated diagram is
  $\rightd{D}_{i-1}$. Two cases are possible:
  \begin{itemize}
  \item Either $\rightd{D}_{i-1}$ and $\rightd{D}_{i}$ differ by closing an arc
    $\arc{j}{\ell}{i}$ for $j<i$. In this case, the $\harc{j}{\ell}$ is a propagating
    arc of $\rightd{D}_{i-1}$ with label
    $\ell =\max(\PropLab( \ssp \rightd{D}_{i-1}))$. Closing this arc ensures that
    $\PropLab(\ssp \rightd{D}_{i})=\PropLab(\ssp \rightd{D}_{i-1})\setminus\{ \ell \}$. This
    corresponds to a cover in the $\YFS$-lattice, and so $\Chain \ssp \rightd{D}_{i}$ must
    be the saturated chain whose associated diagram is $\rightd{D}_{i}$.
  \item Otherwise $\rightd{D}_{i-1}$ and $\rightd{D}_{i}$ differ by a
    propagating arc $\harc{i}{\ell}$. The label $\ell$ of this half arc is
    larger than the label of any other propagating arc due to
    \cref{okada-arc-cond-nested}. So
    $\PropLab(\ssp \rightd{D}_{i})=\PropLab(\ssp \rightd{D}_{i-1})\cup\{\ell\}$. Again, this
    corresponds to a
    cover of $\YFS$-lattice and so $\Chain \ssp \rightd{D}_{i}$ must be a saturated chain whose
    associated diagram is $\rightd{D}_{i}$.
  \end{itemize}
  \medskip

  Conversely, we argue by induction on $N$ that $\Diag( {\bf C})$ is a rank $N$
  Okada half arc-diagram whenever ${\bf C} =C_0\lessdot\cdots\lessdot C_N$
  is a saturated chain. Suppose that $S \lessdot T$ is a cover
  in the $\YFS$-lattice between two Fibonacci set $S$ and $T$ of
  ranks $i-1$ and $i$ respectively. Assume that $\Diag( {\bf C} /[i-1])$ is a rank $i-1$ Okada
  half arc-diagram
  with propagating set $S$ where ${\bf C} /[i-1] = C_0 \lessdot \cdots \lessdot C_{i-1}$
  is the $i\ordinal$-truncation of the saturated chain ${\bf C}$.
  We must prove that $\Diag( {\bf C}  /[i])$ is a rank $i$ Okada half arc-diagram with
  propagating set $T$ such that $\Diag( {\bf C}  /[i])/\setN[i-1] = \Diag( {\bf C}  /[i-1]) $. There are
  two cases to handle:
  \begin{itemize}
  \item $T=S\setminus \max(S)$. In this case we are closing 
    a propagating arc $\harc{j}{\ell}$ in $\Diag({\bf C}  /[i-1]) $ and forming the arc
    $\arc{j}{\ell}{i}$ in $\Diag({\bf C}  /[i]) $ where $\ell=\max(S)$. No crossings are created since all arcs
    incident to nodes $k>j$ are non-propagating arcs (otherwise $S$ would have had a
    propagating label larger than $\ell$). It is easy to check that the
    \cref{okada-arc-cond-interv,okada-arc-cond-parity,okada-arc-cond-nested}
    are satisfied since they already hold for $\Diag( {\bf C}  /[i-1]) $.
  \item $T=S\cup\{\ell\}$ with $\ell > \max(S)$. This means that $\Diag({\bf C}/[i]) $ is obtained
  simply by adding the propagating arc
  $\harc{i}{\ell}$ to $\Diag( {\bf C}  /[i-1]) $. Clearly $\Diag( {\bf C}  /[i]) $ is a valid Okada half arc-diagram with propagating set is $T$.
    \qedhere
  \end{itemize}
\end{proof}

\subsection{The Robinson-Schensted correspondence as an isomorphism}
Starting from a permutation $\sigma$, the two Okada half arc-diagrams encoding its
associated chains  have the same
propagating set and therefore can be merged into an entire Okada arc-diagram
\[
  \RS(\sigma) \eqdef \glue{\Diag(P(\sigma))\:}{\Diag(Q(\sigma))}\,.
\]
This diagram, as we shall see, is precisly $\A{\sigma}$ from \Cref{def:A-sigma-diagram}.

\begin{theorem}\label{theo:RS-isomorphism}
  Let $L_\sigma = \Diag(P(\sigma))$ and $R_\sigma = \Diag(Q(\sigma))$
  denote the rank $N$ Okada half arc-diagrams of the respective saturated
  chains $P(\sigma)$ and $Q(\sigma)$ associated to $\sigma \in \SG[N]$ under
  Fomin's RS-correspondence.  Then
  $\A{\sigma} = \glue{L_\sigma}{R_\sigma}$. Moreover
  $\A{\sigma}^\star = \A{\sigma^{-1}} =
  \glue{R_\sigma}{L_\sigma}$.
\end{theorem}
\begin{proof}
  We've already proven that $\RS$ is a bijection $\SG\to\Okada$. We'll prove, by induction
  on $N$, that $\RS$ and $\Theta$ coincide. The base
  cases $N=0,1$ are trivial. Choose $\sigma\in\SG$ and consider the
  permutation $\sigma'$ obtained by removing $N$ from the one line notation of
  $\sigma$. If $p$ is the position of $N$ in the one line notation of
  $\sigma$ then $\sigma= \sigma^\flat  s_{N-1}\cdots s_p$ where $\sigma^\flat$
  is the permutation obtained by appending $N$ to the one line notation of
  $\sigma'$. We want to show that
  $\RS(\sigma)=\A{\sigma}$ given that $\RS(\sigma')=\A{\sigma'}$.

  The growth diagram $\Grow(\sigma)$ of $\sigma$ is obtained
  by inserting an extra column before the $p\ordinal$ column
  in the growth diagram  $\Grow(\sigma')$ of $\sigma'$
  and putting a box $\Box$ in the square cell situated at grid position $(p,N)$.
  Except for the upper row, the rows of  $\Grow(\sigma)$ and  $\Grow(\sigma')$ are the
  same, with the $p\ordinal$ Fibonacci set repeated twice.
  Let $\varnothing=S_0\lessdot S_1\lessdot\dots\lessdot S_{N-1}$
  and $\varnothing=T_0\lessdot T_1\lessdot\dots\lessdot T_{N-1}=S_{N-1}$ be the 
  top row and rightmost column of  $\Grow(\sigma')$ respectively. Clearly both $S_p$ and $T_p$ are Fibonacci
  sets of rank $p$. Recall that $S^\sharp$ is the same set as $S$ but adding
  $2$ to its rank. Then top row of $\Grow(\sigma)$ is:
\[
\newcommand\es{\emptyset}%
% sage: p = OM[9, 6, 3, 10, 1, 5, 4, 7, 2, 8]
% sage: view(GrowthPicture(p.perm()))  # not tested
% sage: latex(GrowthPicture(p.perm()))  # not tested
\begin{tikzpicture}[ultra thick,
  baseline={(current bounding box.east)},
  column sep={1.5cm,between origins},
  row sep={1.0cm,between origins}]
\matrix (M) [matrix of math nodes, nodes={inner sep=2pt, anchor=center}] {
  \varnothing & S_1 & \cdots & S_{p-1} & S_{p-1} \cup \{p\} & S_{p-1}^\sharp &
  S_{p}^\sharp & \cdots & S_{N-3}^\sharp & S_{N-2}^\sharp \\
  \varnothing & S_1 & \cdots & S_{p-1} & S_{p-1} & S_{p} & S_{p+1} &
  \cdots & S_{N-2} & S_{N-1}\\
};
\foreach \n in {1,2,...,9} {
  \foreach \m in {1,2} {
    \draw[thin] (M-\m-\n.east) -- (M-\m-\the\numexpr \n + 1\relax.west);
  }
}
\foreach \n in {1,2,...,10} {
  \draw[thin] (M-1-\n.south) -- (M-2-\n.north);
}
\node[draw, very thick] at ($(M-1-4)!0.5!(M-2-5)$) {};
\end{tikzpicture}
\]
This gives the top chain as well as the rightmost chain, namely
$T_0\lessdot\dots\lessdot T_{N-1}\lessdot S_{N-2}^\sharp$.
These chains coincide respectively with 
$\Chain \ssp \rightd{D}$ and $\Chain \ssp \leftd{D}$ 
where $D$ is the rank $N$ Okada arc-diagram
given by the factorization
\[ D \, = \, \iota_{N-1}(\A{\sigma'}) \ssp \boldsymbol{\cdot} \A{N-1} \boldsymbol{\cdot} \A{N-2} \ssp \boldsymbol{\cdots} \A{p} \]
Notice that $\iota_{N-1}(\A{\sigma'})$ is precisely $\A{\sigma^\flat} = \A{\sigma}^\flat$ and using
\cref{prop:right-code-factor} we have $D = \A{\sigma}$ and thus
$\RS(\sigma) = \A{\sigma}$.
\end{proof}
As an example, we show in \cref{fig:example-chains-RS} the chains and
permutations associated to the diagram factorisation illustrated in
\cref{fig:right-code-factor}.
\begin{figure}[ht]
  \centering
\begin{tikzpicture}[xscale=0.7, yscale=0.4, thick, baseline={(current bounding box.east)}]
\tikzstyle{vertex} = [shape=rectangle, minimum height=15pt, minimum width=0pt, inner sep=0pt]
\tikzstyle{mid} = [draw, fill=white, shape=circle, minimum size=2pt, inner sep=1pt]
\foreach \i in {0,...,9} {
  \node[vertex] (G-\i) at ($(-1.44, \i*1.5-1.5)$) {};
  \node[vertex] (G--\i) at ($(2.64, \i*1.5-1.5)$) {};
}
\foreach \chainset [count=\xi] in {
  {$\{1\}$}, {$\varnothing$}, {$\{1\}$}, {$\varnothing$}, {$\{1\}$},
  {$\{1, 2\}$}, {$\{1 ,2, 5\}$}, {$\{1, 2\}$}, {$\{1, 2, 3\}$}
} { \node[anchor=east] at (G-\xi) {\chainset}; }
\foreach \chainset [count=\xi] in {
  {$\{1\}$}, {$\{1, 2\}$}, {$\{1, 2, 3\}$}, {$\{1, 2\}$}, {$\{1, 2, 3\}$},
  {$\{1, 2, 3, 4\}$},
  {$\{1, 2, 3, 4, 5\}$}, {$\{1, 2, 3, 4\}$}, {$\{1, 2, 3\}$}
} { \node[anchor=west] at (G--\xi) {\chainset}; }
\draw (G--1) .. controls +(-2.59, 1.30) and +(1.63, -0.82) .. (G-5) node[pos=0.6470588235294118,mid] (M-1-5) { 1 };
\draw (G--2) .. controls +(-2.59, 1.30) and +(1.63, -0.82) .. (G-6) node[pos=0.6470588235294118,mid] (M-2-6) { 2 };
\draw (G--5) .. controls +(-2.59, 1.30) and +(1.63, -0.82) .. (G-9) node[pos=0.6470588235294118,mid] (M-5-9) { 3 };
\draw (G-3) .. controls +(0.70, 0.35) and +(0.70, -0.35) .. (G-4) node[pos=0.5,mid] (M3-4) { 1 };
\draw (G--9) .. controls +(-2.10, -1.05) and +(-2.10, 1.05) .. (G--6) node[pos=0.5,mid] (M-9--6) { 4 };
\draw (G-7) .. controls +(0.70, 0.35) and +(0.70, -0.35) .. (G-8) node[pos=0.5,mid] (M7-8) { 5 };
\draw (G--8) .. controls +(-0.70, -0.35) and +(-0.70, 0.35) .. (G--7) node[pos=0.5,mid] (M-8--7) { 5 };
\draw (G-1) .. controls +(0.70, 0.35) and +(0.70, -0.35) .. (G-2) node[pos=0.5,mid] (M1-2) { 1 };
\draw (G--4) .. controls +(-0.70, -0.35) and +(-0.70, 0.35) .. (G--3) node[pos=0.5,mid] (M-4--3) { 3 };
% \draw[color=red, <-] (M7-8) -- (M-2-6);
% \draw[color=red, <-] (M-4--3) -- (M-2-6);
% \draw[color=red, <-] (M-8--7) -- (M-9--6);
% \draw[color=red, <-] (M-9--6) -- (M-5-9);
% \draw[color=red, <-] (M-5-9) -- (M-2-6);
% \draw[color=red, <-] (M-2-6) -- (M-1-5);
\path (G-0) -- (G--0) node[pos=0.5] {$A_{\sigma}$};
\end{tikzpicture}
%%%%%%%%%%%%%%%%%%%%%%%%%%%%%%%%%%%%%%%%%%%%%%
\begin{tikzpicture}[xscale=0.7, yscale=0.4, thick, baseline={(current bounding box.east)}]
\tikzstyle{vertex} = [shape=rectangle, minimum height=15pt, minimum width=0pt, inner sep=0pt]
\tikzstyle{mid} = [draw, fill=white, shape=circle, minimum size=2pt, inner sep=1pt]
\foreach \i in {0,...,9} {
  \node[vertex] (G-\i) at ($(-1.44, \i*1.5-1.5)$) {};
  \node[vertex] (G--\i) at ($(3.84, \i*1.5-1.5)$) {};
}
\foreach \chainset [count=\xi] in {
  {$\{1\}$}, {$\varnothing$}, {$\{1\}$}, {$\varnothing$}, {$\{1\}$},
  {$\{1, 2\}$}, {$\{1,2,5\}$}, {$\{1, 2\}$}
} { \node[anchor=east] at (G-\xi) {\chainset}; }
\foreach \chainset [count=\xi] in {
  {$\{1\}$}, {$\{1, 2\}$}, {$\{1, 2, 3\}$}, {$\{1, 2, 3, 4\}$},
  {$\{1, 2, 3, 4, 5\}$}, {$\{1, 2, 3, 4\}$}, {$\{1, 2, 3\}$}, {$\{1, 2\}$}
} { \node[anchor=west] at (G--\xi) {\chainset}; }
\draw (G--1) .. controls +(-3.55, 1.78) and +(1.63, -0.82) .. (G-5) node[pos=0.7272727272727272,mid] (M-1-5) { 1 };
\draw (G--2) .. controls +(-3.55, 1.78) and +(1.63, -0.82) .. (G-6) node[pos=0.7272727272727272,mid] (M-2-6) { 2 };
\draw (G-3) .. controls +(0.70, 0.35) and +(0.70, -0.35) .. (G-4) node[pos=0.5,mid] (M3-4) { 1 };
\draw (G--8) .. controls +(-3.50, -1.75) and +(-3.50, 1.75) .. (G--3) node[pos=0.5,mid] (M-8--3) { 3 };
\draw (G-7) .. controls +(0.70, 0.35) and +(0.70, -0.35) .. (G-8) node[pos=0.5,mid] (M7-8) { 5 };
\draw (G--7) .. controls +(-2.10, -1.05) and +(-2.10, 1.05) .. (G--4) node[pos=0.5,mid] (M-7--4) { 4 };
\draw (G-1) .. controls +(0.70, 0.35) and +(0.70, -0.35) .. (G-2) node[pos=0.5,mid] (M1-2) { 1 };
\draw (G--6) .. controls +(-0.70, -0.35) and +(-0.70, 0.35) .. (G--5) node[pos=0.5,mid] (M-6--5) { 5 };
%\draw (G--9) -- (G-9) node[pos=0.7272727272727272,mid] (M-9-9) { 9 };
% \draw[color=red, <-] (M7-8) -- (M-2-6);
% \draw[color=red, <-] (M-6--5) -- (M-7--4);
% \draw[color=red, <-] (M-7--4) -- (M-8--3);
% \draw[color=red, <-] (M-8--3) -- (M-2-6);
% \draw[color=red, <-] (M-9-9) -- (M-2-6);
% \draw[color=red, <-] (M-2-6) -- (M-1-5);
\path (G-0) -- (G--0) node[pos=0.5] {$A_{\sigma'}$};
\end{tikzpicture}
  \caption{The chains associated to $\sigma=349578261$ and $\sigma'=34578261$}
  \label{fig:example-chains-RS}
\end{figure}
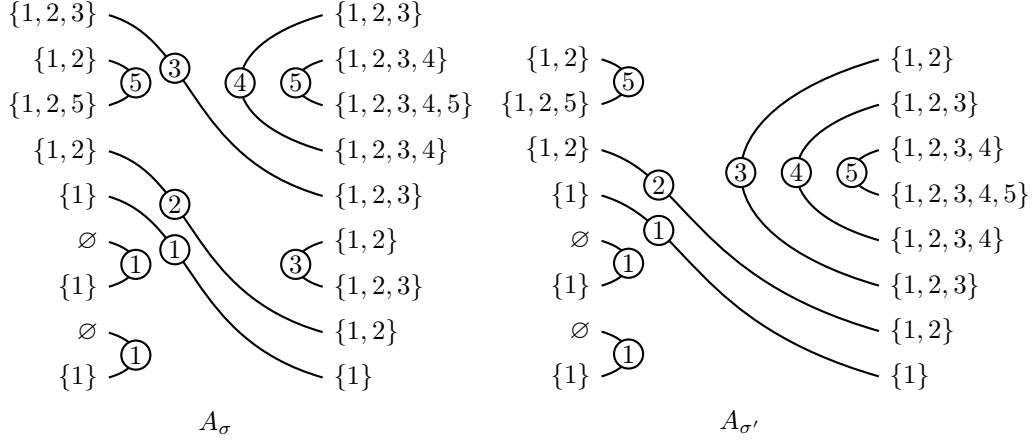

\begin{remark}\label{rem:redblue}
  As a final remark, there is shortcut which allows us
  to compute the arc-diagram $\A{\sigma}$ associated to a permutation
  $\sigma \in \SG$ without "filling out" the
  entire growth diagram $\Grow(\sigma)$.
  The idea is that one can follow the positions of
  the propagating arcs as one adds more rows and columns of the permutation
  matrix. This is illustrated on \cref{fig.growth-diagram} with the blue and
  red lines. Here are the details of the procedure.

  Label each box $\begin{young} {\white \ell} \end{young}$ in the growth diagram
  $\begin{young} \ell \end{young}$ where $\ell$ is the number of
  boxes to its south-west, including itself. Proceed from right to left and pair the box
  in the $i\ordinal$ column with
  the box in the $j\ordinal$ column
  where $j \leq i$ and $\sigma(j)$ is the
  largest index to the left of $\sigma(i)$
  when $\sigma$ is written in one-line
  notation; see~\cref{rem:pairing}. The upper box of each pair
  will give an arc of $\Diag(Q(\sigma))$ situated on the top,
  while the lower box will give an arc of $\Diag(P(\sigma))$
  situated on the right; unpaired boxes will
  give propagating arcs in both half diagrams.

  To compute the arc of $\Diag(Q(\sigma))$ associated to an upper or unpaired box, start at
  the upper-right corner of the square cell containing the
  box. Repeat the following protocol to generate a walk which ends at the top of the grid:
  Scan the row and column of square cells immediately north and east of the
  current grid position.
  \begin{itemize}
  \item Move one step north if the box in the row is east of the current
    grid position (which means that no green line is crossed);
  \item Move one step east if the box in the row is west of the current position
    and the box in the column is north of the current position;
  \item Move one step north-east otherwise (which means the row contains a box to
    the west and the column contains a box to the south).
  \end{itemize}
  Upon reaching the top of the grid, say at position $(i,N)$, create a propagating arc $\harc{i}{\ell}$
  where $\ell$ is the label inside the box at the beginning of the walk. Do this for each of the $N$
  labelled boxes. Reading from right to left, create an arc $\arc{i}{\ell}{j}$
  by "closing off" the propagating arc $\harc{i}{\ell}$
  with the nearest unoccupied grid position $(j,N)$ where $j > i$.
  If there is no unoccupied position, simply retain $\harc{i}{\ell}$
  as a propagating arc. The result is the rank $N$ Okada half arc-diagram $\Diag(Q(\sigma))$.
  The diagram $\Diag(P(\sigma))$ is computed by transposing the
  procedure. \medskip

  The reader should be able to convince himself that this visual procedure
  exactly parallels the construction of the growth diagram. A large
  example of this procedure can be found in annex in
  \cref{fig.compute-arc-large}.
\end{remark}

\pagebreak[2]
We close this section with two small combinatorial remarks:
  \begin{figure}
  \[
    \begin{array}{l|rrrrrrr}
      \hspace{0.4cm}k & 0 & 1 & 2 & 3 & 4 & 5 & 6 \\[-2mm]
      N & \\\hline
      0 & 1 \\
      1 & 0& 1 \\
      2 & 1& 0& 1 \\
      3 & 0& 3& 0& 1 \\
      4 & 3& 0& 6& 0& 1 \\
      5 & 0& 15& 0& 10& 0& 1 \\
      6 & 15& 0& 45& 0& 15& 0& 1 \\
    \end{array}\qquad
    \begin{array}{l|rrrrrrr}
      \hspace{0.4cm}k & 0 & 1 & 2 & 3 & 4 & 5 & 6 \\[-2mm]
      N & \\\hline
      0 & 1 \\
      1 & 0& 1 \\
      2 & 1& 0& 1 \\
      3 & 0& 5& 0& 1 \\
      4 & 9& 0& 14& 0& 1 \\
      5 & 0& 89& 0& 30& 0& 1 \\
      6 & 225& 0& 439& 0& 55& 0& 1 \\
      \end{array}\qquad
      \begin{tikzpicture}[scale=0.6, thick, baseline={(current bounding box.east)}]
\tikzstyle{vertex} = [shape=rectangle, minimum height=15pt, minimum width=0pt, inner sep=0pt]
\tikzstyle{mid} = [draw, fill=white, shape=circle, minimum size=2pt, inner sep=1pt]
\node[vertex] (G--4) at (0.96, 3.90) {};
\node[vertex] (G--3) at (0.96, 2.60) {};
\node[vertex] (G--2) at (0.96, 1.30) {};
\node[vertex] (G--1) at (0.96, 0.00) {};
\node[vertex] (G-1) at (-0.96, 0.00) {};
\node[vertex] (G-2) at (-0.96, 1.30) {};
\node[vertex] (G-3) at (-0.96, 2.60) {};
\node[vertex] (G-4) at (-0.96, 3.90) {};
\draw (G-3) .. controls +(0.70, 0.35) and +(0.70, -0.35) .. (G-4) node[pos=0.5,mid] (M3-4) { 1 };
\draw (G--2) .. controls +(-0.70, -0.35) and +(-0.70, 0.35) .. (G--1) node[pos=0.5,mid] (M-2--1) { 1 };
\draw (G-1) .. controls +(0.70, 0.35) and +(0.70, -0.35) .. (G-2) node[pos=0.5,mid] (M1-2) { 1 };
\draw (G--4) .. controls +(-0.70, -0.35) and +(-0.70, 0.35) .. (G--3) node[pos=0.5,mid] (M-4--3) { 3 };
\end{tikzpicture}
\]
\caption[Number of propagating edge in arc-diagrams]{Triangles showing
  the number of Okada half arc-diagrams (resp. full Okada arc-diagrams)
  of rank $N$ with $k$ propagating edges together with the arc-diagram associated to
  $3241 \in \SG[4]$.}
\label{fig:triangle-propag}

  \end{figure}
\begin{remark}
  The left triangle in~\cref{fig:triangle-propag} shows the number of Okada
  half arc-diagrams of rank $N$ with $k$ propagating edges. These are the
  coefficients of the modified Hermite polynomials~\cite[Sequence
  A099174]{OEIS}. This is because the number of Okada half arc-diagrams of
  rank $N$ with $k$ propagating edges is the same as the number of involution
  of $\SG$ with $k$ fixed point. Indeed, the number of propagating edges of a
  half arc-diagram $H$ is the same as the number of propagating edges
  of the symmetric arc-diagram $\glue{H}{H}$.
  Symmetric arc-diagrams are in bijection with
  involutions under the  Robinson-Schensted map, and it is clear from the growth diagram that each propagating
  arc corresponds to a fixed point of the associated involution.
\end{remark}

\begin{remark}
  The second triangle in~\cref{fig:triangle-propag} shows the number of full
  Okada arc-diagrams of rank $N$ with $k$ propagating edges. This
  is~\cite[Sequence A060524]{OEIS} (i.e. number of permutations in $\SG$ with
  $k$ odd cycles). We offer a proof which was explained to us by
  Izumi~Watanabe~\cite{Watanabe2025} and we thank him. Note that Fomin's RS-bijection 
  doesn't give a proof: For example the permutation
  $3241 = (2)(134)$ has cycle type $(3,1)$ while its associated arc-diagram,
  on the far right in~\cref{fig:triangle-propag}, has no propagating arcs.

  \begin{proof}
  The triangle~$T(n,k)$ of \cite[Sequence A060524]{OEIS} can be defined by
  the following recursion:
  \begin{gather*}
    T(0, 0)= T(1, 1) = 1,\qquad T(1, 0)=0 \qandq T(n, k)=0\quad\text{for $k>n$},\\
    T(n, k)=T(n-1, k-1) + (n-1)^2\,T(n-2, k)\,.
  \end{gather*}
  We show that the number $\card{D(n, k)}$ of Okada arc-diagrams of rank $n$
  with $k$ propagating arcs satisfies the same recursion. The base cases are
  immediate and we only explain the inductive step. We remark that
  $\card{D(n-1, k-1)}$ is also the number of Okada arc-diagrams of rank $n$ that
  contain $\arc{n}{n}{\ov{n}}$, so that we only need to prove that
  $(n-1)^2\,T(n-2, k)$ is the number of Okada arc-diagrams of rank $n$ with $k$
  propagating arcs which don't contain $\arc{n}{n}{\ov{n}}$. For such an arc-diagram 
  $D$, call $r(D) \eqdef \LDes(D)$ and $\ell(D)\eqdef\LDes(D^*)$. Then applying
  \cref{prop:right-code-factor} twice, one can define
  ${}^\flat\!{D}^\flat\eqdef(((D^\flat)^*)^\flat)^*$ which satisfies
  \begin{equation}
    D =
    \A{\ell(D)} \ssp \boldsymbol{\cdots}  \A{N-2}\A{N-1}
    \,{}^\flat\!{D}^\flat\,
    \A{N-1}\A{N-2} \ssp \boldsymbol{\cdots} \A{r(D)}\,.
  \end{equation}
  Moreover the map
  \begin{equation}
    \left.D\in \left\{ D(n, k)\;\middle|\;\arc{n}{n}{\ov{n}}\notin D \right\}  \right.
    \quad\longrightarrow\quad (\ell(D), {}^\flat\!{D}^\flat, r(D))
  \end{equation}
  defines a bijection with the Cartesian product
  \begin{equation}
    \setN[n-1] \times
    \set{D\in D(n, k)}{\arc{n}{n}{\ov{n}}, \, \arc{n-1}{n-1}{\ov{n-1}}\in D}
    \times \setN[n-1]\,.
  \end{equation}
  Our assertion follows because the middle term of the cartesian product stands in bijection
  with $D(n-2, k)$ when removing the arcs
  $\arc{n}{n}{\ov{n}}$ and  $\arc{n-1}{n-1}{\ov{n-1}}$ from $D$.
\end{proof}
We also give an alternative proof based on the RS-correspondence
and using Fibonacci words (instead of sets). In 
what follows, \defn{$\dim(w)$} will denote the number of
saturated chains from $\varnothing$ to $w$ in the $\YF$-lattice.
\begin{proof}
  A rank $N$ Okada arc-diagram $D$ has $k$ propagating arcs if and only if the
  rank $N$ Fibonacci word $w$ associated to $D$ under the RS-correspondence
  has exactly $k$ occurrences of the 1-digit. Notice that $D$ contains
  the arc $\arc{n}{n}{\ov{n}}$ if and only if $w = 1v$ for some suffix $v$ of
  rank $N-1$. Otherwise $w= 2v$ for some suffix of rank $N-2$. Notice also
  that $\dim(2v) = (N-1) \cdot \dim(v)$ which follows from the hook-length
  formula (See \cite[Proposition 2.3]{Nzeutchap2009}).  The RS-correspondence
  implies that the number $\# D(n,k)$ of rank $N$ Okada arc-diagrams $D$ with
  $k$ propagating arcs can computed using the following two-step recursion based
  on whether or not the Fibonacci word $w$ begins with 1 or 2,
  namely
  \[\begin{split}
\# \ssp D(n,k)
\ = \,
\sum_{\stackrel{\scriptstyle |w| \ssp = \ssp N}{\mathbcal{r}(w) \ssp = \ssp k  }}
\dim^2(w)
&\, = \,
\sum_{\stackrel{\scriptstyle |v| \ssp = \ssp N-1}{\mathbcal{r}(v) \ssp = \ssp k-1  }}
\dim^2(1v)
\ + \
\sum_{\stackrel{\scriptstyle |v| \ssp = \ssp N-2}{\mathbcal{r}(v) \ssp = \ssp k }}
\dim^2(2v) \\
&\, = \,
\sum_{\stackrel{\scriptstyle |v| \ssp = \ssp N-1}{\mathbcal{r}(v) \ssp = \ssp k-1  }}
\dim^2(v)
\ + \
\sum_{\stackrel{\scriptstyle |v| \ssp = \ssp N-2}{\mathbcal{r}(v) \ssp = \ssp k }}
(N-1)^2 \cdot \dim^2(v) \\
&\, = \, \# \ssp D(N-1,k-1) \ + \ \# \ssp D(N-2,k) \cdot (N-1)^2
    \end{split}\]
with the obvious initial conditions and
where $\mathbcal{r}(w)$ denotes the total number of {\bf 1}'s in $w$.
\end{proof}
\end{remark}

\subsection{Okada arc-diagrams and Young-Fibonacci Tableaux}
\label{subsection:YF-tableaux}
This subsection explains
the relationship between our notion of Okada half arc-diagrams
and other, previously known encodings
of saturated chains in the $\YF$-lattice, in particular
an encoding based on a notion of Young-Fibonacci tableaux
due to T. Roby~\cite{Roby1991}.
The reader unfamiliar with Roby's work may skip this subsection
 without any problem as we don't use any of its results later in the text.
Accordingly, and in order to keep the discussion short, we only 
describe some relevant bijections, leaving proofs to the
reader. \smallskip

Recall that a Fibonacci word $w = a_1 \cdots a_k$ of rank $|w|=n$ can be depicted by
its \defn{Young-Fibonacci diagram}. This is
a profile of $n$ {\it boxes} arranged,
from left to right,
into $k$ adjacent columns such that the $i\ordinal$ column consists of $a_i$ vertically stacked boxes.
The following example with $w=12112211$
should illustrate the concept clearly:
\begin{equation*}
\begin{young}
, &  & , & , &  &   \\
& &  &  &  &  & & \\
, 1 & , 2  & , 1 &  , 1 &  , 2 & , 2 & , 1 & , 1
\end{young}
\end{equation*}

\begin{definition}[\cite{Roby1991}]
A \defn{Roby-tableau} of shape $w \in \YFn[n]$
is a bijective labelling of the boxes of the
Young-Fibonacci diagram associated to $w$ by
integers in $\{1, \dots, n\}$ such that:

 \begin{enumerate}
 \item box entries are strictly decreasing in columns
when read top to bottom
\item  the top entry of any column of height one 
is greater than any entry to its right
\item the entries to
the right of a column of height two with entries $a < b$
are not contained in the interval $[a,b]$.
\end{enumerate}
\end{definition}

\begin{example}
  \newcommand\ylw{\Yfillcolour{yellow}}
For example, if $w=2212$ and
\[ \raisebox{1.5ex}{$S \, = $}  \ \  \text{\begin{young}  $3$ & $5$ & , & $6$    \\ $2$ & $4$ & $7$ & $1$  \end{young}}
\qquad\qquad  \raisebox{1.5ex}{$T \, = $}  \ \  \text{\begin{young}  $7$ & $5$
    & , & $2$    \\ $6$ & $3$ & $\red\bf4$ & $1$  \end{young}} \]

\noindent
then $S$ is a Roby-tableau but $T$ is {\bf not} because $4$ resides in a column to the
right of the column containing the pair $\{3,5\}$ but nevertheless $4 \in [3,5]$.
\end{example}

We first expalain how to re-encode a Roby-tableau as a non-crossing
half arc-diagram equipped with a system of arc labels:
\begin{definition}
  Fix a non-crossing unlabelled half arc-diagram $H$ of rank $N$ with $d$ arcs (propagating or not). A
  \defn{linear extension labelling} of $H$ (or LE-labelling for short) is a
  linear extension of the associated nesting order, i.e. a bijective labelling
  of its arcs by indices $\{1, \dots, d \}$ which is increasing with respect to the nesting order.
\end{definition}
A LE-labelling of $H$ will be presented as an $N$-tuple
$\boldsymbol{\underline{\ell}} = (\ell_1 , \dots, \ell_N)$ where $\ell_k$ is
the LE-label of the arc incident to vertex $k$ of $H$.

\begin{lemma}
\label{Roby-to-LinearExtension}
Let $T$ be the filling of a shape $w \in \YFn[n]$. Number its columns from
right to left starting with $1$. Let $(H, \boldsymbol{\underline{\ell}})$ be the labelled
half arc-diagram containing
\begin{itemize}
\item $\arc{i}{c}{j}$ whenever the $c\ordinal$ column consists of two boxes
  containing $i$ and $j$.
\item $\harc{i}{c}$ whenever the $c\ordinal$ column consists of a single box
  containing $i$.
\end{itemize}
Then $T$ is a Roby-tableau if and only if $H$ is
a non-crossing half arc-diagram and $\boldsymbol{\underline{\ell}}$
is  an LE-labelling of $H$. Furthermore the map
$T\mapsto(H,\boldsymbol{\underline{\ell}})$ defined a bijection between Roby tableaux of
rank $N$ and LE-labelled non-crossing half arc-diagrams of rank $N$.
\end{lemma}
See the left of~\cref{fig.roby-arc} for an example.

\begin{definition}
\label{LE-to-HLabels}
Given a LE-labelling of a non-crossing half arc-diagram $H$
of rank $N$ encoded by the $N$-tuple
$\boldsymbol{\underline{\ell}} = (\ell_1 , \dots, \ell_N)$ we define
$\boldsymbol{\underline{h}} = (h_1, \dots, h_N)$
where $h_k \eqdef \# \{ i < k \, : \, \ell_i < \ell_k \} \, + \, 1$.

Note that the nesting condition ensures that $h_i = h_j$ whenever
$\ell_i = \ell_j$. Since $\ell_i = \ell_j$ if and only if vertices $i$ and
$j$ are joined by an arc in the non-crossing half arc-diagram $H$, we may
unambiguously label this arc by the common value $h_i = h_j$. We get a
new labelling $\boldsymbol{\underline{h}}$ of $H$.
\end{definition}

\begin{proposition}
\label{LinearExtentions-to-HLabels}
Fix a non-crossing unlabelled half arc-diagram $H$ of rank $N$ and let
$\boldsymbol{\underline{\ell}}$ be a LE-labelling of $H$. Then the half arc-diagram $H$
labelled by the associated $N$-tuple $\boldsymbol{\underline{h}}$ is an Okada half arc-diagram.
Furthermore the map
$(H, \boldsymbol{\underline{\ell}}) \mapsto (H, \boldsymbol{\underline{h}})$ defines a bijection
between LE-labelled non-crossing half arc-diagrams of rank $N$ and Okada half
arc-diagrams of rank $N$.
\end{proposition}
See~\cref{fig.roby-arc} for an example.

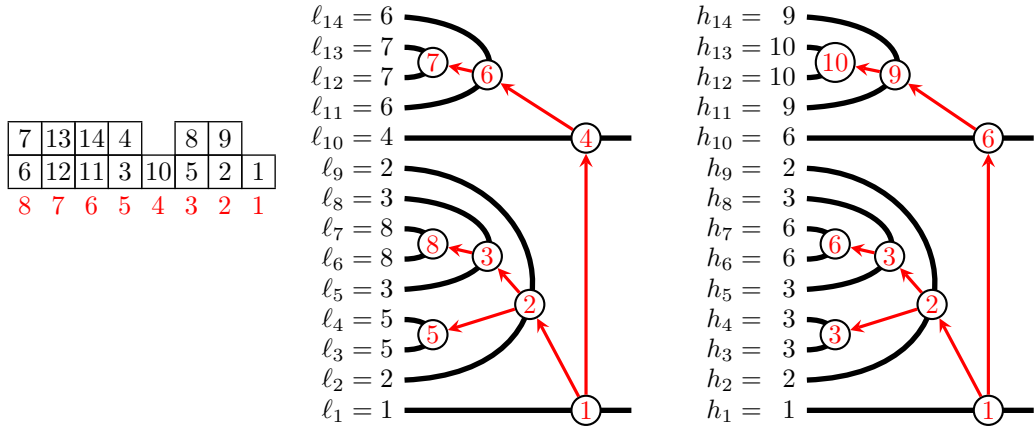
\begin{figure}[ht]
\[
\text{\begin{young} $7$ & $13$ & $14$ & $4$ & , & $8$ & $9$ & ,
        \\ $6$ & $12$ & $11$ & $3$ & $10$ & $5$ & $2$ & $1$  \\
        , {\red 8} & , {\red 7} & , {\red 6} & , {\red 5} & , {\red 4} & ,
        {\red 3} & , {\red 2} & , {\red 1}
      \end{young}}
\quad
\begin{tikzpicture}[ultra thick,yscale=0.40, baseline={(current bounding
        box.east)}]
\tikzstyle{mid}=[draw,thick,black,anchor=center,fill=white,
            shape=circle,minimum size=2pt,inner sep=1pt]
\tikzstyle{walkv}=[draw, line width=2pt]
\node[anchor=east] (1) at (0, 1) {$\ell_1 = 1$};
\node[anchor=east] (2) at (0, 2) {$\ell_2 = 2$};
\node[anchor=east] (3) at (0, 3) {$\ell_3 = 5$};
\node[anchor=east] (4) at (0, 4) {$\ell_4 = 5$};
\node[anchor=east] (5) at (0, 5) {$\ell_5 = 3$};
\node[anchor=east] (6) at (0, 6) {$\ell_6 = 8$};
\node[anchor=east] (7) at (0, 7) {$\ell_7 = 8$};
\node[anchor=east] (8) at (0, 8) {$\ell_8 = 3$};
\node[anchor=east] (9) at (0, 9) {$\ell_9 = 2$};
\node[anchor=east] (10) at (0, 10) {$\ell_{10} = 4$};
\node[anchor=east] (11) at (0, 11) {$\ell_{11} = 6$};
\node[anchor=east] (12) at (0, 12) {$\ell_{12} = 7$};
\node[anchor=east] (13) at (0, 13) {$\ell_{13} = 7$};
\node[anchor=east] (14) at (0, 14) {$\ell_{14} = 6$};
\draw[walkv] (1) -- ($(1.east) + (3, 0)$) node[pos=0.8,mid] (L1) {\red$1$};
\draw[walkv] (2.east) .. controls +(2.50, 0.25) and +(2.00, -0.25) .. (9.east)
    node[pos=0.4,mid] (L2) {$\red2$};
\draw[walkv] (3.east) .. controls +(0.50, 0.) and +(0.50, -0.) .. (4.east)
    node[pos=0.5,mid] (L5) {$\red5$};
\draw[walkv] (5.east) .. controls +(1.60, 0.25) and +(1.40, -0.25) .. (8.east)
     node[pos=0.4,mid] (L3) {$\red3$};
\draw[walkv] (6.east) .. controls +(0.50, 0.) and +(0.50, -0.) .. (7.east)
     node[pos=0.5,mid] (L8) {$\red8$};
\draw[walkv] (10) -- ($(10.east) + (3,0)$) node[pos=0.8,mid] (L4) {$\red4$};
\draw[walkv] (11.east) .. controls +(1.60, 0.25) and +(1.40, -0.25) .. (14.east)
    node[pos=0.4,mid] (L6) {$\red6$};
\draw[walkv] (12.east) .. controls +(0.50, 0.) and +(0.50, -0.) .. (13.east)
    node[pos=0.5,mid] (L7) {$\red 7$};
      % \draw[line width=1.2pt,color=red,<-] (L2) -- (L1);
\draw[line width=1.2pt,>=stealth,color=red,<-] (L2) -- (L1);
\draw[line width=1.2pt,>=stealth,color=red,<-] (L4) -- (L1);
\draw[line width=1.2pt,>=stealth,color=red,<-] (L3) -- (L2);
\draw[line width=1.2pt,>=stealth,color=red,<-] (L8) -- (L3);
\draw[line width=1.2pt,>=stealth,color=red,<-] (L5) -- (L2);
\draw[line width=1.2pt,>=stealth,color=red,<-] (L6) -- (L4);
\draw[line width=1.2pt,>=stealth,color=red,<-] (L7) -- (L6);
\end{tikzpicture}
\qquad
\begin{tikzpicture}[ultra thick,yscale=0.40, baseline={(current bounding
    box.east)}]
\tikzstyle{mid}=[draw,thick,black,anchor=center,fill=white,
    shape=circle,minimum size=2pt,inner sep=1pt]
\tikzstyle{walkv}=[draw, line width=2pt]
\node[anchor=east] (1) at (0, 1) {$h_1 = \ \,1$};
\node[anchor=east] (2) at (0, 2) {$h_2 = \ \,2$};
\node[anchor=east] (3) at (0, 3) {$h_3 = \ \,3$};
\node[anchor=east] (4) at (0, 4) {$h_4 = \ \,3$};
\node[anchor=east] (5) at (0, 5) {$h_5 = \ \,3$};
\node[anchor=east] (6) at (0, 6) {$h_6 = \ \,6$};
\node[anchor=east] (7) at (0, 7) {$h_7 = \ \,6$};
\node[anchor=east] (8) at (0, 8) {$h_8 = \ \,3$};
\node[anchor=east] (9) at (0, 9) {$h_9 = \ \,2$};
\node[anchor=east] (10) at (0, 10) {$h_{10} = \ \,6$};
\node[anchor=east] (11) at (0, 11) {$h_{11} = \ \,9$};
\node[anchor=east] (12) at (0, 12) {$h_{12} = 10$};
\node[anchor=east] (13) at (0, 13) {$h_{13} = 10$};
\node[anchor=east] (14) at (0, 14) {$h_{14} = \ \,9$};
\draw[walkv] (1) -- ($(1.east) + (3, 0)$) node[pos=0.8,mid] (L1) {$\red1$};
\draw[walkv] (2.east) .. controls +(2.50, 0.25) and +(2.00, -0.25) .. (9.east)
   node[pos=0.4,mid] (L2) {$\red2$};
\draw[walkv] (3.east) .. controls +(0.50, 0.) and +(0.50, -0.) .. (4.east)
   node[pos=0.5,mid] (L5) {$\red3$};
\draw[walkv] (5.east) .. controls +(1.60, 0.25) and +(1.40, -0.25) .. (8.east)
   node[pos=0.4,mid] (L3) {$\red3$};
\draw[walkv] (6.east) .. controls +(0.50, 0.) and +(0.50, -0.) .. (7.east)
   node[pos=0.5,mid] (L8) {$\red6$};
\draw[walkv] (10) -- ($(10.east) + (3, 0)$) node[pos=0.8,mid] (L4) {$\red6$};
\draw[walkv] (11.east) .. controls +(1.70, 0.25) and +(1.50, -0.25) .. (14.east)
   node[pos=0.4,mid] (L6) {$\red9$};
\draw[walkv] (12.east) .. controls +(0.50, 0.) and +(0.50, -0.) .. (13.east)
   node[pos=0.5,mid] (L7) {$\red 10$};
\draw[line width=1.2pt,>=stealth,color=red,<-] (L2) -- (L1);
\draw[line width=1.2pt,>=stealth,color=red,<-] (L4) -- (L1);
\draw[line width=1.2pt,>=stealth,color=red,<-] (L3) -- (L2);
\draw[line width=1.2pt,>=stealth,color=red,<-] (L8) -- (L3);
\draw[line width=1.2pt,>=stealth,color=red,<-] (L5) -- (L2);
\draw[line width=1.2pt,>=stealth,color=red,<-] (L6) -- (L4);
\draw[line width=1.2pt,>=stealth,color=red,<-] (L7) -- (L6);
    \end{tikzpicture}
  \]
  \caption{A Roby tableau and its associated LE-labelling (left and middle) together with 
  the corresponding Okada half arc-diagram (right).}
  \label{fig.roby-arc}
\end{figure}

\begin{remark}
Roby's insertion algorithm,
and the $\Chain$ map in \Cref{def:chain-map} are consistent:
Given a Roby-tableaux $T$,
and the Okada half arc-diagram $(H, \boldsymbol{\underline{h}})$
specified by \Cref{LE-to-HLabels} and
\Cref{LinearExtentions-to-HLabels},
the saturated chain in the $\YF$-lattice which corresponds to
$T$ under Roby's insertion algorithm
is identical to the saturated
chain in the $\YFS$-lattice
corresponding to $(H, \boldsymbol{\underline{h}})$
under the $\Chain$ map.
\end{remark}

\section{The structure of the Okada monoid}

\textbf{From now on, we won't distinguish $\Okada$ and $\OkArc$ and
  following~\cref{eq.perm-of-diag}, we identify a diagram $D$ with its
  associated Okada monoid element $\e{\perm{D}}$. In particular, the
  diagram~$\A{i}$ is identified with the generator~$\e{i}$.} \bigskip

Thanks to its diagrammatic realization, it is now much easier to
compute in the Okada monoid. The goal of this section is to understand its
structure theory. In particular, we describe its green relations and we
show that the monoid is aperiodic. This will allows us to prove
cellularity of the Okada algebra in \cref{section:okada-algebra-cellular-structure}.
We start by examining an analogue of the
dominance order on Fibonacci sets (instead of partitions).

\subsection{Dominance orders}
Recall that Fibonacci sets appear as both propagating indexing sets and
labelling sets. To describe how these two descriptions are related, we need an analogue of
the dominance order for Fibonacci sets.
\begin{definition}
  Let $S=\{s_1<\dots<s_k\}$ and $T=\{t_1<\dots<t_\ell\}$ be two Fibonacci sets of
  the same rank $N$. We says that \defn{$S$ is dominated by $T$} and write
  $S \isdomeq T$ if $k \leq \ell$ and $s_{k-i}\leq t_{\ell-i}$ for any $0\leq i<k$. We
  write $S \isdom T$ if $S \isdomeq T$ but $S \neq T$.
\end{definition}
\begin{proposition}\label{prop.dominance.lattice}
  $(\YFSn, \isdomeq)$ is a ranked distributive lattice such that
  \begin{itemize}
  \item the minimal element is $\emptyset$ if $N$ is even and $\{1\}$ if $N$
    is odd;
  \item the maximal element is $\setN$;
  \item $S \isdom T$ is a cover if and only if  either $T = (S\setminus \{i\})\sqcup \{i+2\}$ whenever
    $i\in S$ and $i+1\notin S$ or $T= S \sqcup\{1,2\}$ provided $1\notin S$ and $2\notin S$;
  \item the rank of $S$ is $\left(\sum_{s \in S} s  - (\#S+p)/2\right)/2$ where
    $p \eqdef N \bmod 2$;
  \item for $S=\{s_1<\dots<s_k\}$ and $T=\{t_1<\dots<t_\ell\}$ the respective meet and join are
  \begin{alignat}{1}
    S \wedge T&
    =  \set{\min(s_{k-i},t_{\ell-i})}{0<i<\min(k,\ell)},\\
   S \vee T&
    = \set{\max(s_{k-i},t_{\ell-i})}{0<i<\max(k,\ell)},
  \end{alignat}
  where $t_i=s_i=0$ is understood whenever $i\leq 0$
  in the last formula.
  \end{itemize}
\end{proposition}
\Cref{fig.dominance} shows the Hasse diagrams for the dominance order up to rank $N=6$.
\begin{proof}
  All the assertions are straightforward and left to the reader, except for the rank
  formula: One can easily check that the minimal element has
  rank $0$ and that each cover adds one to the value 
  returned by the proposed rank formula (in particular it's always an integer). This proves that the lattice is ranked.
\end{proof}
\begin{figure}[ht]
\[
  % sage: P = Poset([Fib_words(4), lambda x,
  % y:is_dominated(y,x)]).relabel(FibSet_of_FibWord)
\begin{tikzpicture}[>=latex,line join=bevel,draw,black]
\node (00) at (-1,0) {$\emptyset$};
\node (10) at (-0,0) {$\left\{1\right\}$};
\node (20) at (1.2,0) {$\emptyset$};
\node (21) at (1.2,1) {$\left\{1, 2\right\}$};
\draw[black,->] (20) -- (21);
\node (30) at (2.5,0) {$\left\{1\right\}$};
\node (31) at (2.5,1) {$\left\{3\right\}$};
\node (32) at (2.5,2) {$\left\{1, 2, 3\right\}$};
\draw [black,->] (30) -- (31);
\draw [black,->] (31) -- (32);
\node (40) at (4,0) {$\emptyset$};
\node (41) at (4,1) {$\left\{1, 2\right\}$};
\node (42) at (4,2) {$\left\{1, 4\right\}$};
\node (43) at (4,3) {$\left\{3, 4\right\}$};
\node (44) at (4,4) {$\left\{1, 2, 3, 4\right\}$};
\draw [black,->] (40) -- (41);
\draw [black,->] (41) -- (42);
\draw [black,->] (42) -- (43);
\draw [black,->] (43) -- (44);
\node (50) at (6.5,0) {$\left\{1\right\}$};
\node (51) at (6.5,1) {$\left\{3\right\}$};
\node (52) at (5.5,2) {$\left\{1, 2, 3\right\}$};
\node (53) at (7.5,2) {$\left\{5\right\}$};
\node (54) at (6.5,3) {$\left\{1, 2, 5\right\}$};
\node (55) at (6.5,4) {$\left\{1, 4, 5\right\}$};
\node (56) at (6.5,5) {$\left\{3, 4, 5\right\}$};
\node (57) at (6.5,6) {$\left\{1, 2, 3, 4, 5\right\}$};
\draw [black,->] (50) -- (51);
\draw [black,->] (51) -- (52);
\draw [black,->] (51) -- (53);
\draw [black,->] (52) -- (54);
\draw [black,->] (53) -- (54);
\draw [black,->] (54) -- (55);
\draw [black,->] (55) -- (56);
\draw [black,->] (56) -- (57);
\node (6-0) at (10.5,0) {$\emptyset$};
\node (6-1) at (10.5,1) {$\left\{1, 2\right\}$};
\node (6-2) at (10.5,2) {$\left\{1, 4\right\}$};
\node (6-3) at (9,3) {$\left\{3, 4\right\}$};
\node (6-5) at (12,3) {$\left\{1, 6\right\}$};
\node (6-4) at (9,4) {$\left\{1, 2, 3, 4\right\}$};
\node (6-6) at (12,4) {$\left\{3, 6\right\}$};
\node (6-7) at (9,5) {$\left\{1, 2, 3, 6\right\}$};
\node (6-8) at (12,5) {$\left\{5, 6\right\}$};
\node (6-9) at (10.5,6) {$\left\{1, 2, 5, 6\right\}$};
\node (6-10) at (10.5,7) {$\left\{1, 4, 5, 6\right\}$};
\node (6-11) at (10.5,8) {$\left\{3, 4, 5, 6\right\}$};
\node (6-12) at (10.5,9) {$\left\{1, 2, 3, 4, 5, 6\right\}$};
\draw [black,->] (6-0) -- (6-1);
\draw [black,->] (6-1) -- (6-2);
\draw [black,->] (6-2) -- (6-3);
\draw [black,->] (6-2) -- (6-5);
\draw [black,->] (6-3) -- (6-4);
\draw [black,->] (6-3) -- (6-6);
\draw [black,->] (6-4) -- (6-7);
\draw [black,->] (6-5) -- (6-6);
\draw [black,->] (6-6) -- (6-7);
\draw [black,->] (6-6) -- (6-8);
\draw [black,->] (6-7) -- (6-9);
\draw [black,->] (6-8) -- (6-9);
\draw [black,->] (6-9) -- (6-10);
\draw [black,->] (6-10) -- (6-11);
\draw [black,->] (6-11) -- (6-12);
\end{tikzpicture}
\]
  \caption[Dominance orders of Fibonacci sets]{The dominance order on Fibonacci
    sets of rank $0$ to $6$.}
  \label{fig.dominance}
\end{figure}
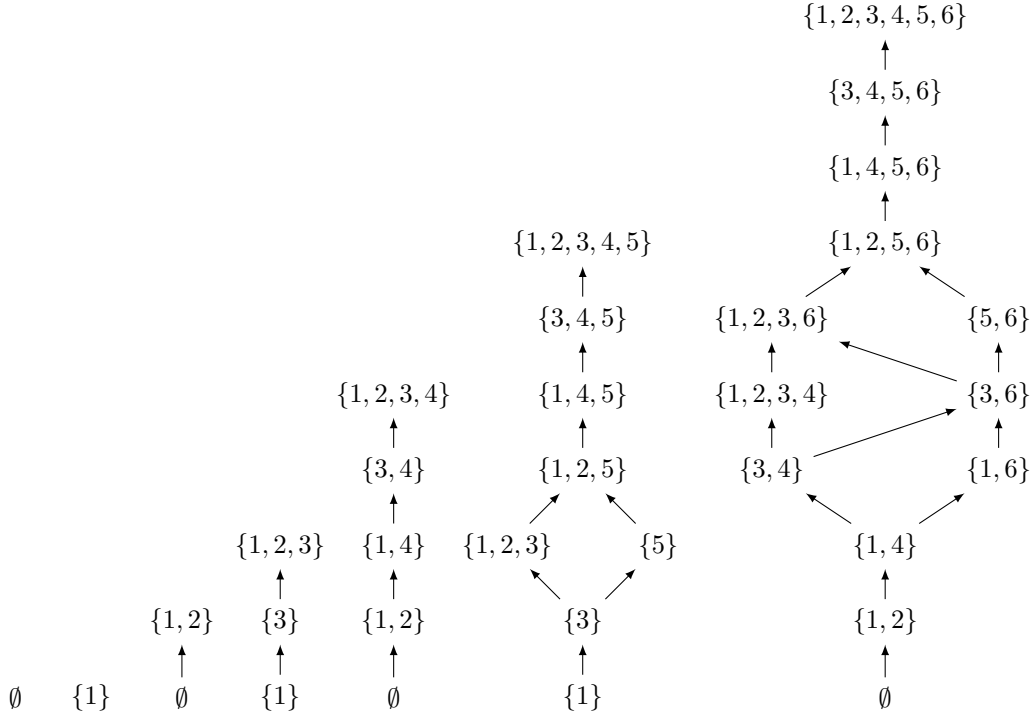
\begin{remark}
   Okada's original description of the dominance order (comparing partial
   sums of suffixes of Fibonacci words) introduced in  ~\cite[p.~560 and Theorem
  4.2]{Okada1994} can be recovered using the bijection
  between Fibonacci words and Fibonacci sets given in \Cref{prop-bij-fib-set-word}.
\end{remark}

\subsection{Half diagram surgery}
The goal of this section is to present a few ways to construct Okada half
arc-diagrams from Fibonacci sets. We starts with a remark which shows a
first link between arc-diagrams and dominance order:
\begin{remark}
  Let $H$ be an Okada half arc-diagram, then $\PropLab(H)\isdomeq\PropInd(H)$.
\end{remark}
\begin{proof}
  This is just a consequence of \cref{okada-arc-cond-interv}.
\end{proof}
The following lemma-definition is a kind of converse:
\begin{definition}\label{lemma:connect-isdomeq}
  Let $S=\{s_1<\dots<s_k\}$ and $T=\{t_1<\dots<t_\ell\}$ be two Fibonacci sets of
  rank $N$ such that $S\isdomeq T$ (in particular $k\leq \ell$).  Let
  $\tilde{T}\eqdef\{t_{\ell-k+1},\dots,t_\ell\}$. There exists a unique half Okada
  arc-diagram $\linkd{T}{S}$ of rank $N$ such that
  \begin{itemize}
  \item $\PropInd \ssp \linkd{T}{S} = \tilde T$ and
    $\PropLab \ssp \linkd{T}{S} =S$.
    In other words, the propagating arcs of $\linkd{T}{S}$ are the arcs
    $\harc{t_{\ell-i}}{s_{k-i}}$ for all $0\leq i<k$.
  \item All the other arcs are of the form $\arc{i}{i}{i+1}$.
  \end{itemize}
\end{definition}
\begin{proof}
  First of all, $\tilde T$ is clearly a rank $N$ Fibonacci set. As in the previous
  lemmas there exists a unique way to complete with $\arc{i}{i}{i+1}$. The
  definition of Fibonacci sets ensure that $t_{k-i}$ and $s_{\ell-i}$ have the
  same parity, so that \cref{okada-arc-cond-parity} holds. One gets the two
  other \cref{okada-arc-cond-interv,okada-arc-cond-nested} because
  $s_{k-i}\leq t_{\ell-i}$.
\end{proof}
We will use two extremal cases of the previous definition. Given
$S=\{s_1<\dots<s_k\}$ a Fibonacci sets of rank $N$, the minimal possible $T$
such that $S\isdomeq T$ is $T= S$, leading to the half diagram
$\freeld{S}$\ssp. The maximal possible is $T = \setN$ leading to
the diagram $\topd{S}$\ssp. See~\cref{fig.ex-half-arcs} for
some examples.

\begin{figure}[ht]
\begin{tikzpicture}[ultra thick,yscale=0.5,
  baseline={(current bounding box.east)}]
  \tikzstyle{mid}=[draw,thick,black,anchor=center,fill=white,
                   shape=circle,minimum size=2pt,inner sep=1pt]
  \tikzstyle{walkv}=[draw, line width=2pt]
\node[anchor=east] (1) at (0, 1) {$\{1\}$};
\node[anchor=east] (2) at (0, 2) {$\emptyset$};
\node[anchor=east] (3) at (0, 3) {$\{1\}$};
\node[anchor=east] (4) at (0, 4) {$\{1, 4\}$};
\node[anchor=east] (5) at (0, 5) {$\{1\}$};
\node[anchor=east] (6) at (0, 6) {$\{1, 4\}$};
\node[anchor=east] (7) at (0, 7) {$\{1, 4, 5\}$};
\draw[walkv] (1.east) .. controls +(0.50, 0.25) and +(0.50, -0.25) .. (2.east)
    node[pos=0.5,mid] {$1$};
\draw[walkv] (3) -- ($(3.east) + (1, 0)$) node[pos=0.5,mid] {$1$};
\draw[walkv] (4.east) .. controls +(0.50, 0.25) and +(0.50, -0.25) .. (5.east)
    node[pos=0.5,mid] {$4$};
\draw[walkv] (6) -- ($(6.east) + (1, 0)$) node[pos=0.5,mid] {$4$};
\draw[walkv] (7) -- ($(7.east) + (1, 0)$) node[pos=0.5,mid] {$5$};
\end{tikzpicture}
  \qquad
  %%%%%%%%%%%%%%%%%%
\begin{tikzpicture}[ultra thick,yscale=0.5,
  baseline={(current bounding box.east)}]
  \tikzstyle{mid}=[draw,thick,black,anchor=center,fill=white,
                   shape=circle,minimum size=2pt,inner sep=1pt]
  \tikzstyle{walkv}=[draw, line width=2pt]
\node[anchor=east] (1) at (0, 1) {$\{1\}$};
\node[anchor=east] (2) at (0, 2) {$\{1, 2\}$};
\node[anchor=east] (3) at (0, 3) {$\{1\}$};
\node[anchor=east] (4) at (0, 4) {$\{1, 4\}$};
\node[anchor=east] (5) at (0, 5) {$\{1, 4, 5\}$};
\node[anchor=east] (6) at (0, 6) {$\{1, 4, 5, 6\}$};
\node[anchor=east] (7) at (0, 7) {$\{1, 4, 5\}$};
\draw[walkv] (1) -- ($(1.east) + (1, 0)$) node[pos=0.5,mid] {$1$};
\draw[walkv] (2.east) .. controls +(0.50, 0.25) and +(0.50, -0.25) .. (3.east)
    node[pos=0.5,mid] {$2$};
\draw[walkv] (4) -- ($(4.east) + (1, 0)$) node[pos=0.5,mid] {$4$};
\draw[walkv] (5) -- ($(5.east) + (1, 0)$) node[pos=0.5,mid] {$5$};
\draw[walkv] (6.east) .. controls +(0.50, 0.25) and +(0.50, -0.25) .. (7.east)
    node[pos=0.5,mid] {$6$};
  \end{tikzpicture}
  \qquad
  %%%%%%%%%%%%%%%%%%
  \begin{tikzpicture}[ultra thick,yscale=0.5,
  baseline={(current bounding box.east)}]
  \tikzstyle{mid}=[draw,thick,black,anchor=center,fill=white,
                   shape=circle,minimum size=2pt,inner sep=1pt]
  \tikzstyle{walkv}=[draw, line width=2pt]
\node[anchor=east] (1) at (0, 1) {$\{1\}$};
\node[anchor=east] (2) at (0, 2) {$\emptyset$};
\node[anchor=east] (3) at (0, 3) {$\{3\}$};
\node[anchor=east] (4) at (0, 4) {$\emptyset$};
\node[anchor=east] (5) at (0, 5) {$\{1\}$};
\node[anchor=east] (6) at (0, 6) {$\{1, 4\}$};
\node[anchor=east] (7) at (0, 7) {$\{1, 4, 5\}$};
\draw[walkv] (1.east) .. controls +(0.50, 0.25) and +(0.50, -0.25) .. (2.east)
    node[pos=0.5,mid] {$1$};
\draw[walkv] (3.east) .. controls +(0.50, 0.25) and +(0.50, -0.25) .. (4.east)
    node[pos=0.5,mid] {$3$};
\draw[walkv] (5) -- ($(5.east) + (1, 0)$) node[pos=0.5,mid] {$1$};
\draw[walkv] (6) -- ($(6.east) + (1, 0)$) node[pos=0.5,mid] {$4$};
\draw[walkv] (7) -- ($(7.east) + (1, 0)$) node[pos=0.5,mid] {$5$};
\end{tikzpicture}
\caption[Half diagram surgery examples]{The half arc-diagrams
  $\linkd{T}{S}$ \ssp , $\freeld{S}$ and $\topd{S}$ for
  $S=\{1,4,5\}_\mathrm{\bf 7}$ and $T=\{1,2,3,6,7\}_\mathrm{\bf 7}$}
  \label{fig.ex-half-arcs}
\end{figure}
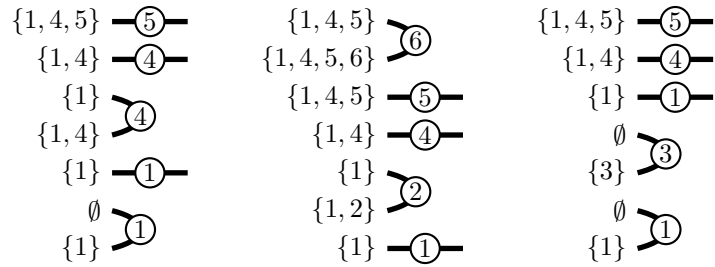

\subsection{Aperiodicity of the Okada monoid}

We now return to structural aspects of the Okada monoid. In this subsection
$N$ is a fixed positive integer.
\begin{lemma}\label{lemma:mult-compat}
  Suppose that $F$, $G$,  and $H$ are three half Okada
  arc-diagrams of rank $N$ with the same propagating set. Then
  $\glue{F}{G}\cdot\glue{G}{H} = \glue{F}{H}$.
\end{lemma}
\begin{proof}
  Since $F$, $G$, and $H$ have the same  $\PropLab$-set
  $S$, their respective propagating arcs labelled by $s \in S$ are linked together
  in $\glue{F}{G}$, $\glue{G}{H}$ and $\glue{F}{H}$. Applying the product
  rule, the pair of propagating arcs in $\glue{F}{G}$ and $\glue{G}{H}$ 
  labelled by $s$ concatenate to form a propagating arc in $\glue{F}{G}\cdot\glue{G}{H}$
  which is also labelled $s$. Clearly the endpoints of this concatenated
  arc match the endpoints of the propagating arc labelled $s$ in $\glue{F}{H}$
  Likewise, non-propagating left (resp. right) arcs of
  $\glue{F}{G}\cdot\glue{G}{H}$ and $\glue{F}{H}$ are exactly the non-propagating
  arcs of $F$ (resp. $H$).
\end{proof}
Setting $F = G = H$ in the previous lemma, one gets:
\begin{corollary}
  For any half Okada arc-diagram $H$, the element $\glue{H}{H}$ is 
  idempotent.
\end{corollary}
\begin{definition}
  An element $\mathsf{e}  \in\Okada$ is said to be \defn{involutive} if
  $\mathsf{e} =\glue{H}{H}$
  for some half Okada arc-diagram $H$; equivalently if $\mathsf{e}^\star= \mathsf{e} $.
\end{definition}
Thanks to the $\RS$-correspondence, $\mathsf{e}$ is involutive if and only if $\mathsf{e}  =\e{\sigma}$
for some involution $\sigma \in \SG[N]$ (\ie $\sigma^2=1$).
\begin{remark}
  The involutive elements do not exhaust, by far, the idempotents of
  $\Okada$. For example in $\Okada[3]$ all element are idempotents and in
  $\Okada[4]$ the only non-idempotents are $\e{1}\e{2}\e{3}$ and
  $\e{3}\e{2}\e{1}$. We have not yet managed to fully understand which elements of
  $\Okada$ are itempotents. This problem seems to be non-trivial even 
  in the unlabelled case, i.e. for the Jones
  monoid~(see~\cite{DOLINKA2019}). Nevertheless, thanks to some computer
  experiments, we know their numbers for $n=0,\dots,10$:
  \[1,\quad1,\quad 2,\quad 6,\quad 22,\quad 108,\quad 594,\quad 4116,\quad 30500,\quad 274006,\quad 2560400.\]
\end{remark}
The usual, group-theoretic definition of \emph{inverse} isn't very useful for monoids
and doesn't even make sense for semigroups (which don't have an
identity). Therefore, in the context of monoids, one usually uses the following
definition which is equivalent when the monoid is a group: $y$ is a
\emph{(monoid) inverse} of $x$ if $xyx=x$ and $yxy=y$. A \emph{regular} monoid
is a monoid where each element has an inverse. An \emph{inverse} monoid is a
regular monoid where the inverse is unique. Therefore a group can be defined as an
inverse monoid which is {\it cancellative} (i.e. right and left multiplication are
both injective). The following is a simple consequence of \cref{lemma:mult-compat}:
\begin{corollary}\label{cor:regular}
  For any $\mathsf{e}\in\Okada$, one has
  $\mathsf{e}\mathsf{e}^*\mathsf{e} = \mathsf{e}$. In particular, $\Okada$ is
  a regular monoid,
\end{corollary}
However for $N\geq3$, the monoid $\Okada$ is not inverse. Indeed, the four
elements $\e{2}$, $\e{1}\e{2}$, $\e{2}\e{1}$, and $\e{1}\e{2}\e{1}$ are all pairwise inverse.
\medskip

We can now relate the product of the Okada monoid and the dominance order:
\begin{lemma}\label{lemma:left-eq-or-prop-le}
  Let $\mathsf{e}, \mathsf{f}$ be two elements in $\Okada$. Then
  either $\rightd{\mathsf{ef}}=\rightd{\mathsf{e}}$ and thus
  $\PropLab(\mathsf{ef})=\PropLab(\mathsf{e})$ or
  $\PropLab(\mathsf{ef})\isdom\PropLab(\mathsf{e})$.
\end{lemma}
\begin{proof}
  By induction, it suffice to show the statement when $\mathsf{f}=\e{i}$ for some $i$.
  Consider the product $\mathsf{e} \e{i}$. Four cases are possible depending whether the
  arc $\alpha_i=\arc{\ov{i}}{\ell_i}{a_i}$ and
  $\alpha_{i+1}=\arc{\ov{i+1}}{\ell_{i+1}}{a_{i+1}}$ are propagating in $\mathsf{e}$ or not:
  \begin{itemize}
  \item If $\alpha_{i}$ and $\alpha_{i+1}$ are not propagating, then the arc
    $\arc{i}{i}{i+1}$ of $\e{i}$ joins two right arcs of $\mathsf{e}$, so that
    $\rightd{\mathsf{e} \e{i}}=\rightd{\mathsf{e}}$.
  \item If $\alpha_{i}$ is propagating but not $\alpha_{i+1}$, then $\alpha_{i+1}$ is nested
    in $\alpha_i$ so that the resulting arc is labelled $\ell_i$. Again,
    $\rightd{\mathsf{e} \e{i}}=\rightd{\mathsf{e}}$.
  \item If $\alpha_{i}$ is not propagating but $\alpha_{i+1}$ is, then either the label
    of $\ell_{i+1}\leq \ell_{i}$ and the resulting arc is labelled $\ell_{i+1}$, and once again
    $\rightd{\mathsf{e} \e{i}}=\rightd{\mathsf{e}}$. Otherwise the new label is $\ell_{i}$ and 
    $\PropLab(\mathsf{e} \e{i})=(\PropLab(\mathsf{e})\setminus\{\ell_{i+1}\})\cup\{\ell_{i}\}$. 
    By the nesting condition, $\ell_{i}$ must be larger than any other propagating
    label. It follows that $\PropLab(\mathsf{e} \e{i})\isdom\PropLab(\mathsf{e})$.
  \item Finally, if both $\alpha_i$ and $\alpha_{i+1}$ are propagating then, together
    with $\arc{i}{i}{i+1}$ in $\e{i}$, they create a non-propagating arc
    $\arc{a_i}{\ell_i}{a_{i+1}}$. Then
    $\PropLab(\mathsf{e} \e{i})=\PropLab(\mathsf{e})\setminus\{\ell_{i},\ell_{i+1}\}$ and again
    we have $\PropLab(\mathsf{e} \e{i})\isdom\PropLab(\mathsf{e})$. \qedhere
  \end{itemize}
\end{proof}
By mirror symmetry, one obtain the following corollary:
\begin{corollary}\label{cor:prod-inf-prop}
  Let $\mathsf{e}, \mathsf{f} \in\Okada$. Then $\PropLab(\mathsf{ef} )\isdomeq \PropLab(\mathsf{e} ) \wedge \PropLab(\mathsf{f} )$.
\end{corollary}
We are now ready to describe the structure of the Okada monoid:
\begin{theorem}\label{theo:aperiodic}
  For all $N$, the monoid $\Okada$ is aperiodic.
\end{theorem}
\begin{proof}
  Let $\mathsf{e}  \in\Okada$. Consider the sequence $P_k\eqdef\PropLab(\mathsf{e}^k)$. According to
  \cref{cor:prod-inf-prop}, it is decreasing with respect to $\isdomeq$ and therefore must
  eventually stabilize. Find an index $i$ such that $P_i=P_{i+1}$. Applying
  \cref{lemma:left-eq-or-prop-le} to $\mathsf{e}^i$ and $\mathsf{e} $, we get that
  $\rightd{\mathsf{e} ^{i+1}}=\rightd{\mathsf{e} ^i}$ \ssp . Likewise
$\leftd{\mathsf{e}^{i+1}}=\leftd{\mathsf{e}^i}$, and so $\mathsf{e}^{i+1}=\mathsf{e}^i$.
\end{proof}

\subsection{Green relations in the Okada monoid}
We are now ready to describe combinatorially the Green relations and
orders of the Okada monoid. We first study the double sided case:
Recall that $\e{}\leJJ\f{}$ means that there exists $\mathsf{u}$ and
$\mathsf{v}$ such that $\e{}=\mathsf{u}\f{}\mathsf{v}$ and that
$\e{} \JJ \f{}$ means $\e{} \leJJ \f{}$ and $\f{} \leJJ \e{}$. Finally, recall
that $\JJ$ is an equivalence relation. We begin by
describing a set of representative for $\JJ$ classes.
\begin{definition}
  A Okada half arc-diagram is \defn{free} if it contains only propagating arcs of the
  form $\harc{i}{i}{}$ and non-propagating arcs of the form $\arc{i}{i}{i+1}$. 
  An Okada arc-diagram
  $D\in\Okada$ is \defn{free} if $D=\glue{H}{H}$ for some free Okada half
  arc-diagram~$H$.
\end{definition}
The following two lemmas are easy.
\begin{lemma} The map
  $H \mapsto \PropLab(H)$ is a bijection between rank $N$ free Okada half arc-diagrams and rank $N$ Fibonacci
  sets.
\end{lemma}
\begin{lemma}
  Let $\e{}  \in\Okada$ be a free element. Let $S \eqdef \PropLab(\e{})$ and define
  \begin{equation}
    F(S) \eqdef \set{i}{i-\max\set{j\in S}{j\leq i}\text{ is odd}} \, .
  \end{equation}
  Then $\e{} =\prod_{i\in F(S)}\e{i}$. The set $F(S)$ doesn't contain both $i$ and $i+1$
  for any $i\in\setN$. In particular the product is fully commutative.
  \end{lemma}
\begin{proof}
  Note we can recover $S$ from $F(S)$ thanks to
  \begin{equation}
    S = \set{i}{i\notin F(S) \text{ and }i-1\notin F(S)}\,.
  \end{equation}
  The map $S \mapsto F(S)$ is a bijection between Fibonacci sets of rank
  $N$ and subsets of $\setN$ which don't contain consecutive indices. The
  Lemma follows from the fact that $F(S)$ is exactly the set of indices $i \in \setN$ such that
  $\arc{i}{i}{i+1}$ is an arc of $\e{}$.
\end{proof}
\begin{proposition}
\label{JJclass-free-rep}
  Each $\JJ$-class of $\Okada$ contains a unique free element. Moreover, the
  free representative of the $\JJ$-class of an element $\e{} \in\Okada$ is the
  unique free element having the same propagating labelling set as $\e{}$.
\end{proposition}
\begin{proof}
  Let $\e{} \in\Okada$. Set $H\eqdef\freeld{S}$ where $S=\PropLab(\e{})$. By
  \cref{lemma:mult-compat}, we have the factorization
  \[
    \e{} = \glue{\rightd{\e{}}}{H} \boldsymbol{\cdot} \glue{H}{\leftd{\e{}}}
    = \glue{\rightd{\e{}}}{H} \boldsymbol{\cdot} \glue{H}{H} \boldsymbol{\cdot} \glue{H}{\leftd{\e{}}}\,.
  \]
  Moreover
  \[
    \glue{H}{H} = \glue{H}{\leftd{\e{}}} \boldsymbol{\cdot} \e{} \boldsymbol{\cdot} \glue{\rightd{\e{}}}{H}\,.
  \]
  So each element lies in the $\JJ$-class of some free element. Conversely, by \cref{cor:prod-inf-prop}, two elements within the same
  $\JJ$-class must share the same $\PropLab$ set.
\end{proof}
We a now in a position to describe combinatorially the $\JJ$-relation
and order.
\begin{theorem}\label{theo:j-order-okada}
  Let $\e{}, \f{} \in\Okada$, then $\e{} \leJJ \f{}$ if and only if
  $\PropLab(\e{})\isdomeq\PropLab(\f{})$. In particular $\e{}, \f{}$
  belongs to the same $\JJ$-class if and only if
  $\PropLab(\e{})=\PropLab(\f{})$.
\end{theorem}
\begin{proof}
  The implication ($\Rightarrow$) is already known from \cref{cor:prod-inf-prop}.
  To prove the converse, we can assume, without loss of generality, that $\e{}$
  and $\f{}$ are free. Let $S = \PropLab(\e{})$ and $T = \PropLab( \f{} )$ and
  define
  \begin{equation*}
    \mathsf{g} \eqdef \glue{\linkd{S}{S}}{\linkd{T}{S}}
  \end{equation*}
  from~\cref{lemma:connect-isdomeq}. Then $\mathsf{g} \boldsymbol{\cdot} \f{} \boldsymbol{\cdot} \mathsf{g}^\star = \e{}$, which
  proves that $\e{} \leJJ \f{}$.
\end{proof}

We now turn our attension to the one-sided Green relations.
Recall that $\e{} \leRR \f{}$ means that there exists $\mathsf{g}$ such that $\e{} = \f{} \boldsymbol{\cdot} \mathsf{g}$ and that
$\e{} \RR \f{}$ means $\e{} \leRR \f{}$ and $\f{} \leRR \e{}$. Thanks to the previous lemmas, we
are ready to describe the $\RR$-classes of $\Okada$:
\begin{proposition}\label{prop:r-class}
  Let $\e{}, \f{} \in\Okada$. Then $\e{} \RR \f{}$ if and only if \, $\rightd{\e{}}=\rightd{ \f{} }$.
\end{proposition}

\begin{proof}
  If $\e{} \RR \f{}$ but $\rightd{\e{}}\neq\rightd{\f{} }$ by \cref{lemma:left-eq-or-prop-le},
  both $\PropLab(\e{})\isdom\PropLab(\f{})$ and $\PropLab(\f{})\isdom\PropLab(\e{})$ holds.
  This is absurd since $\isdom$ is a strict order.  Conversely,
  \cref{lemma:mult-compat} shows that if $\rightd{\e{}}=\rightd{\f{}}$ then
  $\e{} \boldsymbol{\cdot} \glue{\rightd{\e{}}}{\leftd{\f{}}}=\glue{\rightd{\f{} }}{\leftd{\e{} }} \boldsymbol{\cdot}
  \glue{\rightd{\e{}  }}{\leftd{ \f{} }}= \f{}$.  This show that $\f{} \leRR \e{} $. By symmetry
  $\e{} \leRR \f{}$ and so $\e{} \RR \f{}$.
\end{proof}
\begin{corollary}
  Each $\RR$-class of $\Okada$ contains a unique involutive element.
\end{corollary}
Involutive elements are not the only representative of right classes. Another
possible choice of representative is $\glue{H}{\freeld{S}}$ where $S = \PropLab(H)$, 
whose reduced factorization in $\Okada$ has shortest length.
A third representative is
$\glue{H}{\topd{S}}$ whose
reduced factorization is longest.  \bigskip

\newcommand{\brel}{\triangleleft}%
\newcommand{\breleq}{\trianglelefteq}%
\begin{figure}[ht]
  \centering
  $\begin{array}{c}
    \begin{tikzpicture}[ultra thick,yscale=0.5,
  baseline={(current bounding box.east)},
background rectangle/.style={fill=gray!25, rounded corners}, show background rectangle]
  \tikzstyle{mid}=[draw,thick,black,anchor=center,fill=white,
                   shape=circle,minimum size=2pt,inner sep=1pt]
  \tikzstyle{walkv}=[draw, line width=2pt]
\node[inner sep=0pt, anchor=east] (1) at (0, 1) {$\!\{1\}$\,};
\node[inner sep=0pt, anchor=east] (2) at (0, 2) {$\!\{1, 2\}$\,};
\node[inner sep=0pt, anchor=east] (3) at (0, 3) {$\!\{1\}$\,};
\node[inner sep=0pt, anchor=east] (4) at (0, 4) {$\!\{1, 4\}$\,};
\node[inner sep=0pt, anchor=east] (5) at (0, 5) {$\!\{1, 4, 5\}$\,};
\node[inner sep=0pt, anchor=east] (6) at (0, 6) {$\!\{1, 4\}$\,};
\node[inner sep=0pt, anchor=east] (7) at (0, 7) {$\!\{1\}$\,};
\node[inner sep=0pt, anchor=east] (8) at (0, 8) {$\!\{1, 4\}$\,};
\draw[walkv] (1) -- ($(1.east) + (1, 0)$) node[pos=0.5,mid] {$1$};
\draw[walkv] (2.east) .. controls +(0.50, 0.25) and +(0.50, -0.25) .. (3.east)
    node[pos=0.5,mid] {$2$};
\draw[walkv] (4.east) .. controls +(1.30, 0.25) and +(1.30, -0.25) .. (7.east)
    node[pos=0.5,mid] {$4$};
\draw[walkv] (5.east) .. controls +(0.50, 0.25) and +(0.50, -0.25) .. (6.east)
    node[pos=0.5,mid] {$5$};
\draw[walkv] (8) -- ($(8.east) + (1, 0)$) node[pos=0.5,mid] {$4$};
\end{tikzpicture}
    \\
     \rotatebox[origin=c]{90}{\scalebox{1.8}{\ $\ntrianglerighteq$\ }}
    \\
    \begin{tikzpicture}[ultra thick,yscale=0.5,
  baseline={(current bounding box.east)},
background rectangle/.style={fill=gray!25, rounded corners}, show background rectangle]
  \tikzstyle{mid}=[draw,thick,black,anchor=center,fill=white,
                   shape=circle,minimum size=2pt,inner sep=1pt]
  \tikzstyle{walkv}=[draw, line width=2pt]
\node[inner sep=0pt, anchor=east] (1) at (0, 1) {$\!\{1\}$\,};
\node[inner sep=0pt, anchor=east] (2) at (0, 2) {$\!\emptyset$\,};
\node[inner sep=0pt, anchor=east] (3) at (0, 3) {$\!\{3\}$\,};
\node[inner sep=0pt, anchor=east] (4) at (0, 4) {$\!\{3, 4\}$\,};
\node[inner sep=0pt, anchor=east] (5) at (0, 5) {$\!\{3, 4, 5\}$\,};
\node[inner sep=0pt, anchor=east] (6) at (0, 6) {$\!\{3, 4\}$\,};
\node[inner sep=0pt, anchor=east] (7) at (0, 7) {$\!\{3\}$\,};
\node[inner sep=0pt, anchor=east] (8) at (0, 8) {$\!\{3, 6\}$\,};
\draw[walkv] (1.east) .. controls +(0.50, 0.25) and +(0.50, -0.25) .. (2.east)
    node[pos=0.5,mid] {$1$};
\draw[walkv] (3) -- ($(3.east) + (1, 0)$) node[pos=0.5,mid] {$3$};
\draw[walkv] (4.east) .. controls +(1.30, 0.25) and +(1.30, -0.25) .. (7.east)
    node[pos=0.5,mid] {$4$};
\draw[walkv] (5.east) .. controls +(0.50, 0.25) and +(0.50, -0.25) .. (6.east)
    node[pos=0.5,mid] {$5$};
\draw[walkv] (8) -- ($(8.east) + (1, 0)$) node[pos=0.5,mid] {$6$};
\end{tikzpicture}
    \\
     \rotatebox[origin=c]{90}{\scalebox{1.8}{\ $\trianglelefteq$\ }}
    \\
  \begin{tikzpicture}[ultra thick,yscale=0.5,
  baseline={(current bounding box.east)},
background rectangle/.style={fill=gray!25, rounded corners}, show background rectangle]
  \tikzstyle{mid}=[draw,thick,black,anchor=center,fill=white,
                   shape=circle,minimum size=2pt,inner sep=1pt]
  \tikzstyle{walkv}=[draw, line width=2pt]
\node[inner sep=0pt, anchor=east] (1) at (0, 1) {$\!\{1\}$\,};
\node[inner sep=0pt, anchor=east] (2) at (0, 2) {$\!\emptyset$\,};
\node[inner sep=0pt, anchor=east] (3) at (0, 3) {$\!\{1\}$\,};
\node[inner sep=0pt, anchor=east] (4) at (0, 4) {$\!\{1, 4\}$\,};
\node[inner sep=0pt, anchor=east] (5) at (0, 5) {$\!\{1, 4, 5\}$\,};
\node[inner sep=0pt, anchor=east] (6) at (0, 6) {$\!\{1, 4\}$\,};
\node[inner sep=0pt, anchor=east] (7) at (0, 7) {$\!\{1\}$\,};
\node[inner sep=0pt, anchor=east] (8) at (0, 8) {$\!\emptyset$\,};
\draw[walkv] (1.east) .. controls +(0.50, 0.25) and +(0.50, -0.25) .. (2.east)
    node[pos=0.5,mid] {$1$};
\draw[walkv] (3.east) .. controls +(2.10, 0.25) and +(2.10, -0.25) .. (8.east)
    node[pos=0.5,mid] {$1$};
\draw[walkv] (4.east) .. controls +(1.30, 0.25) and +(1.30, -0.25) .. (7.east)
    node[pos=0.5,mid] {$4$};
\draw[walkv] (5.east) .. controls +(0.50, 0.25) and +(0.50, -0.25) .. (6.east)
    node[pos=0.5,mid] {$5$};
  \end{tikzpicture}
   \end{array}$
   %%%%%%%%%%%%%%%%%%%%%%%%%
   \qquad\qquad
  \begin{tikzpicture}[>=latex,line join=bevel,yscale=0.8,baseline=(current bounding box.center)]
\node (node_0) at (122.0bp,32.5bp) [draw,draw=none] {$\begin{tikzpicture}[ultra thick,yscale=0.5,  baseline={(current bounding box.east)},background rectangle/.style={ fill=gray!25, rounded corners }, show background rectangle]  \tikzstyle{mid}=[draw,thick,black,anchor=center,fill=white,                   shape=circle,minimum size=2pt,inner sep=1pt]  \tikzstyle{walkv}=[draw, line width=2pt]\node[inner sep=0pt, anchor=east] (1) at (0, 1) {$\{1\}\,$};\node[inner sep=0pt, anchor=east] (2) at (0, 2) {$\emptyset\,$};\node[inner sep=0pt, anchor=east] (3) at (0, 3) {$\{1\}\,$};\node[inner sep=0pt, anchor=east] (4) at (0, 4) {$\emptyset\,$};\draw[walkv] (1.east) .. controls +(0.50, 0.25) and +(0.50, -0.25) .. (2.east)    node[pos=0.5,mid] {$1$};\draw[walkv] (3.east) .. controls +(0.50, 0.25) and +(0.50, -0.25) .. (4.east)    node[pos=0.5,mid] {$1$};\end{tikzpicture}$};
  \node (node_1) at (122.0bp,137.5bp) [draw,draw=none] {$\begin{tikzpicture}[ultra thick,yscale=0.5,  baseline={(current bounding box.east)},background rectangle/.style={ fill=gray!25, rounded corners }, show background rectangle]  \tikzstyle{mid}=[draw,thick,black,anchor=center,fill=white,                   shape=circle,minimum size=2pt,inner sep=1pt]  \tikzstyle{walkv}=[draw, line width=2pt]\node[inner sep=0pt, anchor=east] (1) at (0, 1) {$\{1\}\,$};\node[inner sep=0pt, anchor=east] (2) at (0, 2) {$\emptyset\,$};\node[inner sep=0pt, anchor=east] (3) at (0, 3) {$\{1\}\,$};\node[inner sep=0pt, anchor=east] (4) at (0, 4) {$\{1, 2\}\,$};\draw[walkv] (1.east) .. controls +(0.50, 0.25) and +(0.50, -0.25) .. (2.east)    node[pos=0.5,mid] {$1$};\draw[walkv] (3) -- ($(3.east) + (1, 0)$) node[pos=0.5,mid] {$1$};\draw[walkv] (4) -- ($(4.east) + (1, 0)$) node[pos=0.5,mid] {$2$};\end{tikzpicture}$};
  \node (node_2) at (122.0bp,247.5bp) [draw,draw=none] {$\begin{tikzpicture}[ultra thick,yscale=0.5,  baseline={(current bounding box.east)},background rectangle/.style={ fill=gray!25, rounded corners }, show background rectangle]  \tikzstyle{mid}=[draw,thick,black,anchor=center,fill=white,                   shape=circle,minimum size=2pt,inner sep=1pt]  \tikzstyle{walkv}=[draw, line width=2pt]\node[inner sep=0pt, anchor=east] (1) at (0, 1) {$\{1\}\,$};\node[inner sep=0pt, anchor=east] (2) at (0, 2) {$\emptyset\,$};\node[inner sep=0pt, anchor=east] (3) at (0, 3) {$\{1\}\,$};\node[inner sep=0pt, anchor=east] (4) at (0, 4) {$\{1, 4\}\,$};\draw[walkv] (1.east) .. controls +(0.50, 0.25) and +(0.50, -0.25) .. (2.east)    node[pos=0.5,mid] {$1$};\draw[walkv] (3) -- ($(3.east) + (1, 0)$) node[pos=0.5,mid] {$1$};\draw[walkv] (4) -- ($(4.east) + (1, 0)$) node[pos=0.5,mid] {$4$};\end{tikzpicture}$};
  \node (node_8) at (122.0bp,358.5bp) [draw,draw=none] {$\begin{tikzpicture}[ultra thick,yscale=0.5,  baseline={(current bounding box.east)},background rectangle/.style={ fill=gray!25, rounded corners }, show background rectangle]  \tikzstyle{mid}=[draw,thick,black,anchor=center,fill=white,                   shape=circle,minimum size=2pt,inner sep=1pt]  \tikzstyle{walkv}=[draw, line width=2pt]\node[inner sep=0pt, anchor=east] (1) at (0, 1) {$\{1\}\,$};\node[inner sep=0pt, anchor=east] (2) at (0, 2) {$\emptyset\,$};\node[inner sep=0pt, anchor=east] (3) at (0, 3) {$\{3\}\,$};\node[inner sep=0pt, anchor=east] (4) at (0, 4) {$\{3, 4\}\,$};\draw[walkv] (1.east) .. controls +(0.50, 0.25) and +(0.50, -0.25) .. (2.east)    node[pos=0.5,mid] {$1$};\draw[walkv] (3) -- ($(3.east) + (1, 0)$) node[pos=0.5,mid] {$3$};\draw[walkv] (4) -- ($(4.east) + (1, 0)$) node[pos=0.5,mid] {$4$};\end{tikzpicture}$};
  \node (node_3) at (41.0bp,137.5bp) [draw,draw=none] {$\begin{tikzpicture}[ultra thick,yscale=0.5,  baseline={(current bounding box.east)},background rectangle/.style={ fill=gray!25, rounded corners }, show background rectangle]  \tikzstyle{mid}=[draw,thick,black,anchor=center,fill=white,                   shape=circle,minimum size=2pt,inner sep=1pt]  \tikzstyle{walkv}=[draw, line width=2pt]\node[inner sep=0pt, anchor=east] (1) at (0, 1) {$\{1\}\,$};\node[inner sep=0pt, anchor=east] (2) at (0, 2) {$\{1, 2\}\,$};\node[inner sep=0pt, anchor=east] (3) at (0, 3) {$\{1\}\,$};\node[inner sep=0pt, anchor=east] (4) at (0, 4) {$\emptyset\,$};\draw[walkv] (1.east) .. controls +(1.30, 0.25) and +(1.30, -0.25) .. (4.east)    node[pos=0.5,mid] {$1$};\draw[walkv] (2.east) .. controls +(0.50, 0.25) and +(0.50, -0.25) .. (3.east)    node[pos=0.5,mid] {$2$};\end{tikzpicture}$};
  \node (node_4) at (41.0bp,247.5bp) [draw,draw=none] {$\begin{tikzpicture}[ultra thick,yscale=0.5,  baseline={(current bounding box.east)},background rectangle/.style={ fill=gray!25, rounded corners }, show background rectangle]  \tikzstyle{mid}=[draw,thick,black,anchor=center,fill=white,                   shape=circle,minimum size=2pt,inner sep=1pt]  \tikzstyle{walkv}=[draw, line width=2pt]\node[inner sep=0pt, anchor=east] (1) at (0, 1) {$\{1\}\,$};\node[inner sep=0pt, anchor=east] (2) at (0, 2) {$\{1, 2\}\,$};\node[inner sep=0pt, anchor=east] (3) at (0, 3) {$\{1\}\,$};\node[inner sep=0pt, anchor=east] (4) at (0, 4) {$\{1, 2\}\,$};\draw[walkv] (1) -- ($(1.east) + (1, 0)$) node[pos=0.5,mid] {$1$};\draw[walkv] (2.east) .. controls +(0.50, 0.25) and +(0.50, -0.25) .. (3.east)    node[pos=0.5,mid] {$2$};\draw[walkv] (4) -- ($(4.east) + (1, 0)$) node[pos=0.5,mid] {$2$};\end{tikzpicture}$};
  \node (node_5) at (41.0bp,358.5bp) [draw,draw=none] {$\begin{tikzpicture}[ultra thick,yscale=0.5,  baseline={(current bounding box.east)},background rectangle/.style={ fill=gray!25, rounded corners }, show background rectangle]  \tikzstyle{mid}=[draw,thick,black,anchor=center,fill=white,                   shape=circle,minimum size=2pt,inner sep=1pt]  \tikzstyle{walkv}=[draw, line width=2pt]\node[inner sep=0pt, anchor=east] (1) at (0, 1) {$\{1\}\,$};\node[inner sep=0pt, anchor=east] (2) at (0, 2) {$\{1, 2\}\,$};\node[inner sep=0pt, anchor=east] (3) at (0, 3) {$\{1\}\,$};\node[inner sep=0pt, anchor=east] (4) at (0, 4) {$\{1, 4\}\,$};\draw[walkv] (1) -- ($(1.east) + (1, 0)$) node[pos=0.5,mid] {$1$};\draw[walkv] (2.east) .. controls +(0.50, 0.25) and +(0.50, -0.25) .. (3.east)    node[pos=0.5,mid] {$2$};\draw[walkv] (4) -- ($(4.east) + (1, 0)$) node[pos=0.5,mid] {$4$};\end{tikzpicture}$};
  \node (node_9) at (122.0bp,469.5bp) [draw,draw=none] {$\begin{tikzpicture}[ultra thick,yscale=0.5,  baseline={(current bounding box.east)},background rectangle/.style={ fill=gray!25, rounded corners }, show background rectangle]  \tikzstyle{mid}=[draw,thick,black,anchor=center,fill=white,                   shape=circle,minimum size=2pt,inner sep=1pt]  \tikzstyle{walkv}=[draw, line width=2pt]\node[inner sep=0pt, anchor=east] (1) at (0, 1) {$\{1\}\,$};\node[inner sep=0pt, anchor=east] (2) at (0, 2) {$\{1, 2\}\,$};\node[inner sep=0pt, anchor=east] (3) at (0, 3) {$\{1, 2, 3\}\,$};\node[inner sep=0pt, anchor=east] (4) at (0, 4) {$\{1, 2, 3, 4\}\,$};\draw[walkv] (1) -- ($(1.east) + (1, 0)$) node[pos=0.5,mid] {$1$};\draw[walkv] (2) -- ($(2.east) + (1, 0)$) node[pos=0.5,mid] {$2$};\draw[walkv] (3) -- ($(3.east) + (1, 0)$) node[pos=0.5,mid] {$3$};\draw[walkv] (4) -- ($(4.east) + (1, 0)$) node[pos=0.5,mid] {$4$};\end{tikzpicture}$};
  \node (node_6) at (208.0bp,247.5bp) [draw,draw=none] {$\begin{tikzpicture}[ultra thick,yscale=0.5,  baseline={(current bounding box.east)},background rectangle/.style={ fill=gray!25, rounded corners }, show background rectangle]  \tikzstyle{mid}=[draw,thick,black,anchor=center,fill=white,                   shape=circle,minimum size=2pt,inner sep=1pt]  \tikzstyle{walkv}=[draw, line width=2pt]\node[inner sep=0pt, anchor=east] (1) at (0, 1) {$\{1\}\,$};\node[inner sep=0pt, anchor=east] (2) at (0, 2) {$\emptyset\,$};\node[inner sep=0pt, anchor=east] (3) at (0, 3) {$\{3\}\,$};\node[inner sep=0pt, anchor=east] (4) at (0, 4) {$\emptyset\,$};\draw[walkv] (1.east) .. controls +(0.50, 0.25) and +(0.50, -0.25) .. (2.east)    node[pos=0.5,mid] {$1$};\draw[walkv] (3.east) .. controls +(0.50, 0.25) and +(0.50, -0.25) .. (4.east)    node[pos=0.5,mid] {$3$};\end{tikzpicture}$};
  \node (node_7) at (208.0bp,358.5bp) [draw,draw=none] {$\begin{tikzpicture}[ultra thick,yscale=0.5,  baseline={(current bounding box.east)},background rectangle/.style={ fill=gray!25, rounded corners }, show background rectangle]  \tikzstyle{mid}=[draw,thick,black,anchor=center,fill=white,                   shape=circle,minimum size=2pt,inner sep=1pt]  \tikzstyle{walkv}=[draw, line width=2pt]\node[inner sep=0pt, anchor=east] (1) at (0, 1) {$\{1\}\,$};\node[inner sep=0pt, anchor=east] (2) at (0, 2) {$\{1, 2\}\,$};\node[inner sep=0pt, anchor=east] (3) at (0, 3) {$\{1, 2, 3\}\,$};\node[inner sep=0pt, anchor=east] (4) at (0, 4) {$\{1, 2\}\,$};\draw[walkv] (1) -- ($(1.east) + (1, 0)$) node[pos=0.5,mid] {$1$};\draw[walkv] (2) -- ($(2.east) + (1, 0)$) node[pos=0.5,mid] {$2$};\draw[walkv] (3.east) .. controls +(0.50, 0.25) and +(0.50, -0.25) .. (4.east)    node[pos=0.5,mid] {$3$};\end{tikzpicture}$};
  \draw [black,->] (node_0) -- (node_1);
  \draw [black,->] (node_1) -- (node_2);
  \draw [black,->] (node_2) -- (node_8);
  \draw [black,->] (node_3) -- (node_4);
  \draw [black,->] (node_4) -- (node_5);
  \draw [black,->] (node_5) -- (node_9);
  \draw [black,->] (node_6) -- (node_7);
  \draw [black,->] (node_6) -- (node_8);
  \draw [black,->] (node_7) -- (node_9);
  \draw [black,->] (node_8) -- (node_9);
\end{tikzpicture}
\caption[The Green $\RR$ order for the Okada monoid]{
 On the left are three half diagrams: The bottom-most diagram is smaller
  (in Green $\RR$ order) than the
  middle diagram, while the top-most diagram is not. On the right is a picture of the Hasse diagram of the Green
  $\RR$ order for the Okada monoid $\Okada[4]$}.
  \label{fig:Green-order}
\end{figure}
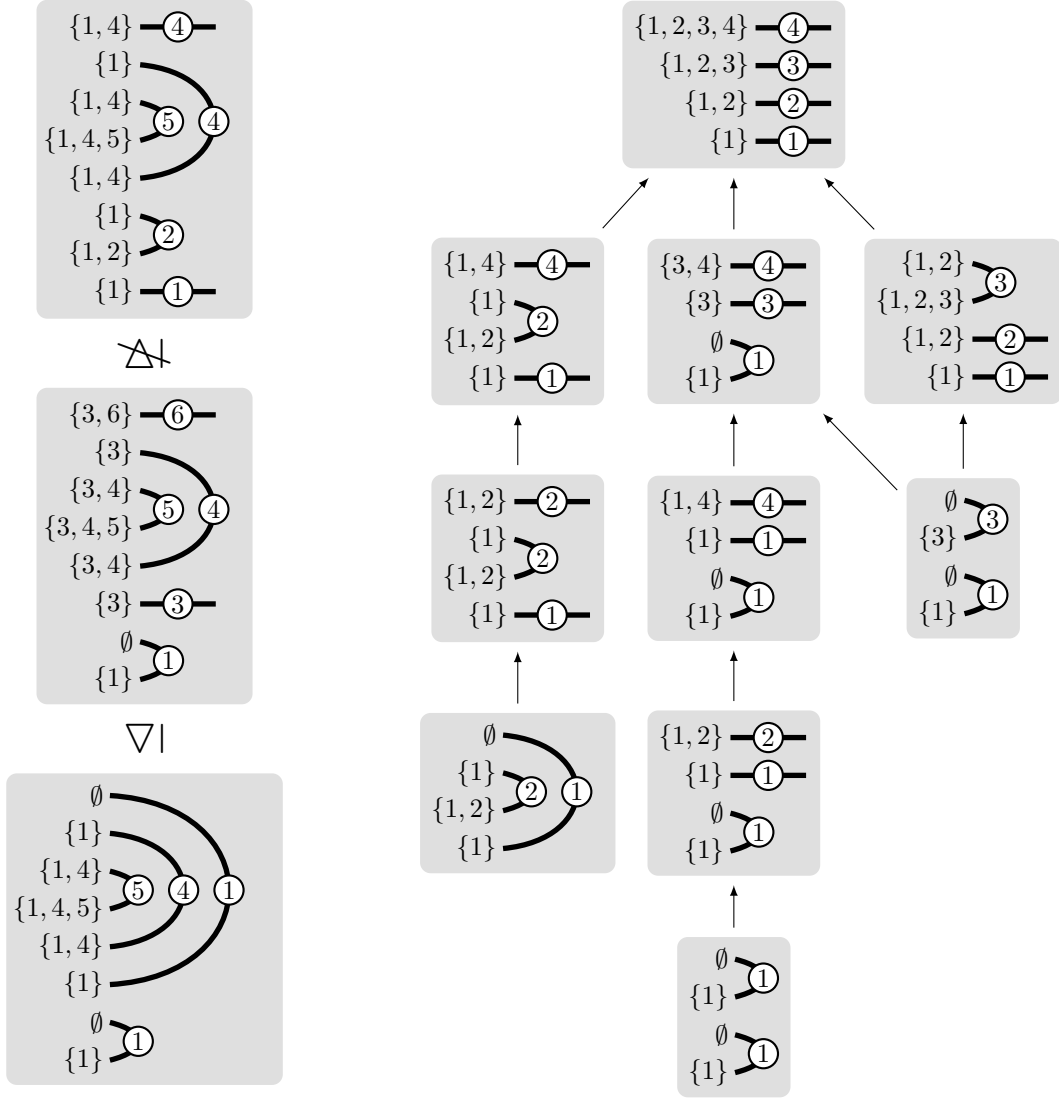
\bigskip

We now present a combinatorial description of the right Green order. See ~\cref{fig:Green-order}
for an illustration of the Hasse diagram of this order when $N=4$.
We proceed backwards: First we define
an auxiliary order on half-diagrams, we then examine its combinatorial structure, and then finish by showing that this order
on half-diagrams is isomorphic to the right Green order.
\pagebreak[2]
\begin{definition}
\label{def:breleq-order}
  Let $ \breleq $ denote the binary relation on half diagrams where  $G \breleq H$ indicates that
    \begin{enumerate}
  \item $\arc{i}{\ell}{j}$ is an arc of $G$ whenever it is an arc of $H$
  \item the arc of $G$ incident to node $i$ (which could be propagating or not) has a label
    smaller or equal than $\ell$ whenever $\harc{i}{\ell}$ is a propagating arc of $H$.
  \end{enumerate}
  We write $G \brel H$ when $G \breleq H$ and $G \ne H$.
  \end{definition}
An immediate consequence of the definition is that $G \brel H$ implies that
$\PropLab(G)\isdom\PropLab(H)$.
\begin{proposition} \label{prop:r-combi-order}
  The relation $\breleq$ defines a ranked join-semilattice such that
  $G \brel H$ is a covering relation if and only if $G$ can be obtained by
  % $G \brel H$ is a covering relation if and only if $G$ can be obtained by
  % the elements covered by a half diagram $H$ are the half-diagrams $G$ obtained by
    \begin{enumerate}
    \item \label{cover-R-glue} either gluing together two propagating arcs
       in $H$, with consecutive labels $\ell$ and $\ell+1$, to obtain a non-propagating arc labelled $\ell$
    \item\label{cover-R-label} or replacing a label $\ell$ by $\ell-2$ in some
      propagating arc $\harc{i}{\ell}$ of $H$.
    \end{enumerate}
 Futhermore, the rank of $H$ in the dual poset is given by
    \[
      \frac12\left(\operatorname{NP}(H) + \sum_{i=1}^N i - \ell(i)\right)\,.
    \]
    where $\operatorname{NP}(H)$ is the number of non-propagating arcs of $H$
    and $\ell(i)$ is the label of the arc incident to node $i$ in $H$. The maximal rank is
    $\left\lfloor\frac{N^2}4\right\rfloor$ (\cite[Sequence A002620]{OEIS}).
\end{proposition}
\begin{proof}
  Reflexivity, antisymmetry and transitivity are immediate.
  It is also clear that $F \brel H$ whenever $F$ and $H$ satisfies either
  \cref{cover-R-glue} or \ref{cover-R-label}. Conversely, suppose that
  $F\brel H$. We'll show that a half arc-diagram $G$
  can always be constructed from $H$ using \cref{cover-R-glue} or \ref{cover-R-label}
  such that  $F\breleq G \brel H$. We need a technical lemma.

\begin{lemma}
\label{lemma:r-combi-order}
 The labels of the arcs incident to node $i$ of $F$ and $H$ must
 differ for some $i \in \setN$.
 \end{lemma}

 \begin{proof}
 There are two cases to consider: Either the underlying (i.e. unlabelled)
 arc-diagrams of $F$ and $H$ are the same or different. If they're the same
 the labels must clearly differ because $F \ne H$. If the underlying diagrams differ then,
 by condition (2) of the \cref{def:breleq-order}, there must exist a pair $i < j$ such that $\arc{i}{}{j}$
 is an arc in $F$ while $\harc{i}{}$ and $\harc{j}{}$ are
 propagating arcs in $H$. The nesting condition forces the labels $\ell(i)$ and $\ell(j)$ of the
 propagating arcs $\harc{i}{}$ and $\harc{j}{}$ to be different. So either the
 labels of the arcs of $F$ and $H$ incident to node $i$ differ or else
 the labels of the arcs of $F$ and $H$ incident to node $j$ differ.
 \end{proof}

\begin{remark}
\label{remark-r-combi-order}
Let $i_0 \in \setN$ be the minimal node where the labels of $F$ and $H$ differ.
The following observations are straight forward consequences of \cref{lemma:r-combi-order}:
\begin{enumerate}
\item The restrictions $F/\setN[j]$ and $H/\setN[j]$
  agree for all $1 \leq j < i_0$.
\item The arc incident to $i_0$ in $H$ must be propagating
 with some label $\ell$.
\end{enumerate}
\end{remark}
We now return to the proof of~\cref{prop:r-combi-order}: If the arc incident to
$i_0$ in $F$ is propagating with label $k$ then, by construction of $i_0$, it
must be the case that $k \ne \ell $. Since $F \brel H$ it must also be the
case that $k \leq \ell -2 $ (parity!).  Let $G$ be the half diagram obtained
from $H$ by replacing the label $\ell$ of the propagating arc attached to
$i_0$ by $\ell -2$. Note that this new label $\ell - 2$ does not violate the
nesting conditions. Indeed, the only possible violation could arise from a propagating
arc $\harc{j}{p}$ of $H$ incident to a node $j < i_0$ with label
$p \geq \ell -2 $.  However $\harc{j}{p}$ must also be present in $F$, by part
(2) of \cref{remark-r-combi-order}, and this would be in conflict with
$\harc{i_0}{k}$ since $k < \ell -2 $. Clearly $F \breleq  G \brel H$ and $G$ is
constructed from $H$ using ~\cref{cover-R-label}.

  If the arc incident to $i_0$ in $F$ is non-propagating, let $i_1$ be its other endpoint
  and let $k$ be its label. There are two cases to consider depending on whether $i_1>i_0$ or not: If $i_1>i_0$
  then $k \leq \ell  -2$ since $F \brel H$.
  As before $G$ will be the half diagram obtained from $H$ by replacing
  the label $\ell$ of the propagating arc attached to $i_0$ by $\ell -2$. Again $G$
  will satisfy the nesting conditions since any hypothetical violation would originate from
  $H/\setN[j] = F/\setN[j]$ for some $1 \leq j < i_0$. This would
  be in conflict with the arc $\arc{i_0}{k}{i_1}$ in $F$ whose label satisfies
  $k \leq \ell -2$. Thus $F \breleq G \brel H$ and $G$ is constructed from
  $H$ using ~\cref{cover-R-label}.

  If $i_1 < i_0$ then it must be the case that the propagating arc $\harc{i_1}{k}$ belongs to $H$. Two subcases
  emerge: Either $k = \ell -1$ or $k \leq \ell -2$.

  If $k \leq \ell -2$ then
  let $G$ be the half diagram obtained from $H$ by replacing
  the label $\ell$ of the propagating arc attached to $i_0$ by $\ell -2$.
  Any nesting conflict would have to arise from a propagating arc
  $\harc{j}{p}$ of $H$ incident to a node $j < i_0$
  with label $p \geq \ell  -2 $.  By part (2) of \cref{remark-r-combi-order}
  the arc $\harc{j}{p}$ also belongs to $F$, in which case $\harc{j}{p}$ must not cross
  the arc $\arc{i_0}{k}{i_1}$. This forces $j < i_1$ and then
  $\arc{i_0}{k}{i_1}$ would have to be nested in $\harc{j}{p}$. However
  this constitutes a nesting violation for $F$ since $k \leq \ell -2 \leq p$.
  Again $G$ is a legitimate Okada half arc-diagram and
  $F \breleq G \brel H$ where $G$ is constructed from $H$ using ~\cref{cover-R-label}.

  Finally, if $k = \ell -1$ then let $G$ be the half diagram obtained from $H$ by replacing
  the pair propagating arcs $\harc{i_0}{\ell}$ and $\harc{i_1}{\ell-1}$ by
  the single arc $\arc{i_0}{\ell-1}{i_1}$. There are no nesting conflicts (by the same kind
  of reasoning used above) and
   $F\breleq  G \brel H$ where $G$ is constructed from $H$ using ~\cref{cover-R-glue}.

  \medskip

  The formula for the rank is first checked on $\rightd{\id_N}$ which is maximal with
  respect to $\breleq$ and rank zero in the dual poset. One checks that whenever $G$
  and $H$ satisfy either \cref{cover-R-glue} or \ref{cover-R-label}
  then the value of the formula for $G$ is indeed one plus the value of the formula applied to $H$.

  Clearly, the maximal rank is obtained when all the edges are labelled $1$.
  This happens when $H$ is the half diagram with
  arcs $\arc{2k-1}{1}{2k}$ for
  $k=1\dots\left\lfloor\frac{N}{2}\right\rfloor$ together with $\harc{N}{1}$ when $N$ is
  odd. Its rank is
  $\frac12(N \bmod 2 + \sum(i-1)) =\frac12(N \bmod 2 + N(N-1)/2) =
  \left\lfloor\frac{N^2}4\right\rfloor$.  \medskip

  Finally the join-semilattice property is easily verified using the (dual)
  Björner-Edelman-Ziegler criterion (see~\cite[Lemma
  9-2.5]{Reading-PosetRegionsChapter}) together with
  the description of the cover relations
  given by Conditions \ref{cover-R-glue} and \ref{cover-R-label}.
\end{proof}
Going back to the Okada monoid, we have:
\begin{theorem}\label{theo:r-order}
  Let $   \mathsf{e}, \mathsf{f} \in \Okada$. Then $ \mathsf{e} \leRR  \mathsf{f} $ if and only
  $\rightd{ \mathsf{e} }\breleq\rightd{ \mathsf{f} }$\,.
\end{theorem}
\begin{proof}
Recall that $ \mathsf{e} \leRR  \mathsf{f} $ means that $ \mathsf{e} \ssp \Okada \! \subseteq  \mathsf{f} \ssp \Okada$ which is
equivalent to $ \mathsf{e} =  \mathsf{f} \mathsf{g}  $ for some $\mathsf{g}  \in \Okada$.
If this is the case, then  $\rightd{ \mathsf{f} \mathsf{g} } \breleq\rightd{ \mathsf{f} }$
is a direct consequence of how the product of labelled non-crossing arc-diagrams
is defined.

In order to prove the implication ($\Leftarrow$) it is sufficient, in light of  \cref{prop:r-class}, to show that whenever
$\rightd{ \mathsf{e} }\breleq\rightd{ \mathsf{f} }$ is a covering relation
between the left half arc-diagrams associated to $\mathsf{e}, \mathsf{f} \in \Okada$
we can find $\mathsf{h}, \mathsf{g} \in \Okada$ such that $\rightd{ \mathsf{h} } =  \rightd{ \mathsf{f} }$
and $\rightd{ \mathsf{h} \mathsf{g} } =  \rightd{ \mathsf{e} }$\,.
We consider the two kinds of covering relations described in \cref{prop:r-combi-order}:
\begin{enumerate}
  \item In this case let $S= \PropLab(\mathsf{f})$ and
   define $\mathsf{h}\eqdef \glue{ \rightd{\mathsf{f}}}{\freerd{S}}$ and take
    $\mathsf{g}=\e{\ell}$. The two propagating arcs of $\mathsf{h}$ labelled by $\ell$ and
    $\ell+1$ have their right end points situated at $\ell$ and $\ell+1$. Multiplying on the right by $\e{\ell}$ glues those
    propagating arcs together and we obtain $\rightd{ \mathsf{h} \mathsf{g} } =  \rightd{ \mathsf{e} }$\,.
  \item In this case we define
    $\mathsf{h} \eqdef \glue{\rightd{\mathsf{f}}}{\rightd{\mathsf{f}}}$
    and
    $\mathsf{g} = \glue{\rightd{\mathsf{e}}}{\rightd{\mathsf{e}}}$\,.
    The product rule of arc-diagrams shows that
    $\rightd{ \mathsf{h} \mathsf{g} } = \rightd{ \mathsf{e} }$\,.
    \qedhere
  \end{enumerate}
\end{proof}
\begin{remark}
  There is an alternative description of the right Green order related to the
  pointwise dominance order on saturated chains. Recall that if
  ${\bf C} =C_0\lessdot\cdots\lessdot C_N$ and
  ${\bf C}' =C_0' \lessdot\cdots\lessdot C_N'$ are two
  saturated chains in the $\YFS$-lattice, then
  ${\bf C} \isdomeq_{\text{ch}} {\bf C}'$
  if $C_r \isdomeq C'_r$ for all $r$.
  Using \cref{lemma:left-eq-or-prop-le}, it is easy to see
  that for any $\mathsf{e}, \mathsf{f} \in\Okada$ and $r\leq N$,
  \begin{equation}
    \PropLab(\,\rightd{\mathsf{ef}}/\setN[r])
    \isdomeq
    \PropLab(\,\rightd{\mathsf{e}}/\setN[r])\,.
  \end{equation}
  Consequently $\Chain(\mathsf{ef}) \isdomeq_{\text{ch}} \Chain(\mathsf{e})$.
  The converse, however, is false. Nevertheless one can compute $\leRR $ from $\isdomeq_{\text{ch}}$ thanks to the
  following observation: $\mathsf{e} <_{\RR} \mathsf{f}$ is a cover if and only if
  $\Chain(\mathsf{e})\isdom_{\text{ch}}\Chain(\mathsf{f})$ is a cover and
  $\PropLab(\mathsf{e})\neq\PropLab(\mathsf{f})$. Note that this extra
  condition can be rewritten as
  $\Chain(\mathsf{e})_N\neq\Chain(\mathsf{f})_N$.

  It is worth noting that this modified dominance order on chains makes sense for
  the Jones Monoid (see~\cite{Lau2006}) as well as for the symmetric group
  (where it defines an order on standard Young tableaux). This order is different
  and strictly stronger than the chain order considered by~\cite{Taskin2006}
  since she uses the pointwise comparison of the restriction to any interval
  $[i, j]$ while we only consider the restriction to intervals of the form
  $[1, j]$.  The usual, pointwise dominance order on chains in the Young lattice
  (standard Young tableaux) also appears
  in~\cite{Thomas2019}, where the authors study its restriction to tableaux of a
  fixed shape;  to get the correct analogue of the Green order, one would
  instead retain only covering relations of the pointwise dominance order 
  between tableaux with different shapes.
\end{remark}

\subsection{Length, descents and reduced words}
We finally export our results for loop configurations to arc-diagrams:
\begin{proposition}\label{prop:length-okada}
  Let $D$ be a rank $N$ Okada arc-diagram and let $\length(D) \eqdef \mathcal{l}(\sigma(D))$ be the length of
  the corresponding permutation $\sigma(D) \in\SG[N]$ then
%  The minimal length $\length(D)$ of a word for an Okada arc-diagram $D$ is
%  half the sum over all nodes $n$ of the difference between the absolute value
%  of the index of $n$ and the label of the arc starting at $n$. Otherwise said
  \begin{equation}\label{equ:okada-arc-length}
    \length(D) = \frac{N(N+1)}{2} - \sum \ell = \frac12\sum |i| + |j| - 2\ell\,,
  \end{equation}
  where both sums are taken over all arcs $\arc{i}{\ell}{j}$ in $D$.
\end{proposition}
\begin{proof}
  Each arc-diagram $D$ is associated to a unique normal form FPLC
  $\mathcal{D}$ which is obtained from the greedy diamond diagram of
  $\sigma(D)$. By \cref{prop:loop-to-diag}, the arcs of $D$ are in bijective
  correspondance with the arcs of $\mathcal{D}$. The right hand side of
  \cref{equ:okada-arc-length} can be seen as a sum over the arcs of
  $\mathcal{D}$.

  The length of $D$ equals the number of U-turn diamonds in $\mathcal{D}$
  which is half the total number of U-turns. By~\cref{prop:red-expr-loop},
  there are no loops. So each U-turn belongs to a unique arc in
  $\mathcal{D}$. The arc $\arc{i}{\ell}{j}$ traverses $|i|-\ell + |j|-\ell$
  many U-turns. Summing over all arcs gives the formula.
\end{proof}
\begin{proposition}\label{prop:simplify-desc}
  Let $\e{}$ be an arc-diagram of rank $N$ and let $i<N$. Then $i$ is a right $\Okada$-descent of
  $\e{}$ if and only if $\e{}$ contains the arc $\arc{\ov{i}}{i}{\ov{i+1}}$. In
  this case, there is a unique arc-diagram $\f{}$ such that $\e{} = \f{} \e{i}$ and
  $\length(\e{}) = \length(\f{}) + 1$\,.
\end{proposition}
\begin{proof}
  This is a consequence of~\cref{lemma:descents-loop} and
  \cref{prop:loop-to-diag}. Uniqueness follows from ~\cref{lemma:cancel}.
\end{proof}
We now describe a way to construct $\f{}$ in~\cref{prop:simplify-desc}. Recall
that the set $[N]\cup\ov{[N]}$ is ordered as in \cref{eq:orderN-NN}. Consider
the set $C$ of elements of $[N] \cup \ov{[N]}$ that are strictly smaller that
$\ov{i+1}$ and incident to an arc whose label is smaller or equal to $i$. The
set $C$ is not empty since it contains at least $1$. Let $j$ be the largest
element of $C$ and let $\arc{j}{\ell}{k}$ be the corresponding arc. We leave it to
the reader to check that the diagram
\[
  \f{} \eqdef
  (\e{} \setminus \{\arc{\ov{i}}{i+1}{\ov{i+1}},\ \arc{j}{\ell}{k}\})
  \cup\{\arc{j}{i+1}{\ov{i+1}},\ \arc{k}{\ell}{\ov{i}}\}
\]
is an Okada arc-diagram (in particular is non-crossing thanks to the choice of
$j$), that $\e{}=\f{}\e{i}$, and that $\length(\e{}) =
\length(\f{}) + 1$\,.
\medskip

\begin{remark} \label{rem:descents+chains}
\Cref{prop:simplify-desc}
can be reformulated in terms of propagating labelling sets.
Specifically, $i \in \setN[N-1]$ is a left $\Okada$-descent of  $D$ if and only if
\[ \PropLab(\ssp \rightd{D}/\setN[i])  \, = \, \PropLab(\ssp \rightd{D}/\setN[i+1] \ssp) \cup \ssp  \{ i \}. \]
Said differently,  $i \in \setN[N-1]$ is a left $\Okada$-descent of  $D$
if and only if $S_i = S_{i+1} \cup \{i \}$ in the saturated chain
$\Chain \ssp \rightd{D} =  S_0 \lessdot \cdots \lessdot S_N$
associated to $\rightd{D}$ under the RS-correspondence.
\end{remark}

\begin{proposition}
\label{prop:Okada-Eulerian-recurrence}
Let $\Des_{\ssp \YF}(D)$ denote the number of left $\Okada$-descents of
a rank $N$ Okada arc-diagram $D$, then the
\defn{Okada-Eulerian polynomial}
$a_N(t) \eqdef \ssp \sum_{D\in\Okada} t^{\Des_{\ssp \YF}(D)}$
satisfies the recurrence
\[ a_N(t) \, = \, N a_{N-1}(t) \, + \, (N-1)(t-1) a_{N-2}(t)  \]
starting with $a_0(t) = a_1(t) = 1$.
\end{proposition}
Here are the first few values:
\begin{alignat*}{2}
  a_2(t)&= t + 1&
  a_3(t)&=5 \, t + 1\\
  a_4(t)&=3 \, t^{2} + 20 \, t + 1&
  a_5(t)&=35 \, t^{2} + 84 \, t + 1\\
  a_6(t)&=15 \, t^{3} + 295 \, t^{2} + 409 \, t + 1\qquad\qquad&
  a_7(t)&=315 \, t^{3} + 2359 \, t^{2} + 2365 \, t + 1 \\
  %a_8(t)&=105 \, t^{4} + 4480 \, t^{3} + 19670 \, t^{2} + 16064 \, t + 1
  %        \qquad\qquad&
  %a_9(t)&= 3465 \, t^{4} + 56672 \, t^{3} + 177078 \, t^{2} + 125664 \, t + 1
\end{alignat*}
\begin{proof}[Sketch of the proof]
  Let \defn{$a_{N,T}(t)$} be a refined version of the Okada-Eulerian
  polynomial which depends upon a rank $N$ Fibonacci set $T$, namely
  \begin{equation*}
    a_{N,T}(t) \, \eqdef \, \sum_{\PropLab \ssp \rightd{D} \ssp = \ssp T}
    t^{\Des_{\ssp \YF}(D)}
  \end{equation*}
  where the sum is taken over all rank $N$ Okada arc-diagrams with a fixed
  propagating set $T$. Given a saturated chain $S_0 \lessdot \cdots \lessdot S_N$
  in the $\YFS$-lattice let
  \defn{$\Des(S_0 \lessdot \cdots \lessdot S_N)$} count the number of
  indices $i \in \setN[N-1]$ for which $S_{i+1} = S_i \cup \{ i \}$. Moreover,
  given a rank $N$ Fibonacci set $S$ let's introduce a $t$-deformed version of
  its dimension, namely
  \[
    \dim_t(S) \, \eqdef \,
    \sum_{S_0 \lessdot \cdots \lessdot S_N} t^{\Des(S_0 \lessdot \cdots
      \lessdot S_N)}
  \]
  where the sum is taken over all saturated chains in the $\YFS$-lattice
  starting at $S_0 = \emptyset_\text{\bf 0}$ and ending at $S_N = S$.  Since
  left $\Okada$-descents of $D$ only depend upon the half diagram $\rightd{D}$ it
  follows from \Cref{rem:descents+chains} that $a_{N,T} = \dim_t(T) \dim(T)$.
  The proposition is a consequence of the remark that
  $\dim_t(T) = \dim_t(T \setminus \{ \! N \! \})$ whenever $\max T = N$ and
  \begin{align*}
    \dim_t(T) &= t  \dim_t \big( T \cup \{ \!  N-1 \! \} \big)  \, +
                \sum_{\stackrel{\scriptstyle S \ssp \lessdot \ssp T}{S \ssp \ne \ssp T \cup \{ \! N-1 \! \} }}  \dim_t (S) \\
              &= (t-1)  \dim_t \big( T \cup \{ \!  N-1 \! \} \big)  \, +
                \sum_{ S \ssp \lessdot \ssp T} \dim_t (S)
  \end{align*}
  otherwise.
\end{proof}

\begin{corollary}
\label{cor:tridiagonal-determinant-okada-eulerian}
The recurrence from \Cref{prop:Okada-Eulerian-recurrence} implies
that
\begin{equation}
\label{eq:tridiagonal-determinant-okada-eulerian}
 a_N(t)  \eqdef
\det
\underbrace{\begin{pmatrix}
1 & 1 & 0 &0 & \cdots\\
1-t& 2 & 2 &0  &\\
0 & 1-t & 3 &3  & \\
0 & 0 & 1-t & 4 & \\
\vdots & & & &\ddots
\end{pmatrix}}_{N \times N \, \mathrm{tridiagonal} \, \mathrm{matrix}}\,. 
\end{equation}
Furthermore, the exponential generating function of the sequence $(a_N(t))_{N \geq 0}$ is given by:
\[ \sum_{N \geq 0}\frac{a_N(t)}{N!}z^N = e^{(1-t)z} (z-1)^{-t}\,. \]
\end{corollary}
\bigskip

The following proposition will be crucial in proving cellularity of the Okada algebra:
\begin{proposition}[LJR-factorization]\label{prop:canonical-fact}
  Let $\e{} \in\Okada$ and let $\e{\JJ}$ be the free representative of its
  $\JJ$-class. There exists a unique pair $(\e{\LL} , \e{\RR}) \in\Okada$ such
  that $\e{} = \e{\LL} \cdot \e{\JJ} \cdot \e{\RR}$ is a reduced
  factorization, i.e.
  \[
    \length(\e{})=\length(\e{\LL})+\length(\e{\JJ})+\length(\e{\RR})\,.
  \]
  Moreover $ \e{\LL}\cdot \e{\JJ}$ and $\e{\JJ} \cdot \e{\RR}$ are minimal
  representatives of the respective $\RR$ and $\LL$-classes of $\e{}$.
\end{proposition}
\begin{proof}
  By \cref{JJclass-free-rep} the free representative of $\e{}$
  is $\e{\JJ} = \glue{H}{H}$ where  $H\eqdef\freeld{S}$ and $S = \PropLab(\e{})$.
  Keep in mind that $\e{\JJ}$ can also be expressed as the (commutative) product
  $\prod_{i \in F(S)} \e{i}$ where $F(S)$ is the free set associated to $S$.
  By \cref{lemma:mult-compat} we also know that $\e{}$ factorizes as
  $\e{}=\glue{\rightd{\e{}}}{H}\cdot\glue{H}{\leftd{\e{}}}$.
  However \cref{prop:length-okada} tells us that
  this factorization is not reduced: The product contains a loop
  for each $i \in F(S)$ and thus
  \[
    \length(\ssp \glue{\rightd{\e{}}}{H} \ssp)+\length(\ssp \glue{H}{\leftd{\e{}}} \ssp) = \length(\e{}) + |F(S)|
  \]

  Each $i \in F(S)$ is an $\Okada$-descent of
  $\glue{\rightd{\e{}}}{H}$ since $\glue{\rightd{\e{}}}{H} = \glue{\rightd{\e{}}}{H} \cdot \glue{H}{H}$.
  So $\e{i}$ can be factored out (on the right) from
  $\glue{\rightd{\e{}}}{H}$ for each $i \in F(S)$.
  As a consequence $\e{\JJ}$ can be factored out (on the right)
  and we obtain an element $\e{\LL}$ satisfying $\glue{\rightd{\e{}}}{H}= \e{\LL} \cdot \e{\JJ}$ and
  $\length(\ssp   \glue{\rightd{\e{}}}{H})=\length(\e{\LL}) + |F(S)|$. Similarly one can factor
  $\glue{H}{\leftd{e}}= \e{\JJ} \cdot \e{\RR}$ with
  $\length(\glue{H}{\leftd{e}})=\length(\e{\RR})+|F(S)|$. The result follows
  from the equality $\e{\JJ}^2= \e{\JJ}$.  Uniqueness follows from ~\cref{lemma:cancel}.
\end{proof}
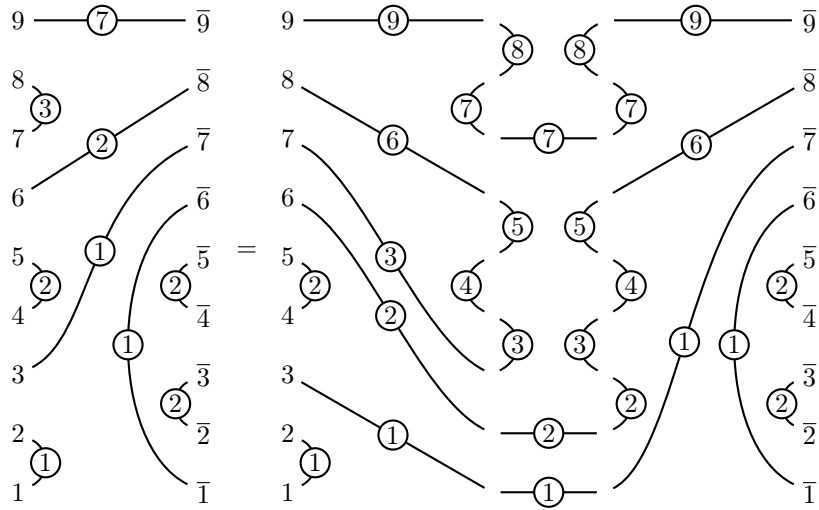
\begin{figure}[ht]
  \[\begin{tikzpicture}[scale=0.6, thick, baseline={(current bounding box.east)}]
\tikzstyle{vertex} = [shape=circle, minimum size=7pt, inner sep=1pt]
\tikzstyle{mid} = [draw, fill=white, shape=circle, minimum size=2pt, inner sep=1pt]
\node[vertex] (G--9) at (2.64, 10.40) {$\overline{9}$};
\node[vertex] (G--8) at (2.64, 9.10) {$\overline{8}$};
\node[vertex] (G--7) at (2.64, 7.80) {$\overline{7}$};
\node[vertex] (G--6) at (2.64, 6.50) {$\overline{6}$};
\node[vertex] (G--5) at (2.64, 5.20) {$\overline{5}$};
\node[vertex] (G--4) at (2.64, 3.90) {$\overline{4}$};
\node[vertex] (G--3) at (2.64, 2.60) {$\overline{3}$};
\node[vertex] (G--2) at (2.64, 1.30) {$\overline{2}$};
\node[vertex] (G--1) at (2.64, 0.00) {$\overline{1}$};
\node[vertex] (G-1) at (-1.44, 0.00) {$1$};
\node[vertex] (G-2) at (-1.44, 1.30) {$2$};
\node[vertex] (G-3) at (-1.44, 2.60) {$3$};
\node[vertex] (G-4) at (-1.44, 3.90) {$4$};
\node[vertex] (G-5) at (-1.44, 5.20) {$5$};
\node[vertex] (G-6) at (-1.44, 6.50) {$6$};
\node[vertex] (G-7) at (-1.44, 7.80) {$7$};
\node[vertex] (G-8) at (-1.44, 9.10) {$8$};
\node[vertex] (G-9) at (-1.44, 10.40) {$9$};
\draw (G--7) .. controls +(-2.59, -1.30) and +(1.63, 0.82) .. (G-3) node[pos=0.45,mid] (M-7-3) { 1 };
\draw (G--8) -- (G-6) node[pos=0.55,mid] (M-8-6) { 2 };
\draw (G-7) .. controls +(0.70, 0.35) and +(0.70, -0.35) .. (G-8) node[pos=0.5,mid] (M7-8) { 3 };
\draw (G--6) .. controls +(-2.10, -1.05) and +(-2.10, 1.05) .. (G--1) node[pos=0.5,mid] (M-6--1) { 1 };
\draw (G-4) .. controls +(0.70, 0.35) and +(0.70, -0.35) .. (G-5) node[pos=0.5,mid] (M4-5) { 2 };
\draw (G--5) .. controls +(-0.70, -0.35) and +(-0.70, 0.35) .. (G--4) node[pos=0.5,mid] (M-5--4) { 2 };
\draw (G--9) -- (G-9) node[pos=0.55,mid] (M-9-9) { 7 };
\draw (G-1) .. controls +(0.70, 0.35) and +(0.70, -0.35) .. (G-2) node[pos=0.5,mid] (M1-2) { 1 };
\draw (G--3) .. controls +(-0.70, -0.35) and +(-0.70, 0.35) .. (G--2) node[pos=0.5,mid] (M-3--2) { 2 };
\end{tikzpicture}
\ =\
\begin{tikzpicture}[scale=0.6, thick, baseline={(current bounding box.east)}]
\tikzstyle{vertex} = [shape=circle, minimum size=7pt, inner sep=1pt]
\tikzstyle{interm} = [shape=circle, minimum size=1pt, inner sep=1pt]
\tikzstyle{mid} = [draw, fill=white, shape=circle, minimum size=2pt, inner sep=1pt]
\foreach \i in {1,...,9} {
  \node[vertex] (G-A-\i) at ($(0, \i*1.3-1.3)$) {$\i$};
  \node[interm] (G-B-\i) at ($(4.5, \i*1.3-1.3)$) {$\ $};
  \node[interm] (G-C-\i) at ($(7, \i*1.3-1.3)$) {$\ $};
  \node[vertex] (G-D-\i) at ($(11.5, \i*1.3-1.3)$) {$\overline{\i}$};
}
\draw (G-B-1) -- (G-A-3) node[pos=0.5,mid] (M-1-3) { 1 };
\draw (G-B-2) .. controls +(-1.63, 0.82) and +(1.63, -0.82) .. (G-A-6) node[pos=0.5,mid] (M-2-6) { 2 };
\draw (G-B-3) .. controls +(-1.63, 0.82) and +(1.63, -0.82) .. (G-A-7) node[pos=0.5,mid] (M-3-7) { 3 };
\draw (G-A-4) .. controls +(0.70, 0.35) and +(0.70, -0.35) .. (G-A-5) node[pos=0.5,mid] (M4-5) { 2 };
\draw (G-B-5) .. controls +(-0.70, -0.35) and +(-0.70, 0.35) .. (G-B-4) node[pos=0.5,mid] (M-5--4) { 4 };
\draw (G-B-6) -- (G-A-8) node[pos=0.5,mid] (M-6-8) { 6 };
\draw (G-A-1) .. controls +(0.70, 0.35) and +(0.70, -0.35) .. (G-A-2) node[pos=0.5,mid] (M1-2) { 1 };
\draw (G-B-8) .. controls +(-0.70, -0.35) and +(-0.70, 0.35) .. (G-B-7) node[pos=0.5,mid] (M-8--7) { 7 };
\draw (G-B-9) -- (G-A-9) node[pos=0.5,mid] (M-9-9) { 9 };
%%%%%%%%
\draw (G-C-1) -- (G-B-1) node[pos=0.5,mid] (M-1-1) { 1 };
\draw (G-C-2) -- (G-B-2) node[pos=0.5,mid] (M-2-2) { 2 };
\draw (G-B-3) .. controls +(0.70, 0.35) and +(0.70, -0.35) .. (G-B-4) node[pos=0.5,mid] (M3-4) { 3 };
\draw (G-C-4) .. controls +(-0.70, -0.35) and +(-0.70, 0.35) .. (G-C-3) node[pos=0.5,mid] (M-4--3) { 3 };
\draw (G-B-5) .. controls +(0.70, 0.35) and +(0.70, -0.35) .. (G-B-6) node[pos=0.5,mid] (M5-6) { 5 };
\draw (G-C-6) .. controls +(-0.70, -0.35) and +(-0.70, 0.35) .. (G-C-5) node[pos=0.5,mid] (M-6--5) { 5 };
\draw (G-C-7) -- (G-B-7) node[pos=0.5,mid] (M-7-7) { 7 };
\draw (G-B-8) .. controls +(0.70, 0.35) and +(0.70, -0.35) .. (G-B-9) node[pos=0.5,mid] (M8-9) { 8 };
\draw (G-C-9) .. controls +(-0.70, -0.35) and +(-0.70, 0.35) .. (G-C-8) node[pos=0.5,mid] (M-9--8) { 8 };
%%%%%%%%
\draw (G-D-7) .. controls +(-2.59, -1.30) and +(1.63, 0.82) .. (G-C-1) node[pos=0.54,mid] (M-7-1) { 1 };
\draw (G-C-2) .. controls +(0.70, 0.35) and +(0.70, -0.35) .. (G-C-3) node[pos=0.5,mid] (M2-3) { 2 };
\draw (G-D-6) .. controls +(-2.10, -1.05) and +(-2.10, 1.05) .. (G-D-1) node[pos=0.5,mid] (M-6--1) { 1 };
\draw (G-C-4) .. controls +(0.70, 0.35) and +(0.70, -0.35) .. (G-C-5) node[pos=0.5,mid] (M4-5) { 4 };
\draw (G-D-5) .. controls +(-0.70, -0.35) and +(-0.70, 0.35) .. (G-D-4) node[pos=0.5,mid] (M-5--4) { 2 };
\draw (G-D-8) -- (G-C-6) node[pos=0.54,mid] (M-8-6) { 6 };
\draw (G-C-7) .. controls +(0.70, 0.35) and +(0.70, -0.35) .. (G-C-8) node[pos=0.5,mid] (M7-8) { 7 };
\draw (G-D-3) .. controls +(-0.70, -0.35) and +(-0.70, 0.35) .. (G-D-2) node[pos=0.5,mid] (M-3--2) { 2 };
\draw (G-D-9) -- (G-C-9) node[pos=0.54,mid] (M-9-9) { 9 };
\end{tikzpicture}\]
  \caption{Example of an $\mathrm{LJR}$-factorization}
  \label{fig.ljr-fact}
\end{figure}

\section{The Okada algebra as a diagram algebra}
\label{section:okada-algebra-cellular-structure}

We can now realize the Okada algebra as a diagram algebra. Recall that
$\perm{D}$ is the unique permutation such that $D=\e{\perm{D}}$.
\begin{definition}
  For any Okada arc-diagram $D\in\OkArc$ we define an element $\E{D}$ of the Okada
  algebra by $\E{D}\eqdef\E{\perm{D}}$.
\end{definition}
\begin{proposition}
\label{prop:lambda-coefficient}
  The family $(\E{D})_{D\in\OkArc}$ is a basis of the okada algebra
  $\Okada(X, Y)$ such that for any two Okada arc-diagrams $C, D$ there exists a scalar
  $\lambda(C, D)$ satsifying
  \begin{equation}
  \label{eq:lambda-coefficient}
    \E{C}\E{D} = \lambda(C,D)\,\E{C \boldsymbol{\cdot} D}\,,
  \end{equation}
  where $C \boldsymbol{\cdot} D$ is the diagram product in $\OkArc$. Moreover, if $\length(C \boldsymbol{\cdot}
  D)=\length(C)+\length(D)$ then $\lambda(C,D) = 1$.
\end{proposition}
\begin{proof}
  Recall that $D\mapsto\perm{D}$ is a bijection between diagram and
  permutations. Therefore $(\E{D})_{D\in\OkArc}$ is just a re-indexing of the
  usual basis $\E{\sigma}$ of $\Okada(X,Y)$.

  By definition, for any $D$ the basis element $\E{D}$ is the monomial
  $\E{\lexmin(\perm{D})}$, so that $\E{C}\E{D}=\E{\ind{w}}$ where $\ind{w}$ is
  the concatenation $\lexmin(\perm{C})\lexmin(\perm{D})$. On one hand
  \cref{coro-weight-basis}, tells us that
  \begin{equation}
  \label{eq:product-basis-elements-weights-canonical}
    \E{C}\E{D} = \weight\big(\Delta(\ind{w})\rewto\norf(\Delta(\ind{w})) \big)
    \,\E{\norf(\Delta(\ind{w}))}\,.
  \end{equation}
  On the other hand, the diagram $\e{\norf(\Delta(\ind{w}))}$ is known to be
  the diagram product $C \boldsymbol{\cdot} D$. This settles \cref{eq:lambda-coefficient}. Now, if
  $\length(C \boldsymbol{\cdot} D)=\length(C)+\length(D)$ the reduction
  $\Delta(\ind{w})\rewto\norf(\Delta(\ind{w}))$ preserves the number of black
  squares which means that the only possible rewriting rule which can be applied is commutation
  $(C)$ whose weight is $1$. Consequently
  $\weight \big(\Delta(\ind{w})\rewto\norf(\Delta(\ind{w})) \big)=1$.
\end{proof}
In \cref{sec:coeffprod}, we'll provide a method to compute $\lambda(C, D)$.
\bigskip

Thanks to our diagram realization, we can now prove that the Okada algebra is
cellular. Recall that, a \defn{cell datum} for an algebra $A$ is a tuple
$(\Lambda, i, \mathfrak{M}, C)$ consisting of
\begin{itemize}
\item a finite partially ordered set $(\Lambda, \leq )$.
\item a $\K$-linear anti-automorphism $i:A\to A$ with $i^2=\id_A$.
\item for each $\lambda\in\Lambda$ a non-empty finite set $\mathfrak{M}(\lambda)$ of
  indices.
\item an injective map
  \[
    C:\bigsqcup_{\lambda\in\Lambda} \mathfrak{M}(\lambda)\times \mathfrak{M}(\lambda) \to A
  \]
  The image of $(\mathfrak{s},\mathfrak{t}) \in \mathfrak{M}(\lambda) \times \mathfrak{M}(\lambda)$ under this map is denoted $C_{\mathfrak{st}}^\lambda$
and the following conditions must be satisfied:
\end{itemize}
\begin{enumerate}
\item the image of $C$ is an $R$-basis of $A$
\item $i(C_{\mathfrak{st}}^\lambda)=C_{\mathfrak{ts}}^\lambda$
\item \label{cond-cell-indep}
  for every $\lambda\in\Lambda$, $\mathfrak{s}, \mathfrak{t} \in \mathfrak{M}(\lambda)$ and
  $a \in A$,
  \[
    aC^\lambda_{\mathfrak{st}} = \sum_{\mathfrak{u} \in \mathfrak{M}(\lambda)} r_a(\mathfrak{u},\mathfrak{s})C^\lambda_{\mathfrak{ut}} \pmod {A<\lambda}
  \]
  where the scalar $r_a(\mathfrak{u}, \mathfrak{s}) \in R$ depends only on $a, \mathfrak{u}$, and $\mathfrak{s}$, but {\bf not} on $\mathfrak{t}$.
\end{enumerate}
\begin{theorem}\label{theo-cellular}
  The Okada algebra $\Okada(X, Y)$ is cellular with the following data:
  \begin{enumerate}
  \item\label{cond-cell-basis} The cell-poset is $(\Lambda, \leq ) =(\YFS_N , \isdomeq)$.
  \item\label{cond-cell-invol} The involution is the mirror involution
    $i:\E{D}\mapsto \E{D^\star}$.
  \item\label{cond-cell-invar} For $S \in \YFS_N$ let $\mathfrak{M}(S)$
    be the set of Okada half arc-diagrams $H$ such that $\PropLab(H) = S$.
  \item The cell element associated to $L, R \in M(S)$ is
    $C^\lambda_{LR}\eqdef\E{\glue{L}{R}}$.
  \end{enumerate}
\end{theorem}
\begin{proof}
  We already proved \cref{cond-cell-basis,cond-cell-invol}.  Suppose that
  $L_1,R_1,L_2,R_2$ are half arc-diagrams all sharing the same propagating labelling set $S$.
  Suppose also that the propagating labelling set of the
  product $D \eqdef\glue{L_1}{R_1} \boldsymbol{\cdot} \glue{L_2}{R_2}$ 
  is $S$. This forces $D =\glue{L_1}{R_2}$. One needs to
  prove that the scalar $\lambda$ appearing in
  $\E{\glue{L_1}{R_1}}\E{\glue{L_2}{R_2}} = \lambda \E{D}$
  depends only on $R_1$ and $L_2$. We claim that $\lambda = \mu$ where
  $\f{} = \mu  \ssp \E{\, \glue{\freeld{S}}{\ssp \freeld{S}}}$
  and $\f{}  \eqdef\E{\, \glue{\freeld{S}}{R_1}}\E{\glue{L_2}{\ssp \freeld{S}}}$

   Factor the elements
  $\glue{L_i}{R_i} = \mathsf{l}_i \boldsymbol{\cdot} \mathsf{j}_i   \boldsymbol{\cdot} \mathsf{r}_i$
  using \cref{prop:canonical-fact}.
  Since these products are reduced
  in the monoid, they are also reduced in the algebra so that
  \[
    \E{\glue{L_i}{R_i}} = \E{\mathsf{l}_i}\E{\mathsf{j}_i}\E{\mathsf{r}_i}\,.
  \]
  We know that
  \[
    \glue{\freeld{S}}{R_i}= \mathsf{j}_i \boldsymbol{\cdot} \mathsf{r}_i
    \qandq
    \glue{L_i}{\freerd{S}}= \mathsf{l}_i \boldsymbol{\cdot} \mathsf{j}_i\,,
  \]
  and thus
  \[
    \E{\ssp \glue{\freeld{S}}{R_i}}= \E{\mathsf{j}_i} \ssp \E{\mathsf{r}_i}
    \qandq
    \E{\glue{L_i}{\ssp \freerd{S}}}= \E{\mathsf{l}_i} \ssp \E{\mathsf{j}_i}\,.
  \]
  Now
  \begin{alignat*}{2}
    \E{\glue{L_1}{R_1}}\E{\glue{L_2}{R_2}}
    &=\E{\mathsf{l}_1}\E{\mathsf{j}_1}\E{\mathsf{r}_1}\E{\mathsf{l}_2}\E{\mathsf{j}_2}\E{\mathsf{r}_2}
    =\E{\mathsf{l}_1}\E{\mathsf{j}_1 \boldsymbol{\cdot} \mathsf{r}_1}\E{\mathsf{l}_2 \boldsymbol{\cdot} \mathsf{j}_2} \E{\mathsf{r}_2} \\
    &=\E{\mathsf{l}_1}\E{\ssp \glue{\freeld{S}}{R_1}}\E{\glue{L_2}{\ssp \freerd{S}}} \ssp \E{\mathsf{r}_2}
    =\E{\mathsf{l}_1} \f{} \ssp \E{\mathsf{r}_2} =  \mu \ssp \E{\mathsf{l}_1}\E{\ssp \glue{\freeld{S}}{\ssp \freerd{S}}} \ssp \E{\mathsf{r}_2} 
  \end{alignat*}
  However we know that the product $\mathsf{l}_1 \boldsymbol{\cdot} \glue{\freeld{S}}{\freerd{S}} \boldsymbol{\cdot} \mathsf{r}_2$ is
  reduced since it is the canonical factorization $\glue{L_1}{R_2} = D$.
\end{proof}

\subsection{Restriction and induction of cell modules}\label{sec:cell-module}
In this section we use the combinatorics of arc-diagrams to describe the
structure of cell modules and to recover and extend Okada's
results expressing restriction and induction of
simple $\Okada[N](X, Y)$-modules
in terms of the covering relations of the Young-Fibonacci lattice.

We fix two infinite sequences of generic scalars $X\eqdef(x_1,x_2,\dots)$ and
$Y\eqdef(y_1,y_2,\dots)$ and let
$\Okada(X,Y) \eqdef \Okada(x_1,\dots,x_{N-1},y_1,\dots,y_{N-2})$
for all $N$. The standard inclusion
makes $\Okada[N](X,Y)$  into a subalgebra of
$\Okada[N+1](X,Y)$ and leads to an
increasing sequence of algebras:
\[
  \mathbb{K}=\Okada[0](X,Y)\subset\Okada[1](X,Y)\subset\Okada[2](X,Y)\subset\Okada[3](X,Y)\subset\Okada[4](X,Y)\subset\cdots
\]

The left $\Okada(X,Y)$ \defn{cell module} $V^S$ associated to
$S \in \YFS_N$ can be realized by the vector space spanned by rank~$N$
half Okada arc-diagrams $\rightd{D}$ such that $\PropLab\rightd{D}=S$
equipped with the left action $\bullet$ obtained by extending linearly
the formula
\begin{equation}
  \label{eq.cell-module-action}
  \E{C}  \bullet \rightd{D} \ \eqdef
  \begin{cases}
    \lambda(C,D) \ssp \rightd{C \boldsymbol{\cdot} D}
    &\text{if $\PropLab(C \boldsymbol{\cdot} D) = S$} \\
    0
    &\text{otherwise}
  \end{cases}
\end{equation}
where $D \in \Okada(X,Y)$. See~\cref{fig.cell211}
for an example. According to~\cite{Okada1994}, the module $V^S$ is simple provided
the $X,Y$ parameters are sufficiently generic~(see \cref{thm:okada-semisimple-parameters}).
\bigskip
\begin{remark}\label{rem-cell-realisation-as-full-diags}
  The result of \cref{eq.cell-module-action} depends only on the half diagram
  $\rightd{D}$ and not on the choice of representative full diagram
  $D$.  In particular, given $S$ and a fixed half arc-diagram $L$ with
  $\PropLab(L)=S$, the module $V^S$ can be realized as the span 
  half arc-diagrams $\rightd{D}$ where $D$ varies over
  the set of full arc-diagram $D$ with $\leftd{D}=L$.
  \end{remark}
We now turn to the branching rule starting with a simple combinatorial lemma:
\begin{lemma}\label{lem:YF-pred_totally-ordered}
  Let $S$ be a Fibonacci set of rank $N$. The set of predecessors of $S$ in
  $\YF_{N-1}$ is totally ordered by the dominance order. Similarly the set of
  successors of $S$ in $\YF_{N+1}$ is totally ordered by the dominance order.
\end{lemma}
For example, the Fibonacci set $\{3, 4\}_8$ covers the three sets
$\{3\}_7 \isdom \{3, 4, 5\}_7 \isdom \{3, 4, 7\}_7$
\begin{proof}
  If $S$ is empty (and thus $N$ is even), then the predecessors of $S$ are the
  singletons $\{i\}$ with $1\leq i <N$ odd. Two singletons satisfy
  the dominance order 
  $\{i\}\isdom\{j\}$ if and only if $i < j$. When $S$ is not empty write $S=\{s_1<\dots<s_\ell\}$. Then
  the predecessors of $S$ are $S\setminus\{s_{\ell}\}$ and $S\cup\{s_\ell+k\}$
  where $k$ odd and $s_\ell+k<N$. These predecessors are totally ordered as
  well. The proof for successors of $S$ is similar.
\end{proof}
\begin{proposition}\label{prop:restr-okada}
  Fix $N>0$ and $T$ a Young-Fibonacci set of rank $N$. Let
  $S_1\isdom S_2\isdom\cdots\isdom S_k$ be the list of predecessors of
  $T$ in $\YFSn[N-1]$, sorted by dominance order. For $i=1,\dots,k$, define
  $V^T_{\leq i}$ to be the span of the rank $N$ half arc-diagrams
  $H$ such that $\PropLab(H/\setN[N-1]) \isdomeq S_i$. We also set
  $V^T_{\leq 0}$ to be the zero vector space. Then the $V^T_{\leq i}$ form an
  increasing sequence of $\Okada[N-1](X,Y)$-submodules
  \[ 0 = V^T_{\leq 0}\subset V^T_{\leq 1}\subset\dots\subset V^T_{\leq k} =
    V^T\] such that $V^T_{\leq i}/V^T_{\leq i-1}\isom V^{S_i}$ for any
  $i=1,\dots,k$. In particular, if these modules are simple given the choice of
  parameters $X, Y$ defining $\Okada[N-1](X,Y)$, the following isomorphism
  holds:
  \begin{equation}\label{eq.restriction}
    \operatorname{Res}_{\ssp\Okada[N-1](X,Y)}^{\ssp\Okada[N](X,Y)}(V^T)
    \ssp \isom \ssp \bigoplus_{i=1}^k V^{S_i}
    \ssp \isom \ssp \bigoplus_{S\lessdot T} V^S\,.
  \end{equation}
  where $S\lessdot T$ means that $S$ is covered by $T$ in $\YFS$.
\end{proposition}
\begin{proof}
  The fact that $V^T_{\leq i}$ is submodule comes from~\cref{lemma:left-eq-or-prop-le}.
\end{proof}
See~\cref{fig.cell211} for an illustration of~\cref{eq.restriction}.
In the semi-simple case we recover Okada's results (see in particular
\cite[Lemma 2.4]{Okada1994}).
\begin{figure}[ht]
  \includegraphics[width=0.95\textwidth]{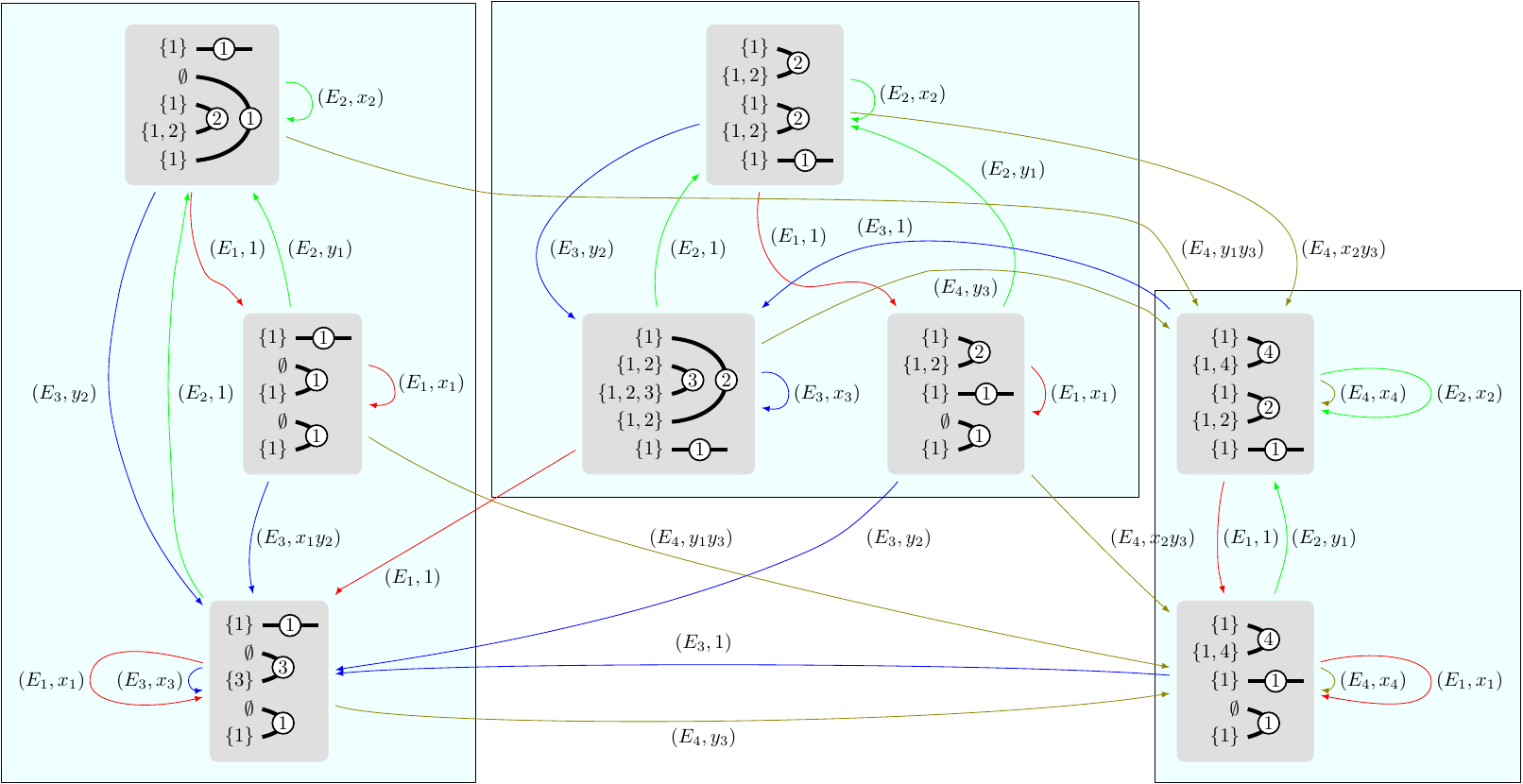}
  \caption[Restriction of a cell module]{The cell module associated to
    $S=\{1\}_5$ and its restriction to $\Okada[4]$. An edge $(H \mapsto K)$
    labelled by $(\E{i}, \lambda)$ means that $\E{i} \bullet H = \lambda  K
    $. The blue boxes show the irreducible components. Ignoring the gray arrow
    ($\E{4}$ generators), we see the following filtration:\\
    ${~}\qquad\quad
    \operatorname{Res}_{\Okada[3](X,Y)}^{\Okada[4](X,Y)}\left(V_{\{1\}_5}\right) =
    V_{\varnothing_4}\oplus V_{\{1,2\}_4}\oplus V_{\{1,4\}_4} \supset
    V_{\varnothing_4}\oplus V_{\{1,2\}_4} \supset
    V_{\varnothing_4}\,.$
  }
  \label{fig.cell211}
\end{figure}
Let $S$ be a Fibonacci set of rank $N$. By definition of the
Young-Fibonacci lattice, the successors of $S$ are the predecessors of
$S^\sharp$ (the same set as $S$ but with rank $N+2$). This is reflected 
by the following isomorphism:
\begin{proposition}
  \newcommand{\IndVS}{\widehat{V^S}}%
  Fix some $N\geq0$ and $S$ a Young-Fibonacci set of rank $N$ and let $V^S$ be
  the associated $\Okada(X,Y)$-cell module. One has the following
  isomorphim of $\Okada[N+1](X,Y)$-modules
  \[
    \operatorname{Ind}_{\ssp\Okada[N](X,Y)}^{\ssp\Okada[N+1](X,Y)}\left(V^S\right)
    \ssp \isom \ssp 
    \operatorname{Res}_{\ssp\Okada[N+1](X,Y)}^{\ssp\Okada[N+2](X,Y)}
    \left(V^{S^\sharp}\right)
  \]
  where $S^\sharp$ is the same set as $S$ but with rank $N+2$. In
  particular induced modules have the structure described in
  \cref{prop:restr-okada}.
\end{proposition}
\begin{proof}
  Fix a Fibonacci set $S$ of rank $N$ together with an arbitrary half
  diagram $L$ such that $\PropLab{L}=S$, for example
  $L=\freerd{S}$. Acording to
  \cref{rem-cell-realisation-as-full-diags}, the cell module $V^S$ can
  be realized as the span of the arc-diagrams of the form
  $\glue{K}{L}$ with the natural action. By
  \cref{eq:right-coset-decomposition}, a basis of
  $
  \operatorname{Ind}_{\ssp\Okada[N](X,Y)}^{\ssp\Okada[N+1](X,Y)}\left(V^S\right)$
  consists of tensors of the form
  \[
    \E{\ell}\cdots\E{N} \otimes \glue{K}{L}
  \]
  where $K$ is any half diagram of with $\PropLab{K}=S$ and
  $1\leq\ell\leq N+1$ with the proviso that $\E{\ell}\cdots\E{N}=1$
  when $\ell=N+1$. Moreover the map
  \[
    \E{\ell}\cdots\E{N} \otimes \glue{K}{L}
    \ \longmapsto\
    \A{\ell} \ssp \boldsymbol{\cdots} \A{N} \boldsymbol{\cdot} \iota_N(\glue{K}{L})
  \]
  extends linearly to an $\Okada[N+1](X, Y)$-module isomorphism from
  $
  \operatorname{Ind}_{\ssp\Okada[N](X,Y)}^{\ssp\Okada[N+1](X,Y)}\left(V^S\right)$
  to the span of the set $\mathcal{S}$ consisting of rank $N+1$
  diagrams $D$ for which $\leftd{D}/\setN[N] = L$.  Furthermore, the
  map $D \mapsto \rightd{D}$ is a bijection between $\mathcal{S}$ and
  half diagrams $H$ such that $S\lessdot \PropLab(H)$.  The successors
  of $S$ in $\YF$ are precisely the predecessors on $S^\sharp$, so
  there is a unique way to complete $\rightd{D}$ to a rank $N+2$ half
  diagram $\rightd{D}^\sharp$ whose propagating labels is
  $S^\sharp$. It is easy to see that the bijection
  \[
    \E{\ell}\cdots\E{N} \otimes \glue{K}{L}
    \ \longmapsto\
    \rightd{\A{\ell} \ssp \boldsymbol{\cdots} \A{N} \boldsymbol{\cdot} \iota_N(\glue{K}{L})}^\sharp
  \]
  extends linearly to an $\Okada[N+1](X,Y)$-module morphism.
\end{proof}

One can summarize the entire discussion of this subsection as follows:
\begin{theorem}\label{theo:coherent-celullar}
  The sequence $\left(\Okada[N](X,Y)\right)_{N\geq0}$ is a coherent tower of
  cellular algebras in the sense of~\cite[Definition 2.15]{Goodman2011}.
\end{theorem}

\subsection{Computation of the coefficient of a product in the Okada Algebra}
\label{sec:coeffprod}
\newcommand{\U}{\textup{\texttt{U}}} %
\newcommand{\D}{\textup{\texttt{D}}} %

The goal of this section is to provide a combinatorial rule which allows us to
compute the coefficients of products of basis elements in the Okada
algebra $\Okada(X, Y)$. Until now, we've relied on weights of rewriting
sequences of diamond diagrams and their
associated fully-packed loop configurations. Here we describe a shortcut where we only
need to record up and down steps in the paths/loops occurring in products of arc-diagrams.
\smallskip

First, we recall how to compute the product of
$\E{C} \E{D}$ of two Okada arc-diagrams $C$ and $D$
using their associated, normal-form FPLCs 
$\mathcal{C}$ and $\mathcal{D}$.
Recall that $\mathcal{C}$ and $\mathcal{D}$ are
the unique, normal form FPLCs for which
$[\mathcal{C}]=C$ and $[\mathcal{D}]=D$.
Following the
prescription detailed in \Cref{rem:fusion-FPLCs}, create a new FPLC
$\mathcal{B} \eqdef \mathcal{C} \circ \mathcal{D}$ by gluing $\mathcal{C}$
and $\mathcal{D}$ together and filling the triangular region between
them with horizontal arc-segments; see \cref{fig.prod-arc-fplc} for an
example. Its normal form $\norf(\mathcal{B})$ satisfies
$[\norf(\mathcal{B})] = C \boldsymbol{\cdot} D$ and
\begin{equation}
  \E{C} \E{D} = \weight\big( \mathcal{B} \rewto \norf(\mathcal{B}) \big)\,\E{C \boldsymbol{\cdot} D}\,.
\end{equation}
We'll show that $\weight(\mathcal{B} \rewto \norf(\mathcal{B}))$ can be computed by looking at
paired steps in lattice walks which record the up and down moves taken when traveling
along paths in the loop diagram $\mathcal{C} \circ \mathcal{D}$. Moreover, these walks can be computed directly from the Okada
arc-diagrams $C$ and $D$.

\subsubsection{Lattice walks}

We consider lattice walks on the integer grid from $(0,a)$ to $(k,b)$
consisting only of steps of the form $\U=(+1,+1)$ and $\D=(+1,-1)$ and staying
within the vertical band bounded by $y=1$ and $y=N$. A walk of this kind is
encoded by a triple $\mathbb{w}=(a, \boldsymbol{w}, b)$ where $\boldsymbol{w}$ is a word in
the symbols $\{\U, \D\}$. The notation is redundant since the equality
$b = a + |\boldsymbol{w}|_{\U} -|\boldsymbol{w}|_{\D}$ must clearly hold, where $|\boldsymbol{w}|_\U$ and $|\boldsymbol{w}|_\D$
denote the number of occurrence of the letter $\U$ and $\D$ in $\boldsymbol{w}$. It is
customary in the context of Dyck paths to replace the $\U$ by an opening
parenthesis symbol $"("$ and $\D$ by a closing parenthesis symbol $")"$. In
this way we obtain a word in the alphabet of parenthesis $\{ ( \ssp , ) \}$
which supports a natural notion of pairing. We transfer this notion onto $\boldsymbol{w}$
itself:
\begin{definition}
  Let $\mathbb{w}=(a, \boldsymbol{w}, b)$ be a walk. A $\U$-step in $w$ \defn{occurs} at
  $p \in \{0, \dots, k \}$ if there is a step from $(p, h)$ to $(p+1, h+1)$
  along the walk $\mathbb{w}$. We denote $\height_\mathbb{w}(p)\eqdef h$. We say $p$ is \defn{paired
    with $q$} if there is a $\D$-step from $(q-1, h+1)$ to $(q, h)$ with
  $q > p$ and $q$ is minimal with that property. We emphasize that $p$ and
  $h$ are the coordinates of the \emph{left end-point} of the corresponding
  $\U$-step while $q$ is the \emph{right end-point} of its $\D$-step. Clearly
  $q-p$ is even, and so \defn{$\lpair_w(p)\eqdef\frac{q - p}{2}$} is an integer.

  Some $\U$ and $\D$-steps taken along a walk $\mathbb{w}=(a, \boldsymbol{w}, b)$ may
  be unpaired.  We denote by \defn{$\unpair(\mathbb{w})$} the walk obtained by removing
  the paired steps. It must be of the form $(a, \D^k\U^\ell, b)$.
\end{definition}
When $\mathbb{w}$ is obvious from the context, we suppress it and simply write \defn{$\lpair(p)$} and \defn{$\height(p)$}.
\begin{example}
  The following picture shows the walk
  $\mathbb{w}= (2, \D\U\D\D\U\U\U\D\D\U\U\U\D\D\U\D\D\U\U,3)$.
  \[
    \begin{tikzpicture}[scale=0.5]
      \draw[dotted] (0, 0) grid (19, 4);
      \foreach \x in {0,...,19} {
        \node [below] at (\x, 0) {$\x$};
      }
      \foreach \y in {0,...,4} {
        \node [left] at (0,\y) {$\y$};
      }
      \draw[rounded corners=1, color=black, line width=2]
      (0, 2) -- (1, 1) -- (2, 2) -- (3, 1) -- (4, 0) -- (5, 1) -- (6, 2)
      -- (7, 3) -- (8, 2) -- (9, 1) -- (10, 2) -- (11, 3) -- (12, 4)
      -- (13, 3) -- (14, 2) -- (15, 3) -- (16, 2) -- (17, 1) -- (18, 2)
      -- (19, 3);
      \draw[rounded corners=1, color=red, line width=2] (0, 2) -- (1, 1);
      \draw[rounded corners=1, color=red, line width=2] (3, 1) -- (4, 0) -- (5, 1);
      \draw[rounded corners=1, color=red, line width=2]
       (17, 1) -- (18, 2) -- (19, 3);
    \end{tikzpicture}
  \]
  Notice the $\U$-step occurring at $p=5$ is a paired with the $\D$-step at $q=9$ and
  so $\height(5)=1$ and $\lpair(5)=2$. Similarly $p=9$ is paired with $q=17$ and
  so $\height(9)=1$ and $\lpair(9)=4$. The unpaired occurrences of $\U$
  are $4, 17$ and $18$. The unpaired occurrences of $\D$ are $1$ and $4$
  (recall that $\D$ steps are indexed by their endpoints). The unpaired walk is
  $\unpair(\mathbb{w})=(2,DDUUU,3)$ and is depicted in red.
\end{example}

\subsubsection{Walks associated to FPLCs}

Consider an arbitrary FPLC $\mathcal{D}$. We begin by choosing
an origin on each path or loop
$\pi$ in $\mathcal{D}$ (the precise choice in not
very important, provided loops originate at their bottom row). For a path
we choose the leftmost and lowest endpoint. For a loop, we choose the
leftmost point in the lowest row. Starting
from the origin, we traverse (counter-clockwise) each path (or loop) $\pi$ and build a
corresponding lattice walk which mirrors how $\pi$ moves up and down.
The walk starts at $(0, a)$ where $a$
is the height of the row containing the origin and proceeds according to the
following local rules:
\begin{itemize}
\item take no step if $\pi$ moves horizontally within a square;
\item take a $\U$-step if $\pi$ makes a U-turn going up;
\item take a $\D$-step if $\pi$ makes a U-turn going down.
\end{itemize}
We continue this procedure until we reach the other end-point if $\pi$ is a path or we
return to the origin if $\pi$ is a loop. This algorithm generates the
\defn{lattice walk associated to $\pi$}.
\begin{definition}
  Given a loop configuration $\mathcal{D}$, let \defn{$\walks(\mathcal{D})$}
  denote the set of walks associated to its paths and loops; let
  \defn{$\loops(\mathcal{D})$} denote the subset of the walks associated to
  loops.
\end{definition}
See \cref{fig.walks} for an example.
\begin{remark}
  By \cref{lemma-FPLC-height-invariance},
  if $\mathbb{w} = (a, \boldsymbol{w}, b)$ is a walk associated to a path in
  $\mathcal{D}$, then $\unpair(\mathbb{w})$ is the walk associated to the corresponding
  unique path starting at $a$ and ending at $\ov{b}$ in the normal form FPLC
  $\norf(\mathcal{D})$.
\end{remark}
\medskip

We are now in a position to state the two main results of this section:
\begin{theorem}   \label{theorem-coeff-mult}
  Let $\mathcal{D}$ be an arbitrary fully packed loop configuration. Then the
  weight of any rewriting sequence $\mathcal{D} \rewto
  \norf(\mathcal{D})$ equals to $R(\mathcal{D})$ where
  \begin{equation}\label{equ:weight-rewriting-loops}
    R(\mathcal{D}) \eqdef
    \left(\prod_{(a, \boldsymbol{w}, a)\, \in\, \loops(\mathcal{D})} x_a\right)
    \left(
      \prod_{\mathbb{w}\, \in\, \walks(\mathcal{D})}
      \prod_{p} y_{\height_\mathbb{w}(p)}\right)\,,
  \end{equation}
  The last product is taken over all $p \in \{1, \dots, k \}$
  for which there is an occurrence of $\U$ at $p$
  such that $p$ is paired and $\lpair_\mathbb{w}(p)$ is even.
\end{theorem}
\begin{example}
  \cref{fig.walks} details the computation of $R(\mathcal{D})$. The
  calculation of the factor associated to the single red path from $1$ to
  $\ov3$ is illustrated in~\cref{fig.walk-ref}.
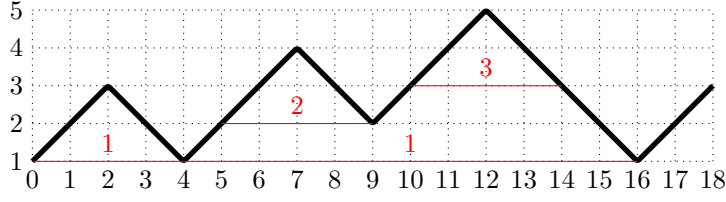
\begin{figure}[ht]
  \[
    \begin{tikzpicture}[scale=0.5]
      \draw[dotted] (0, 0) grid (18, 4);
      \foreach \x in {0,...,18} {
        \node [below] at (\x, 0) {$\x$};
      }
      \foreach \y in {1,...,5} {
        \node [left] at ($(0,\y) + (0, -1)$) {$\y$};
      }
      \draw[rounded corners=1, color=black, line width=2] (0, 0) -- (1, 1) -- (2,
      2) -- (3, 1) -- (4, 0) -- (5, 1) -- (6, 2) -- (7, 3) -- (8, 2) -- (9, 1) --
      (10, 2) -- (11, 3) -- (12, 4) -- (13, 3) -- (14, 2) -- (15, 1) -- (16, 0) --
      (17, 1) -- (18, 2);
      \draw[red] (0, 0) -- (4, 0) node[pos=0.5, above] {$1$};
      \draw[red] (5, 1) -- (9, 1) node[pos=0.5, above] {$2$};
      \draw[red] (10, 2) -- (14, 2) node[pos=0.5, above] {$3$};
      \draw[red] (4, 0) -- (16, 0) node[pos=0.5, above] {$1$};
    \end{tikzpicture}
  \]
  \caption[Lattice walk and its associated factor]{Lattice walk $(1,\U\U\D\D\U\U\U\D\D\U\U\U\D\D\D\D\U\U, 3)$. The
    horizontal lines indicates paired steps $p$ for which $\lpair(p)$ is
    even. The number above each line is $\height(p)$. The factor associated to
    this walk is therefore $y_1^2y_2y_3$.}
\label{fig.walk-ref}
\end{figure}
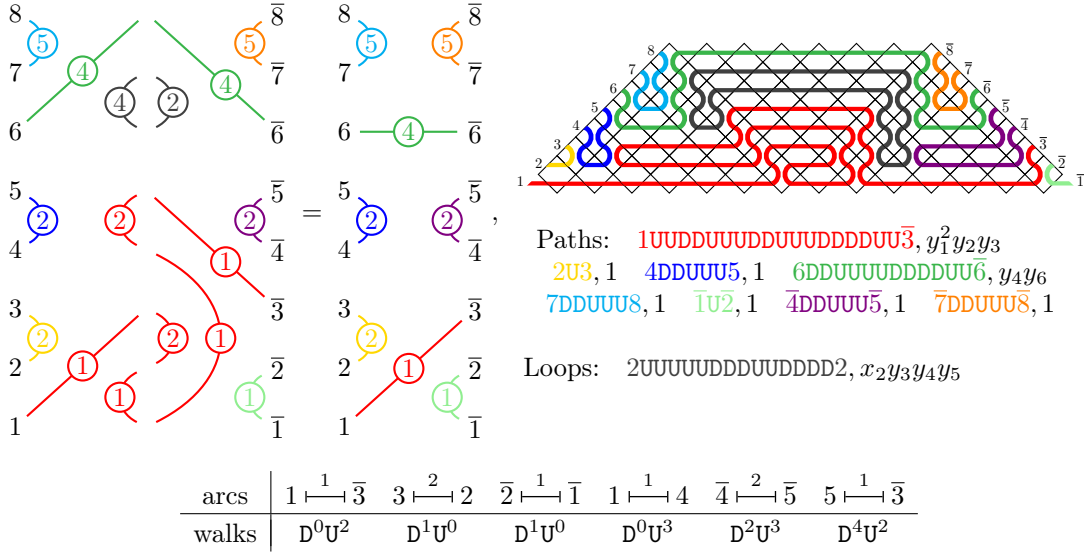
\begin{figure}[ht]
\[
% sage: e = OkadaMonoid(8).from_heights_pair((1,2,0,2,0,4,5,0), (1,0,1,2,0,4,0,4))
% sage: f = OkadaMonoid(8).from_heights_pair((1,2,0,0,1,2,0,4), (1,0,1,2,0,4,5,0))
% sage: d = LCO(e).join(LCO(f))
% sage: d.walk_all(True)
  \begin{tikzpicture}[scale=0.6, thick, baseline={(current bounding box.east)}]
\tikzstyle{vertex} = [shape=circle, minimum size=7pt, inner sep=1pt]
\tikzstyle{mid} = [draw, fill=white, shape=circle, minimum size=2pt, inner sep=1pt]
\node[vertex] (G--8) at (1.44, 9.10) {};
\node[vertex] (G--7) at (1.44, 7.80) {};
\node[vertex] (G--6) at (1.44, 6.50) {};
\node[vertex] (G--5) at (1.44, 5.20) {};
\node[vertex] (G--4) at (1.44, 3.90) {};
\node[vertex] (G--3) at (1.44, 2.60) {};
\node[vertex] (G--2) at (1.44, 1.30) {};
\node[vertex] (G--1) at (1.44, 0.00) {};
\node[vertex] (G-1) at (-1.44, 0.00) {$1$};
\node[vertex] (G-2) at (-1.44, 1.30) {$2$};
\node[vertex] (G-3) at (-1.44, 2.60) {$3$};
\node[vertex] (G-4) at (-1.44, 3.90) {$4$};
\node[vertex] (G-5) at (-1.44, 5.20) {$5$};
\node[vertex] (G-6) at (-1.44, 6.50) {$6$};
\node[vertex] (G-7) at (-1.44, 7.80) {$7$};
\node[vertex] (G-8) at (-1.44, 9.10) {$8$};
\draw[red] (G--3) -- (G-1) node[pos=0.5,mid] (M-3-1) { 1 };
\draw[color=blue] (G-4) .. controls +(0.70, 0.35) and +(0.70, -0.35) .. (G-5) node[pos=0.5,mid] (M4-5) { 2 };
\draw[red] (G--5) .. controls +(-0.70, -0.35) and +(-0.70, 0.35) .. (G--4) node[pos=0.5,mid] (M-5--4) { 2 };
\draw[green] (G--8) -- (G-6) node[pos=0.5,mid] (M-8-6) { 4 };
\draw[cyan] (G-7) .. controls +(0.70, 0.35) and +(0.70, -0.35) .. (G-8) node[pos=0.5,mid] (M7-8) { 5 };
\draw[darkgray] (G--7) .. controls +(-0.70, -0.35) and +(-0.70, 0.35) .. (G--6) node[pos=0.5,mid] (M-7--6) { 4 };
\draw[Gold] (G-2) .. controls +(0.70, 0.35) and +(0.70, -0.35) .. (G-3) node[pos=0.5,mid] (M2-3) { 2 };
\draw[red] (G--2) .. controls +(-0.70, -0.35) and +(-0.70, 0.35) .. (G--1) node[pos=0.5,mid] (M-2--1) { 1 };
\end{tikzpicture}\!\!\!\!
\begin{tikzpicture}[scale=0.6, thick, baseline={(current bounding box.east)}]
\tikzstyle{vertex} = [shape=circle, minimum size=7pt, inner sep=1pt]
\tikzstyle{mid} = [draw, fill=white, shape=circle, minimum size=2pt, inner sep=1pt]
\node[vertex] (G--8) at (1.44, 9.10) {$\overline{8}$};
\node[vertex] (G--7) at (1.44, 7.80) {$\overline{7}$};
\node[vertex] (G--6) at (1.44, 6.50) {$\overline{6}$};
\node[vertex] (G--5) at (1.44, 5.20) {$\overline{5}$};
\node[vertex] (G--4) at (1.44, 3.90) {$\overline{4}$};
\node[vertex] (G--3) at (1.44, 2.60) {$\overline{3}$};
\node[vertex] (G--2) at (1.44, 1.30) {$\overline{2}$};
\node[vertex] (G--1) at (1.44, 0.00) {$\overline{1}$};
\node[vertex] (G-1) at (-1.44, 0.00) {};
\node[vertex] (G-2) at (-1.44, 1.30) {};
\node[vertex] (G-3) at (-1.44, 2.60) {};
\node[vertex] (G-4) at (-1.44, 3.90) {};
\node[vertex] (G-5) at (-1.44, 5.20) {};
\node[vertex] (G-6) at (-1.44, 6.50) {};
\node[vertex] (G-7) at (-1.44, 7.80) {};
\node[vertex] (G-8) at (-1.44, 9.10) {};
\draw[red] (G--3) -- (G-5) node[pos=0.35294117647058826,mid] (M-3-5) { 1 };
\draw[darkgray] (G-6) .. controls +(0.70, 0.35) and +(0.70, -0.35) .. (G-7) node[pos=0.5,mid] (M6-7) { 2 };
\draw[violet] (G--5) .. controls +(-0.70, -0.35) and +(-0.70, 0.35) .. (G--4) node[pos=0.5,mid] (M-5--4) { 2 };
\draw[green] (G--6) -- (G-8) node[pos=0.35294117647058826,mid] (M-6-8) { 4 };
\draw[red] (G-1) .. controls +(2.10, 1.05) and +(2.10, -1.05) .. (G-4) node[pos=0.5,mid] (M1-4) { 1 };
\draw[orange] (G--8) .. controls +(-0.70, -0.35) and +(-0.70, 0.35) .. (G--7) node[pos=0.5,mid] (M-8--7) { 5 };
\draw[red] (G-2) .. controls +(0.70, 0.35) and +(0.70, -0.35) .. (G-3) node[pos=0.5,mid] (M2-3) { 2 };
\draw[LightGreen] (G--2) .. controls +(-0.70, -0.35) and +(-0.70, 0.35) .. (G--1) node[pos=0.5,mid] (M-2--1) { 1 };
\end{tikzpicture}
=
\begin{tikzpicture}[scale=0.6, thick, baseline={(current bounding box.east)}]
\tikzstyle{vertex} = [shape=circle, minimum size=7pt, inner sep=1pt]
\tikzstyle{mid} = [draw, fill=white, shape=circle, minimum size=2pt, inner sep=1pt]
\node[vertex] (G--8) at (1.44, 9.10) {$\overline{8}$};
\node[vertex] (G--7) at (1.44, 7.80) {$\overline{7}$};
\node[vertex] (G--6) at (1.44, 6.50) {$\overline{6}$};
\node[vertex] (G--5) at (1.44, 5.20) {$\overline{5}$};
\node[vertex] (G--4) at (1.44, 3.90) {$\overline{4}$};
\node[vertex] (G--3) at (1.44, 2.60) {$\overline{3}$};
\node[vertex] (G--2) at (1.44, 1.30) {$\overline{2}$};
\node[vertex] (G--1) at (1.44, 0.00) {$\overline{1}$};
\node[vertex] (G-1) at (-1.44, 0.00) {$1$};
\node[vertex] (G-2) at (-1.44, 1.30) {$2$};
\node[vertex] (G-3) at (-1.44, 2.60) {$3$};
\node[vertex] (G-4) at (-1.44, 3.90) {$4$};
\node[vertex] (G-5) at (-1.44, 5.20) {$5$};
\node[vertex] (G-6) at (-1.44, 6.50) {$6$};
\node[vertex] (G-7) at (-1.44, 7.80) {$7$};
\node[vertex] (G-8) at (-1.44, 9.10) {$8$};
\draw[red] (G--3) -- (G-1) node[pos=0.5,mid] (M-3-1) { 1 };
\draw[blue] (G-4) .. controls +(0.70, 0.35) and +(0.70, -0.35) .. (G-5) node[pos=0.5,mid] (M4-5) { 2 };
\draw[violet] (G--5) .. controls +(-0.70, -0.35) and +(-0.70, 0.35) .. (G--4) node[pos=0.5,mid] (M-5--4) { 2 };
\draw[green] (G--6) -- (G-6) node[pos=0.5,mid] (M-6-6) { 4 };
\draw[cyan] (G-7) .. controls +(0.70, 0.35) and +(0.70, -0.35) .. (G-8) node[pos=0.5,mid] (M7-8) { 5 };
\draw[orange] (G--8) .. controls +(-0.70, -0.35) and +(-0.70, 0.35) .. (G--7) node[pos=0.5,mid] (M-8--7) { 5 };
\draw[Gold] (G-2) .. controls +(0.70, 0.35) and +(0.70, -0.35) .. (G-3) node[pos=0.5,mid] (M2-3) { 2 };
\draw[LightGreen] (G--2) .. controls +(-0.70, -0.35) and +(-0.70, 0.35) .. (G--1) node[pos=0.5,mid] (M-2--1) { 1 };
\end{tikzpicture},
\begin{array}{c}
    \begin{tikzpicture}[diamond, diamond/bound={ 7 }{ 9 }]
      \matrix[, diamond/matrix] {
    \TLlab{1}&\straightLoop{, draw=none}{red}{white}\\
\TLlab{2}&\straightLoop{}{Gold}{red}&\straightLoop{, draw=none}{red}{white}\\
\TLlab{3}&\tuC{Gold}{blue}&\straightLoop{}{blue}{red}&\straightLoop{, draw=none}{red}{white}\\
\TLlab{4}&\tuC{blue}{blue}&\tuC{blue}{red}&\straightLoop{}{red}{red}&\straightLoop{, draw=none}{red}{white}\\
\TLlab{5}&\tuC{blue}{green}&\straightLoop{}{green}{red}&\straightLoop{}{red}{red}&\straightLoop{}{red}{red}&\straightLoop{, draw=none}{red}{white}\\
\TLlab{6}&\tuC{green}{cyan}&\straightLoop{}{cyan}{green}&\straightLoop{}{green}{red}&\straightLoop{}{red}{red}&\straightLoop{}{red}{red}&\straightLoop{, draw=none}{red}{white}\\
\TLlab{7}&\tuC{cyan}{cyan}&\tuC{cyan}{green}&\tuC{green}{darkgray}&\straightLoop{}{darkgray}{red}&\straightLoop{}{red}{red}&\straightLoop{}{red}{red}&\straightLoop{, draw=none}{red}{white}\\
\TLlab{8}&\tuC{cyan}{green}&\tuC{green}{darkgray}&\tuC{darkgray}{darkgray}&\tuC{darkgray}{red}&\tuC{red}{red}&\tuC{red}{red}&\tuC{red}{red}&\straightLoop{, draw=none}{red}{white}\\
&\straightLoop{, draw=none}{white}{green}&\straightLoop{}{green}{darkgray}&\straightLoop{}{darkgray}{darkgray}&\straightLoop{}{darkgray}{red}&\straightLoop{}{red}{red}&\straightLoop{}{red}{red}&\straightLoop{}{red}{red}&\straightLoop{}{red}{red}&\straightLoop{, draw=none}{red}{white}\\
&&\straightLoop{, draw=none}{white}{green}&\straightLoop{}{green}{darkgray}&\straightLoop{}{darkgray}{darkgray}&\straightLoop{}{darkgray}{red}&\straightLoop{}{red}{red}&\straightLoop{}{red}{red}&\tuC{red}{red}&\tuC{red}{red}&\straightLoop{, draw=none}{red}{white}\\
&&&\straightLoop{, draw=none}{white}{green}&\straightLoop{}{green}{darkgray}&\straightLoop{}{darkgray}{darkgray}&\straightLoop{}{darkgray}{red}&\straightLoop{}{red}{red}&\tuC{red}{red}&\tuC{red}{darkgray}&\straightLoop{}{darkgray}{red}&\straightLoop{, draw=none}{red}{white}\\
&&&&\straightLoop{, draw=none}{white}{green}&\straightLoop{}{green}{darkgray}&\straightLoop{}{darkgray}{darkgray}&\straightLoop{}{darkgray}{red}&\tuC{red}{darkgray}&\tuC{darkgray}{darkgray}&\tuC{darkgray}{violet}&\straightLoop{}{violet}{red}&\straightLoop{, draw=none}{red}{white}\\
&&&&&\straightLoop{, draw=none}{white}{green}&\straightLoop{}{green}{darkgray}&\straightLoop{}{darkgray}{darkgray}&\tuC{darkgray}{darkgray}&\tuC{darkgray}{green}&\straightLoop{}{green}{violet}&\straightLoop{}{violet}{violet}&\straightLoop{}{violet}{red}&\straightLoop{, draw=none}{red}{white}\\
&&&&&&\straightLoop{, draw=none}{white}{green}&\straightLoop{}{green}{darkgray}&\tuC{darkgray}{green}&\tuC{green}{orange}&\straightLoop{}{orange}{green}&\straightLoop{}{green}{violet}&\straightLoop{}{violet}{violet}&\straightLoop{}{violet}{red}&\straightLoop{, draw=none}{red}{white}\\
&&&&&&&\straightLoop{, draw=none}{white}{green}&\tuC{green}{orange}&\tuC{orange}{orange}&\tuC{orange}{green}&\tuC{green}{violet}&\tuC{violet}{violet}&\tuC{violet}{red}&\tuC{red}{LightGreen}&\straightLoop{, draw=none}{LightGreen}{white}\\
&&&&&&&&\TLlab{\ov{8}}&\TLlab{\ov{7}}&\TLlab{\ov{6}}&\TLlab{\ov{5}}&\TLlab{\ov{4}}&\TLlab{\ov{3}}&\TLlab{\ov{2}}&\TLlab{\ov{1}}\\
    };
  \end{tikzpicture}\\
  \\
  \text{Paths:}\quad
  {\red1\U\U\D\D\U\U\U\D\D\U\U\U\D\D\D\D\U\U\ov3}, y_1^2y_2y_3\qquad{~} \\
  {\color{Gold}2\U3}, 1 \quad
  {\blue4\D\D\U\U\U5}, 1 \quad
  {\green6\D\D\U\U\U\U\D\D\D\D\U\U\ov6}, y_4y_6 \\
  {\cyan7\D\D\U\U\U8}, 1 \quad
  {\color{LightGreen}\ov1\U\ov2}, 1 \quad
  {\color{violet}\ov4\D\D\U\U\U\ov5}, 1 \quad
  {\color{orange}\ov7\D\D\U\U\U\ov8}, 1 \\
  \\
  \text{Loops:}\quad
  {\color{darkgray}2\U\U\U\U\U\D\D\D\U\U\D\D\D\D2}, x_2y_3y_4y_5\qquad\qquad{~}\\
\end{array}
\]
\medskip
\[
  \begin{array}{c|cccccc}
    \text{arcs} &
\arc{1}{1}{\ov3} & \arc{3}{2}{2} & \arc{\ov2}{1}{\ov1} & \arc{1}{1}{4}
& \arc{\ov4}{2}{\ov5} & \arc{5}{1}{\ov3}\\\hline\\[-3.5mm]
\text{walks} & \D^0\U^2&\D^1\U^0&\D^1\U^0&\D^0\U^3&\D^2\U^3&\D^4\U^2\\
\end{array}
\]
\caption[Computation of a product in the Okada algebra]{A product in the Okada algebra, the associated loop configuration and
  its associated walks for paths and loops together with their weights. The
  array underneath describes the walk corresponding to the red arc joining $1$ and $\ov{3}$.}
  \label{fig.walks}
\end{figure}
\end{example}
\begin{proof}
  We show that the equality $\weight(\mathcal{D} \rewto \norf(\mathcal{D})) = R(\mathcal{D})$
  holds for any loop configuration $\mathcal{D}$ by induction on the length of the
  rewriting sequence $\mathcal{D} \rewto \norf(\mathcal{D})$.

  If the rewriting sequence is of length $0$, then $\mathcal{D}$ is already a
  normal form. According to \cref{prop:red-expr-loop}, it contains no loops
  and the walks are all of the form $\D^i\U^j$ so that they have no paired
  $\U$-steps. As a consequence both side are equal to $1$.

  We now suppose that there is a non-trivial rewriting sequence
  $\mathcal{D}\mapsto\mathcal{D}'\rewto\norf(\mathcal{D})$. Note that
  $\norf(\mathcal{D})=\norf(\mathcal{D}')$. By induction hypothesis, we know that
  $\weight(\mathcal{D}' \rewto \norf(\mathcal{D}))=R(\mathcal{D}')$. There are three cases
  depending on which rewriting rule is applied in the step $\mathcal{D}\mapsto\mathcal{D}'$:
  \begin{itemize}
  \item[(C)] If $\mathcal{D}'$ is obtain from $\mathcal{D}$ by rule
    $C = \begin{tikzpicture}[diamond, diamond/bound={3}{0}]\matrix[diamond/matrix] {
        \nn \\
        & \st & \st \\
        & \st & \tu & \nn& \nn\\
      };
    \end{tikzpicture}
    \longmapsto
    \begin{tikzpicture}[diamond, diamond/bound={3}{0}]\matrix[diamond/matrix] {
        \nn \\
        & \tu & \st \\
        & \st & \st & \nn& \nn\\
      };
    \end{tikzpicture}$ then $\walks(\mathcal{D}) = \walks(\mathcal{D}')$ and
    $\loops(\mathcal{D}) = \loops(\mathcal{D}')$. Therefore $R(\mathcal{D})=R(\mathcal{D}')$. On the
    other hand, by definition, $\weight(\mathcal{D}\mapsto\mathcal{D}')=1$, so that
    $\weight(\mathcal{D}\rewto\norf(\mathcal{D}))=\weight(\mathcal{D}'\rewto\norf(\mathcal{D}))$.
    The conclusion follows from the induction hypothesis.

  \item[(I)] Suppose that $\mathcal{D}'$ is obtain from $\mathcal{D}$ by rule
    $I = \begin{tikzpicture}[diamond, diamond/bound={3}{0}]\matrix[diamond/matrix] {
        \nn \\
        & \tu & \st \\
        & \st & \tu & \nn& \nn\\
      };
    \end{tikzpicture}
    \longmapsto
    \begin{tikzpicture}[diamond, diamond/bound={3}{0}]\matrix[diamond/matrix] {
        \nn \\
        & \tu & \st \\
        & \st & \st & \nn& \nn\\
      };
    \end{tikzpicture}$. Let $h$ be the height of the middle row of the
    pattern. Then
    $\walks(\mathcal{D})=\walks(\mathcal{D}')\cup\{(h,\U\D,h)\}$. Moreover
    $\loops(\mathcal{D})=\loops(\mathcal{D}')\cup\{(h,\U\D,h)\}$. The only
    paired occurence of $\U$ in the loop walk $(h,\U\D,h)$ is $0$ and
    $\lpair(0) = 1$, which is odd. Hence the walk $(h,\U\D,h)$ doesn't
    contribute any $y$-parameter but, as a loop, it contributes
    $x_h$. Consequently $R(\mathcal{D}) = x_hR(\mathcal{D}')$. On the other
    hand, by definition of the weight,
    $\weight(\mathcal{D}\rewto\norf(\mathcal{D}))=x_h\weight(\mathcal{D}'\rewto\norf(\mathcal{D}))$.
    As in the previous case, the conclusion follows from the induction hypothesis.

  \item[(S)] The last case is when $\mathcal{D}'$ is obtained from $\mathcal{D}$ by rule
    $S=\begin{tikzpicture}[diamond,
      diamond/bound={3}{0}]\matrix[diamond/matrix] {
        \nn \\
        & \tu & \tu \\
        & \st & \tu & \nn& \nn\\
      };
    \end{tikzpicture}
    \longmapsto
    \begin{tikzpicture}[diamond, diamond/bound={3}{0}]\matrix[diamond/matrix] {
        \nn \\
        & \tu & \st \\
        & \st & \st & \nn& \nn\\
      };
    \end{tikzpicture}$.
    Let $h$ be the height of the bottom row of the pattern.  Let $\pi$ be the
    path (or loop) in $\mathcal{D}'$ passing though the bottom horizontal line 
    of the pattern on the right hand side of rule $S$.  The
    corresponding walk of $\mathcal{D}'$ is of the form
    $\mathbb{w}'=(a, \boldsymbol{v}_1 \boldsymbol{v}_2, b)$, where $\boldsymbol{v}_1$ is
    the portion of $\mathbb{w}'$ which precedes the horizontal line and
    $\boldsymbol{v}_2$ is the succeeding portion of $\mathbb{w}'$.
    The walk of the corresponding path/loop in $\mathcal{D}$
    is  $\mathbb{w}=(a,\boldsymbol{v}_1\U\U\D\D\boldsymbol{v}_2,b)$
    so that
    \[\walks(\mathcal{D}) =
      \walks(\mathcal{D}')\setminus\{\mathbb{w}'\}
      \cup\{\mathbb{w} \}\,.\]
     When comparing $R(\mathcal{D})$ and $R(\mathcal{D}')$ we
     only need to examine the $y$-parameter contributions of $\mathbb{w}'$
     and $\mathbb{w}$ respectively. Let's compare the paired steps in
    $\mathbb{w}'$ and $\mathbb{w}$. Suppose $\mathbb{w}'$
    contains a $\U$-step which occurs at $p'$ which is
    paired with a $\D$-step occuring at $q' > p'$.
    Let  $q > p$ denote the positions of the corresponding $\D$ and $\U$-steps in $\mathbb{w}$.
    There are three distinct cases to consider:
    \begin{itemize}
    \item the $\U$ and $\D$-steps are both in $\boldsymbol{v}_1$ (i.e. $p' < q' \leq |\boldsymbol{v}_1|$) so that $p = p'$ and $q = q'$
    \item the $\U$ and $\D$-steps are both in $\boldsymbol{v}_2$ (i.e. $ |\boldsymbol{v}_1| \leq p' < q'$) so that $p = p'+4$ and $q = q'+4$
    \item the $\U$-step is in $\boldsymbol{v}_1$ and the $\D$-step is in $\boldsymbol{v}_2$ (i.e. $p' < |\boldsymbol{v}_1| < q'$) so that
    $p = p'$ and $q = q'+4$.
    \end{itemize}
    In all cases it follows that
      \[
      \lpair_{\mathbb{w}'} (p')=\lpair_{\mathbb{w}}(p) \ssp \pmod 2
      \qandq
      \height_{\mathbb{w}'}(p')=\height_{\mathbb{w}}(p).
    \]
    This implies that both $p'$ and $p$ contribute the same factor in $R(\mathcal{D}')$ and
    $R(\mathcal{D})$ respectively. The extra $\U$ and $\D$-steps in $\mathbb{w}$
    which occur at positions $p = |\boldsymbol{v}_1|$ and $q = |\boldsymbol{v}_1| + 4$
    contribute a factor of $y_h$ to $R(\mathcal{D})$ because $\height_{\mathbb{w}}(p) = h$.
 The remaining $\U$ and $\D$-steps
    in $\mathbb{w}$ occurring at positions $p = |\boldsymbol{v}_1|+1$ and $q = |\boldsymbol{v}_1| +3$
    do not contribute a factor since $\lpair_{\mathbb{w}}(p) = 1$ is odd. This shows that
    \[
      \prod_{r \, : \,  \lpair_{\mathbb{w}}(r) \text{ is even}} y_{\height_{\mathbb{w}}(r)}
      \ =\  y_h
      \prod_{r \, : \, \lpair_{\mathbb{w}'}(r) \text{ is even}} y_{\height_{\mathbb{w}'}(r)}\,.
    \]
    and therefore $R(\mathcal{D}) = y_h R(\mathcal{D}')$. As in the two previous cases,
    the result follows from the induction hypothesis and the prescription for weights.
  \end{itemize}
  We proved that the equality $\weight(\mathcal{D} \rewto \norf(\mathcal{D})) =
  R(\mathcal{D})$ holds in all cases, which finishes the proof by induction.
\end{proof}

\subsubsection{Product of two Okada arc-diagrams}

We introduce a technique which uses \cref{theorem-coeff-mult}
to calculate the scalar coefficient $\lambda(C,D)$ arising in the product
\begin{equation*}
    \E{C}\E{D} = \lambda(C,D)\,\E{C \boldsymbol{\cdot}  D}
  \end{equation*}
where $C$ and $D$ are two Okada arc-diagrams.
To do this, let $\mathcal{C}$ and $\mathcal{D}$ be the
(unique) normal form FPLCs such that
$[\mathcal{C}] = C$ and $[\mathcal{D}] = D$.
Since $\mathcal{C}$ and $\mathcal{D}$ are
normal form FPLCs, each  arc/loop $\alpha$ in $C \circ D$
is isotopic to a unique arc/loop in $\mathcal{C} \circ \mathcal{D}$.
Now consider a path
$\alpha$ in $C \circ D$ composed of a sequence of arcs
$\alpha_1, \dots, \alpha_n$ in $C$ and $D$
where each $\alpha_k$ is of the form
$\arc{i_k}{\ell_k}{j_k}$ and where $i_1$
is the index of the leftmost and lowest
end-point of $\alpha$.
 We assign to $\alpha$ the
lattice walk $\mathbb{w} = (i_1, \boldsymbol{w}, |j_n|)$
where
\[ \boldsymbol{w} \,  =  \, \D^{|i_1|-\ell_1} \U^{|j_1|-\ell_1} \cdots \D^{|i_n|-\ell_n} \U^{|j_n|-\ell_n}.  \]
The procedure is slightly different for a loop $\alpha$ in $C \circ D$ due to the fact that
we must identify an origin for the corresponding loop in 
$\mathcal{C} \circ \mathcal{D}$. Let
$\mathcal{A}(\alpha)$ be the set of arc fragments
in $C \circ D$ which comprise the loop $\alpha$. Take the minimum label
\[ \ell_1 = \mathrm{min} \big\{ \ell : \arc{a}{\ell}{b} \in \mathcal{A}(\alpha) \big\} \]
and let $\mathcal{A}_C (\alpha)$ (resp. $\mathcal{A}_D (\alpha)$) be the subset of $\mathcal{A}(\alpha)$
consisting of those arcs $\arc{a}{\ell}{b}$
in $C$ (resp. $D$) for which $\ell = \ell_1$. If  $\mathcal{A}_C (\alpha)$ is
non-empty let $\arc{u}{\ell_1}{v}$ be the
arc in $\mathcal{A}_C (\alpha)$ such that
$\max\{u,v \}$ is maximal. Set $i_1 = \max\{u,v \}$
and $j_1= \min\{u,v \}$.
Otherwise let $\arc{u}{\ell_1}{v}$ be the arc in
$\mathcal{A}_D (\alpha)$
such that $\min\{u, v\}$ is minimal.
In this case set $i_1 = \min\{u,v \}$
and $j_1= \max\{u,v \}$.
In both cases, set
$\alpha_1 = \arc{i_1}{\ell_1}{j_1}$.
Starting with $\alpha_1$, and following the arc $\alpha$ counter-clockwise, we iteratively
list the remaining arcs $\alpha_2, \dots, \alpha_n$
in $\mathcal{A}(\alpha)$. We assign to $\alpha$
the lattice walk  $\mathbb{w} = (\ell_1, \boldsymbol{w}, \ell_1)$
where
\[ \boldsymbol{w} \,  =  \, \U^{|j_1|-\ell_1} \D^{|i_2|-\ell_2} U^{|j_2| - \ell_2}  \cdots
\D^{|i_n|-\ell_n} \U^{|j_n|-\ell_n} \D^{|i_1|-\ell_1}. \]
Note that the initial and terminal steps $\U^{|j_1|-\ell_1}$ and $\D^{|i_1|-\ell_1}$ in
$\boldsymbol{w}$ come from $\alpha_1$. The construction of the arc $\alpha_1$ and
the word $\boldsymbol{w}$ ensures that the walk $\mathbb{w}$
coincides with the walk of the associated loop in the
FPLC $\mathcal{C} \circ \mathcal{D}$.

Gather these walks into two sets
depending on whether $\alpha$ is an arc or a loop,
and denote them  $\walks(C, D)$ and $\loops(C, D)$ respectively,
then
\[ \walks(C,D) = \walks(\mathcal{C} \circ \mathcal{D}) \quad \text{and} \quad
\loops(C,D) = \loops(\mathcal{C} \circ \mathcal{D}). \]
We can now restate \cref{theorem-coeff-mult} as
\begin{corollary}\label{coro:coeff-mult}
  For Okada arc-diagrams $C , D$ the product $\E{C}\E{D}$ is given by
  \begin{equation}\label{equ:weight-rewriting-walks}
    \E{C}\E{D} =
    \left(\prod_{(a, \boldsymbol{w}, a)\, \in\, \loops(C, D)} x_a\right)
    \left(
      \prod_{\mathbb{w}\, \in\, \walks(C, D)}
      \prod_{p} y_{\height_\mathbb{w}(p)}\right) \E{C \boldsymbol{\cdot} D}
  \end{equation}
  The last product is taken over all $p \in \{1, \dots, k \}$
  for which there is an occurrence of $\U$ at $p$
  such that $p$ is paired and $\lpair_\mathbb{w}(p)$ is even.
\end{corollary}
\subsubsection{Combinatorial digression: Grammars for walks}

We make a small, combinatorial digression concerning the Dyck paths
that appear in walks associated to FPLCs. The results here are not used in the rest of
the paper and are presented without detailed proof in the hope that they
may interest combinatorists.

We consider the maximal factors consisting only of paired steps of walks in FPLCs.
Such a factor $\boldsymbol{w}$ is \defn{even reducible} if 
there exists a \defn{reduction sequence}, i.e. there exists a sequence
$(\boldsymbol{w}_i)_{0 \leq i \leq k}$ where $\boldsymbol{w}_0=\varepsilon$ and $\boldsymbol{w}_k=\boldsymbol{w}$ such that for any $i<k$,
the word $\boldsymbol{w}_{i+1}$ is obtained from $\boldsymbol{w}_i$ by inserting $\U\U\D\D$
somewhere. Similarly, a walk~$\mathbb{w}$ is \defn{odd reducible} if
there is a reduction sequence starting with $\U\D$ and ending at
$\boldsymbol{w}$. Even (resp. odd) reducible word are Dyck word of even (resp. odd)
semi-length.

\begin{definition}
  We define, inductively, two sets of Dyck words, namely the set $\evenD$
  (resp. $\oddD$ of \defn{even (resp. odd) tileable} Dyck words:
  \begin{itemize}
  \item $\evenD$ is the set of Dyck words which are either empty or of the
    form $\boldsymbol{w} = \U \boldsymbol{v}_1\D\U \boldsymbol{v}_k\D \dots\U \boldsymbol{v}_k\D$ where $\boldsymbol{v}_i\in\oddD$ for all $i$;
  \item $\oddD$ is the sets of Dyck words of the form
    $\boldsymbol{w} = \boldsymbol{v}_1\U \boldsymbol{v}_2\D \boldsymbol{v}_3$ where $\boldsymbol{v}_1,\boldsymbol{v}_2,\boldsymbol{v}_3\in\evenD$;
  \end{itemize}
\end{definition}
\begin{lemma}
  The factorization $\U \boldsymbol{v}_1\D\U \boldsymbol{v}_k\D \dots\U \boldsymbol{v}_k\D$ of a non-empty even
  tileable Dyck word is unique. Likewise, the factorization $\boldsymbol{v}_1\U \boldsymbol{v}_2\D \boldsymbol{v}_3$ of an odd
  tileable Dyck word is unique.
\end{lemma}
\begin{proposition}
  A word $\boldsymbol{w}$ is even (resp. odd) tileable if and only if  $\boldsymbol{w}$ is even
  (resp. odd) reducible.
\end{proposition}
\newcommand{\Even}{\operatorname{Even}}
\newcommand{\Odd}{\operatorname{Odd}}
As a consequence the generating functions which track the
semi-length of, respectively, even and odd reducible Dyck words
are related by the identities $\Even(z) = (1 - z\Odd(z))^{-1}$ and
$\Odd(z) = z(\Even(z))^3$. One finds that
\begin{align*}
  \Even(z) &= 1 +  z^{2} + 4 z^{4} + 22 z^{6} + 140 z^{8} + 969 z^{10} + 7084
             z^{12} + 53820 z^{14} + O(z^{15}) \\
  \Odd(z) &= z + 3 z^{3} + 15 z^{5} + 91 z^{7} + 612 z^{9} + 4389 z^{11} + 32890 z^{13} + O(z^{14})
\end{align*}
The sequence of coefficients 
of $\Even$ is~\cite[Sequence A002293]{OEIS}, the
coefficients of $z^{2n}$ is $\frac{1}{n}\binom{4n}{n-1}$. The sequence of
coefficients of $\Odd$ is~\cite[Sequence A006632]{OEIS} and the coefficient of
$z^{2n-1}$ is $\frac{3}{4n-1}\binom{4n-1}{n-1}$. Finally $\Even+\Odd$
is~\cite[Sequence A369082]{OEIS}.

%%%%%%%%%%%%%%%%%%%%%%%%%%%%%%%%%%%%%%

\subsection{Okada, Temperley-Lieb and Blob algebras}
In this section, we comment on the morphisms between the Okada algebra
and the blob and Temperley-Lieb algebras.

Recall that the Temperley-Lieb algebra $\TL(\delta)$ is the algebra
generated by $\tl{1}, \dots, \tl{n-1}$ with relations:
\begin{alignat*}{3}
  \tl{i}^2&=\delta \ssp \tl{i} &\qquad& 1\leq i \leq N-1,\\
  \tl{i}\tl{j}&=\tl{j}\tl{i} &\qquad&\vert i-j\vert \geq 2,\\
  \tl{i+1}\tl{i}\tl{i+1}&=\tl{i+1} &\qquad&  1\leq i \leq N-2,\\
  \tl{i}\tl{i+1}\tl{i}&=\tl{i} &\qquad&  1\leq i \leq N-2,
\end{alignat*}
The reader will recognize a specialized version of the Okada relations after removing the last relation.
This leads to the following observation:

\begin{proposition}
  The map $\E{i}\mapsto\tl{i}$ extends uniquely to an algebra
  epimorphism from the specialized Okada algebra
  $\Okada[N](\delta,\dots,\delta; 1,\dots,1)$ to $\TL(\delta)$. At the level of the diagram basis,
  this amount to removing all arc labels. By specializing $\delta =1$, one gets an
  epimorphism from the Okada monoid to the Jones monoid.
\end{proposition}
This raises the natural, combinatorial problem of computing the
cardinality of the preimage of a Jones monoid element within the Okada
monoid:
\begin{question}
  Is there a formula to compute the number of valid Okada labelling
  of a given unlabelled, non-crossing arc-diagram?
\end{question}
We haven't found a simple answer to this question. For example, the
preimage of $\tl{1}\tl{4}\tl{7}$ in $\Okada[8]$ has
cardinality $61$ which is a prime number, thus ruling out the possibility of
a simple, multiplicative counting formula. 
However, the same question, when posed for half arc-diagrams, has
a simple and elegant answer: Thanks to
\cref{LinearExtentions-to-HLabels}, there is a bijection between
Okada half arc-diagrams with a prescribed underlying (unlabelled) 
half arc-diagram and linear extensions of the associated nesting order. Since the nesting order 
has the structure of a forest, the number of its linear extensions can be obtained from the
hook-length formula for forest \cite[\S5.1.4, Exer.20]{AOCP3}.
% \begin{remark}
%   Note that starting with $N=4$, the specialized Okada algebra
%   $\Okada[N](q,\dots,q; 1,\dots,1)$ is never semisimple whereas
%   $\TL(q)$ might be depending on $q$.
% \end{remark}
\bigskip

We now turn to the blob algebra~\cite{Martin1994,Martin1991}. This is the
two parameter algebra $\Blob[N](\gamma, \delta)$ generated
by $\tl{1}, \dots, \tl{n-1}$ obeying the Temperley-Lieb relations together with an extra generator $\blob$
satisfying
\[
  \blob^2= \blob, \qquad\qquad
  \tl{1}\blob\,\tl{1}= \gamma \ssp \tl{i}\qquad\text{and}\qquad
  \blob\ssp\tl{i}=\tl{i}\blob\quad\text{for $i \geq 2$}\,.
\]
As in the Temperley-Lieb case, these relations contain a specialized
version of Okada's relations, and so the blob algebra is also a quotient
of the Okada algebra:
\begin{proposition}\label{prop:morph-to-blob}
  The map $\E{i}\mapsto\tl{i-1}$ and $\E{1}\mapsto\blob$ extends
  uniquely to an algebra epimorphism from the specialized Okada
  algebra $\Okada[N+1](1,\delta,\dots,\delta; \gamma,1,\dots,1)$ to
  $\Blob[N](\gamma, \delta)$. Specializing both $\delta=1$ and $\gamma=1$, one gets an
  epimorphism from the Okada monoid to the Blob monoid.
\end{proposition}
Recall that the blob algebra is also realized by 
Temperley-Lieb diagrams~\cite{Martin1994}
whose arcs are marked according to the
following rule: Each arc is either
{\it unblobbed} (and left unmarked) or is
{\it blobbed} (and marked with a bullet $\bullet$). However,
only arcs which are not nested may be blobbed. 
The underlying (unmarked) diagrams in the blob algebra are multiplied as diagrams in 
Temperley-Lieb algebra. An arc created 
in the product of two diagrams $C \boldsymbol{\cdot} D$ 
is blobbed whenever
any of the arcs from $C$ or $D$ which
form it are blobbed (i.e. the blob mark $\bullet$ is idemptotent).
As usual, all loops formed in product $C \boldsymbol{\cdot} D$ 
are removed, and a scalar factor $\delta$ or $\gamma$ 
is introduced for each loop depending upon whether or not
it is unblobbed or blobbed.

We now explain how to compute the image of an Okada arc-diagram $D$ of
rank $N+1$ within the blob algebra $\Blob[N](\gamma, \delta)$.
The idea is to remove vertices $1$ and $\ov{1}$ and the arcs incident
to them, retain the arcs of $D$ whose labels are larger than $2$, 
and finally to link together (with a blobed
arc) each consecutive pair of vertices which are not joined by an arc in $D$ but
which are both incident to an arc labelled $1$.
\begin{definition}
  Consider an Okada arc-diagram $D$ of rank $N+1$. Notice first that
  the Okada labelling conditions ensure that vertices $1$ and $\ov{1}$ are incident only to
  arc(s) labelled $1$. Moreover the arcs labelled $1$ cannot be
  nested in one another, so there is a well-defined sequence of vertices
  $(v_i)_{i=1\dots 2k}\in\setN \cup \setNN$ which is strictly
  increasing with respect to the total order on $\setN \cup \setNN$
  and such that the set of arcs labelled $1$ in $D$ is exactly
  \[\left\{
    1=\arc{v_1}{1}{v_2},\ \arc{v_3}{1}{v_4},\ \dots,\
    \arc{v_{2k-1}}{1}{v_{2k}}=\ov{1}\right\}\,.
  \]
  We define \defn{$\mathbcal{b}_N(D)$} to be the blobed arc-diagram of rank $N$ with
  \begin{itemize}
  \item a unblobed arc $\arc{i-1}{}{j-1}$ whenever there is an arc
    $\arc{i}{\ell}{j}$ in $D$ with $\ell>1$\,;
  \item a blobed arc $v_{2i}-1\barc v_{2i+1}-1$ for all $i<k$\,.
  \end{itemize}
\end{definition}
Note that if $1$ is joined to $\ov{1}$ by an arc in $D$ then it is the only arc
labelled $1$ and so $\mathbcal{b}_N(D)$ is just the unlabelled version of $D$
where this arc has been removed. See \cref{fig.compo_Blob} for some
examples.
\begin{proposition}
\label{prop:blob-map}
  The map $\E{D} \mapsto \mathbcal{b}_N(D)$ coincides with the morphism of
  \cref{prop:morph-to-blob}.
\end{proposition}
\begin{proof}
  It is easy to check that this map is compatible with the product and
  that it agrees with the map in \cref{prop:morph-to-blob} when applied to
  generators. Details are left to the reader.
\end{proof}
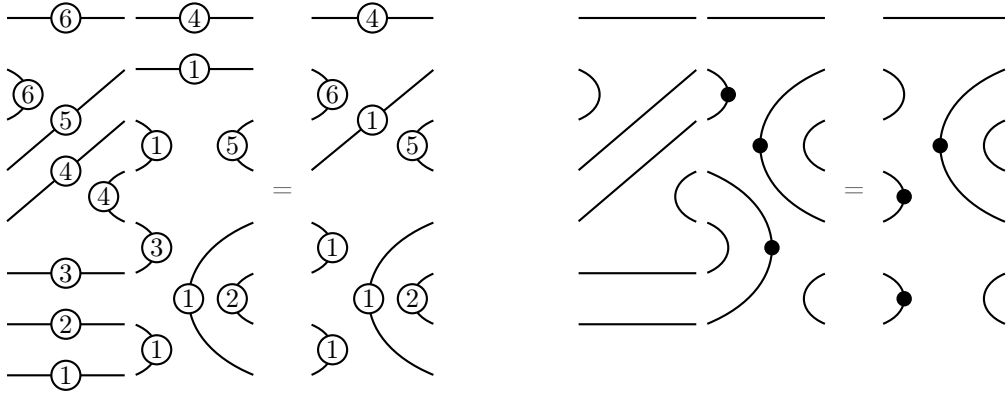
\begin{figure}[ht]
  %\makebox[\textwidth]{
  \begin{tikzpicture}[xscale=0.55,yscale=0.45, thick, baseline={(current bounding box.east)}]
\tikzstyle{vertex} = [shape=rectangle, minimum height=15pt, minimum width=0pt, inner sep=0pt]
\tikzstyle{mid} = [draw, fill=white, shape=circle, minimum size=2pt, inner sep=1pt]
\foreach \i in {1,...,8} {
  \node[vertex] (G-\i) at ($(-1.44, \i*1.5-1.5)$) {};
  \node[vertex] (G--\i) at ($(1.44, \i*1.5-1.5)$) {};
}
\draw (G--3) -- (G-3) node[pos=0.5,mid] (M-3-3) { 3 };
\draw (G--6) -- (G-4) node[pos=0.5,mid] (M-6-4) { 4 };
\draw (G-1)  -- (G--1) node[pos=0.5,mid] (M-1-1) { 1 };
\draw (G-2)  -- (G--2) node[pos=0.5,mid] (M-2-2) { 2 };
\draw (G--7) -- (G-5) node[pos=0.5,mid] (M-7-5) {$5$};
\draw (G--8) -- (G-8) node[pos=0.5,mid] (M-8-8) { 6 };
\draw (G-6) .. controls +(0.70, 0.35) and +(0.70, -0.35) .. (G-7) node[pos=0.5,mid] (M6-7) { 6 };
\draw (G--5) .. controls +(-0.70, -0.35) and +(-0.70, 0.35) .. (G--4) node[pos=0.5,mid] (M-5--4) { 4 };
\end{tikzpicture}
\begin{tikzpicture}[xscale=0.55,yscale=0.45, thick, baseline={(current bounding box.east)}]
\tikzstyle{vertex} = [shape=rectangle, minimum height=15pt, minimum width=0pt, inner sep=0pt]
\tikzstyle{mid} = [draw, fill=white, shape=circle, minimum size=2pt, inner sep=1pt]
\foreach \i in {1,...,8} {
  \node[vertex] (G-\i) at ($(-1.44, \i*1.5-1.5)$) {};
  \node[vertex] (G--\i) at ($(1.44, \i*1.5-1.5)$) {};
}
\draw (G--7) -- (G-7) node[pos=0.5,mid] (M-7-7) { 1 };
\draw (G-5) .. controls +(0.70, 0.35) and +(0.70, -0.35) .. (G-6) node[pos=0.5,mid] (M5-6) { 1 };
\draw (G--4) .. controls +(-2.10, -1.05) and +(-2.10, 1.05) .. (G--1) node[pos=0.5,mid] (M-4--1) { 1 };
\draw (G-3) .. controls +(0.70, 0.35) and +(0.70, -0.35) .. (G-4) node[pos=0.5,mid] (M3-4) { 3 };
\draw (G--3) .. controls +(-0.70, -0.35) and +(-0.70, 0.35) .. (G--2) node[pos=0.5,mid] (M-3--2) { 2 };
\draw (G-1) .. controls +(0.70, 0.35) and +(0.70, -0.35) .. (G-2) node[pos=0.5,mid] (M1-2) { 1 };
\draw (G--6) .. controls +(-0.70, -0.35) and +(-0.70, 0.35) .. (G--5) node[pos=0.5,mid] (M-6--5) { 5 };
\draw (G--8) -- (G-8) node[pos=0.5,mid] (M-8-8) { 4 };
% \draw[color=red, <-] (M-3--2) -- (M-4--1);
% \draw[color=red, <-] (M-8-8) -- (M-7-7);
\end{tikzpicture}
\ = \
\begin{tikzpicture}[xscale=0.55,yscale=0.45, thick, baseline={(current bounding box.east)}]
\tikzstyle{vertex} = [shape=rectangle, minimum height=15pt, minimum width=0pt, inner sep=0pt]
\tikzstyle{mid} = [draw, fill=white, shape=circle, minimum size=2pt, inner sep=1pt]
\foreach \i in {1,...,8} {
  \node[vertex] (G-\i) at ($(-1.44, \i*1.5-1.5)$) {};
  \node[vertex] (G--\i) at ($(1.54, \i*1.5-1.5)$) {};
}
\draw (G--7) -- (G-5) node[pos=0.5,mid] (M-7-5) { 1 };
\draw (G-3) .. controls +(0.70, 0.35) and +(0.70, -0.35) .. (G-4) node[pos=0.5,mid] (M3-4) { 1 };
\draw (G--4) .. controls +(-2.10, -1.05) and +(-2.10, 1.05) .. (G--1) node[pos=0.5,mid] (M-4--1) { 1 };
\draw (G-1) .. controls +(0.70, 0.35) and +(0.70, -0.35) .. (G-2) node[pos=0.5,mid] (M1-2) { 1 };
\draw (G--3) .. controls +(-0.70, -0.35) and +(-0.70, 0.35) .. (G--2) node[pos=0.5,mid] (M-3--2) { 2 };
\draw (G--8) -- (G-8) node[pos=0.5,mid] (M-8-8) { 4 };
\draw (G-6) .. controls +(0.70, 0.35) and +(0.70, -0.35) .. (G-7) node[pos=0.5,mid] (M6-7) { 6 };
\draw (G--6) .. controls +(-0.70, -0.35) and +(-0.70, 0.35) .. (G--5) node[pos=0.5,mid] (M-6--5) { 5 };
\end{tikzpicture}
%%%%%%%%%%%%%%%%%%%%%%%%%%%%%%%%%%%%%%%%%%%%
\hfil\hfil %%%%%%%%%%%%%%%%%%%%%%%%%%%%%%%%%
%%%%%%%%%%%%%%%%%%%%%%%%%%%%%%%%%%%%%%%%%%%%
\begin{tikzpicture}[xscale=0.55,yscale=0.45, thick, baseline={(current bounding box.east)}]
\tikzstyle{vertex} = [shape=rectangle, minimum height=15pt, minimum width=0pt, inner sep=0pt]
\tikzstyle{mid} = [draw, fill=white, shape=circle, minimum size=2pt, inner sep=1pt]
\foreach \i in {1,...,8} {
  \node[vertex] (G-\i) at ($(-1.44, \i*1.5-1.5)$) {};
  \node[vertex] (G--\i) at ($(1.44, \i*1.5-1.5)$) {};
}
\draw (G--3) -- (G-3);
\draw (G--6) -- (G-4);
\draw (G-2)  -- (G--2);
\draw (G--7) -- (G-5);
\draw (G--8) -- (G-8);
\draw (G-6) .. controls +(0.70, 0.35) and +(0.70, -0.35) .. (G-7);
\draw (G--5) .. controls +(-0.70, -0.35) and +(-0.70, 0.35) .. (G--4);
\end{tikzpicture}
\begin{tikzpicture}[xscale=0.55,yscale=0.45, thick, baseline={(current bounding box.east)}]
\tikzstyle{vertex} = [shape=rectangle, minimum height=15pt, minimum width=0pt, inner sep=0pt]
\tikzstyle{blob} = [draw, fill=black, shape=circle, minimum size=5pt, inner sep=0pt]
\foreach \i in {1,...,8} {
  \node[vertex] (G-\i) at ($(-1.44, \i*1.5-1.5)$) {};
  \node[vertex] (G--\i) at ($(1.44, \i*1.5-1.5)$) {};
}
\draw (G-6) .. controls +(0.70, 0.35) and +(0.70, -0.35) .. (G-7) node[pos=0.5,blob] { };
\draw (G--7) .. controls +(-2.10, -1.05) and +(-2.10, 1.05) .. (G--4) node[pos=0.5,blob] { };
\draw (G-3) .. controls +(0.70, 0.35) and +(0.70, -0.35) .. (G-4);
\draw (G--3) .. controls +(-0.70, -0.35) and +(-0.70, 0.35) .. (G--2);
\draw (G--6) .. controls +(-0.70, -0.35) and +(-0.70, 0.35) .. (G--5);
\draw (G--8) -- (G-8);
\draw (G-5) .. controls +(2.10, -1.05) and +(2.10, 1.05) .. (G-2) node[pos=0.5,blob] { };
\end{tikzpicture}
\ = \
\begin{tikzpicture}[xscale=0.55,yscale=0.45, thick, baseline={(current bounding box.east)}]
\tikzstyle{vertex} = [shape=rectangle, minimum height=15pt, minimum width=0pt, inner sep=0pt]
\tikzstyle{blob} = [draw, fill=black, shape=circle, minimum size=5pt, inner sep=0pt]
\foreach \i in {1,...,8} {
  \node[vertex] (G-\i) at ($(-1.44, \i*1.5-1.5)$) {};
  \node[vertex] (G--\i) at ($(1.54, \i*1.5-1.5)$) {};
}
\draw (G-2) .. controls +(0.70, 0.35) and +(0.70, -0.35) .. (G-3) node[pos=0.5,blob] {};
\draw (G-4) .. controls +(0.70, 0.35) and +(0.70, -0.35) .. (G-5) node[pos=0.5,blob] {};
\draw (G--7) .. controls +(-2.10, -1.05) and +(-2.10, 1.05) .. (G--4) node[pos=0.5,blob] (M-4--1) {};
\draw (G--3) .. controls +(-0.70, -0.35) and +(-0.70, 0.35) .. (G--2);
\draw (G--8) -- (G-8);
\draw (G-6) .. controls +(0.70, 0.35) and +(0.70, -0.35) .. (G-7);
\draw (G--6) .. controls +(-0.70, -0.35) and +(-0.70, 0.35) .. (G--5);
\end{tikzpicture}
%
%}
\caption[{Products of arc-diagrams}]{A product in the Okada monoid
  $\Okada[8]$ and its image in the blob monoid $\Blob[7]$ under the morphism $\mathbcal{b}_7$.}
  \label{fig.compo_Blob}
\end{figure}
\Cref{prop:blob-map} gives an alternative diagrammatic realization of the blob algebra:
\begin{corollary}
  The quotient of the Okada algebra
  $\Okada[N+1](1,\delta,\dots,\delta; \gamma, 1,\dots,1)$ (resp. monoid) obtained
  by identifying Okada arc-diagrams whose underlying (unlabelled) 
  arc-diagrams are identical
  and whose arcs labelled $1$ coincide
  is isomorphic to the blob algebra $\Blob[N](\gamma, \delta)$ (resp. monoid).
\end{corollary}
  
  In other words, by introducing blobs onto arcs which are labelled $1$, 
  we get rank
  $N+1$ blobbed arc-diagrams with the extra condition that the arcs
  incident to $1$ and $\ov{1}$ must be blobbed. In particular this
  observation gives us an embedding of the blob
  monoid $\Blob[N]$ into $\Blob[N+1]$
  as a sub-semigroup, namely as $\blob \boldsymbol{\cdot} \Blob[N+1] \boldsymbol{\cdot} \blob \subset \Blob[N+1]$.

\section{Biserial clone Schur functions and Gram determimants}
\label{sec:clone}

In this section, we relate our construction of cell modules with
Okada's description of simple $\Okada(X,Y)$-modules. We start by reviewing 
aspects of Okada's theory of clone Schur functions
which are significant in his treatment of simple modules.

\subsection{Okada's clone Schur functions}
\newcommand{\XY}{\left(X\!\mid\!Y\right)}

Throughout this section, we fix two \emph{infinite},
$\mathbb{K}$-valued sequences $X = (x_1, x_2, x_3, \dots)$ and
$Y= (y_1, y_2, y_3, \dots)$ which serve as a common system of
structure constants for the associated tower of Okada algebras.  The
parameters defining $\Okada(X, Y)$ are understood to be
$(x_1, \dots , x_{N-1} )$ and $(y_1, \dots , y_{N-2} )$ for each
$N \geq 0$.
\begin{definition}
\label{biserial-clone-schur-definition}
The {\bf biserial clone Schur function} $s_w\XY$
associated to a fibonacci word $w \in \YFn$
is defined recursively (with respect to the rank of the fibonacci word) by
\begin{equation}
  s_w\XY \eqdef
  \begin{cases}
    A_k\XY
    &\text{if $w = 1^k$ for some $k \geq 0$} \\
    B_k \Big( X + |v| \, \Big| \, Y+ |v|  \Big) \cdot s_v\XY
    &\text{if $w=1^k2v$ for some $k \geq 0$\,.}
  \end{cases}
\end{equation}
where
\[
A_\ell (X \,  | \, Y) \eqdef
\det
\underbrace{\begin{pmatrix}
x_1 & y_1 & 0 & \cdots\\
1 & x_2 & y_2 &\\
0 & 1 & x_3  & \\
\vdots & & & \ddots
\end{pmatrix}}_{\ell \times \ell \, \mathrm{tridiagonal} \, \mathrm{matrix}}
\quad
B_{\ell -1} (X \, | \, Y) \eqdef
\det
\underbrace{\begin{pmatrix}
y_1 & x_1y_2 & 0 & \cdots\\
1 & x_3 & y_3 &\\
0 & 1 & x_4 & \\
\vdots & & & \ddots
\end{pmatrix}}_{\ell \times \ell \,  \mathrm{tridiagonal} \, \mathrm{matrix}}
\]
for integers $\ell \geq 1$. We use the notation $X + r\eqdef (x_{1+r}, x_{2+r}, x_{3+r}, \dots)$ for any integer $r \geq 0$.
\end{definition}

The biserial clone Schur function are defined in Section 2 of
\cite{Okada1994} and arise
as $\mathbb{K}$-valued {\it evaluations} of associated
{\bf non-commutative} clone Schur functions $s_w(\mathrm{\bf x}, \mathrm{\bf y})$
which are elements in the
free algebra $\mathbb{K}\langle \mathrm{\bf x}, \mathrm{\bf y} \rangle$
generated by two non-commutative variables
$ \mathrm{\bf x}, \mathrm{\bf y} $.  The eponymous use of "Schur"
reflects, in part, on the celebrated relationship
between the symmetric groups and the
classical Schur functions, where the structure constants
for the induction product of irreducible symmetric group representations
coincide with the Littlewood-Richardson
coefficients which govern the multiplication of Schur functions.
The same relationship holds for the
induction product of simple Okada algebra representations
and an analogous Littlewood-Richardson identity
for non-commutative clone Schur functions.
Below, we recall Okada's clone Pieri rule, which is the simplest instance of
the clone Littlewood-Richardson rule and which recapitulates
the structure of covering relations of the $\YF$-lattice.
\begin{proposition}[{\cite[Proposition 3.3]{Okada1994}}]
\label{eq:clone-pieri}
For any $v \in \YF$ we have
\[
  s_1 \big(X + |v|\ \big|\ Y + |v| \ssp \big) \cdot s_v\XY \, = \,
  x_{|v|} \cdot s_v\XY \, = \,
  \sum_{w \gtrdot v} \ssp s_w\XY\,.
\]
\end{proposition}

The general form of the \defn{clone Littlewood-Richardson rule} is given here:
\begin{proposition}[{\cite[Theorem 4.4]{Okada1994}.}]
For any triple of Fibonacci words $u, v, w \in \YF$ such
that $|w| = |u| + |v|$  there exist non-negative integers $c_{u,v}^w$
such that
\begin{equation}
  \label{eq:clone-littlewood-richardson}
  s_u \big(X + |v|\ \big|\ Y + |v| \ssp \big) \cdot s_v\XY \, = \,
\sum_{|w| \, = \, |u| + |v|} \ssp c_{u,v}^w s_w\XY\,.
\end{equation}
Furthermore the coefficient $c_{u,v}^w$ counts certain saturated chains in
the $\YF$-lattice determined by $u$, $v$, and $w$.
\end{proposition}
One must bear in mind that the order of the factors in
\Cref{eq:clone-littlewood-richardson} matters, since the indices of
the variables occurring in the left factor are shifted by the rank of
the Fibonacci word attached to the right factor. This is a reflection
of the fundamental non-commutativity inherited from free algebra under
specialization and expresses the fact that, in contrast with the
symmetric groups, the two induced modules
\[
  \operatorname{Ind}_{\ssp \Okada[M](X,Y)\otimes\Okada[N](X,Y)}^{\ssp\Okada[M+N](X,Y)}
  \left(V^S\otimes V^T\right)
  \qandq
  \operatorname{Ind}_{\ssp \Okada[N](X,Y)\otimes\Okada[M](X,Y)}^{\ssp\Okada[N+M](X,Y)}
  \left(V^T\otimes V^S\right)
\]
are not isomorphic. The vanishing of the clone Schur functions
characterizes parameter sequences $X, Y$ for which the Okada algebra
$\Okada(X,Y)$ is semi-simple. Specfically,
\begin{theorem}[{\cite[Theorem 2.9]{Okada1994}}]
\label{thm:okada-semisimple-parameters}
$\Okada(X,Y)$ is semi-simple if and only if
$s_w\XY \ne 0$ for Fibonacci words $w \in \YF$
of rank $|w| < N$.
\end{theorem}

Recall that in the semi-simple case there is a simple  $\Okada(X,Y)$-module $V^w$
associated to each each Fibonacci word $w \in \YFn[N]$
whose basis $\Omega^w$ consists of all saturated chains
$\underline{\boldsymbol{v}} = (v_0 \lessdot \ssp \cdots \ssp \lessdot v_N)$
in the $\YF$-lattice which terminate at $v_N =w$.
The biserial clone Schur functions appear
in the formula for the action of the generators
$\E{1}, \dots, \E{N-1}$ on $V^w$ with
respect to this basis.

\begin{proposition}[{\cite[Equation 2.11]{Okada1994}}]
\label{action-generators-clone-schur}
For $w \in \YFn[N]$, let
$\varrho^w: \Okada(X,Y) \rightarrow \End(V^w)$ be the corresponding
irreducible representation, and let $\underline{\boldsymbol{v}} \in \Omega^w$ be a
saturated chain terminating at $w$, then for each $k = 1, \dots, N-1$
the action of the generator $\E{k}$ on $V^w$ is given by
\begin{equation*}
\varrho^w(\E{k}) \ssp \underline{\boldsymbol{v}} =
\begin{cases}
\displaystyle
\sum_{u \ssp \gtrdot \ssp v_{k-1}}  \frac{s_u\XY}{s_{v_{k-1}}\XY}  \,
\big( v_0 \lessdot \cdots \lessdot v_{k-1} \lessdot u \lessdot v_{k+1} \lessdot \cdots \lessdot v_N  \big)
&\text{if $v_{k+1} = 2v_{k-1}$,} \\
0 &\text{otherwise.}
\end{cases}
\end{equation*}
\end{proposition}
Again, in the semi-simple case, $V^w$ is isomorphic as an
$\Okada(X,Y)$-module to the cell module $V^S$ where
$S =\FibSet(w)$ is the Fibonacci set
of rank $N = |w|$ associated to $w$. By construction, the dimension of $V^w$
counts the number of saturated chains $\underline{\boldsymbol{v}}$
in the $\YF$-lattice terminating at $v_N = w$,
i.e. the cardinality of $\Omega^w$.
This cardinality can be obtained as a value of the corresponding biserial
clone Schur function, namely
\[ \dim V^w \, = \
 \# \ssp  \Omega^w = \,  \, s_w(X_{\scriptscriptstyle \mathrm{PL}} \, | \, Y_{\scriptscriptstyle \mathrm{PL}} )  \]

\noindent
where
$X_{\scriptscriptstyle \mathrm{PL} } = Y_{\scriptscriptstyle \mathrm{PL}} =
(1, 2, 3, \dots )$ is the \defn{Plancherel specialisation}
--- so named because of its relationship with the system of coherent
Plancherel measures supported the $\YF$-lattice (owing to its differential
structure).  Alternatively $\dim V^w$ can be computed using the
\defn{hook-length} formula for rooted trees
(see~\cite[Proposition 2.3]{Nzeutchap2009}).

\subsection{Cellular invariant bilinear forms and Gram determinants}

We introduce the following system of biserial weights on
covering relations and saturated chains of the $\YF$-lattice.

\begin{definition}
If $v \lessdot w$ is a covering relation
in the $\YF$-lattice
between Fibonacci words 
the corresponding \defn{$\gamma$-weight} is defined by
\[ \gamma_{\scriptscriptstyle X,Y} \big( v \lessdot w \big)
 \eqdef \,
 \frac{ s_v\XY}{s_u\XY} \]
where $u$ is the maximal common suffix of $v$ and $w$ (that is
either $v = 2^ku$ and $w = 2^k1u$ or
else $v= 2^k1u$ and $w= 2^{k+1}u$ for
some $k \geq 0$).
For a saturated chain $\underline{\boldsymbol{v}} = \big(v_0 \lessdot \cdots \lessdot v_N \big) $ in $\YF$ we declare that
\[ \gamma_{\scriptscriptstyle X,Y} \big( \underline{\boldsymbol{v}} \big)
 \eqdef \, \prod_{k \ssp = \ssp 1}^N  \, \gamma_{\scriptscriptstyle X,Y} \big( v_{k-1} \lessdot v_k \big) \, .
\]
\end{definition}

\begin{definition}[{\cite[Definition 2.3]{Graham1996CellularA}}]
Let $S$ be a rank $N$ Fibonacci set and let $V^S$ be the associated
cell module of $\Okada(X,Y)$.
Attached to this module is an \defn{$\Okada(X,Y)$-invariant bilinear form}
$\varphi_S: V^S \times V^S \rightarrow \mathbb{K}$
implicitly defined for two half diagrams $H, K$ with propagating sets
$\PropLab(H) = \PropLab(K) = S$ by
\[
\E{\glue{H}{H}} \bullet K \, = \,
 \varphi_S\big(H, K\big) H\,. \]
\end{definition}
In general $\varphi_S\big(H, K\big)$ will be expressed
as a monomial in the $X$ and $Y$ parameters that can be computed by the
algorithm described in~\cref{sec:coeffprod}. The {\it Gram matrix}
$G^S_N$ of the bilinear form $\varphi_S$ is the
$\dim V^S \times \dim V^S$ matrix whose entries
are given by
\[ \big[ G^S_N \big]_{H, K} \, \eqdef \, \varphi_S\big(H, K\big)\,.\]
According to \cite[Proposition 3.2]{Graham1996CellularA}, $V^S$ is
irreducible if and only if $\det G^S_N \neq0$. Here is a conjectural
formula which computes the value of this determinant:
\begin{conjecture}
\label{conj::SGram-Determinant}
Let $S \in \YFSn[N]$ be a Fibonacci set
of rank $N$ then
\begin{equation}
\label{eqn:SGram-Determinant}
\det G^S_N \, = \,
\prod_{v_0 \lessdot v_1
    \lessdot \, \cdots \, \lessdot v_N}
\gamma_{\scriptscriptstyle X,Y}
\big(v_0  \lessdot  v_1 \lessdot \,  \cdots \, \lessdot v_N \big)
 \end{equation}
 where the product is taken over all saturated chains $v_0  \lessdot  v_1
 \lessdot \,  \cdots \, \lessdot v_N$ in the $\YF$-lattice
 beginning at $v_0 = \varnothing$ and terminating at $v_N = \FibWord(S)$.
\end{conjecture}
See \cref{expl::SGram-Determinant} for an example and
\cref{app:table-Gram-det} for a table listing the Gram determinants
up to rank five.

Notice that while the $\gamma_{\scriptscriptstyle X,Y}$-weights are
rational functions of $X$ and $Y$-parameters, the matrix entries
of $G^S_N$ are monomials and so the Gram determinant $\det G^S_N$ is
always a polynomial in the $X$'s and $Y$'s.
\cref{eqn:SGram-Determinant} is an adaptation of Proposition 5.15 in
\cite{DeGier2009} where the Gram-determinant of a certain
$2^N$-dimensional module for the two-boundary Tempereley-Lieb algebra
$\boldsymbol{\mathrm{2BTL}}_N$ is expressed as a product of local edge
weights for all saturated chains terminating at level $N$ in the
corresponding Bratteli diagram (which is the Pascal lattice). Unlike
Proposition 5.15 in \cite{DeGier2009}, the individual factors
$\gamma_{\scriptscriptstyle X,Y} \big(v_0 \lessdot v_1 \lessdot \,
\cdots \, \lessdot v_N \big)$ in \cref{eqn:SGram-Determinant} are not,
in general, eigenvalues of the Gram matrix $G^S_N$.

Recall the {\it discriminant} of the Okada algebra
is determinant of the $N! \times N!$ matrix $G_N$
with entries given by
\[ \big[ G_N \big]_{\sigma, \tau}
\, \eqdef \,
\mathrm{tr} \big[ \varrho^{\mathrm{reg}}(\E{\sigma}) \varrho^{\mathrm{reg}}( \E{\tau}) \big]
\]
where $\varrho^{\mathrm{reg}}$ is the left-regular representation of $\Okada(X,Y)$
and $\sigma, \tau \in \SG$. Recently, R.P. Stanley raised the issue
of computing this discriminant in MathOverflow \cite{StanleyMathOverflow}.
We propose the following:
\begin{conjecture}
\label{conj:discriminant-factorization}
\[ \det  G_N  \, = \,
\prod_{S \, \in \, \YFSn[N]}\,  \Big[ \det G^S_N \Big]^{2 \dim(\mathrm{Fib}(S))}\,. \]
\end{conjecture}

\begin{example}\label{expl::SGram-Determinant} As an example of
  \cref{conj::SGram-Determinant} consider the Fibonacci set $S=\{3\}_\text{\bf 5}$ whose
  associated Fibonacci word is $212$. The four Okada half arc-diagrams
  with propagating labelling $S$ are shown below, together with the Gram matrix $G^S_5$
  (where the ordering of the row and columns of Gram matrix matches the order of the list of half arc-diagrams).
%   sage: R = PolynomialRing(QQ, 'x1,x2,x3,x4,y1,y2,y3')
%   sage: R.inject_variables()
%   Defining x1, x2, x3, x4, y1, y2, y3
%   sage: A = OkadaAlgebra(R, [x1, x2, x3, x4], [y1, y2, y3])
%   sage: ord, m = A.gram_matrix_of_shape((2,1,2))
%   sage: view(ord)
%   sage: ordt = tuple(YFSetChain(FibSet_of_FibWord(w) for w in chain_from_word(e.perm())) for e in ord)
\[
  \begin{pNiceMatrix}[first-row]
    \begin{tikzpicture}[ultra thick,yscale=0.4,
      baseline={(current bounding box.east)}]
      \tikzstyle{mid}=[draw,thick,black,anchor=center,fill=white,
      shape=circle,minimum size=2pt,inner sep=1pt]
      \tikzstyle{walkv}=[draw, line width=2pt]
      \foreach \i in {1,...,5} {
        \node[anchor=east] (\i) at ($(0, \i)$) {$\i$};
      }
      \draw[walkv] (1.east) .. controls +(0.50, 0.25) and +(0.50, -0.25) .. (2.east)
      node[pos=0.5,mid] {$1$};
      \draw[walkv] (3.east) .. controls +(0.50, 0.25) and +(0.50, -0.25) .. (4.east)
      node[pos=0.5,mid] {$1$};
      \draw[walkv] (5) -- ($(5.east) + (1, 0)$) node[pos=0.5,mid] {$3$};
    \end{tikzpicture}
    & %%%%%%%%%%%%%%%%%%%%%%
    \begin{tikzpicture}[ultra thick,yscale=0.4,
      baseline={(current bounding box.east)}]
      \tikzstyle{mid}=[draw,thick,black,anchor=center,fill=white,
      shape=circle,minimum size=2pt,inner sep=1pt]
      \tikzstyle{walkv}=[draw, line width=2pt]
      \foreach \i in {1,...,5} {
        \node[anchor=east] (\i) at ($(0, \i)$) {$\i$};
      }
      \draw[walkv] (1.east) .. controls +(1.30, 0.25) and +(1.30, -0.25) .. (4.east)
      node[pos=0.5,mid] {$1$};
      \draw[walkv] (2.east) .. controls +(0.50, 0.25) and +(0.50, -0.25) .. (3.east)
      node[pos=0.5,mid] {$2$};
      \draw[walkv] (5) -- ($(5.east) + (1, 0)$) node[pos=0.5,mid] {$3$};
    \end{tikzpicture}
    & %%%%%%%%%%%%%%%%%%%%%%
    \begin{tikzpicture}[ultra thick,yscale=0.4,
      baseline={(current bounding box.east)}]
      \tikzstyle{mid}=[draw,thick,black,anchor=center,fill=white,
      shape=circle,minimum size=2pt,inner sep=1pt]
      \tikzstyle{walkv}=[draw, line width=2pt]
      \foreach \i in {1,...,5} {
        \node[anchor=east] (\i) at ($(0, \i)$) {$\i$};
      }
      \draw[walkv] (1.east) .. controls +(0.50, 0.25) and +(0.50, -0.25) .. (2.east)
      node[pos=0.5,mid] {$1$};
      \draw[walkv] (3.east) .. controls +(0.50, 0.25) and +(0.50, -0.25) .. (4.east)
      node[pos=0.5,mid] {$3$};
      \draw[walkv] (5) -- ($(5.east) + (1, 0)$) node[pos=0.5,mid] {$3$};
    \end{tikzpicture}
    & %%%%%%%%%%%%%%%%%%%%%%
    \begin{tikzpicture}[ultra thick,yscale=0.4,
      baseline={(current bounding box.east)}]
      \tikzstyle{mid}=[draw,thick,black,anchor=center,fill=white,
      shape=circle,minimum size=2pt,inner sep=1pt]
      \tikzstyle{walkv}=[draw, line width=2pt]
      \foreach \i in {1,...,5} {
        \node[anchor=east] (\i) at ($(0, \i)$) {$\i$};
      }
      \draw[walkv] (1.east) .. controls +(0.50, 0.25) and +(0.50, -0.25) .. (2.east)
      node[pos=0.5,mid] {$1$};
      \draw[walkv] (3) -- ($(3.east) + (1, 0)$) node[pos=0.5,mid] {$3$};
      \draw[walkv] (4.east) .. controls +(0.50, 0.25) and +(0.50, -0.25) .. (5.east)
      node[pos=0.5,mid] {$4$};
    \end{tikzpicture}\\[10mm]
    x_{1}^{2} y_{1} y_{2} y_{3} & x_{1} y_{1} y_{2} y_{3} & x_{1}^{2} y_{2} y_{3} & 0 \\
    x_{1} y_{1} y_{2} y_{3} & x_{1} x_{2} y_{2} y_{3} & x_{1} y_{2} y_{3} & 0 \\
    x_{1}^{2} y_{2} y_{3} & x_{1} y_{2} y_{3} & x_{1} x_{3} y_{3} & x_{1} y_{3} \\
    0 & 0 & x_{1} y_{3} & x_{1} x_{4}
  \end{pNiceMatrix}\]
The determinant of $G^S_5$ is
\[
  x_1^4 \, y_2^2 \, y_3^3 \, (x_1x_2- y_1) \, (x_3x_4y_1 - x_1x_4y_2 -y_1y_3 )\,.
\]
The four saturated chains terminating at $212$ in the $\YF$-lattice,
which correspond to the Okada half arc-diagrams, are
listed below along with their  $\gamma$-weights.
The reader can check that the product of all $\gamma$-weights equals the determinant of $G^S_5$.

%interleaved with their
%$\gamma$-weights. for space reason we didn't write the first edge
%$\varnothing\lessdot(1)$ of $\gamma$-weight $1$ at the beginning nor the
%endpoint $(2,1,2)$.
%\[
%  \begin{array}{ccccccc}
%    x_{1}& \left(2\right)& 1& \left(1, 2\right)& 1& \left(1, 1, 2\right)& \frac{x_{3} x_{4} y_{1} - x_{1} x_{4} y_{2} - y_{1} y_{3}}{x_{3} y_{1} - x_{1} y_{2}} \\ % & \frac{{\left(x_{3} x_{4} y_{1} - x_{1} x_{4} y_{2} - y_{1} y_{3}\right)}x_{1}}{x_{3} y_{1} - x_{1} y_{2}}\\
%    x_{1}& \left(2\right)& 1& \left(1, 2\right)& \frac{x_{3} y_{1} - x_{1} y_{2}}{y_{1}}& \left(2, 2\right)& y_{3} \\ % & \frac{{\left(x_{3} y_{1} - x_{1} y_{2}\right)} x_{1} y_{3}}{y_{1}}\\
%    1& \left(1, 1\right)& \frac{x_{1} x_{2} - y_{1}}{x_{1}}& \left(2, 1\right)& x_{1} y_{2}& \left(2, 2\right)& y_{3}\\ % & {\left(x_{1} x_{2} - y_{1}\right)} y_{2} y_{3}\\
%    x_{1}& \left(2\right)& y_{1}& \left(2, 1\right)& x_{1} y_{2}& \left(2, 2\right)& y_{3}\\ % & x_{1}^{2} y_{1} y_{2} y_{3}
%  \end{array}
%\]

\[
\newcommand{\XYs}{\left(X\mid Y\right)}
\setlength\arraycolsep{1.0pt}
\renewcommand\arraystretch{1.6}
\begin{array}{lllllll}
\gamma_{\scriptscriptstyle X,Y}
\big( \varnothing
&\lessdot \ 1
&\lessdot \ 2
&\lessdot \ 21
&\lessdot \ 22
&\lessdot \ 212 \ \big)\ {}
&= \ \frac{s_\varnothing\XYs}{s_\varnothing\XYs}  \cdot   \frac{\, s_{1}\XYs}{s_\varnothing\XYs}
 \cdot   \frac{\, s_{2}\XYs}{s_\varnothing\XYs}
 \cdot   \frac{s_{21}\XYs}{\, s_\varnothing\XYs}
 \cdot   \frac{s_{22}\XYs}{\ s_{2}\XYs}
 \\
 & & & & & &= \ \frac{1}{1} \cdot \frac{x_1}{1} \cdot  \frac{y_1}{1} \cdot \frac{y_2 x_1}{1}  \cdot
\frac{y_3y_1}{y_1}
\\
\gamma_{\scriptscriptstyle X,Y}
 \big( \varnothing
&\lessdot  \ 1
&\lessdot  \ 11
&\lessdot  \ 21
&\lessdot  \ 22
&\lessdot  \ 212 \ \big)
&= \  \frac{ s_\varnothing\XYs}{s_\varnothing\XYs}  \cdot   \frac{ s_{1}\XYs}{ s_1\XYs}
 \cdot   \frac{s_{11}\XYs}{\  s_{1}\XYs}
 \cdot   \frac{s_{21}\XYs}{\, s_\varnothing\XYs}
 \cdot   \frac{s_{22}\XYs}{\  s_{2}\XYs}
\\
& & & & & &= \ \frac{1}{1} \cdot \frac{1}{1} \cdot  \frac{x_1 x_2 - y_1}{x_1} \cdot \frac{y_2 x_1}{1}  \cdot
\frac{y_3y_1}{y_1}
\\
\gamma_{\scriptscriptstyle X,Y}
\big( \varnothing
&\lessdot \ 1
&\lessdot \ 2
&\lessdot \ 12
&\lessdot \ 22
&\lessdot \ 212 \ \big)
&= \ \frac{s_\varnothing\XYs}{s_\varnothing\XYs}  \cdot   \frac{\, s_{1}\XYs}{s_\varnothing\XYs}
 \cdot   \frac{s_{2}\XYs}{s_{2}\XYs}
 \cdot   \frac{s_{12}\XYs}{\ s_{2}\XYs}
 \cdot   \frac{s_{22}\XYs}{\ s_{2}\XYs}
\\
& & & & & &= \ \frac{1}{1} \cdot \frac{x_1}{1} \cdot  \frac{y_1}{y_1} \cdot \frac{x_3y_1 - y_2x_1}{y_1}  \cdot
\frac{y_3y_1}{y_1}
\\
\gamma_{\scriptscriptstyle X,Y}
\big( \varnothing
&\lessdot \ 1
&\lessdot \ 2
&\lessdot \ 12
&\lessdot \ 112
&\lessdot \ 212  \ \big)
&= \ \frac{s_\varnothing\XYs}{s_\varnothing\XYs}  \cdot   \frac{\, s_{1}\XYs}{s_\varnothing\XYs}
 \cdot   \frac{s_{2}\XYs}{s_{2}\XYs}
 \cdot   \frac{s_{12}\XYs}{s_{12}\XYs}
 \cdot   \frac{s_{112}\XYs}{\ s_{12}\XYs}
\\
 & & & & & &= \ \frac{1}{1} \cdot \frac{x_1}{1} \cdot  \frac{y_1}{y_1} \cdot \frac{x_3y_1 - y_2x_1}{x_3y_1 - y_2x_1}  \cdot
\frac{x_4x_3y_1 -y_3y_1 - x_4y_2 x_1}{x_3y_1 - y_2x_1}
\end{array}
\]
\end{example}
\medskip

\subsection{Okada's Markov trace, the Jones map, and meanders}

The biserial clone Schur functions also appear in
Okada's construction
of a normalized trace $\underline{\mathrm{Tr}}^{\scriptscriptstyle (N)} \! : \Okada(X,Y) \rightarrow \mathbb{K}$
defined for elements $A \in \Okada(X,Y)$ by

\begin{equation}
\label{def:markov-trace}
\mathrm{\underline{Tr}}^{\scriptscriptstyle (N)} (A)
\eqdef \, (x_1 \cdots x_N)^{-1} \sum_{w \ssp \in \ssp \YFn}
s_w\XY \ssp \mathrm{tr} \big[ \varrho^w(A) \big].
\end{equation}
See \cite{Okada1994} for details. It is not hard to see that
\begin{itemize}
\item[(1)] $\mathrm{\underline{Tr}}^{\scriptscriptstyle (N)} (1) = 1$
\item[(2)] $\mathrm{\underline{Tr}}^{\scriptscriptstyle (N)} (AB) = \mathrm{\underline{Tr}}^{\scriptscriptstyle (N)} (BA) $
\end{itemize}

\noindent
Indeed, property (1) is just a consequence of the clone {\it Pieri rule}, i.e.
\[ x_1 \cdots x_N = \sum_{w \ssp \in \ssp \YF_N} \dim(w) \ssp s_w\XY  \]
while property (2) follows from the fact that the right hand side of
\Cref{def:markov-trace} is a linear combination of traces.
In addition, Okada showed that  $\mathrm{\underline{Tr}}^{\scriptscriptstyle (N)}$
is {\it Markovian}, in the sense that

\begin{itemize}
\item[(3)]  $\mathrm{\underline{Tr}}^{\scriptscriptstyle (N)} (A\E{N-1}) = \frac{y_{N-1}}{x_N} \ssp
\mathrm{\underline{Tr}}^{\scriptscriptstyle (N)} (A) $ whenever $A \in \Okada[N-1](X,Y)$
\item[(4)] $\mathrm{\underline{Tr}}^{\scriptscriptstyle (N)} (A) =
\mathrm{\underline{Tr}}^{\scriptscriptstyle (N-1)} (A) $ whenever $A \in \Okada[N-1](X,Y)$
\end{itemize}
Properties (1)-(4) in fact characterize
$\mathrm{\underline{Tr}}^{\scriptscriptstyle (N)}$.
The trace $\mathrm{\underline{Tr}}^{\scriptscriptstyle (N)}$ can be realized using a version
of the Jones basic contruction~\cite{Halverson1995CharactersOA} for the tower of Okada algebras.

%\begin{definition}
%The relations $\E{N} \E{k} \ssp = \ssp \E{k} \E{N}$ for $1 \leq k \leq N-2$ along with
% $\E{N} \E{N-1} \E{N} \ssp = \ssp y_{N-1} \E{N}$
% imply that for any $A \in \Okada[N](X,Y) \subset \Okada[N+1](X,Y)$ there exists a unique element
% $\partial_N(A)  \in \Okada[N-1](X,Y)$ such that $\E{N} A \E{N} \ssp =  \partial_N (A) \E{N}$.
%The  map \defn{$A \mapsto  \partial_N (A)$} from $\Okada[N](X,Y)$ to  $\Okada[N-1](X,Y)$
%is called the \defn{$N\ordinal$ Jones map}.
%\end{definition}

\begin{definition}
The relations $\E{N} \E{k} \ssp = \ssp \E{k} \E{N}$ for $1 \leq k \leq N-2$ along with
 $\E{N} \E{N-1} \E{N} \ssp = \ssp y_{N-1} \E{N}$
 imply that for any $A \in \Okada[N](X,Y) \subset \Okada[N+1](X,Y)$ there exists a unique element
 $\partial_N(A)  \in \Okada[N-1](X,Y)$ such that $\E{N} A \E{N} \ssp =  x_N \partial_N (A) \E{N}$.
The  map \defn{$A \mapsto  \partial_N (A)$} from $\Okada[N](X,Y)$ to  $\Okada[N-1](X,Y)$
is called the \defn{$N\ordinal$ Jones map}.
\end{definition}

Since $\E{N}$ centralizes $\Okada[N-1](X,Y)$ within $\Okada[N+1](X,Y)$ it follows that
$\partial_N(A) = A$ whenever $A \in  \Okada[N-1](X,Y) \subset \Okada[N](X,Y)$ and so
$\partial_N :\Okada[N](X,Y) \rightarrow
\Okada[N-1](X,Y)$ is both a left and right $\Okada[N-1](X,Y)$-module
epimorphism.
This construction makes sense for the Okada monoid $\Okada[N]$
and (abusing notion) we get a $\Okada[N-1]$ equivariant projection
$\partial_N : \Okada[N] \rightarrow \Okada[N-1]$ .
Seen from this perspective, the Jones map is a monoid counterpart of the Vershik-Okounkov $\SG[N-1]$ equivariant
projection $p_N : \SG[N] \rightarrow \SG[N-1]$ where $p_N(\sigma)$
is the permutation in $\SG[N-1]$ obtained by removing $N$ from the disjoint cycle decomposition
of $\sigma \in \SG[N]$ (see~\cite{Okounkov1996}). Just as \defn{virtual permutations} are defined using the
the inverse system $\varprojlim\SG[N]$, so too are \defn{virtual Okada arc-diagrams}
defined using the inverse system $\varprojlim\Okada[N]$ associated to the
Jones maps.
The following proposition relates the Jones maps with Okada's trace.

\pagebreak[2]
\begin{proposition} $\underline{\mathrm{Tr}}^{\scriptscriptstyle (N)}  (A) =  \partial_1 \circ \cdots \circ \partial_N (A)$
for any $A \in \Okada(X,Y)$\,.
\end{proposition}
\begin{proof}
We proceed by induction on $N$. Let $\tau_N(A) \eqdef  \partial_1 \circ \cdots \circ \partial_N (A)$.
It's easy to check that $\tau_N(A) = \underline{\mathrm{Tr}}^{\scriptscriptstyle (N)} (A)$ when
$N=1$. Now assume that $\tau_M(A) = \underline{\mathrm{Tr}}^{\scriptscriptstyle (M)} (A)$
for all $M < N$ with $N \geq 2$. By induction we have
\[ \tau_N(A) \, = \,  \tau_{N-1} \big( \partial_N(A ) \big) \, = \,
\underline{\mathrm{Tr}}^{\scriptscriptstyle (N-1)}
\big( \partial_N(A) \big)\,.  \]
According to \cref{prop:right-code-factor} we may uniquely express $A$ as 
\[ A \, = \, A_0 \ + \ \sum_{k \ssp = \ssp 1}^{N-1} \,  A_k \ssp \E{N-1} \E{N-2} \cdots \E{N-k} \]
for some elements $A_k \in \Okada[N-1](X,Y)$ with $0 \leq k \leq N-1$. Therefore
\[ \partial_N (A) \, = \, A_0 \ + \ {y_{N-1} \over {x_N}} \ssp A_1 \  + \ 
{y_{N-1} \over {x_N}} \ssp \sum_{k \ssp = \ssp 2}^{N-1} \, A_k  \ssp \E{N-2} \cdots \E{N-k}   \]
and so
\[ \tau_N(A) \, =  \,
%\underline{\mathrm{Tr}}^{\scriptscriptstyle (N-1)}  \big( \partial_N(A) \big)
\underline{\mathrm{Tr}}^{\scriptscriptstyle (N-1)} (A_0)  \ + \ {y_{N-1} \over {x_N}} \, \underline{\mathrm{Tr}}^{\scriptscriptstyle (N-1)} (A_1) \ + \ 
{y_{N-1} \over {x_N}} \ssp \sum_{k=2}^{N-1}  \, \underline{\mathrm{Tr}}^{\scriptscriptstyle (N-1)}  \big( A_k \ssp \E{N-2} \cdots \E{N-k} \big).  \]
On the other hand
\bigskip
\bigskip
\[ \begin{array}{ll}
\underline{\mathrm{Tr}}^{\scriptscriptstyle (N)} (A)
&=  \,  \underline{\mathrm{Tr}}^{\scriptscriptstyle (N)} (A_0) \, + \,
\sum_{k \ssp = \ssp 1}^{N-1} \, \underline{\mathrm{Tr}}^{\scriptscriptstyle (N)} \big( A_k \ssp \E{N-1} \cdots \E{k} \big) \\ \\
&=  \,  \underline{\mathrm{Tr}}^{\scriptscriptstyle (N)} (A_0) \, + \, \underline{\mathrm{Tr}}^{\scriptscriptstyle (N)} (A_1 \E{N-1})  \, + \,
\sum_{k \ssp = \ssp 2}^{N-1} \, \underline{\mathrm{Tr}}^{\scriptscriptstyle (N)} \big(\E{N-2} \cdots \E{k} A_k \ssp \E{N-1} \big) \\ \\
&=  \,  \underline{\mathrm{Tr}}^{\scriptscriptstyle (N-1)} (A_0) \, + \, {y_{N-1} \over {x_N}} \ssp \underline{\mathrm{Tr}}^{\scriptscriptstyle (N-1)} (A_1)  \, + \,
{y_{N-1} \over {x_N}} \ssp \sum_{k \ssp = \ssp 2}^{N-1} \, \underline{\mathrm{Tr}}^{\scriptscriptstyle (N-1)} \big(\E{N-2} \cdots \E{k} A_k \big) \\ \\
&=  \,  \underline{\mathrm{Tr}}^{\scriptscriptstyle (N-1)} (A_0) \, + \, {y_{N-1} \over {x_N}} \ssp \underline{\mathrm{Tr}}^{\scriptscriptstyle (N-1)} (A_1)  \, + \,
{y_{N-1} \over {x_N}} \ssp \sum_{k \ssp = \ssp 2}^{N-1} \, \underline{\mathrm{Tr}}^{\scriptscriptstyle (N-1)} \big(A_k \E{N-2} \cdots \E{k} \big) \\ \\
\end{array}    \]
and we are done.
\end{proof}

Finally, there is a diagrammatic way to interpret
$\mathrm{\underline{Tr}}^{\scriptscriptstyle (N)}(\A{D})$ for an Okada
arc-diagram $D$.  The \defn{height-labelled meandric system} associated
to a rank $N$ Okada arc-diagram $D$ is the system of loops obtained by
placing $D$ on the forward facing part of a vertical cylinder and
linking each of its endpoints $a$ and $\ov{a}$ by an arc of \emph{height
$a+1$} passing behind the cylinder. \Cref{fig:trace} illustrates this construction
where dashed arcs pass behind the cylinder
(see also~\cite[Figure 10]{Di_Francesco1998-ot}). This creates a
collection of loops each of which contributes a monomial weight in the
$X,Y$-parameters as in \cref{theorem-coeff-mult}. Define
\begin{equation}
\label{def:meander-trace}
\mathrm{Tr}^{\scriptscriptstyle (N)}(A) \eqdef \,
(x_1 \cdots x_N) \ssp \mathrm{\underline{Tr}}^{\scriptscriptstyle (N)}(A)\,,
\end{equation}
then the value of $\mathrm{Tr}^{\scriptscriptstyle (N)}(\E{D})$
is just the product of these loop monomials.
Furthermore $\mathrm{Tr}^{\scriptscriptstyle (N)}(A)$
coincides with the standard trace $\mathrm{tr}[ \varrho^{\mathrm{reg}}(A)]$
associated to the left-regular representation $\varrho^{\mathrm{reg}}$
of $\Okada(X,Y)$ when $X_{\scriptscriptstyle \mathrm{PL} } = Y_{\scriptscriptstyle \mathrm{PL}} =
(1, 2, 3, \dots )$ is the Plancherel specialization.

\begin{figure}[ht]
  \[
  \begin{tikzpicture}[scale=0.6, thick, baseline={(current bounding box.east)}]
\tikzstyle{vertex} = [shape=rectangle, minimum height=15pt, minimum width=0pt, inner sep=0pt]
\tikzstyle{mid} = [draw, fill=white, shape=circle, minimum size=2pt, inner sep=1pt]
\foreach \i in {1,...,7} {
  \node[vertex] (G-\i) at ($(-1.44, \i*1.3-1.3)$) {$\i$};
  \node[vertex] (G--\i) at ($(2.64, \i*1.3-1.3)$) {$\ov{\i}$};
}

\foreach \i / \col [evaluate=\i as \j using int(\i+1)] in {1/black,2/black,3/blue,4/red,5/red,6/blue,7/blue} {
  \draw[color=\col] (G-\i)  -- +(-0.2, 0) .. controls +(-0.5, 0) and +(-0.5, 0)
  .. +(-0.2, +0.8) node[pos=0.5] (C-\j) {};
  \draw[color=\col] (G--\i) -- +(+0.2, 0) .. controls +(+0.5, 0) and +(+0.5, 0)
  .. +(+0.2, +0.8);
}
\foreach \i / \col [evaluate=\i as \j using int(\i+1)] in {1/black,2/black,3/blue,4/red,5/red,6/blue,7/blue} {
  \draw[color=\col, dashed] (G-\i) +(-0.2,+0.8) -- ($(G--\i) +(+0.2,+0.8)$)
  node[pos=0.5] (M-\j) {};
}
\foreach \j / \col in {2/black,3/black,4/blue,7/blue,8/blue} {
  \node[color=\col,mid, dashed] at (M-\j) {$\j$};
}
\node[color=red,mid, dashed] at ($(G-5)!0.25!(G--5) +(+0.2,+0.8)$) {$6$};
\node[color=red,mid] at (C-5) {$5$};

\draw[blue] (G--7) .. controls +(-2.59, -1.30) and +(1.63, 0.82) .. (G-3) node[pos=0.6470588235294118,mid] (M-7-3) { 1 };
\draw[red] (G-4) .. controls +(0.70, 0.35) and +(0.70, -0.35) .. (G-5) node[pos=0.5,mid] (M4-5) { 2 };
\draw (G--2) .. controls +(-0.70, -0.35) and +(-0.70, 0.35) .. (G--1) node[pos=0.5,mid] (M-2--1) { 1 };
\draw[blue] (G-6) .. controls +(0.70, 0.35) and +(0.70, -0.35) .. (G-7) node[pos=0.5,mid] (M6-7) { 4 };
\draw[blue] (G--6) .. controls +(-2.10, -1.05) and +(-2.10, 1.05) .. (G--3) node[pos=0.5,mid] (M-6--3) { 3 };
\draw (G-1) .. controls +(0.70, 0.35) and +(0.70, -0.35) .. (G-2) node[pos=0.5,mid] (M1-2) { 1 };
\draw[red] (G--5) .. controls +(-0.70, -0.35) and +(-0.70, 0.35) .. (G--4) node[pos=0.5,mid] (M-5--4) { 4 };
\end{tikzpicture} \quad
\begin{array}{c}
  \begin{array}{cc}
    \text{walk} & \text{weight} \\[2mm]
  {(1,\U^2\D^2\U\D,1)} & x_{1}y_{1}\\
  \blue{(1,\U^7\D^4\U^3\D^4\U\D^3,1)} & x_{1}y_{2}y_{3}y_{4}y_{5}y_{6}\\
  \red{(2,\U^4\D^2\U\D^3,2)} & x_{2}y_{3}y_{4}
  \end{array} \\[1.5cm]
  \text{total}\\[2mm]
x_{1}^2x_{2}y_{1}y_{2} y_{3}^{2} y_{4}^{2} y_{5} y_{6}
\end{array}
\]
  \caption{Computating of $\mathrm{Tr}^{\scriptscriptstyle (N)}(\E{D})$ using the meandric
  system of $D$.}
  \label{fig:trace}
\end{figure}

There is an associated $N! \times N!$ Gram matrix $G^{\ssp \scriptscriptstyle \mathrm{Tr} }_N$
with entries given by
\[ \big[ G^{\ssp \scriptscriptstyle \mathrm{Tr} }_N \big]_{C, D}
\, \eqdef \,
\mathrm{Tr}^{\scriptscriptstyle (N) } \big[ \E{C} \E{D} \big]
\]
whose determinant $\det G^{\ssp \scriptscriptstyle \mathrm{Tr} }_N$
we conjecture factorizes into biserial clone Schur functions
similar to \Cref{conj::SGram-Determinant}
and \Cref{conj:discriminant-factorization}.

\section{Follow-up projects}
We start by presenting two short term projects we are actively working on.
\subsection{Truncations and higher blob algebras}
For a fixed choice of threshold $k \geq 1$, we can \defn{truncate} any
Okada arc-diagram $D$, replacing its height labels $h$ by
$\min(h, k)$. The $k$-truncated Okada arc-diagrams form a
multiplicative basis for a \defn{higher blob algebra}
$\mathrm{Blob}^{(k)}_N$, which can be realized as a quotient of the
Okada algebra $\Okada(X,Y)$ after specializing the $X,Y$ parameters
appropriately. In particular the Temperely-Lieb and Martin-Saleur blob
algebras~\cite{Martin1994a} are recovered for $k = 1,2$
respectively. The corresponding Bratelli diagram $\YF^{(k)}$ naturally
embeds into the $\YF$-lattice and can be seen as a Fibonacci
counterpart of the sublattice of integer partitions with at most $k$
parts. Furthermore, most issues and questions raised in the present
work extend naturally to higher blob algebras. We are currently working on a
draft paper dealing with these truncations.

\subsection{Height labels and diagram algebras}

It should be possible to incorporate height labels into other diagram
algebras such as the partition and Brauer algebras. Izumi Watanabe has
made first steps in studying the case of the symmetric group and its
representation theory~\cite{Watanabe2026}; more work is being done on Brauer
algebras. Going even further, can one define a suitable notion of
height labelled {\it braids} together with {\it skein relations}
consistent with these labels? A satisfactory answer might shed light
on the problem of identifying appropriate {\it Jucys-Murphy}
elements for the Okada algebras; see \cref{subsection:jucys-murphy-theory}
below.

\section{Prospectives}

We conclude with several long term prospectives.
\subsection{Schur-Weyl duality}
At its simplest level, classical \defn{Schur-Weyl} duality is
the observation that the simultaneous actions of $\SG$ and $\mathrm{GL} (V)$ on
$V^{\otimes N}$ commute with one another and that
\begin{equation}
\label{eq:classical-schur-weyl-duality}  \C[\SG] \cong \End_{\ssp \mathrm{GL}(V) } \big( V^{\otimes N} \big)
\end{equation}
where $V$ is a $\C$-linear space of dimension $\dim(V) \gg N$.
Likewise the action of $\mathrm{GL}(V)$ on $V^{\otimes N}$ factors through
$\End_{\SG} \big( V^{\otimes N} \big)$ while the categories of finite dimensional representations of
$\mathrm{GL}(V)$ and $\bigoplus_{N \geq 0}  \End_{\SG} \big( V^{\otimes N} \big)$
are equivalent.

A more
nuanced view of this relationship can formalized as a
statement about the decomposition
of $V^{\otimes N}$ as a $\mathrm{GL}(V) \times \SG$ representation,
specifically
\begin{equation}
\label{eq:nuanced-classical-schur-weyl-duality} V^{\otimes N} \, \cong \,
\bigoplus_{\lambda \vdash N} \ssp  \mathbb{S}_\lambda (V)
\otimes W_\lambda
\end{equation}
where $W_\lambda$ and $\mathbb{S}_\lambda(V)$ are, respectively, the irreducible
representations of $\SG$ and $\mathrm{GL}(V)$ corresponding to an integer
partition $\lambda \vdash N$. An immediate algebraic consequence
of \Cref{eq:nuanced-classical-schur-weyl-duality} is that the Schur function $s_\lambda$,
understood as the {\it character} of $\mathbb{S}_\lambda(V)$, must branch following the
covering relations of the Young lattice $\mathbb{Y}$, i.e.
$s_\Box \ssp s_\lambda = \sum_{\lambda \ssp \subset \!\!\! \boldsymbol{\cdot} \ \mu} s_\mu$
where $s_\Box$ is the character of $V = \mathbb{S}_\Box(V)$.
\Cref{eq:nuanced-classical-schur-weyl-duality} also tells us that
\begin{equation}
\label{eq:RSK-numerology}
\displaystyle s_{\Box}^N  \, = \,
\sum_{\lambda \vdash N} \,
s_\lambda  \dim W_\lambda
\end{equation}
which, from a combinatorial standpoint, implies
the celebrated RSK-correspondence
between positive integer sequences of length $N$ and
pairs of {\it semi-standard} and {\it standard tableaux}
sharing a common shape of size $N$.

A central question, which remains unresolved, is whether
some kind of \defn{Schur-Weyl} duality holds for the tower of Okada algebras
and the $\YF$-lattice. As in the classical case, this question is
both algebraic and combinatorial in nature. Because the $\mathbb{Y}$ and $\mathbb{YF}$-lattices
agree up to rank four, there are incidental isomorphisms
$\Okada(X,Y) \cong \C [ \SG ]$ whenever
$1 \leq N \leq 4$. For this reason
$\Okada(X,Y)$ can be identified with the $\mathrm{GL}(V)$ intertwinor algebra
of $V^{\otimes N}$ within the range  $1 \leq N \leq 4$.
When $N=5$ we have a strange isomorphism
\[  \Okada[5](X,Y) \, \cong \, \End_{\ssp \mathrm{GL}(V)} \Big( V^{\otimes 4} \otimes V^* \Big)  \]
which is valid provided $\dim(V) \geq 5$. However attempts to realize $\Okada(X,Y)$ as a $\mathrm{GL}(V)$ intertwinor algebra beyond
the threshold $N > 5$ have so far failed.

The presence of clone Schur functions, and the fact
that they branch following the covering relations of the $\mathbb{YF}$-lattice,
is an indication that Schur-Weyl duality might exist in some form.
Clone Schur functions, however, are not {\it characters} of anybody we
know yet (e.g. groups, algebras, etc.) and this frustrates attempts
to identify partners of the Okada algebra under such a duality.
 While the classical Schur functions
are generating functions for semi-standard tableaux,
clone Schur functions are not generating functions
of any obvious combinatorial denomination.
As indicated in \cref{subsection:YF-tableaux}, there are $\mathbb{YF}$-versions
of both standard and semi-standard tableaux. Nevertheless clone Schur functions
are not generating functions for any of these families of tableaux.
Indeed, clone Schur functions are expressed as products of tridiagonal determinants
which involve terms with signs and, as such, are not manifestly positive.
It is non-trivial even to determine specializations of the parameters
$X = (x_1, x_2, x_3, \dots)$ and $Y = (y_1, y_2, y_3, \dots)$ which ensure
that the clone Schur functions $s_w\XY$ are simultaneously
{\it positive} for all $w \in \mathbb{YF}$ (see~\cite{Petrov_Scott_2026}).

Since the classical RSK-correspondence is a combinatorial
avatar of Schur-Weyl duality, one might try
to discern the existence of a $\YF$-variant of the Schur-Weyl duality
by probing an appropriate version of the
RSK-correspondence for the $\YF$-lattice.
Such a correspondence exists
and can be formalized using, for example, Nzeutchap's theory
of standard and semi-standard $\YF$-tableaux.
However troubles present themselves in this approach.
One difficulty is that Nzeutchap's variant of the RSK-correspondence is only an {\it injection}, mapping 
positive integer sequences of length $N$ to pairs
of semi-standard and standard $\YF$-tableaux
sharing a common shape in $\YF_N$. Practically this
means that $1 \leq \mathrm{mult}(Q) \leq \dim(w)$
where $Q$ is a semi-standard Young-Fibonacci tableau of shape $w \in \YF$
and $\mathrm{mult}(Q)$ is the number of standard Fibonacci
tableaux $P$ of shape $w$ for which the pair $P \times Q$ lies in
the image of Nzeutchap's RSK-map.

At this point two obvious generating functions
$\mathbcal{s}_w$ and $\mathbcal{s}_w'$ in $\C[t_1, t_2, t_3, \dots ]$
can be introduced, both of which reproduce to the standard
description of the classical Schur function $s_\lambda$ as a generating function
of semi-standard tableaux of shape $\lambda$, but which
are distinct in the Young-Fibonacci setting, namely:

\[   \mathbcal{s}_w  \, \eqdef \,
\sum_{Q} \,  \prod_{k \geq 1} t_k^{i_k(Q)} \qandq
\mathbcal{s}_w'  \, \eqdef \,
\sum_{Q} \,  \frac{\mathrm{mult}(Q)}{ \dim(w) } \ssp \prod_{k \geq 1} t_k^{i_k(Q)}\,.
\]
Both sums are taken over semi-standard $\YF$-tableaux $Q$
of shape $w \in \YF$ and the statistic $i_k(Q)$ counts the number of times $k$ occurs in $Q$.
These generating functions are in fact {\it quasi-symmetric} owing to
{\it stability} properties of the recording tableau in Nzeutchap's RSK-correspondence.
The function $\mathbcal{s}_w'$ is probably more interesting since
it registers the defect in surjectivity of the RSK-map.
Unfortunately, a Pieri rule governing the expansion of either
 $\mathbcal{s}_1 \cdot \mathbcal{s}_w$ or $\mathbcal{s}_1' \cdot \mathbcal{s}_w'$
is not presently known, and neither is expected to express the covering relations
of the $\YF$-lattice.

\subsection{Jucy-Murphy Theory}
\label{subsection:jucys-murphy-theory}
\newcommand{\BChain}[1]{\mathrm{\bf \underline{#1}}}
\newcommand{\GT}[1]{\mathbcal{g}_{\scriptscriptstyle\BChain{#1}}}
As with any tower of semi-simple algebras, one can define a $\YF$-version of
the \defn{Gelfand-Tsetlin} algebra $\mathrm{\bf GTO}_N(X,Y)$
to be the subalgebra of $\Okada(X,Y)$ generated by
the centers $\mathrm{\bf Z}\Okada[k](X,Y)$ for $1 \leq k \leq N$.
General results of Okounkov-Vershik~\cite{Okounkov1996}
tell us that $\mathrm{\bf GTO}_N(X,Y)$ is a maximal commutative subalgebra
of $\Okada(X,Y)$ with a distinguished basis
$\{\GT{S}\}_{\BChain{S} \ssp \in \ssp \Omega_N}$
indexed by saturated chains $\BChain{S} = S_0 \lessdot \ssp  \cdots \lessdot S_N$
in the $\YFS$-lattice and having the form
$\mathbcal{g}_{\scriptscriptstyle\BChain{S}} = \mathcal{z}_{S_0} \cdots \ssp \mathcal{z}_{S_N}$
where $\mathcal{z}_S$ denotes the central idempotent of $\Okada[k](X,Y)$ corresponding to
the Fibonacci set $S \in \YFS_k$.

One outstanding problem is to identify an (additive) system of so-called \defn{Jucys-Murphy elements}
$\{ J_N \}_{N \geq 1}$ with $J_N \in \Okada(X,Y)$
together with \defn{content} values $\mathrm{\bf c}(S \lessdot T) \in \C$
for each covering relation $S \lessdot T$ in the $\YFS$-lattice
such that the following four characteristic properties hold:
\begin{enumerate}
\item \label{JM1} the Jucys-Murphy elements pairwise commute
\item \label{JM2} the sum $J_1 + \cdots + J_N \in \mathrm{\bf Z}\Okada(X,Y)$ for all $N \geq 1$
\item \label{JM3} $J_k \cdot \GT{S} = \mathrm{\bf c} (S_{k-1}  \lessdot S_k ) \ssp \GT{S}$
\item \label{JM4} the map $\BChain{S} \mapsto \big( \mathrm{\bf c}( S_{0} \lessdot S_1),  \mathrm{\bf c}( S_1 \lessdot S_2), \dots,
\mathrm{\bf c}( S_{N-1} \lessdot S_N) \big)$ is one-to-one.
\end{enumerate}
Once \cref{JM1,JM2,JM3,JM4} are met, the results
of~\cite{DotyLauveSeelinger2018} imply that the Jucys-Murphy elements
$J_1, \dots, J_N$ will generate $\mathrm{\bf GTO}_N(X,Y)$ as an
algebra. Furthermore, if
$\mathbcal{a}_N = \sum_{S \ssp \in \ssp \YFS_N} a_S \ssp
\mathcal{z}_S$ is a sufficiently generic central element in
$\mathrm{\bf Z}\Okada(X,Y)$ for each $N$ such that the values
$\mathrm{\bf c}(S \lessdot T) \eqdef a_T - a_S$ satisfy \cref{JM4} then
$J_N \eqdef \mathbcal{a}_N - \mathbcal{a}_{N-1}$ will automatically
satisfy \cref{JM1,JM2,JM3}. This approach can be used to recover the
classical Jucys-Murphy elements and content values for the tower of
symmetric groups: In this case
$\mathbcal{a}_N \in \mathrm{\bf Z} \C[\SG]$ is the central indicator
function $\delta_{21^{N-2}} \eqdef \sum_{1 \leq i < j \leq N} \ssp (i,j)$
where $(i,j) \in \SG$ denotes the transposition switching
$i, j \in \setN$. The attending contents values
$\mathrm{\bf c}(\lambda \YoungCover \mu)$ can be expressed in terms of
character values.

The problem, as stated, of finding generators of this kind is
underdetermined and some care is needed in selecting a natural system
of Jucys-Murphy elements which conforms best with constraints coming
from the character theory and cellular structure of the Okada algebras
as well as from Okada's clone theory. Clearly accessing any system of
Jucys-Murphy elements entails understanding the Gelfand-Tsetlin basis,
in particular the expansion of each $\GT{S}$ element in the diagram
basis. Based on calculations, we've found that this expansion
exhibits a lot of structure at the combinatorial level.

\subsection{Random ensembles of permutations and Okada monoid elements}

Special families of \defn{coherent measures} $\{M_N \}_{N \geq 0}$
on the $\YF$-lattice are constructed in~\cite{Petrov_Scott_2026}
using \defn{Fibonacci positive} specializations of the $X$ and $Y$ parameters
(i.e. certain values of the $X$ and $Y$ parameters which ensure that the biserial clone Schur functions
$s_w\XY$ are simultaneously
{\it positive} for all Fibonacci words $w \in \mathbb{YF}$).
These measures can be transported to probability distributions on $\SG$ for each $N \geq 0$
by means of Fomin's RS-correspondence for the $\YF$-lattice.
The uniform distribution on $\SG$, for example, corresponds to the coherent system
arising from the Plancherel specialization
$X_{\scriptscriptstyle \mathrm{PL} } = Y_{\scriptscriptstyle \mathrm{PL}} =
(1, 2, 3, \dots )$ of clone Schur functions.

Permutations $\sigma \in \SG$
are synonymous with Okada monoid elements $\e{\sigma} \in \Okada$,
and so the models introduced by~\cite{Petrov_Scott_2026}
for random permutations equally make sense
as models for random Okada monoid elements.
Rather than tracking statistics specific to permutations
(such as the number of right descents of $\sigma \in \SG$) one can instead examine
statistics which are native to Okada monoid elements. 
For example, it is easy to see that the expected number of left $\Okada$-{\it descents} 
of a random monoid element $\e{\sigma}$
drawn uniformly from $\Okada$ is
\[ {\partial \over {\partial t}} \ssp \bigg|_{t = 1} a_N(t)/N! \]
while the probability that $\e{\sigma}$ has $d$ left $\Okada$-descents 
is the coefficient of $t^d$ in $a_N(t)/N!$
where $a_N(t)$ is the {\it Okada-Eulerian}
polynomial addressed in \cref{prop:Okada-Eulerian-recurrence}
and \Cref{cor:tridiagonal-determinant-okada-eulerian}.

Some statistics, such as {\it length}, coincide for both
permutations and Okada monoid elements
while other statistics, such as length or the number of {\it fixed-points} of a permutation,
are simply better adapted to Fomin's RS-correspondence than to the classical
RS-correspondence. It would be interesting to calculate the corresponding expectation values
for those models of random permutations/diagrams coming from particularly
well-behaved deformations of the Plancherel specialization
$X_{\scriptscriptstyle \mathrm{PL} } = Y_{\scriptscriptstyle \mathrm{PL}}$,
such as the {\it Charlier}, {\it shifted-Plancherel}, and {\it shifted-Charlier} specializations
considered in ~\cite{Petrov_Scott_2026}.

% \todo{I think I understand what you mean by transitive below, but I
%   don't think this is a proper use of the word,}
% Unlike multiplication in $\SG$, the monoid product in $\Okada$ is not
% transitive, and consequently a random element $\e{\sigma} \in \Okada$ obtained
% by taking the product $\e{\tau_1} \e{\tau_2}$ of
% two independent and uniformly distributed elements $\e{\tau_1}, \e{\tau_2} \in \Okada$
% is itself not uniformly distributed. Clearly the
% relevant distribution is determined by counting the
% number of ordered, monoid factorizations of the element $\e{\sigma} \in \Okada$.
% This problem can be studied in general,
% taking the monoid product of two elements
% $\e{\tau_1}, \e{\tau_2} \in \Okada$
% sampled with respect to a pair of independent
% probability measures coming from
% the Petrov-Scott construction.

\subsection{Connections with statistical mechanics}
\subsubsection{Razumov-Stroganov conjecture(s)}
Consider the $\delta =1$ specialization of the Temperley-Lieb algebra $\mathrm{TL}_{N}(\delta)$
and assume for simplicity's sake that $N$ is even.
Following prescriptions outlined in~\cite{Pearce2002, DeGier2005}
\[ \mathrm{H}^{\scriptscriptstyle \mathrm{TL}} = \sum_{k=1}^{N-1} \big(1 - \tl{k} \big) \]
defines a {\it loop energy operator}
acting, by diagram multiplication, on the space $V_N$ spanned by all (unlabelled) half-diagrams with $N$ vertices
and no propagating arcs. Here $\tl{1}, \dots, \tl{N-1}$
are the generators of $\mathrm{TL}_N(\delta)$.
The operator $\mathrm{H}^{\scriptscriptstyle \mathrm{TL}}$ determines an {\it intensity matrix}
with respect to the basis of half-diagrams,
meaning that the entries of the operator satisfy $\mathrm{H}_{A, B}^{\scriptscriptstyle \mathrm{TL}} \leq 0$ whenever $A \ne B$ and
$\sum_{B \ne A} \mathrm{H}_{A,B}^{\scriptscriptstyle \mathrm{TL}} = 0$ for all half-diagrams $A$. These properties ensure that
the system of ODEs
\[ \partial_t \ssp P_N(t) \ssp = \ssp  -\mathrm{H}^{\scriptscriptstyle \mathrm{TL}} P_N(t) \quad \text{and its solution} \quad P_N(t) = \sum_{D \in V_N} c_D(t) D  \]
govern a continuous time Markov chain on the set of half-diagrams with no propagating arcs,
whose stationary probability distribution is given by
\[ D \mapsto \lim_{t \rightarrow \infty} c_D(t) \ssp \| P_N(t) \|^{-1}. \]
In this incarnation, the Razumov-Stroganov conjecture stipulates that
$\lim_{t \rightarrow \infty} c_D(t)$ equals the number of fully-packed
loop configurations in a $(N-1) \times  N/2$ rectangular grid
whose non-crossing link pattern is $D$. An analogous
loop energy operator, along with an appropriate variant the Razumov-Stroganov conjecture, also makes sense for the Martin-Saleur blob
algebra $\mathrm{Blob}_N(\gamma,\delta)$ (see~\cite{DeGier2004}).

It's tempting to ask whether there are real parameters (non-negative integers?)  $\varrho_1, \dots, \varrho_{N-1} \geq 0$
and a corresponding loop energy operator
$\mathrm{H} = \sum_{k=1}^{N-1} \varrho_k \big(1 - \e{k} \big)$
which drives a natural
stochastic process within the Okada
monoid $\Okada$.
There is already a well-defined concept
of {\it labelled} fully-packed loop configurations (FPLCs)
as put forth in \cref{subsection:loop-configurations}.
There is also an obvious combinatorial
problem of counting the number
of labelled FPLCs $\mathcal{D}$ of fixed height $N$ and width $L$
whose labelled isotopy class $[\mathcal{D}]$ coincides with
a given Okada arc-diagram $D$ of rank $N$.
In light of this, can one formulate a version of the
Razumov-Stroganov conjecture
which relates the
stationary probability distribution
associated to $\mathrm{H}$
and  this counting problem for
{\it labelled}
fully-packed loop configurations?

\subsubsection{Meander determinants}
There is a trace $\mathrm{Tr}^{\scriptscriptstyle (N)}_{\scriptscriptstyle \mathrm{TL}}$ on the Temperley-Lieb algebra
$\mathrm{TL}_{N}(\delta)$ analogous to
Okada's (un-normalized) Markov trace
$\mathrm{Tr}^{\scriptscriptstyle (N)}$
which obeys an unlabelled version of the
{\it meandric} construction for
$\mathrm{Tr}^{\scriptscriptstyle (N)}$
illustrated in \cref{fig:trace}. Given
an (unlabelled) non-crossing arc-diagram $D$
of rank $N$
the value of
$\mathrm{Tr}^{\scriptscriptstyle (N)}_{\scriptscriptstyle \mathrm{TL}}(D)$
is $\delta^{\kappa(D)}$ where $\kappa(D)$ is the number of
connected components (loops) formed
by placing $D$ on a cylinder and  identifying
the vertices on its left and right hand sides.
A corresponding {\it meander determinant}
is obtained by taking the
determinant of the Gram matrix $\Gamma_N(\delta)$ of the
associated pairing, i.e. $[\Gamma_N]_{A,B} (\delta)\eqdef
\mathrm{Tr}^{\scriptscriptstyle (N)}_{\scriptscriptstyle \mathrm{TL}}(A \boldsymbol{\cdot} B^\star)$
where $B \mapsto B^\star$ is the mirror-involution.
P. Di Francesco has shown in~\cite{Di_Francesco1998-ot}
that the meander determinant admits a factorization
into powers of Chebyshev polynomials of the first kind, specifically
\[  \det \Gamma_N(\delta) \, = \, \prod_{k=1}^N \ssp U_k^{a_k}(\delta) \]
where each exponent is expressed as a difference of
$a_k = c_{2N,2k} - c_{2N,2k+2}$
and where $c_{n,h}$ counts the number of
(unlabelled) half-diagrams with $n$ vertices and
$h$ propagating arcs. Similar results hold
for the blob and two-boundary Temperley-Lieb algebras;
we point readers to the survey~\cite{Jacobsen2008}.

We expect the determinant of the
Gram matrix $G^{\ssp \scriptscriptstyle \mathrm{Tr} }_N$
of Okada's un-normalized trace to factor
into powers of biserial clone Schur functions $s_w^{a_w}\XY$
for Fibonacci words $w$ of rank $|w| < N$. If this is indeed
the case, it would be helpful to understand what the
exponents $a_w$ are counting. Furthermore, the relationship
between this Okada variant of the meander determinant and the
discriminant of the Okada algebra $\Okada(X,Y)$,
which also factorizes according to \cref{conj:discriminant-factorization},
remains unclear.

\printbibliography
\label{sec:biblio}

\newpage

\appendix
\section{Example of the bijection \texorpdfstring{$\SG\to\Okada$}{between permutations and Okada arc-diagrams} }
\begin{figure}[ht]
  \includegraphics[height=\textwidth]{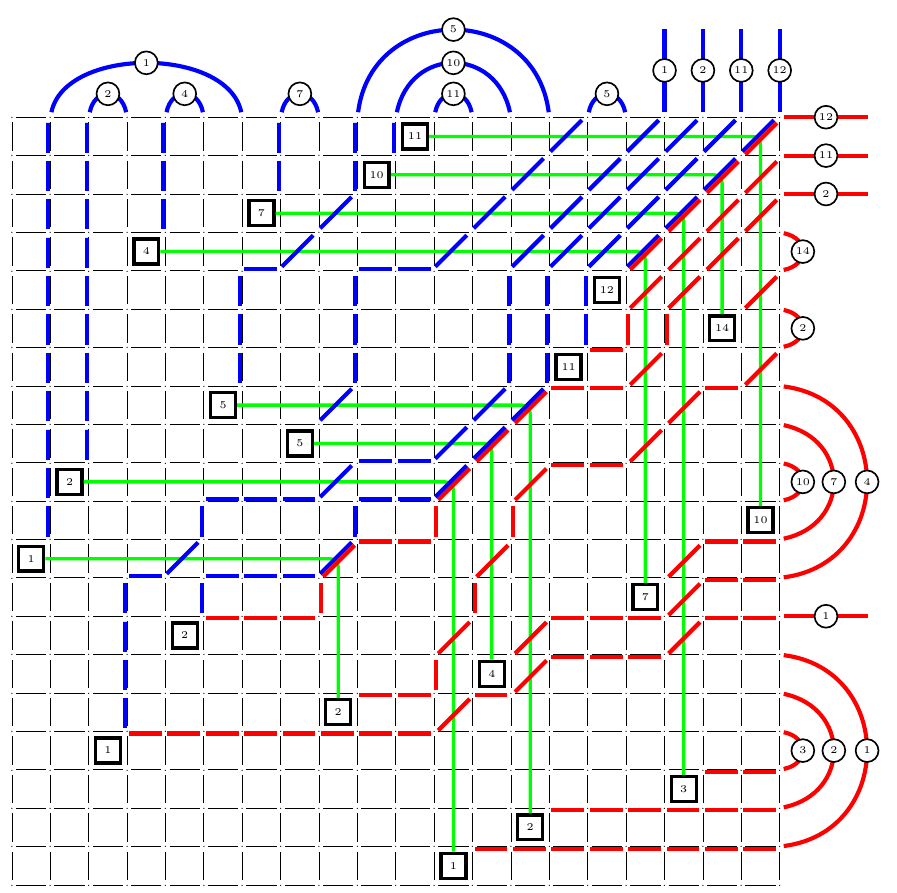}
  \bigskip
  \bigskip

  \caption[An example of the bijection between arc-diagrams and permutations]{The pair of half Okada arc-diagrams
  (red and blue) corresponding to the permutation $\pi =(9, 11, 4, 17, 7, 13, 18, 12, 5, 19, 20, 1, 6, 2, 14, 16, 8, 3, 15, 10)$.  }
  \label{fig.compute-arc-large}
\end{figure}
  \bigskip
  \bigskip
\pagebreak
\section{Table of Gram determinants}\label{app:table-Gram-det}
The following table lists the Fibonacci words of
rank $n$ (for $n=2,3,4,5$), their associated Fibonacci sets, the dimensions of the
associated simple modules, and their Gram determinants.
\medskip

\[\begin{array}{ccll}
  11  & \{1, 2\}    & 1 & 1 \\
  2    & \varnothing & 1 & x_{1}\\
  \hline
  111 & \{1,2,3\} & 1 & 1 \\
  12    & \{3\}     & 1 & x_{1} \\
  21    & \{1\}     & 2 & y_{1}(x_{1} x_{2} - y_{1})\\
  \hline
  1111 & \{1,2,3,4\} & 1 & 1\\
  112    & \{3,4\}     & 1 &x_{1}\\
  121    & \{1,4\}     & 2 & y_{1} (x_{1} x_{2} - y_{1}) \\
  211    & \{1,2\}     & 3 &
     y_{1} y_{2}^{2} (x_{1} x_{2} x_{3} - x_{3} y_{1} + x_{1} y_{2}) \\
  22    & \varnothing & 3 &
     y_{2}^{2}x_{1}^{3}(x_{1} x_{2} - y_{1})(x_{3} y_{1} - x_{1} y_{2}) \\
  \hline
  11111 & \{1,2,3,4,5\} & 1 & 1 \\
  1112 & \{    3,4,5\} & 1 & x_1 \\
  1121 & \{1,    4,5\} & 2 & y_1(x_1x_2 - y_1) \\
  1211 & \{1,2,    5\} & 3 & y_1y_2^2(x_1x_2x_3 - x_3y_1 - x_1y_2)\\
  122   & \{        5\} & 3 & y_2^2x_1^3(x_1x_2 - y_1)(x_3y_1 - x_1y_2) \\
  2111 & \{1,2,3    \} & 4 &
    y_1y_2^2y_3^3(x_1x_2x_3x_4 - x_3x_4y_1 - x_1x_4y_2 - x_1x_2y_3 + y_1y_3) \\
  212    & \{    3    \} & 4 &
    y_2^2y_3^3x_1^4(x_1x_2 - y_1)(x_3x_4y_1 - x_1x_4y_2 - y_1y_3) \\
  221    & \{1        \} & 8 & y_1^5y_2^2y_3^6(x_3y_1 - x_1y_2)(x_4y_2 - x_2y_3)^2(x_1x_2 - y_1)^5 (x_1x_2x_3 - x_3y_1 - x_1y_2) \\
%  & & & \hspace{4cm}(x_1x_2x_3 - x_3y_1 - x_1y_2)
\end{array}\]
\medskip

\end{document}